\documentclass[11pt]{article}
\usepackage{xcolor}   
\newcommand{\mathbbm}[1]{\mathbf{#1}}
\usepackage{paralist}
\usepackage{graphics} 
\usepackage{epsfig} 
\usepackage{epstopdf}
\usepackage[colorlinks=true]{hyperref}
\hypersetup{urlcolor=blue, citecolor=blue}
\usepackage[toc,page]{appendix}
\usepackage{chngcntr}

\usepackage{titlesec}
\usepackage[margin=1in]{geometry}

\usepackage{graphicx,xcolor}
\usepackage{amsthm,amsfonts,euscript,amsmath,comment,slashed,here,todonotes,amssymb, mathrsfs}
\usepackage[affil-it]{authblk}

\newtheorem{theorem}{Theorem}[section]
\newtheorem{lemma}[theorem]{Lemma}
\newtheorem{proposition}{Proposition}[section]
\newtheorem{corollary}{Corollary}[section]

\theoremstyle{definition}
\newtheorem{definition}[theorem]{Definition}
\newtheorem{example}[theorem]{Example}

\theoremstyle{remark}
\newtheorem{remark}[theorem]{Remark}

\numberwithin{equation}{section}

\newcommand\be{\begin{equation*}}
	\newcommand\ee{\end{equation*}}
\newcommand\bea{\begin{eqnarray}}
	\newcommand\eea{\end{eqnarray}}
\newcommand\bi{\begin{itemize}}
	\newcommand\ei{\end{itemize}}
\newcommand\ben{\begin{enumerate}}
	\newcommand\bena{\begin{enumerate}[(a)]}
		\newcommand\een{\end{enumerate}}
	
	\newcommand\bp{\begin{proof}}
		\newcommand\ep{\end{proof}}

	\newcommand{\R}{\ensuremath{\mathbb{R}}}

	\newcommand{\N}{\mathbb{N}}

	\allowdisplaybreaks
	
	\title{Local separation profiles for the stationary Prandtl system}
	\author{Ky Giao Duong\footnote{Department of Mathematics, NYU Abu-Dhabi, \url{gkd1@nyu.edu}} \quad Tej-Eddine Ghoul\footnote{Department of Mathematics, NYU Abu-Dhabi, \url{teg6@nyu.edu}}  \qquad Sameer Iyer\footnote{Department of Mathematics, University of California, Davis, Davis, CA 95616, USA, \url{sameer@math.ucdavis.edu}} \qquad Yasunori Maekawa\footnote{Department of Mathematics, Graduate School of Science, Kyoto University, Kyoto 606-8502, Japan, \url{maekawa.yasunori.3n@kyoto-u.ac.jp}}}
	
	\newcommand{\Bstab}{\mathfrak B_1}
	\newcommand{\cscale}{c_{\mathrm{sc}}}
\begin{document}
		
		\maketitle
		
		\begin{abstract}
			We construct solutions to the two-dimensional stationary Prandtl system under a constant adverse pressure gradient exhibiting a family of distinct separation laws. More precisely, for every integer $\ell\ge1$, we construct smooth initial data for which separation occurs at a finite location $x=x^\ast$, with
			\[
			\partial_y u(x,0)
			=\mathbf C_\ell(x^\ast-x)^{\ell/2}\bigl(1+o(1)\bigr),
			\qquad x\uparrow x^\ast,
			\]
			where $\mathbf C_\ell>0$ depends on the initial data.
			
			We determine the local asymptotic structure of the solution near separation through matched inner and outer expansions based on linearization around the ground state in von Mises variables. This description addresses Open Problem~5 posed by Oleinik and Samokhin~\cite{OleinikSamok-book-99}. We also prove that the square-root separation regime, corresponding to $\ell=1$, is dynamically stable under small admissible perturbations of the initial data in a suitable topology.
		\end{abstract}
		
		\tableofcontents
		
		\section{Introduction}
		A central problem in mathematical fluid mechanics is to justify the inviscid limit, namely the convergence of Navier--Stokes solutions to Euler solutions as the viscosity $\nu$ tends to zero. In the absence of physical boundaries, this convergence is classical for smooth data on the lifespan of the Euler solution; see Swann~\cite{Swann1971} and Kato~\cite{Kato1972}. Important developments include quantitative inviscid-limit estimates for two-dimensional flows with bounded vorticity~\cite{Chemin1996,Abidi_Danchin2004} and convergence results for vortex-patch solutions~\cite{ConstantinWu1995}. In the presence of a stationary solid wall, however, the no-slip condition for the viscous flow generally conflicts with the nonzero tangential velocity permitted by Euler's impermeability condition. This mismatch produces large velocity gradients near the wall, making the formation and stability of viscous boundary layers central to the analysis of the inviscid limit; see Maekawa--Mazzucato~\cite{MaekawaMazzucato2018}.
		
		In 1904, Prandtl \cite{Prandtl1904} proposed to resolve this mismatch by matching an outer Euler flow to a thin viscous layer of formal thickness $O(\sqrt\nu)$. Rescaling the normal variable yields the Prandtl equations, which describe the leading boundary-layer dynamics. Rigorous justifications of this expansion include the analytic setting \cite{Sam_Caf1998a,Sam_Cafl98b,W-W-Z17} and suitable data whose initial vorticity vanishes near the wall \cite{Maekawa2014}. In the time-dependent problem, Kato's criterion \cite{Kato1984} characterizes convergence to a smooth Euler flow in the energy norm through vanishing viscous dissipation in a thinner $O(\nu)$ boundary strip. We refer to \cite{MaekawaMazzucato2018} for a broader account of the inviscid limit and boundary-layer theory.
		
		For steady flows, Prandtl expansions have been justified over a short moving plate by Guo and Nguyen \cite{GuoNguyen2017}, and for a class of scaled layers, including the Blasius profile, at a motionless wall by Guo and Iyer \cite{GuoIyer2023}. An adverse pressure gradient presents a different difficulty: it can decelerate the flow until the wall shear vanishes, leading to boundary-layer separation. Understanding the local structure of the Prandtl solution at this degeneracy is the subject of the present paper. We study the two-dimensional stationary system
		\begin{equation}\label{equa-u-P-E}
			\left\{
			\begin{aligned}
				u\partial_xu+v\partial_yu-\partial_y^2u&=-\dfrac{dp_E}{dx},\\
				\partial_xu+\partial_yv&=0,\\
				u|_{y=0}=v|_{y=0}&=0,\\
				\lim_{y\to\infty}u(x,y)&=u_E(x),\\
				u|_{x=0}&=u_{\rm in}(y).
			\end{aligned}
			\right.
		\end{equation}
		where $x$ is the tangential variable, $y\ge0$ is the distance to the wall, and the outer Euler flow satisfies Bernoulli's law
		\begin{equation}\label{equa-u_E-P}
			u_Eu_E'+p_E'=0.
		\end{equation}
		Throughout the paper, we fix $x_0>0$ and take
		\begin{equation}\label{defi-u_E}
			u_E(x)=\sqrt{2(x_0-x)},\qquad 0\le x<x_0.
		\end{equation}
		Thus $p_E'=1$ by \eqref{equa-u_E-P}, and \eqref{equa-u-P-E} becomes
		\begin{equation}\label{equa-u-P-E-}
			\left\{
			\begin{aligned}
				u\partial_xu+v\partial_yu-\partial_y^2u&=-1,\\
				\partial_xu+\partial_yv&=0,\\
				u|_{y=0}=v|_{y=0}&=0,\\
				\lim_{y\to\infty}u(x,y)&=u_E(x),\\
				u|_{x=0}&=u_{\rm in}(y).
			\end{aligned}
			\right.
		\end{equation}
		Oleinik's local theory applies to positive profiles with nondegenerate wall shear and the boundary compatibility conditions; see \cite{Oleinik1963,Oleinik1966,Oleinik1966b,OleinikSamok-book-99} and the class $\mathcal K_\alpha$ in \eqref{Olenik class}. We study solutions for which the wall shear remains positive on $0\le x<x^*<x_0$ and satisfies $\partial_yu(x,0)\to0$ as $x\uparrow x^*$.  We call $x^*$ the separation point; thus separation occurs while the outer Euler trace remains positive.
		
		Recent advances in steady boundary-layer theory include global regularity up to the wall under favorable pressure gradients, established by Wang--Zhang~\cite{WangZhang2021Regularity}, and global stability of Prandtl expansions for a class including the Blasius profiles, proved by Iyer--Masmoudi~\cite{IyerMasmoudi2026}. Under an adverse pressure gradient, Goldstein's formal expansion \cite{Goldstein1948}, developed further by Stewartson~\cite{Stewartson1958}, predicts a singular local structure at separation. Its rigorous description is the subject of Open Problem~5 in Oleinik--Samokhin~\cite[p.~501]{OleinikSamok-book-99}. Matsui--Shirota~\cite{MatsuiShirota1984,MatsuiShirota1985} gave early rigorous existence results for separation. Their separation criterion allows $\partial_yu$ to vanish along a sequence approaching the wall, and should be distinguished from a sharp pointwise wall-shear asymptotic. Dalibard--Masmoudi~\cite{DalibardMasmoudi-IHES-19} established the sharp wall-shear law $\partial_yu(x,0)\sim C\sqrt{x^*-x}$ for well-prepared monotone $C^7$ data, together with a higher-order local approximation and quantitative remainder estimates. Shen--Wang--Zhang~\cite{Shen_Wang_Zhang2021} proved separation for a broader class of monotone Oleinik data under a smallness condition on the initial slope; their result obtains a space-time curve $y_\ast(x)$ such that $\partial_yu(x,y_\ast(x))\lesssim(x^*-x)^{1/4}$. Under an additional curvature bound, they also established two-sided estimates at selected points approaching separation; these do not prescribe a sharp wall-shear asymptotic. Iyer--Masmoudi~\cite{IyerMasmoudi2022Reversal} developed a mixed-type framework for reversed and recirculating Prandtl flows, see also the related work of Dalibard-Marbach-Rax, \cite{DMR26}. Related separation problems have been studied for compressible boundary layers by Zou--Wang~\cite{ZouWang2023} and for a magnetically damped model by Chen--Zhang--Luo~\cite{ChenZhangLuo2026}. For the classical system, these developments motivate the construction of distinct sharp separation laws and a stability theory for the square-root regime, together with quantitative descriptions of the associated local profiles.
		
		\bigskip
		We address Open Problem~5 in Oleinik--Samokhin~\cite[p.~501]{OleinikSamok-book-99} by constructing separating solutions and determining their local asymptotic structure. For $0<\alpha\le1$, let $C_b^{3,\alpha}(\R_+)$ denote the usual bounded H\"older space and set
		\[
		\mathcal O_\alpha=\{h\in C_b^{3,\alpha}(\R_+):h(0)=0,\ \lim_{y\to\infty}h(y)=0,\ h''(y)=O(y^2)\text{ as }y\downarrow0\}.
		\]
		The perturbation norm is
		\begin{equation}\label{definition-norm-E}
			\|h\|_{\mathcal E}:=
			\sup_{0<y\le1}\left(\frac{|h|}{y}+|h'|+\frac{|h''|}{y^2}\right)
			+\sup_{y\ge1}\sum_{k=0}^2|\partial_y^kh|+\|h\|_{C_b^{3,\alpha}}.
		\end{equation}
		Here $\mathcal K_1$ denotes the class \eqref{Olenik class} with $\alpha=1$. We define the universal separation profile by
		\begin{equation}\label{defi-pro-Q-sep}
			\mathcal Q_{\rm sep}(Y):=4\sqrt{Q\!\left(\mathcal Y_Q^{-1}(Y/2)\right)},
			\qquad \mathcal Y_Q(\zeta):=\int_0^\zeta\frac{d\eta}{\sqrt{Q(\eta)}},
		\end{equation}
		where $Q$ is the ground state defined in \eqref{ground-state-Q}.
		
		\begin{theorem}[Stable square-root separation]\label{theo-main-separate}
			For every $x_0>0$, there exists $\widehat u_{\rm in,1}\in C^\infty(\R_+)\cap\mathcal K_1$ such that the solution $u_1$ of \eqref{equa-u-P-E-} separates at $\widehat x_1^*\in(0,x_0/2)$. There exist $\widehat{\mathbf C}_1>0$ and a sufficiently small $\eta>0$ such that, for every fixed sufficiently small $y_0>0$, with $t=\widehat x_1^*-x$,
			\begin{align}
				u_1(x,y)&=\widehat{\mathbf C}_1^{\,2}t\,
				\mathcal Q_{\rm sep}\!\left(\frac{y}{\widehat{\mathbf C}_1t^{1/2}}\right)
				+\widetilde u_1(x,y),\label{decom-u_1-theo}\\*
				|\widetilde u_1(x,y)|&\lesssim t^{1/2+\eta}y+y^3,
				\qquad 0\le x\le\widehat x_1^*,\quad 0\le y\le y_0.\label{pointwise-physical-remainder-1}
			\end{align}
			For every $\alpha\in(0,1]$, there exists $\epsilon_0>0$ such that, whenever $h\in\mathcal O_\alpha$, $\|h\|_{\mathcal E}<\epsilon_0$, and $u_{\rm in}:=\widehat u_{\rm in,1}+h\in\mathcal K_\alpha$, the corresponding solution separates at $x^*[u_{\rm in}]\in(0,x_0)$ and satisfies \eqref{decom-u_1-theo}--\eqref{pointwise-physical-remainder-1} with $\widehat x_1^*$ and $\widehat{\mathbf C}_1$ replaced by $x^*[u_{\rm in}]$ and $\mathbf C[u_{\rm in}]>0$, respectively, for every fixed sufficiently small $y_0>0$. Moreover, $(x^*[u_{\rm in}],\mathbf C[u_{\rm in}])\to(\widehat x_1^*,\widehat{\mathbf C}_1)$ as $\|h\|_{\mathcal E}\to0$.
		\end{theorem}
		
		\medskip
		Extending this construction yields a countable family of distinct separation regimes.
		\begin{theorem}[A countable family of separation laws]\label{theo-unstab}
			For every integer $\ell\ge2$ and every $x_0>0$, there exists $\widehat u_{{\rm in},\ell}\in C_b^{3,1}(\R_+)\cap C^\infty(\R_+)\cap\mathcal K_1$ such that the solution $u_\ell$ of \eqref{equa-u-P-E-} separates at $x_\ell^*\in(0,x_0)$. There exist $\mathbf C_\ell>0$ and a sufficiently small $\eta>0$ such that, for every fixed sufficiently small $y_0>0$, with $t=x_\ell^*-x$,
			\begin{align}
				u_\ell(x,y)&=\mathbf C_\ell^{\,2}t^\ell\,
				\mathcal Q_{\rm sep}\!\left(\frac{y}{\mathbf C_\ell t^{\ell/2}}\right)
				+\widetilde u_\ell(x,y),\label{decom-u-ell-theo}\\*
				|\widetilde u_\ell(x,y)|&\lesssim t^{\ell/2+(2\ell-1)\eta}y+y^{2\ell+1},
				\qquad 0\le x\le x_\ell^*,\quad 0\le y\le y_0.\label{pointwise-physical-remainder}
			\end{align}
		\end{theorem}
		
		\paragraph{Convention for the remainder bounds.}
		In the remainder estimates above, $y_0$ is fixed independently of the evolution time $s$ and preparation time $s_0$. The implicit constants may depend on the initial datum (hence on $s_0$) and $y_0$, but not on $x$ or $y$. Endpoint values are taken by continuous extension; the rescaled profile is interpreted as $y^2/2$ at separation.
		
		\paragraph{Relation to the classical separation expansion.}
		Motivated by Hartree's numerical evidence of a separation singularity~\cite{Hartree1939}, Goldstein~\cite{Goldstein1948} considered the formal expansions, with $t=x^*-x$,
		\begin{equation}\label{formal-expansion-G_S}
			u(x^*,y)\sim\tfrac12y^2+\sum_{k\ge3}c_ky^k,
			\qquad
			\partial_yu(x,0)\sim\sum_{k\ge1}\alpha_k t^{(k+1)/4}.
		\end{equation}
		Stewartson~\cite{Stewartson1958} showed that the pure-power ansatz generally requires logarithmic corrections, beginning at order $t^{3/2}\log t$ in the nondegenerate wall-shear expansion. His analysis also distinguishes the inner expansion from the terminal profile at fixed $y>0$.
		
		By \eqref{ground-state-Q}, the profile \eqref{defi-pro-Q-sep} satisfies $\mathcal Q_{\rm sep}(Y)=Y+Y^2/2$. Our bounds \eqref{pointwise-physical-remainder-1} and \eqref{pointwise-physical-remainder} imply, with $t=x_\ell^*-x$,
		\begin{equation}\label{our-expansion}
			u_\ell(x_\ell^*,y)=\tfrac12y^2+O(y^{2\ell+1}),
			\qquad
			\partial_yu_\ell(x,0)=\mathbf C_\ell t^{\ell/2}\bigl(1+O(t^{(2\ell-1)\eta})\bigr).
		\end{equation}
		The terminal estimate holds on every fixed sufficiently small interval $[0,y_0]$, with constants depending on the initial datum and $y_0$; the reference stable solution uses the hatted parameters. The~first expansion in \eqref{our-expansion} recovers the leading term of the first expansion in \eqref{formal-expansion-G_S}. Provided that Goldstein's expansion is valid through order $t^{\ell/2}$, the constructed solution realizes the coefficient pattern $\alpha_k=0$ for $1\le k<2\ell-1$ and $\alpha_{2\ell-1}=\mathbf C_\ell$. Leading laws associated with the even-indexed coefficients $\alpha_{2j}$ are not covered by the present analysis and may require a different mechanism.
		
		For $\ell=1$, we recover the square-root law proved by Dalibard--Masmoudi~\cite{DalibardMasmoudi-IHES-19} for well-prepared globally monotone $C^7$ data with high-order compatibility conditions. For $\ell\ge2$, Theorem~\ref{theo-unstab} constructs a countable family of higher-order laws by finite-dimensional tuning.
		
		The second part of Theorem~\ref{theo-main-separate} establishes nonlinear stability in the topology \eqref{definition-norm-E}, with continuous dependence of the separation point and leading coefficient on the initial datum. Perturbations require only $C_b^{3,\alpha}$ regularity, $0<\alpha\le1$, subject to \eqref{Olenik class}. These conditions retain positivity, nondegenerate initial wall shear, and compatibility, but allow profiles that are not globally monotone.
		
		Varying the preparation parameter $s_0$ allows the separation point to be prescribed near the inlet: for each fixed $\ell\ge1$ and $x_0>0$, there is $\widehat{\mathfrak x}_\ell(x_0)\in(0,x_0)$ such that every $x^*\in(0,\widehat{\mathfrak x}_\ell(x_0))$ is attained. For $\ell\ge2$, this follows from $x_\ell^*=e^{-s_0}$ in \eqref{similarity_variable_ell_gener}. For $\ell=1$, continuity of the prepared data in the norm \eqref{definition-norm-E} and continuous dependence of the separation point give a continuous positive map $s_0\mapsto x_1^*(s_0)$ tending to zero. The intermediate value theorem then gives the assertion; see Remark~\ref{remark-prescribed-stable-separation}. Stability preserves the square-root regime while allowing the separation point to vary.
		
		\paragraph{Novelties.}
		The formal defect hierarchy determines the candidate separation rates and profiles (Section~\ref{Section-formal-approach}).  Matching this hierarchy with the spectral expansion around $Q_b$ (Proposition~\ref{diagonalisation-L-b}) yields the adapted decomposition~\eqref{decom-varepvarphi-i-ell}: the parameter $b$ is selected by a geometric orthogonality condition, and a second modulation condition fixes the distinguished amplitude in the stable regime. Wall compatibility is imposed separately in the preparation of the data, while a topological shooting argument selects the unstable modes for $\ell\ge2$.  Pointwise estimates replace high-order weighted Sobolev estimates for the inner remainder: we construct scaled ground-state barriers, following comparison ideas from~\cite{Biernat_Seki2020,BDGJDE2025}, to obtain a dominant zeroth-order contribution, and freeze the diffusion coefficient at the evolving solution to avoid singular comparison coefficients. Semigroup estimates in the intermediate region and eigenfunction barriers in the outer region are coupled through the modulation equations to close the pointwise bootstrap (Section~\ref{subsection-sketch-stable-regime}). The quasilinear term \eqref{defi-B-quadratic-appendix-} is controlled through the residual
		\begin{equation}\label{defi-mathcal-F}
			\mathcal F(s,\xi):=\sqrt{\vartheta(s,\xi)}\,\partial_\xi^2\vartheta(s,\xi)-2,
		\end{equation}
		which corresponds to $2\partial_y^2u-2$ under \eqref{chang-varialbe}. Together with perturbation bounds and positivity, its estimates control $\partial_\xi^2\varepsilon$, propagate wall compatibility, and ensure continuation. The common eigenbasis and spectral shift \eqref{notation-Lhat-b-beta} couple the residual and perturbation estimates, closing the nonlinear bootstrap.
		
		The argument combines modulation with pointwise comparison estimates for a quasilinear operator that degenerates at the wall.  It is related to modulation--energy constructions for Type~II dynamics~\cite{Merle_Raphael2004,Merle_Raphael2005,Raphael_Schweyer2013,Raphael_Schweyer2014,Ghoul_Ibrahim_Nguyen2019,Ghoul_Ibrahim_Nguyen2018,Collot-et-alJAMS-2020,CGMN_CPAM_22} and to the modulation--comparison method used for Prandtl separation in~\cite{DalibardMasmoudi-IHES-19}.
		
		\begin{remark}[Comment on codimensional stability in Theorem~\ref{theo-unstab}]\label{remark-expected-codimension}
			The solutions constructed in Theorem~\ref{theo-unstab} are expected to lie on a finite-codimensional stable set.  A rigorous characterization of this set remains open because the nondegeneracy of the modulation Jacobian has not been established.  Related finite-codimensional constructions appear in~\cite{Merle_Raphael_Rodnianski2015,Raphael_Schweyer2014,Ghoul_Ibrahim_Nguyen2018}; stability and invariant-manifold results in related dispersive and parabolic problems include~\cite{Ghoul_Ibrahim_Nguyen2019,Collot2018Manifolds,Martel_Merle_Nakanishi_Raphael2016}.
			
			We now formulate the expected codimensional-stability statement. Fix $\ell\ge2$ and $\alpha\in(0,1]$, and define
			\[
			\mathcal X_\ell
			:=
			\widehat u_{{\rm in},\ell}+\mathcal O_\alpha
			\]
			as the affine perturbation space of the reference datum, endowed with the topology induced by the norm $\|\cdot\|_{\mathcal E}$. We expect that there exist a neighborhood $\mathcal N_\ell$ of $\widehat u_{{\rm in},\ell}$ in $\mathcal X_\ell$ and a local submanifold
			\[
			\mathcal M_\ell\subset\mathcal N_\ell
			\]
			of finite codimension, with $\widehat u_{{\rm in},\ell}\in\mathcal M_\ell$, such that every admissible initial datum
			\[
			u_{\rm in}\in\mathcal M_\ell
			\]
			generates a solution that separates according to the regime described in Theorem~\ref{theo-unstab}. The stability problem is therefore to identify the local geometry and regularity of $\mathcal M_\ell$.
			
			Motivated by the works cited above, we expect $\mathcal M_\ell$ to have codimension $\ell-1$. More precisely, there should exist an $(\ell-1)$-dimensional subspace $E_\ell^{\mathrm u}\subset\mathcal O_\alpha$ and a closed complementary subspace $\mathcal Z_\ell\subset\mathcal O_\alpha$ such that
			\[
			\mathcal O_\alpha
			=
			\mathcal Z_\ell\oplus E_\ell^{\mathrm u},
			\qquad
			\dim E_\ell^{\mathrm u}=\ell-1.
			\]
			Locally, the manifold should be represented as
			\[
			\mathcal M_\ell
			=
			\left\{
			\widehat u_{{\rm in},\ell}
			+h+\Phi_\ell(h):
			h\in\mathcal Z_\ell,\ \|h\|_{\mathcal E}<r
			\right\},
			\]
			where
			\[
			\Phi_\ell:\mathcal Z_\ell\longrightarrow E_\ell^{\mathrm u},
			\qquad
			\Phi_\ell(0)=0.
			\]
			For each sufficiently small $h\in\mathcal Z_\ell$, the correction $\Phi_\ell(h)$ would be selected so that the corresponding modulated solution remains trapped, suppressing the growth of the $\ell-1$ unstable components. The regularity of $\Phi_\ell$ would determine the regularity of this graph. Lipschitz regularity is proved for the energy-supercritical semilinear wave equation in~\cite{Collot2018Manifolds}, while a $C^1$ threshold manifold is constructed for the critical gKdV equation in~\cite{Martel_Merle_Nakanishi_Raphael2016}.
			\medskip
			
			The expected codimension-$\ell-1$ statement is not proved here.  For $\ell\ge2$, the leading contribution to the Jacobian governing the simultaneous propagation of the geometric and compatibility constraints cancels; the available estimates therefore do not establish its nondegeneracy.  The present argument fixes $\beta=1$ and uses $\ell$ shooting parameters.  Reducing this number to $\ell-1$, and proving continuity or Lipschitz regularity of the resulting selection map, requires a separate analysis, including uniqueness of the trapped correction in the prescribed modulation class; see Section~\ref{subsection-sketch-unstable-regime}.
		\end{remark}
		
		\newpage
		\subsection*{Notation}\label{subsec-notation}
		We use the following conventions.
		\paragraph{General conventions.}
		For a spatial variable $\zeta$,
		\begin{equation}\label{Lambda:op}
			\Lambda_\zeta f
			:=
			f-\frac34\zeta\partial_\zeta f.
		\end{equation}
		We omit the subscript of $\Lambda$ when clear and set
		\begin{align}
			\langle\xi\rangle&:=\sqrt{1+\xi^2},\label{notation-japanese-bracket}\\*
			\int_1^\xi h(\eta)\,d\eta&=-\int_\xi^1 h(\eta)\,d\eta, \qquad 0<\xi<1.\notag
		\end{align}
		$\mathbf 1_E$ and $\mathbbm 1_E$ denote indicators; $A\lesssim B$ means $A\le CB$ with $C>0$ and the stated parameter dependence.
		
		\paragraph{Ground state and rescaled profiles.}
		Let $Q>0$ satisfy (Lemma~\ref{lemma-Q})
		\begin{equation}\label{ground-state-Q}
			\sqrt Q\,Q''=2,\qquad Q(0)=0,\qquad Q'(0)=1.
		\end{equation}
		Set
		\begin{equation}\label{notation-constant-a}
			a:=\left(\frac{3}{\sqrt2}\right)^{2/3}.
		\end{equation}
		For $b>0$,
		\begin{equation}\label{defi-Q-b}
			Q_b(\xi)
			:=
			bQ\left(\frac{\xi}{b^{3/4}}\right).
		\end{equation}
		
		\paragraph{Reference scales and pointwise weights.}
		For integers $\ell\ge1$,
		\begin{align}
			\mathfrak I_\ell(s)&:=e^{(1-2\ell)s},\qquad \ell\ge1.\label{defi-I-ell}\\*
			\mathfrak f_\ell (\xi)&= \xi^{2/3} \langle \xi \rangle^{4/3+4\ell/3}.\label{defi-f-ell}
		\end{align}
		For $0<b<1$, $\eta>0$, and integers $\ell\ge1$, define the weighted norm on $(0,b^{-\eta})$ by
		\begin{equation}\label{definition-weighted-remainder-norm}
			\|f\|_{\mathfrak X_{b,\eta}^{(\ell)}}
			:=b^{1/4}\left\|\frac{f}{\xi}\right\|_{L^\infty(0,b^{3/4})}
			+\left\|\frac{f}{\xi^{2/3}\langle\xi\rangle^{4/3+4\ell/3}}\right\|_{L^\infty(b^{3/4},b^{-\eta})}.
		\end{equation}
		Write $\mathfrak X_{b,\eta}:=\mathfrak X_{b,\eta}^{(1)}$.
		
		For variable $\beta$ (see \eqref{relation-beta-mu-s}),
		\begin{equation}\label{stable-reference-scale}
			\Bstab(s)=e^{-s_0}\exp\left(-\int_{s_0}^s\beta(\sigma)\,d\sigma\right)
			=e^{-s_0}\frac{\mu(s)}{\mu(s_0)}.
		\end{equation}
		These reference scales satisfy $\mathfrak I_\ell'=-(2\ell-1)\mathfrak I_\ell$ and $\Bstab'=-\beta\Bstab$. The modulation equations determine $b$: for $\ell\ge2$, $b/\mathfrak I_\ell$ converges to a positive constant, while in the stable regime the corresponding ratio is $b/\Bstab$; see Sections~\ref{section-prove-theorem-stab} and~\ref{subsection-proof-existen-solu-ell=2}.
		Thresholds $A_i\ge1$, $\eta_i>0$, and $s_i\ge1$ are indexed by result;
		only the thresholds needed in a given statement are introduced.
		The preparation parameters are fixed first; their dependence may be suppressed
		in $A_i$ and $\eta_i$. Any auxiliary spatial cutoff $K$ is then fixed before
		choosing $A\ge A_i$, $0<\eta\le\eta_i$, and finally $s_0\ge s_i$.
		All further restrictions in a proof are absorbed into its thresholds, which
		are independent of the terminal time and uniform over the stated class of data.
		Joint applications require the maximum of the amplitude thresholds, the minimum
		of the exponent thresholds, and the maximum of the initial-time thresholds.
		These are evaluated at the parameters actually used; in particular, applications
		at $A$ and $2A$ require both sets of thresholds, without any monotonicity assumption.
		
		\paragraph{Linearized operators and semigroups.}
		For $b,\beta>0$,
		\begin{align}
			\mathscr L_{b,\beta}f&:= \sqrt{Q_b}\,\partial_\xi^2f + \frac{\partial_\xi^2Q_b}{2\sqrt{Q_b}}f + \beta\Lambda_\xi f,\label{defi-mathscr-L-}\\*
			\mathscr{L}_{0,\beta}f&:= a\xi^{2/3}\partial_\xi^2f + \frac{1}{a^2\xi^{4/3}}f + \beta\Lambda_\xi f.\label{defi-L_0_beta}
		\end{align}
		
		Residual operators:
		\begin{align}
			\widehat{\mathscr L}_{b,\beta}f&:=\sqrt{Q_b}\,\partial_\xi^2f+\frac{f}{Q_b} -\frac{3\beta}{4}\xi\partial_\xi f =\bigl(\mathscr L_{b,\beta}-\beta\operatorname{Id}\bigr)f,\label{notation-Lhat-b-beta}\\*
			\widehat{\mathscr L}_{0,1}f&:=a\xi^{2/3}\partial_\xi^2f+\frac{f}{a^2\xi^{4/3}} -\frac34\xi\partial_\xi f =\bigl(\mathscr L_{0,1}-\operatorname{Id}\bigr)f.\label{definition-Lhat-zero-one}
		\end{align}
		Here $\operatorname{Id}$ is the identity. Define (Section~\ref{section-spectral-L_0-beta})
		\begin{align}
			S_{\beta_0}(t)&:=e^{t\mathscr L_{0,\beta_0}}, \qquad \beta_0>0,\quad t>0,\label{defi-S-beta-0-t}\\*
			\widehat S_1(t)&:=e^{t\widehat{\mathscr L}_{0,1}}, \qquad t>0.\label{notation-semigroup-F}
		\end{align}
		
		In the inner variable $\zeta$, set
		\begin{equation}\label{defi-mathcal-beta}
			\mathcal L_{b,\beta}f
			:=\sqrt Q\,\partial_\zeta^2f+\frac{f}{Q}+b\beta\Lambda_\zeta f.
		\end{equation}
		For $\sigma>0$,
		\begin{align}
			Q_\sigma(\zeta)&:=\sigma Q\!\left(\frac{\zeta}{\sigma^{3/4}}\right),\label{definition-Q-sigma-general-p}\\*
			\mathscr H_\sigma f&:=\partial_\zeta^2f+\frac{f}{Q_\sigma^{3/2}}.\label{notation-H-sigma}
		\end{align}
		
		\paragraph{Weighted spaces and scalar products.}
		For $b,\beta>0$,
		\begin{align}
			P_b(\xi)&=\frac{3\beta}{4}\int_0^\xi \xi' Q_b(\xi')^{-1/2}\,d\xi'.\label{defi-P-b}\\*
			\rho_{b,\beta}(\xi)&=aQ_b(\xi)^{-1/2}e^{-P_b(\xi)}.\label{defi-rho-b-beta}
		\end{align}
		
		The $\beta$-dependence of $P_b$ is implicit. Set
		\begin{equation}\label{defi-rho-infty}
			\rho_{0,\beta}(\xi)
			:=
			\xi^{-2/3}
			\exp\left(
			-\frac{a^2\beta}{8}\xi^{4/3}
			\right).
		\end{equation}
		
		For real $f,g$ and $\rho>0$,
		\begin{equation*}
			\langle f,g\rangle_{L^2_\rho}
			:=\int_0^\infty f(\xi)g(\xi)\rho(\xi)\,d\xi,
			\qquad
			\|f\|_{L^2_\rho}^2:=\langle f,f\rangle_{L^2_\rho}.
		\end{equation*}
		Set
		\begin{equation}\label{definition-Hilbert-b-beta}
			L^2_{b, \beta} :=
			L^2\bigl((0,+\infty),  \rho_{b, \beta} (\xi) d\xi \bigr).
		\end{equation}
		
		Write $L^2_{b,\beta}=L^2_{\rho_{b,\beta}}$, $d\rho_{b,\beta}=\rho_{b,\beta}(\xi)\,d\xi$, and
		\begin{align}
			\|f\|_{H^1_{b,\beta}}^2&:= \int_0^\infty \left( |f|^2 + \sqrt{Q_b}\, |\partial_\xi f|^2 \right) d\rho_{b,\beta},\label{definition-H1-norm-b-beta}\\*
			H^1_{0;b,\beta}&:= \overline{C_c^\infty(0,+\infty)} ^{\,\|\cdot\|_{H^1_{b,\beta}}},\label{definition-H1-zero-b-beta}
		\end{align}
		
		\paragraph{Spectral coordinates and normalization.}
		Normalized eigenfunctions (Lemma~\ref{lemma-spectrum-L0-beta-xi}), $\beta>0$:
		\begin{equation}\label{defi-phi-beta-i}
			\phi_{\beta, n}(\xi)
			=
			\mathcal N_n\,
			a\xi^{2/3}
			L_n^{(1/4)}
			\left(
			\frac{\beta a^2\xi^{4/3}}{8}
			\right),
			\qquad n \in \N,
		\end{equation}
		
		Here $L_n^{(1/4)}$ is the degree-$n$ generalized Laguerre polynomial, and
		\begin{equation}\label{defi-N-beta-i}
			\mathcal N_n
			=
			\frac{2}{\sqrt3}a^{1/4}
			\left(\frac{\beta}{8}\right)^{5/8}
			\sqrt{
				\frac{n!}{\Gamma\!\left(n+\frac54\right)}
			}.
		\end{equation}
		
		We also write $\mathcal N_n(\beta)$. Matching coefficients:
		\begin{equation}\label{defi-c-i-j}
			c_{i,j}
			=
			\frac{2}{\sqrt3}a^{5/4}
			\left(\frac{1}{8}\right)^{5/8}
			\frac{(-1)^j}{(i-j)!}
			\sqrt{i!\Gamma\!\left(i+\frac54\right)},
			\qquad
			0\le j\le i.
		\end{equation}
		
		For the eigenfunctions $\varphi_{j,b,\beta}$ of Proposition~\ref{diagonalisation-L-b},
		\begin{align}
			\mathfrak F_{j,b,\beta}(\varepsilon)&:= \frac{ \left\langle \varepsilon, \varphi_{j,b,\beta} \right\rangle_{L^2_{\rho_{b,\beta}}} }{ \left\| \varphi_{j,b,\beta} \right\|_{L^2_{\rho_{b,\beta}}}^2 }.\label{defi-F-j}\\*
			\Pi_{j,b,\beta}(\varepsilon)&:= \mathfrak F_{j,b,\beta}(\varepsilon) \varphi_{j,b,\beta},\label{eigen-projection-Pi-j}
		\end{align}
		For the iterates $T_j$ of Lemma~\ref{lemma-T-i},
		\begin{equation}\label{notation-scaling-normalization}
			\cscale:=\frac{8}{a\Gamma(5/4)},
			\qquad T_0:=\cscale\Lambda Q.
		\end{equation}
		With \eqref{defi-c-i-j}, set
		\begin{align}
			\mathbf m_0(\beta)&= \frac{1}{\cscale c_{0,0}\beta^{5/8}},\label{defi-m-0}\\*
			\mathbf n_\ell(\beta)&:=\frac{2}{\cscale c_{\ell,0}\beta^{5/8}} =2\frac{c_{0,0}}{c_{\ell,0}}\mathbf m_0(\beta), \qquad \ell\geq1.\label{defi-n-ell-1}
		\end{align}
		$\mathfrak m_0$ denotes the lower bound in \eqref{definition-m0-F}. For $\beta=1$,
		\begin{equation}\label{notation-beta-one}
			\begin{gathered}
				\mathscr L_b:=\mathscr L_{b,1},\quad
				\rho_b:=\rho_{b,1},\quad
				\varphi_{j,b}:=\varphi_{j,b,1},\quad
				\mu_{j,b}:=\mu_{j,b,1},\\
				\mathfrak F_{j,b}:=\mathfrak F_{j,b,1},\quad
				\mathbf m_0:=\mathbf m_0(1),\quad
				\mathbf n_\ell:=\mathbf n_\ell(1).
			\end{gathered}
		\end{equation}
		
		\paragraph{Cutoff function.}
		Choose nonincreasing $\chi_0\in C_c^\infty(\mathbb R_+;[0,1])$ with
		\begin{equation}\label{defi-chi-0}
			\chi_0(z)=1
			\quad\text{for }0\leq z\leq 1,
			\qquad
			\chi_0(z)=0
			\quad\text{for }z\geq 2.
		\end{equation}
		
		\section{Formal   derivation of   separation dynamics}
		\label{Section-formal-approach}
		We derive the formal separation laws and near-wall profiles, then use spectral matching to identify the geometric decomposition underlying the construction. For a smooth solution $u$ of \eqref{equa-u-P-E-}, positive for $y>0$, introduce the von Mises variables
		\begin{equation}\label{defi-von-Mise}
			\psi(x,y):=\int_0^y u(x,y')\,dy',\qquad
			w(x,\psi):=u^2(x,y(x,\psi)).
		\end{equation}
		Then $w$ satisfies
		\begin{equation}\label{wvariable:1}
			\begin{cases}
				\partial_xw-\sqrt w\,\partial_\psi^2w=-2,\\
				w(x,0)=0,\quad w(0,\psi)=w_{\rm in}(\psi),\quad
				w(x,\psi)\to u_E^2(x)\quad(\psi\to\infty).
			\end{cases}
		\end{equation}
		Here $w_{\rm in}=u_{\rm in}^2\circ\psi_{\rm in}^{-1}$, where $\psi_{\rm in}(y)=\int_0^y u_{\rm in}(y')\,dy'$. The equation is invariant under the scaling
		\begin{equation*}
			w_\varpi(x,\psi)=\varpi^{4/3}w(\varpi^{-4/3}x,\varpi^{-1}\psi),\qquad \varpi>0.
		\end{equation*}
		Assume separation occurs at a finite point $\mathfrak X$. Since $\partial_yu(x,0)=\tfrac12\partial_\psi w(x,0)$, define the wall scale by $\lambda(x)^{1/4}=\partial_\psi w(x,0)$; thus $\lambda(x)\to0$ as $x\uparrow\mathfrak X$. Set
		\begin{equation}\label{renormalized-variables-lambda}
			w(x,\psi)=\lambda(x)v(\tau,\zeta),\qquad
			\zeta=\frac{\psi}{\lambda(x)^{3/4}},\qquad
			\tau=\tau_0+\int_0^x\frac{dx'}{\lambda(x')}.
		\end{equation}
		
		\subsection{Construction of an approximate solution}\label{subsection-formal-approximation}
		Motivated by the Type~II modulation schemes in~\cite{Raphael_Schweyer2013,Raphael_Schweyer2014,Merle_Raphael_Rodnianski2015}, we seek a slowly modulated orbit near $Q$, with uniformly small drift
		\begin{equation}\label{slowly-modulated-b-1}
			\mathfrak b_1=-\lambda_\tau/\lambda=-\lambda_x.
		\end{equation}
		Since $\lambda(x)\to0$ as $x\uparrow\mathfrak X$, boundedness of $\mathfrak b_1=-\lambda_x$ gives $0<\lambda(x)\lesssim\mathfrak X-x$, hence $\tau(x)\to+\infty$. Writing $\lambda(\tau)=\lambda(x(\tau))$, equation~\eqref{wvariable:1} becomes
		\begin{equation}\label{equa-v-b}
			\begin{cases}
				\partial_\tau v=\sqrt v\,\partial_\zeta^2v-2+\mathfrak b_1\Lambda_\zeta v,\\
				v(\tau,0)=0,\quad v(\tau_0,\zeta)=v_{\rm in}(\zeta),\quad
				\lambda v\to u_E^2(x(\tau))\quad(\zeta\to\infty).
			\end{cases}
		\end{equation}
		Here $v_{\rm in}(\zeta)=\lambda_0^{-1}w_{\rm in}(\lambda_0^{3/4}\zeta)$,
		$\lambda_0=(w'_{\rm in}(0))^4$, and $\Lambda_\zeta$ is defined in
		\eqref{Lambda:op}. The choice of $\lambda$ fixes the normalization
		$\partial_\zeta v(\tau,0)=\lambda(\tau)^{-1/4}\partial_\psi w(x(\tau),0)=1$.
		The formal inner construction retains the vanishing wall condition and this normalization; matching to the outer flow recovers the far-field condition.
		
		\medskip
		We focus on the regime $0<\mathfrak b_1(\tau)\ll1$, with $\mathfrak b_1(\tau)\to0$ as $\tau\to+\infty$. If $\mathfrak b_1\to\mathfrak b_*>0$, then $\lambda(x)\sim\mathfrak b_*(\mathfrak X-x)$, the self-similar scale with wall-shear exponent $1/4$. This exponent also appears in the upper bound of~\cite{Shen_Wang_Zhang2021}.
		
		\medskip
		\noindent 
		Fix a sufficiently large integer $L\geq2$. Linearizing about the ground state $Q$, the unique positive solution of \eqref{ground-state-Q}, we seek an approximation of the form
		\begin{equation*}
			v_{\rm ap}=Q+h,\qquad h=B+\mathcal S,\qquad
			B=\sum_{j=1}^L\mathfrak b_j\mathcal T_j,\qquad
			\mathcal S=\sum_{k=2}^{L+2}\mathcal S_k.
		\end{equation*}
		Successive inversion of the linearized operator generates the profiles $\mathcal T_j$, while the amplitudes $(\mathfrak b_1,\ldots,\mathfrak b_L)$ govern the modulation dynamics, compare with the modulation construction in~\cite{Merle_Raphael_Rodnianski2015}. To organize the higher-order corrections, set $\mathcal S_1=0$ and write
		\begin{equation*}
			\mathcal S_k=\mathfrak b_1\mathfrak b_{k-1}\mathcal P_{k-1}+\widetilde{\mathcal S}_k
			\quad(2\leq k\leq L+1),\qquad
			\mathcal S_{L+2}=\widetilde{\mathcal S}_{L+2}.
		\end{equation*}
		Substitution into \eqref{equa-v-b} gives the defect
		\begin{align*}
			\Psi_{\rm ap}
			&=\partial_\tau h-\mathscr H_0h-\mathfrak b_1\Lambda_\zeta Q-\mathfrak b_1\Lambda_\zeta h+\mathcal N(h),\\
			\mathcal N(h)
			&=-Q''\left(\sqrt{Q+h}-\sqrt Q-\frac{h}{2\sqrt Q}\right)
			-\left(\sqrt{Q+h}-\sqrt Q\right)\partial_\zeta^2h.
		\end{align*}
		Here $\mathscr H_0=\sqrt Q\,\partial_\zeta^2+1/Q$. For $|h|\ll Q$, Taylor expansion gives
		\begin{equation}\label{decomposition-N-formal}
			\mathcal N(h)
			=
			\sum_{k=2}^{L+2}\mathcal N_k
			+
			\mathcal N_{\geq L+3},
		\end{equation}
		where $\mathcal N_k$ collects the terms of weighted order $k$:
		\begin{align}
			\mathcal N_k
			={}&
			-
			Q''
			\sum_{m=2}^{k}
			\gamma_m Q^{\frac12-m}
			\sum_{\substack{
					i_1+\cdots+i_m=k\\
					i_\nu\geq1
			}}
			h_{i_1}\cdots h_{i_m}\notag\\*
			&-
			\sum_{m=1}^{k-1}
			\gamma_m Q^{\frac12-m}
			\sum_{\substack{
					i_1+\cdots+i_m+r=k\\
					i_\nu\geq1,\ r\geq1
			}}
			h_{i_1}\cdots h_{i_m}
			\partial_\zeta^2h_r.
			\label{definition-N-k-formal}
		\end{align}
		Here $\gamma_m=\binom{1/2}{m}$ and $h=\sum_{k=1}^{L+2}h_k$, with
		\begin{equation}\label{definition-h-k-formal}
			\begin{gathered}
				h_1=\mathfrak b_1\mathcal T_1,\qquad
				h_k=\mathfrak b_k\mathcal T_k+\mathcal S_k\quad(2\leq k\leq L),\\
				h_{L+1}=\mathcal S_{L+1},\qquad h_{L+2}=\mathcal S_{L+2}.
			\end{gathered}
		\end{equation}
		Assign weight $j$ to $\mathfrak b_j$ and weight one to $\partial_\tau$; $\mathcal N_{\geq L+3}$ collects terms of weighted order at least $L+3$. With $\mathfrak b_{L+1}=0$, define
		\begin{equation*}
			\operatorname{Mod}_j=(\mathfrak b_j)_\tau+(j-\tfrac12)(\mathfrak b_1\mathfrak b_j-\mathfrak b_{j+1}),\qquad 1\leq j\leq L.
		\end{equation*}
		Regrouping the defect yields
		\begin{align*}
			\Psi_{\rm ap}&={}\sum\nolimits_{j=1}^L\operatorname{Mod}_j\mathcal T_j
			-\mathfrak b_1[\mathscr H_0\mathcal T_1+\Lambda_\zeta Q]
			-\sum\nolimits_{j=2}^L \mathfrak b_j\bigl[\mathscr H_0\mathcal T_j-(j-\tfrac32)\mathcal T_{j-1}\bigr]\\*
			&\quad{}+\sum\nolimits_{j=1}^L \mathfrak b_1\mathfrak b_j\bigl[-\Lambda_\zeta\mathcal T_j-(j-\tfrac12)\mathcal T_j-\mathscr H_0\mathcal P_j\bigr]\\*
			&\quad{}+\sum\nolimits_{k=2}^{L+2}\bigl[-\mathscr H_0\widetilde{\mathcal S}_k+(\partial_\tau-\mathfrak b_1\Lambda_\zeta)\mathcal S_{k-1}+\mathcal N_k\bigr]
			+(\partial_\tau-\mathfrak b_1\Lambda_\zeta)\mathcal S_{L+2}+\mathcal N_{\geq L+3}.
		\end{align*}
		We cancel successive orders of $\Psi_{\rm ap}$ under the slow-modulation bounds
		\begin{equation}\label{asump-b-j}
			|\partial_\tau^m\mathfrak b_j|\leq C_{m,L}\mathfrak b_1^{j+m},\qquad
			m\geq0,\quad 1\leq j\leq L.
		\end{equation}
		These bounds, satisfied by the explicit orbits below, order the expansion in powers of $\mathfrak b_1$ and control its time derivatives. Using the inverse \eqref{inverse-H-0}, with zero value and derivative at the origin, set
		\begin{equation*}
			\left\{\begin{aligned}
				&\mathscr H_0\mathcal T_1=-\Lambda_\zeta Q,\qquad
				\mathscr H_0\mathcal T_{j+1}=\bigl(j-\tfrac12\bigr)\mathcal T_j,
				\qquad 1\leq j\leq L,\\
				&\mathcal P_j=\mathscr H_0^{-1}\left[-\Lambda_\zeta\mathcal T_j-\bigl(j-\tfrac12\bigr)\mathcal T_j\right],
				\qquad 1\leq j\leq L,\\
				&\widetilde{\mathcal S}_k=\mathscr H_0^{-1}\left[(\partial_\tau-\mathfrak b_1\Lambda_\zeta)\mathcal S_{k-1}+\mathcal N_k\right],
				\qquad 2\leq k\leq L+2.
			\end{aligned}\right.
		\end{equation*}
		Since $\mathcal N_k$ depends only on $h_1,\ldots,h_{k-1}$, this triangular recursion reduces the defect to
		\begin{equation}\label{simply-1}
			\Psi_{\rm ap}=\sum\nolimits_{j=1}^L\operatorname{Mod}_j\mathcal T_j
			+(\partial_\tau-\mathfrak b_1\Lambda_\zeta)\mathcal S_{L+2}+\mathcal N_{\geq L+3}.
		\end{equation}
		Under \eqref{asump-b-j}, induction in the Green representation \eqref{inverse-H-0}, using Lemma~\ref{lemma-Q}, gives the wall expansions
		\begin{alignat*}{2}
			\mathcal T_j(\zeta)&=\widetilde t_j\zeta^{1+3j/2}+O(\zeta^{3/2+3j/2}),
			&\qquad \widetilde t_j&=-\frac1{15}\prod\nolimits_{k=1}^{j-1}\frac{2(2k-1)}{(3k+3)(3k+5)},\\*
			\mathcal P_j(\zeta)&=\frac{j+2}{4(2j-1)}\mathcal T_{j+1}(\zeta)+O(\zeta^{3+3j/2}),
			&\qquad \widetilde{\mathcal S}_k&=O\bigl(\mathfrak b_1^k\zeta^{1+3k/2}\bigr),
		\end{alignat*}
		where $2\leq k\leq L+2$ in the last bound. 
		
		\medskip
		If the parameters satisfy \eqref{asump-b-j} and $\operatorname{Mod}_j=0$, $1\leq j\leq L$, then \eqref{simply-1} and the preceding estimates give $\Psi_{\rm ap}=O(\mathfrak b_1^{L+3}\zeta^{4+3L/2})$, uniformly on a fixed sufficiently small interval near the wall.
		
		The modulation system admits explicit orbits. Fix $1\leq\ell<L$ and seek $\mathfrak b_j^\ell=c_j\tau^{-j}$ for $1\leq j\leq\ell$, with $\mathfrak b_j^\ell=0$ for $j>\ell$. Substitution into $\operatorname{Mod}_j=0$ gives $-jc_j+(j-\tfrac12)(c_1c_j-c_{j+1})=0$. The terminal condition $c_{\ell+1}=0$, with $c_\ell\neq0$, yields
		\begin{equation}\label{for-mal-b-j-asymptotics}
			\mathfrak b_j^\ell(\tau)=\frac{c_j}{\tau^j},\qquad
			c_1=\frac{2\ell}{2\ell-1},\qquad
			c_{j+1}=-\frac{2(\ell-j)}{(2\ell-1)(2j-1)}c_j
			\quad(1\leq j<\ell).
		\end{equation}
		These orbits satisfy \eqref{asump-b-j}, with $0<\mathfrak b_1\ll1$ for $\tau\geq\tau_0\gg1$. Integrating $-\lambda_\tau/\lambda=c_1/\tau$ gives $\lambda=\kappa_\ell^{\rm form}\tau^{-2\ell/(2\ell-1)}$, $\kappa_\ell^{\rm form}>0$. Since $x_\tau=\lambda$ and $x(\tau_0)=0$, we obtain $\mathfrak X=\int_{\tau_0}^\infty\lambda\,d\tau<\infty$ and
		\begin{alignat*}{2}
			\mathfrak X-x&=(2\ell-1)\kappa_\ell^{\rm form}\tau^{-1/(2\ell-1)},
			&\qquad \lambda(x)&=\frac{(\mathfrak X-x)^{2\ell}}{(2\ell-1)^{2\ell}(\kappa_\ell^{\rm form})^{2\ell-1}},\\*
			\partial_yu(x,0)&=\tfrac12\lambda(x)^{1/4}=\mathbf C_\ell(\mathfrak X-x)^{\ell/2},
			&\qquad \mathbf C_\ell&=\frac1{2(2\ell-1)^{\ell/2}(\kappa_\ell^{\rm form})^{(2\ell-1)/4}}.
		\end{alignat*}
		\begin{remark}
			The orbit selects the separation exponent $\ell/2$, while $\kappa_\ell^{\rm form}$ determines its coefficient. Under $w\mapsto w_\varpi$, one has $\kappa_{\ell,\varpi}^{\rm form}=\varpi^{4/3}\kappa_\ell^{\rm form}$ and $\mathbf C_{\ell,\varpi}=\varpi^{-(2\ell-1)/3}\mathbf C_\ell$ for $\varpi>0$. Hence $\mathbf C_\ell$ depends on the reference scale fixed by the initial and matching data.
			
			The inner approximation satisfies the wall condition in \eqref{equa-v-b} and $\partial_\zeta v_{\rm ap}(\tau,0)=1$. The far-field condition $v\to u_E^2/\lambda$ requires a controlled transition to an outer profile.
		\end{remark}

		\subsection{The linearized problem}
		We now formulate the perturbation problem around a rescaled ground state. The rigorous construction determines the scale by modulation and recovers the wall shear a posteriori. Introduce a positive scale $\mu$, with $\mu(0)=\mu_0$, and set
		\begin{equation}\label{chang-varialbe}
			w(x,\psi)=\mu(x)\vartheta(s,\xi), \qquad \xi=\frac{\psi}{\mu(x)^{3/4}}, \qquad \frac{ds}{dx}=\frac1{\mu(x)}, \qquad s(0)=s_0\ge1,
		\end{equation}
		with $\mu(s)=\mu(x(s))$ and
		\begin{equation}\label{relation-beta-mu-s}
			\beta(s):=-\frac{\mu_s(s)}{\mu(s)}.
		\end{equation}
		Then \eqref{wvariable:1} becomes
		\begin{equation}\label{equa-vartheta}
			\begin{cases}
				\partial_s\vartheta=\sqrt\vartheta\,\partial_\xi^2\vartheta-2+\beta(s)\Lambda_\xi\vartheta,\\
				\vartheta(s,0)=0,\quad \vartheta(s_0,\xi)=\vartheta_{\rm in}(\xi),\quad
				\mu\vartheta\to u_E^2(x(s))\quad(\xi\to\infty).
			\end{cases}
		\end{equation}
		Here $U_\infty(s):=u_E^2(x(s))/\mu(s)$. The physical initial data are recovered from $w_{\rm in}(\psi)=\mu_0\vartheta_{\rm in}(\psi/\mu_0^{3/4})$, $y=\int_0^\psi w_{\rm in}(\psi')^{-1/2}\,d\psi'$, and $u_{\rm in}(y)=\sqrt{w_{\rm in}(\psi)}$.
		For $b=b(s)>0$, set $b_0=b(s_0)$ and decompose around the rescaled ground state \eqref{defi-Q-b}:
		\begin{equation}\label{defi-varepsilon-by-vartheta}
			\vartheta(s,\xi)
			=
			Q_{b(s)}(\xi)+\varepsilon(s,\xi).
		\end{equation}
		The parameter $b$ relates the inner and outer scales; formal matching with the wall scale $\lambda$ gives
		\begin{equation}\label{relation-lambda=mu-b}
			\lambda = \mu b.
		\end{equation}
		Substitution into \eqref{equa-vartheta} gives
		\begin{equation}\label{equa-varepsilon-appen-}
			\begin{cases}
				\partial_s\varepsilon=\mathscr L_{b,\beta}\varepsilon+\mathscr B_b(\varepsilon)+\Phi_{b,\beta},\\
				\varepsilon(s,0)=0,\quad \varepsilon(s_0,\xi)=\vartheta_{\rm in}(\xi)-Q_{b_0}(\xi),\quad
				\varepsilon+Q_b\to U_\infty(s)\quad(\xi\to\infty).
			\end{cases}
		\end{equation}
		Here $\mathscr L_{b,\beta}$ is defined in \eqref{defi-mathscr-L-}, and the nonlinear remainder is
		\begin{align}
			\mathscr B_b(\varepsilon)&:={}\sqrt{Q_b+\varepsilon}\,\partial_\xi^2(Q_b+\varepsilon)
			-\sqrt{Q_b}\,\partial_\xi^2Q_b-\sqrt{Q_b}\,\partial_\xi^2\varepsilon
			-\frac{\partial_\xi^2Q_b}{2\sqrt{Q_b}}\varepsilon
			\label{defi-B-quadratic-appendix-}\\*
			&={}\bigl(\sqrt{Q_b+\varepsilon}-\sqrt{Q_b}\bigr)\partial_\xi^2\varepsilon
			+\left(\sqrt{Q_b+\varepsilon}-\sqrt{Q_b}-\frac{\varepsilon}{2\sqrt{Q_b}}\right)\partial_\xi^2Q_b.\nonumber
		\end{align}
		The modulation term is
		\begin{equation}\label{Phi-simple-}
			\Phi_{b,\beta}
			:=
			\left(\beta-\frac{b_s}{b}\right)\Lambda_\xi Q_b.
		\end{equation}
		We seek $\beta(s)=\beta_0+o(1)$, with $\beta_0>0$. A constant $\beta$ fixes a self-similar reference frame, while the relative scale $b(s)$ encodes the non-self-similar dynamics.
		
		With $P_b$ and $\rho_{b,\beta}$ from \eqref{defi-P-b} and \eqref{defi-rho-b-beta}, the operator has divergence form
		\begin{equation*}
			\mathscr L_{b,\beta}f=\frac1{\rho_{b,\beta}}\partial_\xi
			\left(\rho_{b,\beta}\sqrt{Q_b}\,\partial_\xi f\right)
			+\left(\frac1{Q_b}+\beta\right)f.
		\end{equation*}
		As $b\downarrow0$ with $\beta$ fixed, the coefficients and weight converge locally away from zero to \eqref{defi-L_0_beta} and \eqref{defi-rho-infty}. Resolving the transition at $\xi\sim b^{3/4}$ requires matching the inner and outer expansions. We use the weighted space \eqref{definition-Hilbert-b-beta} and form domain \eqref{definition-H1-zero-b-beta}.
		
		\begin{proposition}[Perturbative diagonalization of $\mathscr L_{b,\beta}$]
			\label{diagonalisation-L-b}
			Fix an integer $\ell\geq1$ and $C_0>1$. There exists $b^*=b^*(\ell,C_0)\in(0,1)$ such that all the following conclusions hold for $0<b<b^*$ and $\beta\in[C_0^{-1},C_0]$. The operator $-\mathscr L_{b,\beta}$, initially defined on $C_c^\infty(0,\infty)$, is densely defined, symmetric, and bounded below in $L^2_{b,\beta}$. Its Friedrichs extension $A_{b,\beta}^{\rm F}$ has form domain $H^1_{0;b,\beta}$ and defines a self-adjoint realization $\mathscr L_{b,\beta}=-A_{b,\beta}^{\rm F}$ with compact resolvent.
			
			Its spectrum consists of simple eigenvalues
			\begin{equation}\label{ordered-spectrum-finite-b}
				\mu_{0,b,\beta}>\mu_{1,b,\beta}>\cdots>\mu_{i,b,\beta}>\cdots\longrightarrow-\infty.
			\end{equation}
			The corresponding real eigenfunctions $\varphi_{i,b,\beta}\in D(\mathscr L_{b,\beta})\subset H^1_{0;b,\beta}$ form an orthogonal basis of $L^2_{b,\beta}$, and $\varphi_{i,b,\beta}$ has exactly $i$ simple zeros in $(0,\infty)$. For $0\leq i\leq\ell+1$, fix their normalization by
			\begin{equation}\label{spectral-wall-normalization}
				\partial_\xi\varphi_{i,b,\beta}(0)=\frac{\cscale c_{i,0}}4\,\beta^{5/8}b^{-1/4}.
			\end{equation}
			These eigenvalues and eigenfunctions are $C^1$ in $(b,\beta)$, the latter with values in $C^2_{\rm loc}((0,\infty))$, and satisfy
			\begin{align}
				\mu_{i,b,\beta}&=\beta\left(\frac12-i\right)+\widetilde\mu_{i,b,\beta},
				\label{eigen-value-L-b}\\*
				\varphi_{i,b,\beta}(\xi)&=\sum_{j=0}^i c_{i,j}\beta^{j+5/8}b^{j+1/2}
				T_j\!\left(\frac{\xi}{b^{3/4}}\right)+\widetilde\varphi_{i,b,\beta}(\xi),
				\label{eigenfuc-L-b}
			\end{align}
			where $T_j$ are the iterates of Lemma~\ref{lemma-T-i}, normalized by \eqref{notation-scaling-normalization}, and the coefficients $c_{i,j}$ are defined in \eqref{defi-c-i-j}. Parameter derivatives below are taken at fixed $\xi$, and all bounds are uniform for $0\le i\le\ell+1$ and the indicated parameter range. The eigenvalue corrections satisfy
			\begin{equation*}
				|\widetilde\mu_{i,b,\beta}|+|b\partial_b\widetilde\mu_{i,b,\beta}|
				+|\partial_\beta\widetilde\mu_{i,b,\beta}|\lesssim_{\ell,C_0}b^{1/4}.
			\end{equation*}
			With the notation \eqref{notation-japanese-bracket}, the eigenfunction remainders satisfy, for $k=0,1,2$ and $\xi>0$,
			\begin{align*}
				&|(\xi\partial_\xi)^k\widetilde\varphi_{i,b,\beta}|
				+|(\xi\partial_\xi)^k(b\partial_b)\widetilde\varphi_{i,b,\beta}|
				+\langle\xi\rangle^{-4/3}|(\xi\partial_\xi)^k\partial_\beta\widetilde\varphi_{i,b,\beta}|\\*
				&\qquad\lesssim_{\ell,C_0}
				\xi^2\mathbf1_{\{\xi\leq b^{3/4}\}}
				+b^{1/4}\xi^{2/3}\langle\xi\rangle^{4/3+4i/3}\mathbf1_{\{\xi\geq b^{3/4}\}}.
			\end{align*}
			Moreover, $\|\varphi_{i,b,\beta}\|_{L^2_{b,\beta}}\to1$ as $b\downarrow0$, uniformly in $\beta$ and $0\le i\le\ell+1$. The next eigenvalue satisfies
			\begin{equation}\label{next-eigenvalue-finite-b}
				\mu_{\ell+1,b,\beta}=-\beta\left(\ell+\frac12\right)+O_{\ell,C_0}(b^{1/4}),
				\qquad \mu_{\ell,b,\beta}-\mu_{\ell+1,b,\beta}\ge\frac\beta2.
			\end{equation}
			Consequently, every $f\in H^1_{0;b,\beta}$ with $\langle f,\varphi_{j,b,\beta}\rangle_{L^2_{b,\beta}}=0$ for $0\le j\le\ell$ satisfies the quadratic-form estimate
			\begin{equation}\label{finite-b-spectral-gap}
				\begin{aligned}
					\int_0^\infty\left[\left(Q_b^{-1}+\beta\right)|f|^2-\sqrt{Q_b}\,|\partial_\xi f|^2\right]d\rho_{b,\beta}
					&\le\mu_{\ell+1,b,\beta}\|f\|_{L^2_{b,\beta}}^2.
				\end{aligned}
			\end{equation}
		\end{proposition}
		\begin{proof}
			The construction in Section~\ref{section-proof-L-b diagonal} uses matched asymptotics, following \cite{HadzicRaphael2019,Collot-et-alJAMS-2020,BDGJDE2025}. Section~\ref{subsection-spectral-ordering} identifies the spectral indices by Sturm comparison and proves the gap.
		\end{proof}
		
		\subsection{Geometric decomposition of the perturbation}\label{Toy-model}

		Fix an integer $\ell\ge1$ and $C_0>1$. For $\beta\in[C_0^{-1},C_0]$ and small $b$, self-adjointness and the distinct eigenvalues in Proposition~\ref{diagonalisation-L-b} give orthogonal eigenfunctions in \eqref{definition-Hilbert-b-beta}. With the coefficients and projections \eqref{defi-F-j}--\eqref{eigen-projection-Pi-j}, write
		\begin{equation}\label{decom-simple}
			\varepsilon=\sum_{j=0}^{\ell}\varepsilon_j\varphi_{j,b,\beta}+\varepsilon_-,
			\qquad \varepsilon_j=\mathfrak F_{j,b,\beta}(\varepsilon),\qquad
			\varepsilon_-=(\operatorname{Id}-\sum_{j=0}^{\ell}\Pi_{j,b,\beta})\varepsilon.
		\end{equation}
		Thus
		\begin{equation}\label{orthogonality-epsilon-minus}
			\langle\varepsilon_-,\varphi_{j,b,\beta}\rangle_{L^2_{\rho_{b,\beta}}}=0,
			\qquad 0\le j\le\ell.
		\end{equation}
		The formal wall constraint couples the zeroth and $\ell$th spectral coefficients.  We therefore replace the raw spectral decomposition~\eqref{decom-simple} by the adapted direction introduced below.
		
		\smallskip\noindent\textbf{\underline{Toy model.}}
		The changes of variables \eqref{renormalized-variables-lambda} and \eqref{chang-varialbe}, together with the formal scale relation \eqref{relation-lambda=mu-b}, give
		\begin{equation*}
			\tau_s=1/b,\qquad
			\mathfrak b_1=-\lambda_\tau/\lambda=\beta b-b_s.
		\end{equation*}
		Here $\mathfrak b_1$ is defined in \eqref{slowly-modulated-b-1}. Set $\beta=1$ and use the convention \eqref{notation-beta-one}. Along \eqref{for-mal-b-j-asymptotics}, $Y=1/b$ solves $Y_\tau-c_1Y/\tau=-1$, hence $Y=\tau/(c_1-1)+C\tau^{c_1}$. Since $\lambda=\kappa_\ell^{\rm form}\tau^{-c_1}$, the condition $\mu=\lambda Y\to0$ forces $C=0$. Therefore $b=1/((2\ell-1)\tau)$, $\mathfrak b_1^\ell=2\ell b$, and
		\begin{equation}\label{formal-b-s-b}
			b_s/b=1-2\ell,\qquad b(s)=b_0e^{(1-2\ell)(s-s_0)},\qquad b_0=b(s_0)>0.
		\end{equation}
		For constant $\beta>0$, the same calculation gives $b=[\beta(2\ell-1)\tau]^{-1}$. Reconstruction through \eqref{relation-lambda=mu-b} preserves $\lambda=\kappa_\ell^{\rm form}\tau^{-2\ell/(2\ell-1)}$ and the physical separation law; $\beta$ selects the renormalized frame. To identify the leading profile, consider the finite-mode ansatz
		\begin{equation}\label{formal-finite-mode-ansatz}
			\varepsilon=\varepsilon_+=\sum_{j=0}^{\ell}\varepsilon_j\varphi_{j,b},\qquad
			\varepsilon_-=0.
		\end{equation}
		This truncation is motivated by the spectral gap \eqref{finite-b-spectral-gap} on the orthogonal complement of the first $\ell+1$ eigenfunctions, where $\varepsilon_-$ lies; the full argument controls the forcing of, and coupling to, this stable component.
		By \eqref{eigenfuc-L-b} and \eqref{defi-m-0}, the forcing \eqref{Phi-simple-} has leading direction $\Lambda_\xi Q_b=\mathbf m_0\sqrt b\,\varphi_{0,b}+E_b$, where $E_b=-\mathbf m_0\sqrt b\,\widetilde\varphi_{0,b}$. The remainder bounds in Proposition~\ref{diagonalisation-L-b}, the weight \eqref{defi-rho-b-beta}, and $\|\varphi_{j,b}\|_{L^2_{\rho_b}}\asymp_\ell1$ give
		\begin{equation*}
			\|E_b\|_{L^2_{\rho_b}}+\max_{0\le j\le\ell}|\mathfrak F_{j,b}(E_b)|
			\lesssim_\ell b^{3/4}=o(\sqrt b).
		\end{equation*}
		Thus $E_b$ is smaller than the leading forcing by a factor $b^{1/4}$. Insert \eqref{formal-finite-mode-ansatz} into \eqref{equa-varepsilon-appen-} and take the $L^2_{\rho_b}$ inner products with $\varphi_{i,b}$. Discarding $E_b$ and the nonlinear terms, whose contributions are controlled by \eqref{remainder-ODE-item-i}, gives, by \eqref{formal-b-s-b},
		\begin{equation}\label{ode-matrix-V}
			\mathcal V'=\mathfrak A(s)\mathcal V+2\ell\mathbf m_0\sqrt b\,\mathbf e_0.
		\end{equation}
		Here $\mathcal V=(\varepsilon_0,\ldots,\varepsilon_\ell)^{\mathsf T}$ and $\mathbf e_j$ are the standard coordinate vectors. With $\varphi_j=\varphi_{j,b(s)}$ and $h_i=\|\varphi_i\|_{L^2_{\rho_b}}^2$, differentiation of the moving basis gives $\mathfrak A_{ij}=\mu_{j,b}\delta_{ij}-h_i^{-1}\langle\partial_s\varphi_j,\varphi_i\rangle_{L^2_{\rho_b}}$. Proposition~\ref{diagonalisation-L-b} yields
		\begin{align}
			\mathfrak A&=\mathfrak D_\ell+\mathfrak R,
			\qquad \mathfrak D_\ell=\operatorname{diag}(1/2-j)_{j=0}^{\ell},
			\label{decomposition-matrix-A}\\*
			\|\mathfrak R(s)\|&\lesssim_\ell b(s)^{1/4}\le K e^{-\omega_\ell s},
			\qquad \omega_\ell=(2\ell-1)/4.
			\label{small-off-diagonal-matrix}
		\end{align}
		The constant $K$ includes the fixed factor $b_0^{1/4}e^{\omega_\ell s_0}$.
		
		\smallskip\noindent\textbf{\underline{Matching asymptotics.}} Along \eqref{for-mal-b-j-asymptotics}, $\mathfrak b_1=2\ell b$ and $\mathfrak b_j/b^j=c_j(2\ell-1)^j$ for $1\le j\le\ell$, while $\mathfrak b_j=0$ for $j>\ell$. To compare with the finite-mode ansatz \eqref{formal-finite-mode-ansatz}, replace $v$ by the formal approximation $v_{\rm ap}$ of Subsection~\ref{subsection-formal-approximation}. With $\lambda=\mu b$ and $\zeta=\xi/b^{3/4}$, the corresponding perturbation is
		\begin{equation*}
			\varepsilon=b(v_{\rm ap}-Q)
			=b\sum_{j=1}^{\ell}\mathfrak b_j\mathcal T_j(\zeta)+b\mathcal S(\tau,\zeta).
		\end{equation*}
		This comparison yields two estimates.
		
		\noindent\emph{\underline{Pointwise behavior near the wall.}} The expansions in Subsection~\ref{subsection-formal-approximation} give $\mathcal T_j=O(\zeta^{1+3j/2})$ and $\mathcal S=O_\ell(b^2\zeta^4)$ on a fixed sufficiently small interval. Hence $v_{\rm ap}-Q=\mathfrak b_1\mathcal T_1+O_\ell(b^2\zeta^4)$, and $\mathcal T_1=-\zeta^{5/2}/15+O(\zeta^3)$ gives
		\begin{align}
			\varepsilon(s,\xi)&=-\frac{2\ell}{15}b^2\zeta^{5/2}
			+O_\ell(b^2\zeta^3+b^3\zeta^4),\notag\\*
			|\varepsilon(s,\xi)|&\lesssim_\ell b^2\zeta^{5/2}
			=b^{1/8}\xi^{5/2},\qquad 0\le\xi\lesssim b^{3/4}.
			\label{behavior_formal}
		\end{align}
		\noindent\emph{\underline{Size of the spectral coefficients.}} The recursions in Subsection~\ref{subsection-formal-approximation} and Lemma~\ref{lemma-T-i} give $\mathcal T_j=-p_jT_j/\cscale$, where $p_j=\prod_{k=1}^{j-1}(k-\tfrac12)$ and $p_1=1$. Since $c_{j,j}\ne0$, the leading eigenfunction expansions $\varphi_{i,b}^{[0]}:=\varphi_{i,b}-\widetilde\varphi_{i,b}=\sum_{j=0}^ic_{i,j}b^{j+1/2}T_j(\zeta)$ form a basis of $\operatorname{span}\{T_0,\ldots,T_\ell\}$. Define the matching coefficients by
		\begin{equation*}
			b\sum_{j=1}^{\ell}\mathfrak b_j\mathcal T_j
			=\sum_{i=0}^{\ell}\varepsilon_i^{\rm form}\varphi_{i,b}^{[0]}
			=\sum_{j=0}^{\ell}b^{j+1/2}
			\left(\sum_{i=j}^{\ell}c_{i,j}\varepsilon_i^{\rm form}\right)T_j.
		\end{equation*}
		Comparing the $T_0$ coefficient and then the $T_j$ coefficients gives
		\begin{equation*}
			\sum_{i=0}^{\ell}c_{i,0}\frac{\varepsilon_i^{\rm form}}{\sqrt b}=0,\qquad
			\sum_{i=j}^{\ell}c_{i,j}\frac{\varepsilon_i^{\rm form}}{\sqrt b}
			=-\frac{p_j}{\cscale}\frac{\mathfrak b_j}{b^j},\qquad 1\le j\le\ell.
		\end{equation*}
		The right-hand sides are bounded uniformly in $b$. Backward substitution from $j=\ell$ to $j=1$ bounds $\varepsilon_j^{\rm form}/\sqrt b$; the $T_0$ relation then bounds $\varepsilon_0^{\rm form}/\sqrt b$. The same scale appears in \eqref{asym-T}: $b^{j+1}(\xi/b^{3/4})^{2/3+4j/3}=\sqrt b\,\xi^{2/3+4j/3}$. Replacing $\varphi_{i,b}^{[0]}$ by $\varphi_{i,b}$ changes the matched hierarchy by $O_\ell(b^{3/4})$ in $L^2_{\rho_b}$, by Proposition~\ref{diagonalisation-L-b}. The leading hierarchy therefore selects the scale
		\begin{equation}\label{formal-bound-epsilon-j-b-half}
			|\varepsilon_j(s)|\lesssim_\ell\sqrt{b(s)},\qquad 0\le j\le\ell.
		\end{equation}
		We impose the scale \eqref{formal-bound-epsilon-j-b-half}, obtained by coefficient matching, on the toy system; the full construction controls the higher-order corrections and spectral remainders by bootstrap. This bound implies $\mathcal V\to0$ and excludes the intermediate homogeneous modes: when $\mathfrak R=0$, their ratio to $\sqrt b$ grows like $e^{(\ell-j)(s-s_0)}$ for $1\le j<\ell$. The next lemma gives the asymptotics for the perturbed system.
		
		\begin{lemma}[Asymptotics of the toy system]\label{lemma-asymptotic-integration-toy-system}
			Let $\ell\ge1$ be an integer, $K>0$, and $\mathbf m_0\in\mathbb R$, with the notation \eqref{decomposition-matrix-A}--\eqref{small-off-diagonal-matrix}. For $s_0\ge\max\{1,\omega_\ell^{-1}\log(2K/\omega_\ell)\}$ and $b_0>0$, let $b$ be given by \eqref{formal-b-s-b}. Suppose $\mathfrak R$ is continuous with $\|\mathfrak R(s)\|\le Ke^{-\omega_\ell s}$ in the induced Euclidean norm. If a $C^1$ solution of \eqref{ode-matrix-V} satisfies
			\begin{equation}\label{V-tends-zero-general-decomposition}
				\mathcal V(s)=O(\sqrt{b(s)}),\qquad \mathcal V(s)\to0\quad(s\to\infty),
			\end{equation}
			then, for a unique $a_\ell\in\mathbb R$ and every $s\ge s_0$,
			\begin{align}
				\mathcal V(s)&=\sqrt{b(s)}\,(-2\mathbf m_0\mathbf e_0+a_\ell\mathbf e_\ell)+\mathcal E(s),
				\label{general-asymptotic-expansion-V}\\*
				|\mathcal E(s)|&\le\frac{4K}{\omega_\ell}(|\mathbf m_0|+|a_\ell|)\sqrt{b(s)}\,e^{-\omega_\ell s}.
				\label{estimate-E-i-toy-system}
			\end{align}
			Moreover, $a_\ell=\lim_{s\to\infty}\mathcal V_\ell(s)/\sqrt{b(s)}$.
		\end{lemma}
		\begin{proof}
			Set $W=\mathcal V/\sqrt b$ and $D=\operatorname{diag}(\ell-j)_{j=0}^{\ell}$. By \eqref{V-tends-zero-general-decomposition}, $M:=\sup_{s\ge s_0}|W(s)|<\infty$, and $W'=DW+\mathfrak RW+2\ell\mathbf m_0\mathbf e_0$. Thus $|W_\ell'|\le KM e^{-\omega_\ell s}$ is integrable, giving a limit $a_\ell$. For $Z=W+2\mathbf m_0\mathbf e_0$, backward variation of constants and boundedness give
			\begin{equation*}
				Z(s)-a_\ell\mathbf e_\ell=-\int_s^\infty e^{-D(t-s)}\mathfrak R(t)W(t)\,dt,
				\qquad |Z(s)-a_\ell\mathbf e_\ell|\le\frac K{\omega_\ell}M e^{-\omega_\ell s}.
			\end{equation*}
			Indeed, all diagonal entries of $D$ except the last are positive, while $\|e^{-Dt}\|\le1$ for $t\ge0$, justifying the limit at infinity and the integral. The stated lower bound on $s_0$ gives $(K/\omega_\ell)e^{-\omega_\ell s_0}\le1/2$. Taking the supremum yields $M\le2|\mathbf m_0|+|a_\ell|+M/2$, hence $M\le4|\mathbf m_0|+2|a_\ell|$. Multiplication by $\sqrt b$ proves \eqref{general-asymptotic-expansion-V}--\eqref{estimate-E-i-toy-system}; the limit of $W_\ell$ gives uniqueness.
		\end{proof}
		
		\smallskip\noindent\textbf{\underline{Derivation of the admissible geometric decomposition.}}
		Lemma~\ref{lemma-asymptotic-integration-toy-system} leaves $a_\ell$ undetermined. The wall expansion selects it: by \eqref{eigenfuc-L-b}, $\varphi_{i,b,\beta}(b^{3/4}\zeta)=c_{i,0}\beta^{5/8}\sqrt b\,T_0(\zeta)+O(b^{3/2})$ for $0\le\zeta\le1$. Since $T_0=\cscale\Lambda Q$ vanishes linearly and \eqref{behavior_formal} gives $\varepsilon=O(b^2\zeta^{5/2})$, the order-$b$ contribution in the formal finite-mode ansatz must cancel. For $\beta=1$, this requires $-2\mathbf m_0c_{0,0}+a_\ell c_{\ell,0}=0$, hence $a_\ell=\mathbf n_\ell$ by \eqref{defi-n-ell-1}. Thus, the leading coefficient is predicted to satisfy
		\begin{equation}\label{esym-varp_ell}
			\varepsilon_\ell=\mathbf n_\ell\sqrt b+o(\sqrt b),
		\end{equation}
		This motivates the geometric decomposition
		\begin{equation}\label{decom-varepvarphi-i-ell}
			\varepsilon=\varepsilon_\ell\psi_{\ell,b,\beta}
			+\sum_{j=1}^{\ell-1}\varepsilon_j\varphi_{j,b,\beta}+\varepsilon_-,
		\end{equation}
		where
		\begin{equation}\label{definition-adapted-eigendirection-strategy}
			\psi_{\ell,b,\beta}=\varphi_{\ell,b,\beta}-\frac{c_{\ell,0}}{c_{0,0}}\varphi_{0,b,\beta}.
		\end{equation}
		In the full equation, we select $b$ by the exact constraint $c_{0,0}\varepsilon_0+c_{\ell,0}\varepsilon_\ell=0$. The formal analysis leaves the positive reference value of $\beta$ free.  We normalize $\beta$ near $1$; any fixed positive reference value yields the same physical separation laws.

		\section{Strategy of the proof}
		Section~\ref{Section-formal-approach} determines the candidate separation rates and the adapted decomposition~\eqref{decom-varepvarphi-i-ell}.  We now formulate the nonlinear construction in the renormalized variables $(s,\xi)$.  Writing $\vartheta=Q_b+\varepsilon$, we couple pointwise estimates for $\varepsilon$ with modulation equations for $(b,\beta)$.  If these quantities remain in the shrinking set for all $s\ge s_0$, the inverse von Mises transformation reconstructs a Prandtl solution with the asymptotics stated in Theorems~\ref{theo-main-separate} and~\ref{theo-unstab}.

		\subsection{The stable regime}\label{subsection-sketch-stable-regime}
		Let $\ell=1$. Following the modulation approach of \cite{BDGJDE2025,CGMN_CPAM_22}, we allow $\beta$ to vary near $1$, impose \eqref{esym-varp_ell} exactly, and use the adapted decomposition \eqref{decom-varepvarphi-i-ell}:
		\begin{align}
			\varepsilon_1&=\mathbf n_1(\beta)\sqrt b,\label{exact-stable-finite-mode}\\
			\varepsilon&=\varepsilon_1\psi_{1,b,\beta}+\varepsilon_-.\label{decomposition-stable-sketch}
		\end{align}
		Coupling these constraints to \eqref{equa-varepsilon-appen-} gives the problem for $(\varepsilon,b,\beta)$:
		\begin{equation}\label{defi-stable-problem}
			(\mathbf P_{\rm stab}):\qquad
			\left|\begin{array}{l}
				\varepsilon\text{ solves }\eqref{equa-varepsilon-appen-},\\
				c_{0,0}\mathfrak F_{0,b,\beta}(\varepsilon)+c_{1,0}\mathfrak F_{1,b,\beta}(\varepsilon)=0,\\
				\mathfrak F_{1,b,\beta}(\varepsilon)=\mathbf n_1(\beta)\sqrt b.
			\end{array}
			\right.
		\end{equation}
		We construct a global solution by a bootstrap in the shrinking set $\mathcal V_1[A,\eta]$, inspired by \cite{MZdm97,BDGJDE2025}. For $A\gg1$ and $\eta>0$ small, the following \textit{a priori estimates} close the argument.
		
		\paragraph{\underline{A priori estimates for the modulation parameters.}}
		Since $\mathbf n_1(\beta)\propto\beta^{-5/8}$ by \eqref{defi-n-ell-1}, differentiating the amplitude constraint gives $\varepsilon_1'/\varepsilon_1=\tfrac12b_s/b-\tfrac58\beta_s/\beta$. Applying $\Pi_{j,b,\beta}$, $j=0,1$, to \eqref{equa-varepsilon-appen-} yields \eqref{matrix-modulation-system-correct}. Dividing by the uniformly nonzero factor $\mathbf n_1(\beta)$ gives, with $X_1=b_s/b+\beta$ and $X_2=\beta_s/\beta$,
		\begin{equation*}
			\left[\begin{pmatrix}1/2&1\\0&1\end{pmatrix}+\mathcal E(s)\right]
			\binom{X_1}{X_2}=O(b^{1/4}),\qquad
			\|\mathcal E(s)\|\lesssim b^{1/4}+A^2b^\eta.
		\end{equation*}
		The implicit function theorem applied to \eqref{defi-G-modulation} selects $(b,\beta)$ locally; this system governs their dynamics. For $b^{1/4}+A^2b^\eta$ small, invertibility gives $|b_s/b+\beta|+|\beta_s|\lesssim b^{1/4}$. Exponential decay of $b$ makes these errors integrable, so for some $\beta_\infty>0$ close to $1$ and $b_\infty>0$,
		\begin{equation}\label{target-stable-modulation}
			\beta(s)\to\beta_\infty,\qquad b_s/b\to-\beta_\infty,\qquad
			b(s)\sim b_\infty e^{-\beta_\infty s}.
		\end{equation}
		
		\paragraph{\underline{A priori estimates for $\varepsilon_-$.}}
		The growth $Q_b(\xi)\sim a^2\xi^{4/3}$ and the polynomial tail of $\psi_{1,b,\beta}$ do not match the finite far-field value $U_\infty(s)$ in \eqref{equa-vartheta}; hence $\varepsilon_-$ cannot remain globally small relative to the local expansion. We restrict sharp bounds to $(0,b^{-\eta})$: they determine the local profile and modulation dynamics, while the exponential decay of $\rho_{b,\beta}$ suppresses polynomial tails in the spectral projections. Beyond $b^{-\eta}$, coarse profile and residual bounds control the weighted integrals and preserve positivity. The Oleinik solution and its continuation retain the prescribed condition $\mu\vartheta\to u_E^2(x(s))$.
		Writing $\varepsilon_+=\varepsilon_1\psi_{1,b,\beta}$, subtraction from \eqref{equa-varepsilon-appen-} gives
		\begin{align*}
			\partial_s\varepsilon_-&=\mathscr L_{b,\beta}\varepsilon_-+\mathscr B_b(\varepsilon)+\mathcal R_{1,b}(\varepsilon_+),\\
			\mathcal R_{1,b}(\varepsilon_+)&=-\partial_s\varepsilon_++\mathscr L_{b,\beta}\varepsilon_++(\beta-b_s/b)\Lambda_\xi Q_b,
		\end{align*}
		with $\mathscr B_b$ defined in \eqref{defi-B-quadratic-appendix-}. Our target is
		\begin{equation}\label{target-epsilon-minus-stable}
			\|\varepsilon_-(s)\|_{\mathfrak X_{b(s),\eta}}\lesssim b(s)^{1/2+\eta},
		\end{equation}
		where the norm is defined in \eqref{definition-weighted-remainder-norm}. Fix $K$ large and divide the spatial domain as follows.
		
		\smallskip\noindent\emph{Near the wall: $0<\xi<b^\eta$.}
		The higher modes in $\varepsilon_-$ retain $T_0=\cscale\Lambda Q$ terms in \eqref{eigenfuc-L-b}, so orthogonality gives no pointwise domination by $\varepsilon_+$. In $\psi_{1,b,\beta}=\varphi_{1,b,\beta}-(c_{1,0}/c_{0,0})\varphi_{0,b,\beta}$, these terms cancel, making $\varepsilon_+$ smaller than the target bounds by \eqref{outer-estimates-epsilon-plus}. It therefore suffices to bound $\varepsilon$ and subtract $\varepsilon_+$. To resolve $\xi\sim b^{3/4}$, use $Q_b(\xi)=bQ(\zeta)$ and set
		\begin{equation}\label{definition-inner-theta-zeta}
			\vartheta(s,\xi)=b(s)\theta(s,\zeta),\qquad \zeta=\xi/b(s)^{3/4}.
		\end{equation}
		Then $\varepsilon=b\widetilde\theta$, where $\widetilde\theta=\theta-Q$, and \eqref{equa-vartheta} gives
		\begin{equation*}
			b\partial_s\widetilde\theta=\mathscr H_0\widetilde\theta+(b\beta-b_s)\Lambda_\zeta(Q+\widetilde\theta)-\mathcal N(\widetilde\theta),\qquad \widetilde\theta(s,0)=0.
		\end{equation*}
		Here $b\beta-b_s=2\beta b+O(b^{5/4})$, and $\mathscr H_0=\sqrt Q\,\partial_\zeta^2+Q^{-1}$ and $\mathcal N$ are as in Section~\ref{subsection-formal-approximation}. The fundamental solutions $\Lambda Q,\widetilde Q_0$ in \eqref{wroskian-u-1-2} yield the profiles \eqref{equations-Z-Y-general-p} and the barriers
		\begin{equation*}
			\widetilde\theta^{\rm sub}=c_-\mathcal Y_\sigma-b^\eta\mathcal Z_\sigma,
			\qquad \widetilde\theta^{\rm sup}=c_+\mathcal Y_\sigma+b^\eta\mathcal Z_\sigma,
			\qquad c_\pm=2\beta b\pm b^{1+\eta/2}.
		\end{equation*}
		The profile $\mathcal Y_\sigma$ matches the forcing, while $\mathcal Z_\sigma>0$ supplies the comparison margin. For small $\sigma$, \eqref{near-origin-coercivity-general-p} gives the signs, and \eqref{definition-G1-general-p}--\eqref{definition-G2-general-p} handle the wall degeneracy. Proposition~\ref{propo-priori-near-0} bounds $\varepsilon$; subtracting \eqref{outer-estimates-epsilon-plus} bounds $\varepsilon_-$.
		
		\smallskip\noindent\emph{Intermediate region: $b^\eta\le\xi\le K$.}
		Since $b^\eta\gg b^{3/4}$, the limiting self-similar operator governs the leading dynamics:
		\begin{equation*}
			\partial_s\varepsilon_-=\mathscr L_{0,1}\varepsilon_-+(\beta-1)\Lambda_\xi\varepsilon_-+\widetilde{\mathcal B}_1(\varepsilon)+\widetilde{\mathcal R}_1(\varepsilon_+).
		\end{equation*}
		With the remainders \eqref{definition-B-one-beta-one}--\eqref{definition-R-one-beta-one} and $S_1(t)=e^{t\mathscr L_{0,1}}$, Duhamel's formula reads
		\begin{equation*}
			\varepsilon_-(s)=S_1(s-s_*)\varepsilon_-(s_*)+\int_{s_*}^{s}S_1(s-\tau)\bigl[(\beta-1)\Lambda_\xi\varepsilon_-+\widetilde{\mathcal B}_1(\varepsilon)+\widetilde{\mathcal R}_1(\varepsilon_+)\bigr](\tau)\,d\tau.
		\end{equation*}
		The kernel bounds in Lemmas~\ref{lem-simigroup-L-infty-beta=2} and~\ref{lem-semi-tle-K-0-xi-le-K}, the spectral decay \eqref{stable-semigroup-decay-beta-one}, and the comparison of moving and fixed projections \eqref{comparison-moving-fixed-beta-one} yield Proposition~\ref{propo-priori-b-eta-K}.
		
		\smallskip\noindent\emph{Outer region: $K\le\xi\le b^{-\eta}$.}
		Retaining $\beta$ in the limiting operator gives
		\begin{equation*}
			\partial_s\varepsilon_-=\mathscr L_{0,\beta}\varepsilon_-+\widetilde{\mathcal B}_1(\varepsilon)+\widetilde{\mathcal R}_1(\varepsilon_+).
		\end{equation*}
		For $K$ large, $\phi_{\beta,2}\asymp\xi^{10/3}>0$ and $\mathscr L_{0,\beta}\phi_{\beta,2}=-\tfrac32\beta\phi_{\beta,2}$; hence $\pm A^{-3}b^{1/2+\eta}\phi_{\beta,2}$ provide comparison barriers (Proposition~\ref{pro-priori-epsilon-xi-large}). Combining the three regional estimates gives \eqref{target-epsilon-minus-stable}.
		
		\smallskip\noindent\emph{Far field: $\xi\ge b^{-\eta}$.}
		Constant barriers for the full profile $\vartheta$ give, by Lemma~\ref{lem-comparison-vartheta},
		\begin{equation}\label{target-tail-stable}
			b^{-\eta/2}\lesssim\vartheta\lesssim b^{-1-\eta}.
		\end{equation}
		\textit{- \underline{The residual $\mathcal F$}.}
		The residual serves two purposes. On the one hand, its physical interpretation is
		\begin{equation*}
			\mathcal F=\sqrt\vartheta\,\partial_\xi^2\vartheta-2=2\bigl(\partial_y^2u-1\bigr).
		\end{equation*}
		Since $\xi\asymp_s y^2$ near the nondegenerate wall at each finite $s$, the bound $\mathcal F=O(\xi)$ propagates compatibility \eqref{condition-in-u-locally-exsist}, needed with positivity and regularity for continuation. On the other hand,
		\begin{equation*}
			(\sqrt\vartheta-\sqrt{Q_b})\partial_\xi^2\varepsilon
			=\left(1-\sqrt{\frac{Q_b}{\vartheta}}\right)
			\left[\mathcal F-2\left(\sqrt{\frac{\vartheta}{Q_b}}-1\right)\right]
		\end{equation*}
		controls the highest-derivative term in $\mathscr B_b(\varepsilon)$ pointwise. With $\vartheta=Q_b+\varepsilon$, equation \eqref{equa-mathcal-F} becomes
		\begin{equation*}
			\partial_s\mathcal F=(\mathscr L_{b,\beta}-\beta\operatorname{Id})\mathcal F+(\sqrt\vartheta-\sqrt{Q_b})\partial_\xi^2\mathcal F-\frac{\varepsilon\mathcal F}{Q_b\vartheta}+\frac{\mathcal F^2}{2\vartheta}.
		\end{equation*}
		The linearization shares the eigenfunctions of $\mathscr L_{b,\beta}$, with eigenvalues shifted by $-\beta$ and strictly negative limits $-\beta(j+\tfrac12)$, $j\ge0$. Spectral decay, local parabolic compactness, and comparison yield
		\begin{equation*}
			\|\mathcal F(s)\|_{\mathfrak X_{b(s),\eta}}\lesssim b^{1/2-\eta},
			\qquad -C\le\mathcal F(s,\xi)\le b^{-\eta/8}\quad(\xi\ge b^{-\eta}),
		\end{equation*}
		as proved in Lemma~\ref{lem-priori-F-K} and Section~\ref{section-pointwise-estimates-F}.
		
		\noindent 
		\emph{Closing the stable construction.}
		Strict improvement of the bootstrap bounds, together with the Oleinik continuation criterion, yields global trapping, stable under perturbations of the prepared data. Reconstruction yields $\partial_yu(x,0)=\tfrac12(\mu b)^{1/4}(1+O(b^\eta))$. Since both $\mu$ and $b$ decay like $e^{-\beta_\infty s}$ and $x^*-x=\int_s^\infty\mu(t)\,dt$, the wall shear has exponent $1/2$; see Section~\ref{section-prove-theorem-stab}.
		
		\subsection{The unstable regimes}\label{subsection-sketch-unstable-regime}
		For $\ell\ge2$, the key ingredients are analogous to those of the stable scheme in Section~\ref{subsection-sketch-stable-regime}. However, additional directions in the expansion about $Q_b$ are unstable on the shrinking scale and raise distinct difficulties that must be addressed.
		
		The natural two-parameter modulation scheme would impose
		\begin{equation*}
			c_{0,0}\mathfrak F_{0,b,\beta}(\varepsilon)+c_{\ell,0}\mathfrak F_{\ell,b,\beta}(\varepsilon)=0,
			\qquad \varepsilon_\ell=\mathbf n_\ell(\beta)\sqrt b.
		\end{equation*}
		However, the available estimates do not establish invertibility of the corresponding two-parameter map $G_\ell(b,\beta)$, since the leading term of its Jacobian determinant cancels. Thus, the stable modulation argument does not justify propagating both constraints simultaneously. We therefore propagate only the geometric constraint and fix $\beta=1$ throughout this regime. The resulting geometric decomposition reads
		\begin{equation*}
			\varepsilon=\varepsilon_\ell\psi_{\ell,b}
			+\sum_{j=1}^{\ell-1}\varepsilon_j\varphi_{j,b}+\varepsilon_-.
		\end{equation*}
		Here we omit the fixed $\beta$ index as in \eqref{notation-beta-one}. Coupling the geometric constraint to \eqref{equa-varepsilon-appen-} gives the problem for $(\varepsilon,b)$:
		\begin{equation}\label{defi-unstable-problem}
			(\mathbf P_{\rm unstab}):\qquad
			\left|\begin{array}{l}
				\varepsilon\text{ solves }\eqref{equa-varepsilon-appen-}\text{ with }\beta=1,\\
				c_{0,0}\mathfrak F_{0,b}(\varepsilon)+c_{\ell,0}\mathfrak F_{\ell,b}(\varepsilon)=0.
			\end{array}
			\right.
		\end{equation}
		Fixing $\beta=1$ and normalizing $\mu(s_0)=e^{-s_0}$ gives $\mu(s)=e^{-s}$; hence the preparation time determines the prospective separation point:
		\begin{equation*}
			x^*=e^{-s_0}.
		\end{equation*}
		
		The shrinking-set argument has three components: weighted pointwise control of $\varepsilon_-$ and $\mathcal F$, an integrable modulation estimate for $b$, and a finite-dimensional shooting argument for $(\varepsilon_1,\ldots,\varepsilon_{\ell-1},e_\ell)$.
		
		\smallskip\noindent
		\textit{-\underline{Control of the remainder $\varepsilon_-$ and the residual $\mathcal F$.}}
		The weighted estimates for $\varepsilon_-$ and $\mathcal F$ control $0<\xi\le b^{-\eta}$. In the far field $\xi\ge b^{-\eta}$, we estimate $\vartheta$ directly; the scalar bounds for $\mathcal F$ hold on the whole half-line. With $b=b(s)$, the estimates read
		\begin{align*}
			\|\varepsilon_-(s)\|_{\mathfrak X_{b,\eta}^{(\ell)}}
			&\lesssim b^{1/2+\eta},
			& b^{-\eta/2}&\lesssim\vartheta(s,\xi)\lesssim b^{-1/(2\ell-1)-\eta},\\
			\|\mathcal F(s)\|_{\mathfrak X_{b,\eta}^{(\ell)}}
			&\lesssim b^{1/2-\eta},
			& -C&\le\mathcal F(s,\xi)\le b^{-\eta/8}.
		\end{align*}
		\smallskip\noindent
		\textit{-\underline{Modulation equation for $b$.}}
		As in the stable regime, projecting onto $\varphi_{0,b}$ and $\varphi_{\ell,b}$ yields, by \eqref{dynamics-b-item-i},
		\begin{equation*}
			\left|\frac{b'}b+(2\ell-1)\right|
			\lesssim_{A,\ell} b^{2\eta}.
		\end{equation*}
		Thus, along a globally trapped solution, the modulation error is integrable and
		$b(s)e^{(2\ell-1)s}$ converges to $b_\infty>0$, which determines the leading asymptotic behavior of $b$.
		
		\smallskip\noindent
		\textit{-\underline{Finite-dimensional dynamics and transversality.}}
		Set $\gamma_\ell=2\ell-1$ and normalize the coefficients by $z_j=\varepsilon_j/(Ab^{1/2+\eta})$ for $1\le j<\ell$ and $z_\ell=e_\ell/(Ab^{1/2+2\eta})$. Equations \eqref{ODE-general-item-i}--\eqref{remainder-ODE-item-i}, together with the dynamics of $b$ from the previous step, yield
		\begin{equation}\label{system-z-j}
			\begin{aligned}
				z_j'&=(\ell-j+\gamma_\ell\eta+o(1))z_j+O_{A,\ell}(b^{1/4-\eta}),&&1\le j<\ell,\\
				z_\ell'&=(\ell+2\gamma_\ell\eta+o(1))z_\ell+O_{A,\ell}(b^{1/4-2\eta}).
			\end{aligned}
		\end{equation}
		For $A\ge A_7$, $0<\eta\le\eta_7$, and $s_0\ge s_7$, with the thresholds
		of Proposition~\ref{lemma-transversality-ell}, \eqref{system-z-j} shows that
		all $\ell$ normalized directions are unstable within the bootstrap bounds.
		This gives \emph{transversality}: every face $z_j=\pm1$ is strictly outgoing.
		
		Transversality, strict improvement of the remaining bootstrap bounds, and continuation reduce the infinite-dimensional construction to a finite-dimensional selection of initial amplitudes $(\varepsilon_1,\ldots,\varepsilon_\ell)$, parametrized by $\mathbf d=(d_1,\ldots,d_\ell)$, for which the corresponding solution remains in $\mathcal V_\ell[A,\eta]$ for all $s\ge s_0$.
		
		\medskip\noindent
		\textit{-\underline{Conclusion of the finite-dimensional problem.}}
		The reduced problem is to select suitable parameters $\mathbf d\in\mathcal D_{s_0}\subset[-2,2]^\ell$ such that the corresponding solution of \eqref{defi-unstable-problem} remains globally trapped. The existence proof uses a classical topological argument from blowup constructions; see \cite{MZdm97,BDGJDE2025}. We recall its key contradiction argument. If every prepared trajectory exited, transversality and continuous dependence would yield a continuous exit map from the normalized parameter cube $[-1,1]^\ell$ to its boundary, whose restriction to the boundary is the identity and hence has degree one. This contradicts the no-retraction theorem. Therefore, some datum remains globally trapped, and continuation gives a global renormalized solution; see Section~\ref{subsection-proof-existen-solu-ell=2}. Reconstruction using the shrinking-set bounds yields Theorem~\ref{theo-unstab}.
		
		\section{Existence of a solution in the stable regime}
		We now carry out the stable construction outlined in Section~\ref{subsection-sketch-stable-regime}. We construct a global renormalized solution of $(\mathbf P_{\rm stab})$ in \eqref{defi-stable-problem}.
		
		\subsection{Bootstrap bounds and the shrinking set}
		Fix $s_0>0$ and $s>s_0$. For $\tau\in[s_0,s]$, consider bounded profiles
		$\vartheta(\tau,\cdot)\in C^1([0,\infty))\cap C^2((0,\infty))$,
		with $\vartheta(\tau,0)=0$ and $\vartheta(\tau,\xi)>0$ for $\xi>0$.
		Let $b,\beta\in C^1([s_0,s])$, with $0<b<1$ and $\beta>0$,
		and define $\Bstab$ by the integral in \eqref{stable-reference-scale},
		with $\Bstab(s_0)=e^{-s_0}$. We impose the stable amplitude constraint
		\eqref{exact-stable-finite-mode} throughout $[s_0,s]$.
		
		\begin{definition}[{Stable shrinking set $\mathcal V_1[A,\eta]$}]
			For $A,\eta>0$, we write
			$(\vartheta,b,\beta)\in\mathcal V_1[A,\eta](s)$
			if the geometric decomposition \eqref{decomposition-stable-sketch},
			the orthogonality conditions \eqref{orthogonality-epsilon-minus},
			and the following bounds hold for every $\tau\in[s_0,s]$,
			with $b=b(\tau)$:
			
			\noindent\textup{(i)} Modulation parameters:
			$\tfrac12\le b/\Bstab(\tau)\le2$ and
			$|\beta(\tau)-1|\le Ae^{-\eta s_0}$.
			
			\noindent\textup{(ii)} Infinite-dimensional remainder and far field:
			$\|\varepsilon_-(\tau)\|_{\mathfrak X_{b,\eta}}
			\le A^2b^{1/2+\eta}$ and
			$A^{-1}b^{-\eta/2}\le\vartheta(\tau,\xi)\le Ab^{-1-\eta}$
			for $\xi\ge b^{-\eta}$, where the norm is defined in
			\eqref{definition-weighted-remainder-norm}.
			
			\noindent\textup{(iii)} Residual $\mathcal F$: 
			$-A\le\mathcal F(\tau,\xi)\le Ab^{-\eta/8}$ for $\xi\in\mathbb R_+$,
			and $\|\mathcal F(\tau)\|_{\mathfrak X_{b,\eta}}\le Ab^{1/2-\eta}$.
		\end{definition}
		\begin{remark}[{Design of the shrinking set $\mathcal V_1[A,\eta]$}]
			The three groups of bounds in $\mathcal V_1[A,\eta]$ are chosen to match the modulation, remainder, and residual equations.  Fix $A\gg1$, then $0<\eta\ll1$, and finally $s_0\gg1$.
			
			Item~\textup{(i)} keeps $\beta$ close to $1$ and compares $b$ with $\Bstab$, where $(\log\Bstab)'=-\beta$. For a globally trapped solution, Proposition~\ref{pro-finite-projections} gives $|\beta'|+|b'/b+\beta|\lesssim b^{1/4}$. Exponential decay of $b$ makes these errors integrable, so $\beta\to\beta_\infty>0$, $b'/b\to-\beta_\infty$, and $b/\Bstab$ converges to a positive limit.
			
			For item~\textup{(ii)}, recall \eqref{decomposition-stable-sketch} and the cancellation of the $T_0$ terms:
			\[
			\varepsilon=\varepsilon_++\varepsilon_-,\qquad
			\varepsilon_+=\mathbf n_1(\beta)b^{1/2}\psi,\qquad
			\psi:=\psi_{1,b,\beta}=c_{1,1}\beta^{13/8}b^{3/2}T_1(\xi/b^{3/4})+\widetilde\psi.
			\]
			Denote the displayed $T_1$ term by $\psi^{T_1}$. Proposition~\ref{diagonalisation-L-b}, Lemmas~\ref{lemma-Q}--\ref{lemma-T-i}, and the norm \eqref{definition-weighted-remainder-norm} give Table~\ref{table-shrinking-pointwise-sizes}: the first two rows are comparable sizes, and the remaining rows are upper bounds up to a constant, uniformly for $\beta$ near $1$.
			\par\smallskip\begingroup\centering
			\refstepcounter{table}\label{table-shrinking-pointwise-sizes}
			\resizebox{\linewidth}{!}{$
				\renewcommand{\arraystretch}{1.35}
				\begin{array}{|c|c|c|c|c|}
					\hline
					\text{Quantity}&0<\xi\le b^{3/4}&b^{3/4}\le\xi\le1&1\le\xi\le b^{-\eta}&\xi\ge b^{-\eta}\\ \hline
					Q_b&b^{1/4}\xi&\xi^{4/3}&\xi^{4/3}&\xi^{4/3}\\ \hline
					|\psi^{T_1}|&b^{-3/8}\xi^{5/2}&\xi^2&\xi^2&\xi^2\\ \hline
					|\psi|&\xi^2&\xi^2+b^{1/4}\xi^{2/3}&\xi^2+b^{1/4}\xi^{10/3}&\xi^2+b^{1/4}\xi^{10/3}\\ \hline
					|\varepsilon_+|&b^{1/2}\xi^2&
					b^{1/2}\xi^2+b^{3/4}\xi^{2/3}&
					b^{1/2}\xi^2+b^{3/4}\xi^{10/3}&
					b^{1/2}\xi^2+b^{3/4}\xi^{10/3}\\ \hline
					|\varepsilon_-|&A^2b^{1/4+\eta}\xi&
					A^2b^{1/2+\eta}\xi^{2/3}&A^2b^{1/2+\eta}\xi^{10/3}&
					\text{no information needed}\\ \hline
				\end{array}
				$}
			\smallskip
			
			\textbf{Table~\thetable.} Pointwise profile sizes and perturbation bounds.
			\par\endgroup\smallskip
			Table~\ref{table-shrinking-relative-sizes} gives the corresponding relative bounds; $E_\pm$ denote the positive envelopes listed for $|\varepsilon_\pm|$ in Table~\ref{table-shrinking-pointwise-sizes}.
			\par\smallskip\begingroup\centering
			\refstepcounter{table}\label{table-shrinking-relative-sizes}
			\resizebox{\linewidth}{!}{$
				\renewcommand{\arraystretch}{1.45}
				\begin{array}{|c|c|c|c|c|}
					\hline
					\text{Ratio}&0<\xi\le b^{3/4}&b^{3/4}\le\xi\le1&1\le\xi\le b^{-\eta}&\xi\ge b^{-\eta}\\ \hline
					\rule[-2.3ex]{0pt}{5.8ex}|\varepsilon_+|/Q_b
					&\lesssim b^{1/4}\xi\le b
					&\begin{gathered}\lesssim b^{1/2}\xi^{2/3}+b^{3/4}\xi^{-2/3}\\ \lesssim b^{1/4}\end{gathered}
					&\begin{gathered}\lesssim b^{1/2}\xi^{2/3}+b^{3/4}\xi^2\\
						\lesssim b^{1/2-2\eta/3}+b^{3/4-2\eta}\end{gathered}
					&\begin{gathered}\lesssim b^{1/2}\xi^{2/3}+b^{3/4}\xi^2\\
						\text{not uniformly small}\end{gathered}\\ \hline
					\rule[-2.3ex]{0pt}{5.8ex}|\varepsilon_-|/Q_b
					&\lesssim A^2b^\eta
					&\begin{gathered}\lesssim A^2b^{1/2+\eta}\xi^{-2/3}\\
						\lesssim A^2b^\eta\end{gathered}
					&\begin{gathered}\lesssim A^2b^{1/2+\eta}\xi^2\\
						\lesssim A^2b^{1/2-\eta}\end{gathered}
					&\text{no information needed}\\ \hline
					\rule[-2.3ex]{0pt}{5.8ex}E_-/E_+
					&\lesssim A^2b^{-1/4+\eta}\xi^{-1}
					&\begin{gathered}
						\displaystyle\lesssim
						\frac{A^2b^{1/2+\eta}\xi^{2/3}}
						{b^{1/2}\xi^2+b^{3/4}\xi^{2/3}}\\
						\lesssim A^2b^{-1/4+\eta}
					\end{gathered}
					&\begin{gathered}
						\lesssim A^2b^\eta\xi^{4/3}\\
						\lesssim A^2b^{-\eta/3}
					\end{gathered}
					&\text{no information needed}\\ \hline
				\end{array}
				$}
			\smallskip
			
			\textbf{Table~\thetable.} Relative perturbation bounds and comparison of envelopes.
			\par\endgroup\smallskip
			
			The envelope ratio in the last row is not an estimate for $|\varepsilon_-|/|\varepsilon_+|$ without a corresponding lower bound for $|\varepsilon_+|$. The bounds needed here are relative to $Q_b$: near the wall, $|\varepsilon_-|\lesssim A^2b^{1/4+\eta}\xi$ implies $\varepsilon_-(s,0)=0$ by continuity, while
			\[
			\sup_{0<\xi\le b^{-\eta}}\frac{|\varepsilon|}{Q_b}
			\lesssim b^{1/4}+A^2b^\eta\lesssim A^2b^\eta=o(1).
			\]
			Thus $\vartheta\asymp Q_b$ on this expanding region; in particular, $\vartheta\le2Q_b\lesssim b^{-4\eta/3}\le(A/2)b^{-1-\eta}$ there, extending the upper bound in~\textup{(ii)} to the whole half-line.
			
			For $\xi\ge b^{-\eta}$, we control $\vartheta$ directly. Its matching condition $\vartheta(s,\xi)\to U_\infty(s)$ in \eqref{equa-vartheta} requires cancellation of the growing profiles, and the eigenfunction remainder does not identify a uniform leading power at infinity. No separate pointwise bootstrap norm for $\varepsilon_-$ is imposed beyond $b^{-\eta}$: the identity $\varepsilon_-=\vartheta-Q_b-\varepsilon_+$ gives a sufficient polynomial bound for the weighted projections. Exponential decay of $\rho_{b,\beta}$ reduces its far-field contribution to the modulation projections to $A^Cb^{-C}e^{-cb^{-4\eta/3}}$. Together with the local bounds, this yields
			\[
			\|\varepsilon_-\|_{L^2_{\rho_{b,\beta}}}\lesssim A^2b^{1/2+\eta}
			=o(b^{1/2}),\qquad
			\|\varepsilon_+\|_{L^2_{\rho_{b,\beta}}}\asymp b^{1/2}.
			\]
			Hence the remainder is lower order in the weighted space and in pairings against uniformly bounded test functions.
			
			The weight in \eqref{definition-weighted-remainder-norm} reflects the first remaining spectral direction, $\varphi_{2,b,\beta}$, after removing modes $0$ and $1$ by \eqref{orthogonality-epsilon-minus}. Its leading matched expansion has wall size $b^{-1/4}\xi$ and outer powers $\xi^{2/3}$, $\xi^2$, and $\xi^{10/3}$, motivating the envelope $\mathfrak f_1(\xi)=\xi^{2/3}\langle\xi\rangle^{8/3}$. Finally, item~\textup{(iii)} controls curvature through the residual $\mathcal F=\sqrt\vartheta\,\partial_\xi^2\vartheta-2$. The same spatial norm is used for $\mathcal F$ because its linearized operator is $\mathscr L_{b,\beta}-\beta\operatorname{Id}$ and therefore has the same eigenfunctions as the perturbation operator, with every eigenvalue shifted by $-\beta$. Indeed, substituting $\vartheta=Q_b+\varepsilon$ into \eqref{equa-mathcal-F} gives
			\[
			\partial_s\mathcal F
			=(\mathscr L_{b,\beta}-\beta\operatorname{Id})\mathcal F
			+(\sqrt\vartheta-\sqrt{Q_b})\partial_\xi^2\mathcal F
			-\frac{\varepsilon\mathcal F}{Q_b\vartheta}
			+\frac{\mathcal F^2}{2\vartheta}.
			\]
			Thus the eigenvalues are shifted from $\mu_{j,b,\beta}$ to $\mu_{j,b,\beta}-\beta=-\beta(j+1/2)+O(b^{1/4})$. The slowest linear decay has scale $b^{1/2}$, motivating the slack in the bound $\|\mathcal F\|_{\mathfrak X_{b,\eta}}\le Ab^{1/2-\eta}$. Together with the global bounds in~\textup{(iii)}, this controls the quasilinear terms and supports continuation in the physical variables.
		\end{remark}

		\subsection{Construction of initial data}
		We construct initial data for $(\mathbf P_{\rm stab})$ that satisfy the compatibility conditions and lie strictly inside $\mathcal V_1[A,\eta](s_0)$.  Fix $x_0>0$ as in~\eqref{defi-u_E}, $0<\delta\le3/8$, $M_0\ge1$, and $s_0\ge1$.
		With $\chi_0$ from \eqref{defi-chi-0}, set
		\begin{align}
			\mu_0&=\widetilde b_0=e^{-s_0},\qquad \widetilde\beta_0=1,
			\label{defi-mu-0-l=1}\\*
			\chi_{\delta,M_0}(\xi)&:=\chi_0(M_0\widetilde b_0^\delta\xi),\qquad
			\chi_\delta:=\chi_{\delta,M_0},
			\label{defi-chi-delta}\\*
			R_\delta&:=M_0^{-1}\widetilde b_0^{-\delta}.
			\label{definition-initial-cutoff-radius}
		\end{align}
		Thus $\chi_\delta=1$ on $[0,R_\delta]$ and
		$\operatorname{supp}\chi_\delta\subset[0,2R_\delta]$.
		Define the inner correction and outer profile by
		\begin{align}
			\Psi_1(\xi)&:=\mathbf n_1(\widetilde\beta_0)\widetilde b_0^{1/2}
			\psi_{1,\widetilde b_0,\widetilde\beta_0}(\xi),\nonumber\\*
			\Psi_\infty^1(\xi)&:=U_\infty(s_0)
			-\bigl(U_\infty(s_0)-4a^2\widetilde b_0^{-4\delta/3}\bigr)
			e^{-(\xi-\widetilde b_0^{-\delta})/\Xi_1},
			\label{defi-U-infty-1}
		\end{align}
		where $\psi$ is given by \eqref{definition-adapted-eigendirection-strategy},
		$\Xi_1=\widetilde b_0^{-(1+4\delta/3)}$, and
		$U_\infty(s_0)=\mu_0^{-1}u_E^2(0)=2x_0\widetilde b_0^{-1}$.
		Set
		\begin{equation}\label{defi-initi-varthe-s-0}
			\Psi_1^0(s_0,\xi)=\chi_\delta(Q_{\widetilde b_0}+\Psi_1)
			+(1-\chi_\delta)\Psi_\infty^1.
		\end{equation}
		
		\begin{lemma}[Well-prepared initial data]\label{lemma-initial-data-l=1}
			Fix $x_0>0$ and $0<\delta\le3/8$.  There exists $M_1=M_1(x_0,\delta)\ge1$ such that, for every $M_0\ge M_1$, there exists $s_1=s_1(\delta,x_0,M_0)\ge1$ with the following property. For $s_0\ge s_1$, the datum
			\eqref{defi-initi-varthe-s-0} belongs to
			$C^\infty((0,\infty))\cap C^1([0,\infty))$, vanishes at the wall,
			and increases strictly to $U_\infty(s_0)$ as $\xi\to\infty$.
			
			\noindent\textup{(i)} \emph{Residual estimates.}
			The residual $\mathcal F_1:=\sqrt{\Psi_1^0}\,\partial_\xi^2\Psi_1^0-2$,
			extended by $\mathcal F_1(s_0,0)=0$, satisfies
			$\|\mathcal F_1(s_0)\|_{L^\infty(\mathbb R_+)}\lesssim_{x_0,M_0}1$ and
			\[
			|\mathcal F_1(s_0,\xi)|\lesssim
			\widetilde b_0^{1/4}\xi\,\mathbf1_{\{\xi\le\widetilde b_0^{3/4}\}}
			+\widetilde b_0^{1/2}\mathfrak f_1(\xi)\,
			\mathbf1_{\{\xi\ge\widetilde b_0^{3/4}\}},\qquad 0<\xi\le R_\delta,
			\]
			with $\mathfrak f_1$ from \eqref{defi-f-ell}.
			
			\noindent\textup{(ii)} \emph{Renormalization.}
			There exists $r_0>0$, independent of $s_0$, such that a unique pair
			$(b_0,\beta_0)$ with $|b_0/\widetilde b_0-1|+|\beta_0-1|<r_0$ satisfies
			\begin{gather*}
				\Psi_1^0(s_0,\xi)=Q_{b_0}(\xi)
				+\mathbf n_1(\beta_0)\sqrt{b_0}\,\psi_{1,b_0,\beta_0}(\xi)
				+\varepsilon_-(s_0,\xi),\\*
				\varepsilon_-(s_0)\perp
				\operatorname{span}\{\varphi_{0,b_0,\beta_0},\varphi_{1,b_0,\beta_0}\}
				\quad\text{in }L^2_{\rho_{b_0,\beta_0}}.
			\end{gather*}
			Moreover,
			\begin{gather}
				\widetilde b_0^{1/2}|b_0-\widetilde b_0|
				+\widetilde b_0|\beta_0-1|
				\lesssim e^{-\widetilde b_0^{-\delta}/16},
				\label{distance-b-s-0-tildeb-0}\\*
				\left\|\frac{\varepsilon_-(s_0,\cdot)}{\xi}\right\|_{L^\infty(0,b_0^{3/4})}
				+\left\|\frac{\varepsilon_-(s_0,\cdot)}{\xi^{2/3}}\right\|_{L^\infty(b_0^{3/4},R_\delta/8)}
				\lesssim e^{-\widetilde b_0^{-\delta}/16}.
				\label{initial-epsilon-minus-exponential}
			\end{gather}
			In \eqref{distance-b-s-0-tildeb-0}--\eqref{initial-epsilon-minus-exponential},
			the implicit constants are universal; dependence on $x_0,\delta,M_0$
			enters only through the threshold $s_1$.
		\end{lemma}
		
		\begin{proof}
			Set $q=Q_{\widetilde b_0}$, $Z=q+\Psi_1$, $V=\Psi_\infty^1$,
			$D=\xi\partial_\xi$, $\nu=\mathbf n_1(1)$, and $\kappa=c_{1,0}/c_{0,0}$.
			Unspecified spectral parameters are $(\widetilde b_0,1)$;
			all initial profiles are evaluated at $s_0$.
			
			\noindent\emph{Inner estimates.}
			With $\zeta=\xi/\widetilde b_0^{3/4}$, cancellation of the $T_0$ terms in
			\eqref{eigenfuc-L-b} gives
			\begin{equation}\label{expression-Psi-1}
				\Psi_1=\nu\widetilde b_0^{1/2}
				\bigl[c_{1,1}\widetilde b_0^{3/2}T_1(\zeta)
				+\widetilde\varphi_1-\kappa\widetilde\varphi_0\bigr].
			\end{equation}
			For a small fixed $\xi_*\in(0,1)$, Propositions~\ref{propo-esti-S-j-R-i}
			and~\ref{diagonalisation-L-b}, including \eqref{impro-beha-R_i}, yield
			\[
			|D^k\Psi_1|\lesssim
			\begin{cases}
				\widetilde b_0^{1/2}\xi^2,&0<\xi\le\xi_*,\\
				\widetilde b_0^{1/2}\xi^2+
				\widetilde b_0^{3/4}\xi^{2/3}\langle\xi\rangle^{8/3},&\xi\ge\xi_*,
			\end{cases}
			\qquad k=0,1,2.
			\]
			Indeed, the refined remainder on $[\widetilde b_0^{3/4},\xi_*]$ is
			$O(\widetilde b_0^{1/4}\xi^{5/3})$, whose contribution to
			\eqref{expression-Psi-1} is bounded by $\widetilde b_0^{1/2}\xi^2$.
			Below the inner scale, use $T_1,R_i=O(\zeta^{5/2})$ and the $S_j$ bounds.
			Lemma~\ref{lemma-Q} then gives, since $\delta\le3/8$,
			\[
			\frac{|\Psi_1|}{q}+\frac{|\Psi_1'|}{q'}
			\le C(\widetilde b_0^{1/4}+M_0^{-2})<\tfrac12,
			\qquad 0<\xi\le2R_\delta,
			\]
			for $M_0$ and then $s_1$ large. Hence $Z,Z'>0$;
			$\Psi_1(0)=\Psi_1'(0)=0$ also gives $Z\in C^1([0,\infty))$
			and $Z'(0)=\widetilde b_0^{1/4}$.
			
			On $(0,R_\delta]$, put $r=\Psi_1/q$ and
			$\mathscr A=\mathscr L_{\widetilde b_0,1}-\Lambda_\xi$.
			The eigenvalue equations and $\sqrt q\,q''=2$ give
			\begin{align*}
				\mathscr A\Psi_1&=(\mu_{1,\widetilde b_0}-\Lambda_\xi)\Psi_1
				+\nu\kappa\widetilde b_0^{1/2}
				(\mu_{1,\widetilde b_0}-\mu_{0,\widetilde b_0})\varphi_0,\\*
				\mathcal F_1&=\sqrt{1+r}\,\mathscr A\Psi_1
				-\frac{\sqrt{1+r}+2}{(\sqrt{1+r}+1)^2}\,r^2.
			\end{align*}
			Here $|\mu_{j,\widetilde b_0}|\lesssim1$ and
			$|\varphi_0|\lesssim\widetilde b_0^{-1/4}\xi$ below
			$\widetilde b_0^{3/4}$, whereas
			$|\varphi_0|\lesssim\xi^{2/3}(1+\widetilde b_0^{1/4}\langle\xi\rangle^{4/3})$
			above it. Together with $q\asymp\widetilde b_0^{1/4}\xi$ and
			$q\asymp\xi^{4/3}$ in the respective regions, these bounds control
			$|\mathscr A\Psi_1|+r^2$ by the right-hand side of (i).
			Since $|r|<1/2$, this proves the local residual estimate and its zero wall trace.
			
			\noindent\emph{Gluing and the global residual bound.}
			On $[R_\delta,2R_\delta]$, the preceding estimates and
			\eqref{defi-U-infty-1} give
			\begin{gather*}
				|Z^{(j)}|\lesssim_{M_0}\widetilde b_0^{(j-4/3)\delta}
				\quad(j=0,1,2),\qquad
				V=4a^2\widetilde b_0^{-4\delta/3}
				\bigl[1+O_{M_0,x_0}(\widetilde b_0^{5\delta/3})\bigr],\\*
				V^{(j)}=O_{M_0,x_0}(\widetilde b_0^{j-1+4j\delta/3})
				\quad(j=1,2).
			\end{gather*}
			Thus $Z<V$ there for $M_0$ and then $s_1$ large;
			on $[R_\delta,\infty)$, $0<V<U_\infty(s_0)$ and $V'>0$.
			Since $\chi_\delta'\le0$,
			\[
			(\Psi_1^0)'=\chi_\delta'(Z-V)+\chi_\delta Z'+(1-\chi_\delta)V'>0.
			\]
			Smooth gluing, the wall value, and $V\to U_\infty(s_0)$ prove the
			regularity and profile assertions. On the transition interval,
			$|\chi_\delta^{(j)}|\lesssim M_0^j\widetilde b_0^{j\delta}$ gives
			\[
			(\Psi_1^0)''=\chi_\delta''(Z-V)+2\chi_\delta'(Z'-V')
			+\chi_\delta Z''+(1-\chi_\delta)V'',\qquad
			|(\Psi_1^0)''|\lesssim_{M_0,x_0}\widetilde b_0^{2\delta/3}.
			\]
			As $\Psi_1^0\asymp_{M_0,x_0}\widetilde b_0^{-4\delta/3}$, this yields
			$|\mathcal F_1|\lesssim1$. For $\xi\ge2R_\delta$,
			$\Psi_1^0=V\le2x_0\widetilde b_0^{-1}$ and
			$V''=O(\widetilde b_0^{1+8\delta/3})$, so
			$\mathcal F_1=-2+O(\widetilde b_0^{1/2+8\delta/3})$.
			Finally, on $[\xi_*,R_\delta]$ the identity
			$\mathscr A\Psi_1=\sqrt q\,\Psi_1''+\Psi_1/q$ gives
			\[
			|\mathscr A\Psi_1|\lesssim
			\widetilde b_0^{1/2}\xi^{2/3}
			+\widetilde b_0^{3/4}\xi^{-2/3}\langle\xi\rangle^{8/3}\lesssim1,
			\]
			where the last bound uses $\delta\le3/8$.
			The exact residual identity and the local estimate below $\xi_*$
			complete the global bound in (i).
			
			\noindent\emph{Renormalization and the remainder.}
			Set $H_{b,\beta}=Q_b+\mathbf n_1(\beta)\sqrt b\,\psi_{1,b,\beta}$ and
			$\widetilde\varepsilon^0=(1-\chi_\delta)(V-Z)$, so that
			$\Psi_1^0=H_{\widetilde b_0,1}+\widetilde\varepsilon^0$.
			For $\mathcal R(b,\beta,\varepsilon)=H_{\widetilde b_0,1}+\varepsilon-H_{b,\beta}$,
			define
			$G=(\langle\mathcal R,\varphi_{1,b,\beta}\rangle_\rho,
			\langle\mathcal R,\varphi_{0,b,\beta}\rangle_\rho)^{\mathsf T}$,
			where $\rho=\rho_{b,\beta}$. The required orthogonality is
			$G(b_0,\beta_0,\widetilde\varepsilon^0)=0$.
			Since $\mathcal R(\widetilde b_0,1,0)=0$, derivatives of the test functions
			and weight do not contribute to the Jacobian at the reference datum.
			
			Use the scaled coordinates $b=\widetilde b_0(1+u)$, $\beta=1+v$, and
			$\mathcal G_{\widetilde b_0}(u,v)=\widetilde b_0^{-1/2}
			G(\widetilde b_0(1+u),1+v,\widetilde\varepsilon^0)$.
			The ground-state and spectral estimates give $C^1$ convergence of the
			map with $\widetilde\varepsilon^0=0$ on a fixed small ball.
			Differentiation under the integrals is justified by the uniform integrable
			majorant $C\xi^{-2/3}(1+\xi)^K e^{-c\xi^{4/3}}$.
			The tail and its parameter derivatives are exponentially small, since
			$\operatorname{supp}\widetilde\varepsilon^0\subset[R_\delta,\infty)$.
			Appendix~\ref{appen-jacobian} therefore gives
			\[
			D\mathcal G_{\widetilde b_0}(0,0)=\mathsf J+o(1),\qquad
			\mathsf J=\begin{pmatrix}-\nu/2&-\nu\\0&\kappa\nu\end{pmatrix},\qquad
			\det\mathsf J=-\nu\mathbf m_0(1)\ne0.
			\]
			Here $\kappa\nu=2\mathbf m_0(1)$, and the unnormalized projections
			have the same leading Jacobian because $\|\varphi_j\|_\rho^2=1+o(1)$.
			For $r_0$ small and $s_1$ large,
			$\sup_{|u|+|v|\le r_0}\|I-\mathsf J^{-1}D\mathcal G_{\widetilde b_0}\|\le1/2$,
			while $|\mathcal G_{\widetilde b_0}(0,0)|\lesssim
			\widetilde b_0^{-J}e^{-cR_\delta^{4/3}}$ for fixed $c,J>0$.
			Thus $(u,v)\mapsto(u,v)-\mathsf J^{-1}\mathcal G_{\widetilde b_0}(u,v)$
			is a contraction of this ball into itself, with a unique fixed point satisfying
			$|u|+|v|\lesssim\widetilde b_0^{-J}e^{-cR_\delta^{4/3}}$.
			Since $R_\delta^{4/3}=M_0^{-4/3}\widetilde b_0^{-4\delta/3}$,
			increasing $s_1(\delta,x_0,M_0)$ absorbs the polynomial losses and
			parameter-dependent prefactors, giving \eqref{distance-b-s-0-tildeb-0}
			with a universal implicit constant.
			Set $\varepsilon_-^0=\mathcal R(b_0,\beta_0,\widetilde\varepsilon^0)$.
			
			On $(0,R_\delta/8]$, the cutoff equals one, so
			$\varepsilon_-^0=H_{\widetilde b_0,1}-H_{b_0,\beta_0}$.
			Along the parameter segment $(b_\theta,\beta_\theta)$,
			$b_\theta\asymp b_0\asymp\widetilde b_0$ and $\beta_\theta\asymp1$.
			Proposition~\ref{diagonalisation-L-b} and
			Lemmas~\ref{lemma-Q} and~\ref{lemma-T-i} give
			\[
			|\partial_bH_{b_\theta,\beta_\theta}(\xi)|
			+|\partial_\beta H_{b_\theta,\beta_\theta}(\xi)|
			\lesssim\widetilde b_0^{-2}
			\begin{cases}
				\xi,&0<\xi\le b_0^{3/4},\\
				\xi^{2/3}\langle\xi\rangle^4,&\xi>b_0^{3/4}.
			\end{cases}
			\]
			The factor $\langle\xi\rangle^4$ accounts for the $\beta$-derivative
			of the spectral remainder. Let $\mathcal E_0$ denote the sum of norms in
			\eqref{initial-epsilon-minus-exponential}. Since
			$\langle\xi\rangle^4\lesssim\widetilde b_0^{-4\delta}$ on this interval,
			the mean value formula and the stronger Gaussian bound on $|u|+|v|$ above give
			\begin{align*}
				\mathcal E_0
				&\lesssim\widetilde b_0^{-2-4\delta}
				\bigl(|b_0-\widetilde b_0|+|\beta_0-1|\bigr)\\*
				&\le C_{x_0,M_0}\widetilde b_0^{-J-2-4\delta}
				e^{-cR_\delta^{4/3}}\le e^{-\widetilde b_0^{-\delta}/16},
			\end{align*}
			after increasing $s_1$. Indeed,
			$R_\delta^{4/3}\widetilde b_0^\delta
			=M_0^{-4/3}\widetilde b_0^{-\delta/3}\to\infty$,
			so the stronger exponential absorbs all polynomial losses and fixed prefactors.
			This proves \eqref{initial-epsilon-minus-exponential} with a universal implicit constant.
		\end{proof}
		\begin{remark}[Initial data for $(\mathbf P_{\rm stab})$]\label{remark-stable-initial-trapping}
			Fix $x_0>0$, $\delta=3/8$, and $M_0\ge M_1(x_0,\delta)$ as in
			Lemma~\ref{lemma-initial-data-l=1}. For $s_0\ge s_1(\delta,x_0,M_0)$, set
			$\mu_0=\widetilde b_0=e^{-s_0}$ as in \eqref{defi-mu-0-l=1}, let
			$\vartheta_{\rm in}=\Psi_1^0(s_0)$ be defined by
			\eqref{defi-initi-varthe-s-0}, and let $(b_0,\beta_0)$ be the unique pair
			selected in item~\textup{(ii)} of that lemma. Then the initial datum for
			$(\mathbf P_{\rm stab})$ is
			\[
			(\varepsilon_{\rm in},b_0,\beta_0),
			\qquad \varepsilon_{\rm in}=\vartheta_{\rm in}-Q_{b_0},
			\]
			or, equivalently, $(\vartheta_{\rm in},\mu_0)$ for
			\eqref{equa-vartheta}. Set $\eta_{20}:=\delta/16$.
			The estimates of Lemma~\ref{lemma-initial-data-l=1} give thresholds
			$A_{20}=A_{20}(M_0,x_0)\ge2$ and
			$s_{20}=s_{20}(A,\eta,M_0,x_0)\ge s_1$ such that, for
			$A\ge A_{20}$, $0<\eta\le\eta_{20}$, and $s_0\ge s_{20}$,
			\[
			(\vartheta_{\rm in},b_0,\beta_0)\in\mathcal V_1[A,\eta](s_0)
			\]
			with all defining inequalities strict. The threshold $s_{20}$ includes
			the dependence on $\eta$ needed for this inclusion.
			
			The corresponding physical initial profile is recovered by the inverse
			von Mises transformation
			\[
			w_{\rm in}(\psi)=\mu_0\vartheta_{\rm in}
			\!\left(\frac{\psi}{\mu_0^{3/4}}\right),\qquad
			y_{\rm in}(\psi)=\int_0^\psi\frac{d\psi'}{\sqrt{w_{\rm in}(\psi')}},
			\qquad
			u_{\rm in}(y_{\rm in}(\psi))=\sqrt{w_{\rm in}(\psi)}.
			\]
			Varying $s_0\ge s_1$ therefore produces a family of well-prepared
			physical initial data.
		\end{remark}
		
		\subsection{Propagation of the modulation parameters}
		We implement the local continuation argument outlined in
		Section~\ref{subsection-sketch-stable-regime}: first verify Oleinik's
		compatibility condition for the prepared data, then select the modulation
		parameters and continue them while the solution remains in the shrinking set.
		
		Let $X>0$ and $0<\alpha\leq1$. Assume that the outer data satisfy
		\eqref{equa-u_E-P}, with $u_E,p_E'\in C^1([0,X])$ and $u_E>0$.
		The corner compatibility condition is
		\begin{equation}\label{condition-in-u-locally-exsist}
			u_0''(y)-p_E'(0)=O(y^2)\qquad(y\to0^+).
		\end{equation}
		The corresponding Oleinik class is
		\begin{equation}\label{Olenik class}
			\mathcal K_\alpha:=\left\{u_0\in C_b^{2,\alpha}(\mathbb R_+):
			\begin{array}{l}
				u_0(0)=0,\quad u_0'(0)>0,\quad u_0(y)>0\ (y>0),\\
				u_0(y)\to u_E(0)\ (y\to\infty),\quad
				\text{\eqref{condition-in-u-locally-exsist} holds}
			\end{array}\right\}.
		\end{equation}
		We recall the local theory from \cite[Theorems~2.1.1 and~2.1.17]{OleinikSamok-book-99}.
		\begin{theorem}[Oleinik and Samokhin]\label{Olenik-thm-2.1.1}
			Under the preceding assumptions, every $u_0\in\mathcal K_\alpha$ generates
			a unique local Oleinik solution $(u,v)$ of \eqref{equa-u-P-E} on
			$[0,x_*)\times\mathbb R_+$ for some $0<x_*\leq X$.
			On every compact $x$-strip $[0,\bar x]\times\mathbb R_+$, with
			$\bar x<x_*$, the functions $u,u_y,u_{yy}$ are bounded and continuous,
			while $u_x,v,v_y$ are locally bounded and continuous.
			Moreover, $u(x,y)>0$ for $y>0$, and there exist $m_{\bar x},y_{\bar x}>0$
			such that $u_y\geq m_{\bar x}$ on $[0,\bar x]\times[0,y_{\bar x}]$.
		\end{theorem}
		For $p_E'=1$ and $u_E$ from \eqref{defi-u_E}, we also use the classical
		continuation criterion: before $x_0$, a finite maximal endpoint requires
		loss of wall nondegeneracy or interior positivity; see
		\cite[Section~1.1]{DalibardMasmoudi-IHES-19}.
		Thus continuation holds through $X<x_0$ if, uniformly for $x<X$,
		$u_y(x,0)\geq c>0$ and $\inf_{y\geq r}u(x,y)\geq c_r>0$ for every $r>0$.
		
		The weighted curvature term in \eqref{definition-norm-E} preserves the
		corner compatibility condition: if $|u_0''(y)-1|\le C_0y^2$ for $0<y\le1$,
		then
		\begin{equation*}
			|(u_0+h)''(y)-1|\le(C_0+\|h\|_{\mathcal E})y^2.
		\end{equation*}
		No limit of $h''(y)/y^2$ at zero is required. The following lemma gives
		the local persistence and dependence properties used below.
		
		\begin{lemma}[Local weighted bounds and dependence]\label{lem-local-weighted-dependence}
			Let $u$ be an Oleinik solution with initial datum in $\mathcal K_\alpha$,
			and let $[0,X]$, $X<x_0$, be a compact strip before its maximal endpoint.
			In von Mises variables, set
			\begin{equation}\label{physical-residual-local}
				R:=w_x=\sqrt w\,w_{\psi\psi}-2
				=2u_{yy}(x,y(x,\psi))-2.
			\end{equation}
			Then $|R(x,\psi)|\le C_X\min\{\psi,1\}$ on this strip, and
			\begin{equation}\label{local-weighted-time-profile}
				\sup_{\psi>0}\frac{|w(x,\psi)-w(t,\psi)|}{\min\{\psi,1\}}
				\le C_X|x-t|\qquad(0\le x,t\le X).
			\end{equation}
			For every $L>0$ and $0\le t<X$,
			\begin{equation}\label{local-residual-upper-semicontinuity}
				\limsup_{x\downarrow t}\sup_{0<\psi\le L}\frac{|R(x,\psi)|}{\psi}
				\le\sup_{0<\psi\le L}\frac{|R(t,\psi)|}{\psi}.
			\end{equation}
			Global upper and lower bounds for $R(t,\cdot)$ persist for a short time
			with any prescribed positive slack.
			
			If $u_{0,n},u_0\in\mathcal K_\alpha$, $u_{0,n}-u_0\in\mathcal O_\alpha$,
			and $\|u_{0,n}-u_0\|_{\mathcal E}\to0$, then the corresponding solutions
			exist on $[0,X]$ for all sufficiently large $n$, and
			\begin{equation}\label{local-weighted-data-dependence}
				\sup_{0\le x\le X,\,\psi>0}
				\frac{|w_n(x,\psi)-w(x,\psi)|}{\min\{\psi,1\}}\longrightarrow0.
			\end{equation}
			Their residuals converge uniformly on $[0,X]\times[r,L]$ whenever
			$0<r<L<\infty$.
		\end{lemma}
		\begin{proof}
			\emph{Propagation of compatibility.}
			Outer matching is uniform on the fixed strip. Indeed,
			$v=w-u_E^2$ satisfies $v_x-\sqrt w\,v_{\psi\psi}=0$; on $\psi\ge r$,
			compare it with $\pm(\epsilon+Ce^{Mx-(\psi-r)})$, first choosing $r$
			so that $|v(0,\psi)|\le\epsilon$ there, and then $C,M$ large enough
			for the boundary ordering and the differential inequalities.
			Together with positivity and wall nondegeneracy, this gives
			$w\ge c_X\min\{\psi,1\}$, while $c\psi\le w\le C\psi$ near zero
			and $R$ is bounded.
			The Oleinik construction also gives $|w_x|\le Cw^{3/4}$; see
			\cite[Lemma~2.1.13, with exponent $1/4$, and Theorem~2.1.14]{OleinikSamok-book-99}.
			Thus $R/\sqrt\psi\to0$ uniformly at the wall. Interior differentiation of
			\eqref{wvariable:1} gives
			\begin{equation}\label{physical-residual-local-equation}
				R_x=aR_{\psi\psi}+c_RR,\qquad
				a=\sqrt w,\qquad c_R=\frac{R+2}{2w}.
			\end{equation}
			Here $a\asymp\sqrt\psi$ and $|c_R|\le C/\psi$ near zero; both
			coefficients are bounded away from the wall, with $a$ bounded below there.
			Choose $k$ large and then $r>0$ small. On $(0,r]$,
			$H(\psi)=\psi-k\psi^{3/2}$ satisfies
			\begin{equation}\label{local-wall-comparison-weight}
				aH''+c_RH\le-\tfrac34k\sqrt c+C\le-1,
				\qquad \tfrac12\psi\le H\le\psi.
			\end{equation}
			Extend $H$ smoothly and positively, constant for large $\psi$, with
			$H\asymp\min\{\psi,1\}$. Then $aH''+c_RH\le MH$ globally for some
			$M<\infty$. Initial compatibility gives $|R(0)|\le BH$, and comparison
			yields $|R(x)|\le Be^{Mx}H$.
			
			This comparison does not assume the desired linear wall bound.
			Divide the comparison differences by a positive smooth function $g$
			equal to $\sqrt\psi$ near zero and to $1$ for large $\psi$.
			They tend to zero at the wall by the preceding Oleinik estimate.
			Their zeroth-order coefficient near zero is
			\begin{equation*}
				c_R+a\frac{g''}{g}
				=c_R-\frac{a}{4\psi^2}
				\le C\psi^{-1}-c_1\psi^{-3/2},
			\end{equation*}
			which is bounded above. The maximum principle therefore applies after
			an exponential change in $x$. At infinity, subtract
			$\epsilon e^{M_1x}(1+\psi^2)$ and let $\epsilon\downarrow0$.
			This proves the bound for $R$ and, by integration,
			\eqref{local-weighted-time-profile}.
			
			For \eqref{local-residual-upper-semicontinuity}, put
			$B_0=\sup_{0<\psi\le L}|R(t,\psi)|/\psi$. Given $\epsilon>0$,
			reduce $r<L$ so that $(B_0+\epsilon)H\ge(B_0+\epsilon/2)\psi$ on
			$(0,r]$. The initial ordering is strict, and continuity of $R(x,r)$
			preserves it at $\psi=r$ for a short time after $t$.
			By \eqref{local-wall-comparison-weight} and the same maximum principle,
			$|R(x,\psi)|\le(B_0+\epsilon)H(\psi)$ there. On $[r,L]$, ordinary
			parabolic continuity completes the assertion. For the global bounds,
			the bound $R=O(\psi)$ controls a fixed wall neighborhood. On its
			complement, $c_R$ is bounded; constant-in-space barriers with an
			exponential factor preserve the upper and lower bounds with any
			positive slack. Continuous changes of the thresholds are allowed.
			
			\emph{Dependence on the datum.}
			At the initial time,
			\begin{equation}\label{initial-weighted-von-mises-convergence}
				\sup_{\psi>0}\frac{|w_{0,n}(\psi)-w_0(\psi)|}{\min\{\psi,1\}}
				\longrightarrow0.
			\end{equation}
			Indeed, near zero $u_{0,n},u_0\asymp y$, and their inverse coordinates
			differ by $O(\|u_{0,n}-u_0\|_{\mathcal E}\sqrt\psi)$.
			On each compact interval away from zero the inverse maps converge
			uniformly. In the tail, uniform convergence of $u_{0,n}$ and their
			common positive limit make both squared profiles uniformly close to
			$u_E(0)^2$. No decay rate for the perturbations is needed.
			
			For $z=w_n-w$, subtraction gives
			\begin{equation}\label{physical-profile-difference-equation}
				z_x-\sqrt{w_n}\,z_{\psi\psi}=c_nz,\qquad
				c_n=\frac{R+2}{\sqrt w(\sqrt{w_n}+\sqrt w)}.
			\end{equation}
			As long as $w_n\ge w/2$, we have $|c_n|\le C/\psi$ near zero and
			bounded coefficients away from zero, uniformly in $n$. A uniform upper
			bound for $w_n$ follows from the maximum principle for
			\eqref{wvariable:1}. Construct $H$ as above, using the reference strip
			and these uniform bounds. Comparison gives
			\begin{equation}\label{weighted-profile-difference-estimate}
				\sup_{\psi>0}\frac{|z(x,\psi)|}{H(\psi)}
				\le e^{Mx}\sup_{\psi>0}\frac{|z(0,\psi)|}{H(\psi)}.
			\end{equation}
			The same division by $g$ justifies comparison at the wall, since
			$z=O(\psi)$; no estimate for $(R_n-R)/\psi$ is used.
			The reference solution satisfies $w\ge c_XH$ on $[0,X]$.
			For large $n$, \eqref{initial-weighted-von-mises-convergence} makes the
			right-hand side of \eqref{weighted-profile-difference-estimate} less
			than $c_X/4$. Weighted time continuity for each solution starts the
			comparison, and this strict improvement prevents loss of $w_n\ge w/2$.
			It preserves the wall shear and positivity away from the wall, so the
			continuation criterion extends each such solution through $X$.
			This proves \eqref{local-weighted-data-dependence}.
			
			On compact $\psi$-intervals away from zero, the equations are uniformly
			parabolic. Differentiation of the initial inverse maps gives convergence
			of $w_{0,n}$ in $C^2$ there and uniform $C^{2,\alpha}$ bounds.
			Local parabolic H\"older estimates followed by Schauder estimates give
			uniform bounds for $w_{n,x}$ and $w_{n,\psi\psi}$ on smaller intervals,
			including $x=0$. Compactness with a smaller H\"older exponent, and
			\eqref{local-weighted-data-dependence}, identify the limits with
			$w_x,w_{\psi\psi}$, proving the residual convergence.
		\end{proof}
		
		\noindent\emph{Admissibility of the initial data.}
		With $\mu_0=\widetilde b_0$ from \eqref{defi-mu-0-l=1}, set
		$w_0(\psi):=\mu_0\Psi_1^0(s_0,\psi/\mu_0^{3/4})$.
		The inverse von Mises transformation gives
		\begin{equation}\label{von-Mise-initial data}
			Y_0(\psi):=\int_0^\psi\frac{dt}{\sqrt{w_0(t)}},\qquad
			u_{\rm in}(Y_0(\psi))=\sqrt{w_0(\psi)}.
		\end{equation}
		Lemma~\ref{lemma-initial-data-l=1} makes $Y_0$ an increasing bijection of
		$\mathbb R_+$ and gives $u_{\rm in}(0)=0$, $u_{\rm in}>0$ for $y>0$,
		and $u_{\rm in}(y)\to u_E(0)$ as $y\to\infty$.
		The cancellation in \eqref{expression-Psi-1} yields
		$\mathfrak c_0:=w_0'(0)=(\mu_0\widetilde b_0)^{1/4}
		=\widetilde b_0^{1/2}$ and
		\begin{equation*}
			w_0(\psi)=\mathfrak c_0\psi+O(\psi^{3/2}),\qquad
			Y_0(\psi)=2\sqrt{\psi/\mathfrak c_0}\,(1+o(1)),\qquad
			\psi=\tfrac14\mathfrak c_0y^2(1+o(1)).
		\end{equation*}
		Writing $u_0=u_{\rm in}$ and $\xi=\psi/\mu_0^{3/4}$, differentiation of
		\eqref{von-Mise-initial data} gives
		\begin{equation*}
			u_0'(Y_0(\psi))=\tfrac12w_0'(\psi),\qquad
			u_0''(Y_0(\psi))=\tfrac12\sqrt{w_0(\psi)}\,w_0''(\psi)
			=1+\tfrac12\mathcal F_1(s_0,\xi).
		\end{equation*}
		Hence $u_0'(0)=\mathfrak c_0/2>0$, and the local residual estimate gives
		$u_0''(y)-1=O_{s_0}(y^2)$, as required by
		\eqref{condition-in-u-locally-exsist}.
		To check the remaining H\"older regularity, differentiate the exact residual
		identity in the proof of Lemma~\ref{lemma-initial-data-l=1} with
		$D=\xi\partial_\xi$. The eigenvalue identity for $\mathscr A\Psi_1$ uses
		only $D^k\Psi_1$ ($k\leq2$) and $D\varphi_0$, whose stated bounds give
		\begin{equation*}
			|\mathcal F_1|+|D\mathcal F_1|\lesssim
			\widetilde b_0^{1/4}\xi\quad(0<\xi\leq\widetilde b_0^{3/4}),\qquad
			u_0'''(y)=\frac{u_0(y)}{2\mu_0^{3/4}}
			\partial_\xi\mathcal F_1(s_0,\xi)=O_{s_0}(y).
		\end{equation*}
		Away from the wall, smooth gluing and the exponential tail in
		\eqref{defi-U-infty-1} give bounded derivatives through order three.
		In fact, the prepared datum is smooth at the wall. Write
		$Q(\zeta)=\zeta a_0(\sqrt\zeta)$, where $a_0$ is analytic and $a_0(0)=1$,
		as follows from \eqref{ground-state-Q}. For a Dirichlet eigenfunction,
		$\varphi(\zeta)=\zeta g(\sqrt\zeta)$ reduces its ODE to a regular singular
		equation with indicial roots $0,-2$; its regular branch is analytic.
		The inverse von Mises map is analytic in $\sqrt\psi$ with nonzero derivative
		at zero. Hence $u_0$ is analytic there; smooth gluing and the exponential tail
		give bounded fourth derivatives globally. Thus
		$u_0\in C_b^{3,1}\cap C^\infty\cap\mathcal K_1$, and
		Theorem~\ref{Olenik-thm-2.1.1} gives a local Oleinik solution.
		
		\begin{lemma}[Normalization]\label{lemma-normalization-propagation-flows}
			There exists $s_2=s_2(\delta,x_0,M_0)\ge s_1$ such that the following
			holds for $s_0\ge s_2$. Let $(b_0,\beta_0)$ be given by
			Lemma~\ref{lemma-initial-data-l=1}, and let $(u,v)$ be the local Oleinik
			solution with datum \eqref{von-Mise-initial data}. There is
			$\widetilde x_1>0$ and a locally unique triple
			$\widehat\mu\in C^2([0,\widetilde x_1])$,
			$(\widehat b,\widehat\beta)\in C^1([0,\widetilde x_1])$, with
			$\widehat\mu,\widehat b>0$ and initial value $(\mu_0,b_0,\beta_0)$.
			Under \eqref{chang-varialbe}, its pullback $(\mu,b,\beta)$ satisfies
			\eqref{relation-beta-mu-s} and, for some $s_{\rm loc}>s_0$,
			\begin{equation*}
				\vartheta=Q_b+\mathbf n_1(\beta)\sqrt b\,\psi_{1,b,\beta}
				+\varepsilon_-,\qquad
				\mathfrak F_{j,b,\beta}(\varepsilon_-)=0\quad(j=0,1),\qquad
				s\in[s_0,s_{\rm loc}].
			\end{equation*}
			Here the adapted direction and projections are defined in
			\eqref{definition-adapted-eigendirection-strategy} and \eqref{defi-F-j}.
		\end{lemma}
		\begin{proof}
			For a trial scale $\sigma>0$, write
			$\Theta_{x,\sigma}(\xi):=\sigma^{-1}w(x,\sigma^{3/4}\xi)$ and
			$E:=\Theta_{x,\sigma}-Q_b$. Define
			\begin{equation}\label{defi-G-modulation}
				G(x,\sigma,b,\beta):=
				\begin{pmatrix}
					\mathfrak F_{1,b,\beta}(E)-\mathbf n_1(\beta)\sqrt b\\
					\mathfrak F_{0,b,\beta}(E)+2\mathbf m_0(\beta)\sqrt b
				\end{pmatrix}.
			\end{equation}
			The coefficients are given by \eqref{defi-m-0}--\eqref{defi-n-ell-1}.
			Their relation $(c_{1,0}/c_{0,0})\mathbf n_1=2\mathbf m_0$ and spectral
			orthogonality make $G=0$ equivalent to the required decomposition.
			In particular, $G(0,\mu_0,b_0,\beta_0)=0$.
			
			We justify differentiation of this map before applying the implicit
			function theorem. Put
			$K_j:=\varphi_{j,b,\beta}\rho_{b,\beta}/
			\|\varphi_{j,b,\beta}\|_{L^2_{\rho_{b,\beta}}}^2$ for $j=0,1$.
			Proposition~\ref{diagonalisation-L-b}, Lemma~\ref{lemma-Q}, and
			\eqref{defi-rho-b-beta} yield
			\begin{equation}\label{projection-kernel-parameter-bound}
				|K_j|+|b\partial_bK_j|+|\partial_\beta K_j|
				\leq C\xi^{-2/3}(1+\xi)^M e^{-c\xi^{4/3}},
				\qquad 0<b\ll1,\quad \tfrac12\leq\beta\leq\tfrac32.
			\end{equation}
			Here $C,M,c>0$ are fixed. The profile and spectral bounds control the behavior at zero;
			parameter differentiation of the density adds only polynomial factors at
			infinity. At fixed $\psi$, the von Mises equation gives
			$w_\psi=2u_y$ and $w_x=2u_{yy}-2$, both bounded on compact $x$-strips.
			Also $\partial_\sigma\Theta_{x,\sigma}
			=-\sigma^{-1}\Lambda_\xi\Theta_{x,\sigma}$.
			These identities and \eqref{projection-kernel-parameter-bound} provide
			integrable bounds for the first derivatives, so $G$ is $C^1$ in a
			one-sided neighborhood of the base point.
			
			Appendix~\ref{appen-jacobian} computes the leading Jacobian at a pure
			modulated profile. The initial remainder contributes $o(b^{-1/2})$ to the
			$b$-column and $o(b^{1/2})$ to the $\beta$-column, by
			Lemma~\ref{lemma-initial-data-l=1} and
			\eqref{projection-kernel-parameter-bound}. Thus, at the base point,
			\begin{equation*}
				\det D_{(b,\beta)}G=-\frac{\mathbf n_1(\beta)\mathbf m_0(\beta)}{\beta}+o(1).
			\end{equation*}
			Choose $s_2\ge s_1$ so that this determinant is nonzero for $s_0\ge s_2$. The one-sided $C^1$ implicit
			function theorem therefore yields locally unique $C^1$ functions
			$(\mathcal B,\mathcal C)(x,\sigma)$ with
			$G(x,\sigma,\mathcal B,\mathcal C)=0$ and
			$(\mathcal B,\mathcal C)(0,\mu_0)=(b_0,\beta_0)$.
			
			Solve $\widehat\mu_x=-\mathcal C(x,\widehat\mu)$ with
			$\widehat\mu(0)=\mu_0$, and set
			$\widehat b=\mathcal B(x,\widehat\mu)$,
			$\widehat\beta=\mathcal C(x,\widehat\mu)$.
			Local Lipschitz continuity in $\sigma$ gives a unique solution;
			after shortening the interval, $\widehat\mu,\widehat b>0$.
			Since the right-hand side is $C^1$, $\widehat\mu$ is $C^2$ and
			$\widehat b,\widehat\beta$ are $C^1$.
			Define $s(x)=s_0+\int_0^x\widehat\mu(t)^{-1}\,dt$ and pull the parameters
			back by its inverse. Then $x_s=\mu$, $\mu_s=-\beta\mu$, and $G=0$
			gives the stated decomposition. Uniqueness follows from the implicit
			function theorem and uniqueness of the scalar initial-value problem.
		\end{proof}
		
		\begin{corollary}[Local persistence and parameter continuity]\label{cor-local-trapping-continuity}
			At a finite regular time, inequalities defining
			$\mathcal V_1[A,\eta]$, or the non-discrete part of
			$\mathcal V_\ell[A,\eta]$, with a positive margin in each normalized
			bound persist on a short subsequent interval.
			If the discrete inequalities are also strict, the solution remains
			in the corresponding shrinking set on that interval.
			For data converging in \eqref{definition-norm-E}, use common initial
			scale and time and the locally unique normalization near the reference.
			The normalized parameters and finite-mode coefficients converge uniformly on each
			finite renormalized interval on which the reference normalization
			exists and its Jacobian is invertible.
		\end{corollary}
		\begin{proof}
			Use $\mathcal F(s,\xi)=R(x(s),\mu(s)^{3/4}\xi)$.
			Equation~\eqref{local-weighted-time-profile} and continuity of the
			normalization give weighted continuity of $\varepsilon_-$ at the wall.
			Equation~\eqref{local-residual-upper-semicontinuity} preserves the residual
			bound with any strict margin. Away from zero, parabolic continuity
			handles compact regions and the moving endpoints $b^{3/4}$ and
			$b^{-\eta}$. Each term in $\|\mathcal F\|_{\mathfrak X_{b,\eta}}$
			is consequently upper semicontinuous from the right, as is their sum.
			The global residual bounds persist by Lemma~\ref{lem-local-weighted-dependence}.
			The far-field profile bounds follow from uniform continuity of $w$ in
			time, its prescribed limit at infinity, and continuity of the scale.
			All other inequalities involve continuous parameters and finite modes.
			
			For dependence, apply \eqref{local-weighted-data-dependence} on a
			slightly longer physical strip than the one corresponding to the chosen
			renormalized interval. The bounds \eqref{projection-kernel-parameter-bound}
			give uniform convergence of $G_n$ and its $(b,\beta)$ derivatives on
			compact parameter neighborhoods. For the scale derivative, integrate
			by parts:
			\begin{equation*}
				\partial_\sigma\!\int_0^\infty\Theta_{x,\sigma}K_j\,d\xi
				=-\sigma^{-1}\!\int_0^\infty
				\Theta_{x,\sigma}\bigl(K_j+\tfrac34(\xi K_j)'\bigr)\,d\xi.
			\end{equation*}
			The boundary terms vanish. The eigenfunction and density estimates
			make this kernel integrable against the profiles, uniformly on those
			neighborhoods. The implicit function theorem gives convergent maps
			$(\mathcal B_n,\mathcal C_n)$, locally uniformly Lipschitz in $\sigma$.
			Continuous dependence for $\widehat\mu_x=-\mathcal C_n(x,\widehat\mu)$
			and for the inverse time change yields convergence of
			$(\mu_n,b_n,\beta_n)$. A finite covering gives the whole interval.
			For $\ell\ge2$, use the scalar constraint of
			Remark~\ref{remark-unstable-continuation}; the scale and time change
			are fixed. The finite-mode coefficients converge by their integral
			formulas. No continuity of the physical flow in the full norm
			$\mathcal E$ is used.
		\end{proof}
		
		The next result rules out a finite breakdown of this normalized flow
		while the shrinking-set bounds hold.
		\begin{lemma}[Propagation of the modulation decomposition]\label{lem-modulation}
			Fix $A\geq1$ and $0<\eta<\delta/8$, with $0<\delta\leq3/8$ as in
			Lemma~\ref{lemma-initial-data-l=1}. There exists
			$s_3=s_3(A,\eta,\delta,M_0,x_0)\ge s_2$ such that, for $s_0\ge s_3$, if
			a normalized Oleinik solution, with $\mu(s_0)=e^{-s_0}$, satisfies
			$\bigl(\vartheta(s),b(s),\beta(s)\bigr)\in\mathcal V_1[A,\eta](s)$
			for $s_0\leq s<\mathfrak s^*<\infty$,
			then the solution and its modulation decomposition extend beyond
			$\mathfrak s^*$. The continued parameters have the regularity of
			Lemma~\ref{lemma-normalization-propagation-flows} and are locally unique.
		\end{lemma}
		\begin{proof}
			\noindent\emph{Continuation of the physical solution.}
			Choose $s_3\ge s_2$ so that $A e^{-\eta s_3}\leq1/2$ and $2e^{-s_3}<x_0$.
			By \eqref{relation-beta-mu-s}, \eqref{stable-reference-scale}, and the
			shrinking-set bounds,
			\begin{equation*}
				\mu_0e^{-3(s-s_0)/2}\leq\mu(s)\leq\mu_0e^{-(s-s_0)/2},\qquad
				\tfrac12\mu(s)\leq b(s)\leq2\mu(s),\qquad
				x(s)=\int_{s_0}^s\mu(t)\,dt<2\mu_0.
			\end{equation*}
			Thus $\mu$ and $b$ stay bounded away from zero on this finite interval,
			and $x^*:=\lim_{s\uparrow\mathfrak s^*}x(s)<x_0$ for $s_0\ge s_3$.
			The cancellation in the adapted direction gives
			$\partial_\xi\psi_{1,b,\beta}(0)=0$.
			Using the near-wall bound for $\varepsilon_-$ and $Q_b'(0)=b^{1/4}$,
			we obtain
			\begin{equation*}
				\vartheta_\xi(s,0)=b^{1/4}\bigl(1+O(A^2b^\eta)\bigr),\qquad
				u_y(x(s),0)=\tfrac12(\mu b)^{1/4}\bigl(1+O(A^2b^\eta)\bigr)
				\geq c_{\mathfrak s^*}>0.
			\end{equation*}
			The profile estimates and the two regional bounds for $\varepsilon_-$
			also give
			\begin{equation*}
				\frac{|\vartheta-Q_b|}{Q_b}
				\lesssim A^2b^\eta+b^{1/4}\quad(0<\xi\leq b^{-\eta}),\qquad
				\vartheta\geq A^{-1}b^{-\eta/2}\quad(\xi\geq b^{-\eta}).
			\end{equation*}
			For the first estimate, on $b^{3/4}\leq\xi\leq b^{-\eta}$ the remainder
			contributes at most $C A^2b^{1/2+\eta}\xi^{-2/3}\langle\xi\rangle^{8/3}$;
			the spectral contribution is controlled as in the inner estimates of
			Lemma~\ref{lemma-initial-data-l=1}. The restriction on $\eta$ makes both
			uniformly small. In particular, $\vartheta\geq Q_b/2$ in this region.
			
			The global bounds on $\vartheta$ and $\mathcal F$ imply
			\begin{equation*}
				\|u(x(s))\|_\infty^2\leq A\mu b^{-1-\eta},\qquad
				\|u_{yy}(x(s))\|_\infty
				\leq1+\tfrac A2(1+b^{-\eta/8}).
			\end{equation*}
			Both are uniform up to $\mathfrak s^*$. The wall shear lower bound and
			the curvature bound give $u(x,y)\geq c y$ on a fixed wall strip.
			For each $r>0$, this bounds $\psi(x,r)=\int_0^r u(x,y)\,dy$ away from
			zero; the preceding lower bounds for $\vartheta$ then give
			$\inf_{s<\mathfrak s^*,\,y\geq r}u(x(s),y)>0$.
			The continuation criterion following Theorem~\ref{Olenik-thm-2.1.1}
			therefore extends the physical solution through $x^*$.
			
			\noindent\emph{Uniform invertibility and continuation of the parameters.}
			At a normalized profile, write
			$H_{b,\beta}:=Q_b+\mathbf n_1(\beta)\sqrt b\,\psi_{1,b,\beta}$ and
			$R:=\Theta_{x,\sigma}-H_{b,\beta}=\varepsilon_-$.
			The components of \eqref{defi-G-modulation} are $\int RK_j$,
			ordered as $j=1,0$. At fixed $(x,\sigma)$, for $p=b,\beta$,
			\begin{equation*}
				\partial_pG_j=-\int_0^\infty(\partial_pH_{b,\beta})K_j\,d\xi
				+\int_0^\infty\varepsilon_-\partial_pK_j\,d\xi.
			\end{equation*}
			Split the second integral at $b^{3/4}$ and $b^{-\eta}$.
			The two local bounds for $\varepsilon_-$ and
			\eqref{projection-kernel-parameter-bound} give
			\begin{equation*}
				\int_0^{b^{-\eta}}|\varepsilon_-|
				\bigl(|b\partial_bK_j|+|\partial_\beta K_j|\bigr)\,d\xi
				\lesssim A^2b^{1/2+\eta}.
			\end{equation*}
			Beyond $b^{-\eta}$, use
			$|\varepsilon_-|\leq\vartheta+|H_{b,\beta}|$ and
			$\vartheta\leq A b^{-1-\eta}$: the weighted integral is bounded by
			$C_A b^{-J}e^{-c b^{-4\eta/3}}$ for a fixed $J>0$, and hence is smaller
			than every power of $b$. Scaling the two columns by $\sqrt b$ and $b^{-1/2}$, respectively,
			gives, uniformly along the trapped flow,
			\begin{equation*}
				D_{(b,\beta)}G
				\begin{pmatrix}\sqrt b&0\\0&b^{-1/2}\end{pmatrix}
				=\begin{pmatrix}
					-\mathbf n_1(\beta)/2&-\mathbf n_1(\beta)/\beta\\
					0&2\mathbf m_0(\beta)/\beta
				\end{pmatrix}+O(A^2b^\eta)+o_{b\to0}(1).
			\end{equation*}
			The last error comes from the spectral asymptotics and is uniform for
			$\beta\in[1/2,3/2]$. Since $b\leq2e^{-s_0}$, the determinant stays
			bounded away from zero for $s_0\ge s_3$, after increasing $s_3$.
			
			Finally, $\widehat\mu_x=-\widehat\beta$ gives a positive limit of
			$\widehat\mu$ at $x^*$. The extended physical solution and
			\eqref{projection-kernel-parameter-bound} bound $G_x$ and $G_\sigma$
			on the resulting compact parameter set. Differentiating $G=0$ yields
			\begin{equation*}
				D_{(b,\beta)}G
				\begin{pmatrix}\widehat b_x\\\widehat\beta_x\end{pmatrix}
				=-G_x+\widehat\beta G_\sigma.
			\end{equation*}
			Its uniformly bounded inverse gives bounded derivatives and limits of
			$\widehat b,\widehat\beta$ at $x^*$, where $G=0$ by continuity.
			The implicit function theorem and the scalar equation for
			$\widehat\mu$ now extend all parameters uniquely through $x^*$.
			Since $\widehat\mu(x^*)>0$, the change of time extends the normalized
			solution beyond $\mathfrak s^*$ and preserves its orthogonality conditions.
		\end{proof}
		
		\begin{remark}[Continuation at a finite endpoint]
			The finite endpoint in Lemma~\ref{lem-modulation} is a regular continuation
			point: the shrinking bounds keep the wall shear and interior profile positive,
			and bound $u$ and $u_{yy}$ uniformly. The continuation criterion cited above
			therefore supplies an admissible trace with Euler value $u_E(x^*)$ and the
			propagated wall compatibility condition. This argument uses that criterion;
			uniform $C_b^{2,1}$ control is not inferred from interior parabolic regularity.
		\end{remark}
		
		\subsection{A priori estimates for the stable shrinking set}
		The following strict improvements close the stable bootstrap.
		\begin{proposition}[Bootstrap improvement]\label{propo-priori-ell-1}
			Fix $\delta=3/8$ in the prepared datum of Lemma~\ref{lemma-initial-data-l=1}.
			There exist $A_4=A_4(M_0,x_0)\ge\max\{4,A_{20}\}$,
			$\eta_4(A)\in(0,\eta_{20}]$, and
			$s_4=s_4(A,\eta,M_0,x_0)\ge s_3$ with the following property.
			Let $A\ge A_4$, $0<\eta\le\eta_4(A)$, and $s_0\ge s_4$.
			Suppose that $\vartheta$ solves \eqref{equa-vartheta}, that
			\[
			(\vartheta,b,\beta)(s)\in\mathcal V_1[A,\eta](s)
			\qquad\text{for every }s\in[s_0,\bar s],
			\]
			where $\bar s>s_0$, and that $\vartheta(s_0)=\Psi_1^0$, with
			$\Psi_1^0$ defined in \eqref{defi-initi-varthe-s-0}.
			Then the following improved estimates hold on $[s_0,\bar s]$.
			\begin{itemize}
				\item[(i)] \emph{Estimates for the modulation parameters.}
				\begin{equation*}
					\frac34
					\le
					b(s)\Bstab^{-1}(s)
					\le
					\frac32,\qquad
					|\beta(s)-1|
					\lesssim
					b^{1/4}(s_0),
				\end{equation*}
				where $b_0=b(s_0)$ is selected by Lemma~\ref{lemma-initial-data-l=1}. 
				\item[(ii)] \emph{Profile bounds.}
				For $\xi\ge0$, $\vartheta(s,\xi)\le(A/2)b^{-1-\eta}$; for
				$\xi\ge b^{-\eta}$, $\vartheta(s,\xi)\ge(2/A)b^{-\eta/2}$.
				
				\item[(iii)] \emph{Estimate for $\varepsilon_-$.}
				\[
				\|\varepsilon_-(s)\|_{\mathfrak X_{b,\eta}}\le Ab^{1/2+\eta}.
				\]
				
				\item[(iv)] \emph{Estimate for $\mathcal F$.}
				\[
				-\frac A2\le\mathcal F(s,\xi)\le\frac A2b^{-\eta/8},\qquad
				\|\mathcal F(s)\|_{\mathfrak X_{b,\eta}}\le\frac A2b^{1/2-\eta}.
				\]
			\end{itemize}
		\end{proposition}
		\begin{proof}
			Fix $K\ge\max\{K_{14},K_{16},K_{17},2\}$ once and for all.
			With $C_0$ from \eqref{definition-m0-F}, choose
			\[
			\begin{aligned}
				A_4&\ge\max\{4,A_{20},A_{14}(K),A_{16}(K),A_{17}(K),2C_0\},\\*
				\eta_4(A)&\le\min\{\eta_{20},\eta_{14}(A,K),\eta_{16}(A,K),\eta_{17}(A,K)\},\\*
				s_4&\ge\max\{s_3,s_{14},s_{16},s_{17},s_{20}\}.
			\end{aligned}
			\]
			The subsidiary thresholds in these results are included through their stated
			dependencies, with $\delta=3/8$. The fixed choice of $K$ is suppressed in $s_4$.
			We outline the argument, referring to the subsequent sections for the estimates.
			
			\medskip
			\noindent
			\emph{Proof of (i).}
			Set $r_b=b_s/b+\beta$. For $s_0\ge s_4$, the modulation estimate and
			$\beta\in[3/4,5/4]$ give $b_s/b\le-1/2$. Hence
			\[
			\int_{s_0}^s b(\sigma)^{1/4}\,d\sigma\le8b(s_0)^{1/4}.
			\]
			Since $\partial_s\log(b/\Bstab)=r_b$ and the prepared datum satisfies
			$b(s_0)e^{s_0}=1+o(1)$, integration gives
			$3/4\le b/\Bstab\le3/2$ after increasing $s_4$.
			Also $|\beta(s)-1|\lesssim b(s_0)^{1/4}$, which strictly improves
			$|\beta-1|\le Ae^{-\eta s_0}$ when $\eta<1/4$.
			\medskip\noindent
			\emph{Proof of (ii).}
			This is an immediate consequence of the comparison principle established in
			Lemma~\ref{lem-comparison-vartheta}.
			
			\medskip
			\noindent
			\emph{Proof of (iii).}
			Proposition~\ref{lem-improved-bootstrap-epsilon-negative-pointwise} gives the stated norm bound, which strictly improves $A^2b^{1/2+\eta}$ for $A\ge2$. Its proof combines the inner comparison estimate, after subtracting \eqref{outer-estimates-epsilon-plus}, with the intermediate and outer estimates of Propositions~\ref{propo-priori-b-eta-K} and~\ref{pro-priori-epsilon-xi-large}.
			
			\medskip\noindent
			\emph{Proof of (iv).}
			Lemmas~\ref{lem-priori-F-K}, \ref{lem-esti-F-0-1/K}, and~\ref{lem-esti-F-outer-region} bound each weighted component by $(A/4)b^{1/2-\eta}$ after increasing $A$; summing gives the norm bound. Lemma~\ref{lem-esti-F-le-A} supplies the global bounds.
			
			\medskip
			\noindent 
			This completes the proof of the proposition.
		\end{proof}
		
		\begin{proposition}[Bootstrap improvement for perturbed data]\label{propo-stable-perturbed-data}
			Fix the preparation parameters with $\delta=3/8$ and the spatial cutoff $K$
			chosen in the proof of Proposition~\ref{propo-priori-ell-1}.
			There exist $A_{21}=A_{21}(M_0,x_0)\ge A_4$,
			$\eta_{21}(A)\in(0,\min\{\eta_4(A),\eta_4(2A),1/16\}]$, and
			\[
			s_{21}=s_{21}(A,\eta,M_0,x_0)
			\ge\max\{s_4(A,\eta,M_0,x_0),s_4(2A,\eta,M_0,x_0)\}
			\]
			such that the following holds for $A\ge A_{21}$,
			$0<\eta\le\eta_{21}(A)$, and $s_0\ge s_{21}$.
			Let $\widehat u_0$ be the physical datum associated with
			$\Psi_1^0(s_0)$ by \eqref{von-Mise-initial data}, and fix $0<\alpha\le1$.
			There exists $\gamma=\gamma(s_0)>0$, also depending on the fixed parameters,
			such that, for every
			\begin{equation*}
				h\in\mathcal O_\alpha,\qquad \|h\|_{\mathcal E}<\gamma,
				\qquad u_0=\widehat u_0+h\in\mathcal K_\alpha,
			\end{equation*}
			the initial profile at scale $\mu_0=e^{-s_0}$ admits a unique nearby stable
			normalization $(b_0^h,\beta_0^h)$ and lies strictly in
			$\mathcal V_1[A,\eta](s_0)$.
			If a corresponding classical solution of \eqref{equa-vartheta} and
			\eqref{defi-stable-problem} remains in $\mathcal V_1[2A,\eta](s)$
			on $[s_0,\bar s]$, with $\bar s>s_0$, then all conclusions of
			Proposition~\ref{propo-priori-ell-1} hold with $A$ replaced by $2A$
			and $b(s_0)=b_0^h$.
			The estimates are uniform over this perturbation ball and independent of
			$\bar s$.
		\end{proposition}
		\begin{proof}
			Write $\widetilde b_0=\mu_0=e^{-s_0}$ and denote the normalized parameters
			of the reference datum by $(\widehat b_0,\widehat\beta_0)$.
			All constants denoted by $C_{s_0}$ may depend on the fixed preparation
			parameters, but not on $h$ or the terminal time.
			All further lower bounds on $A$ are absorbed into $A_{21}$ before choosing
			$\eta_{21}$ and $s_{21}$. The latter include the requirements of the
			estimates below at both $A$ and $2A$; the perturbation radius is chosen last.
			We also impose $4e^{-s_{21}}<x_0$ for the physical reconstruction.
			
			\emph{Initial profiles in the von Mises variables.}
			Put $\psi_h(y)=\int_0^yu_0(z)\,dz$ and $y_h=\psi_h^{-1}$, with hats
			for the reference datum. At the fixed scale $\mu_0$, the exact identities are
			\begin{equation}\label{perturbed-initial-von-mises}
				\vartheta_0^h(\xi)=\mu_0^{-1}u_0\!\left(y_h(\mu_0^{3/4}\xi)\right)^2,
				\qquad
				\mathcal F_0^h(\xi)=2u_0''\!\left(y_h(\mu_0^{3/4}\xi)\right)-2.
			\end{equation}
			Near the wall, $u_0,\widehat u_0\asymp y$ uniformly for small
			$\|h\|_{\mathcal E}$, and $|\psi_h(y)-\widehat\psi(y)|\le
			C\|h\|_{\mathcal E}y^2$. Inverting gives
			$|y_h(\psi)-\widehat y(\psi)|\le C_{s_0}\|h\|_{\mathcal E}\sqrt\psi$.
			Use $|h''(y)|\le\|h\|_{\mathcal E}y^2$ and
			$\widehat u_0'''(y)=O_{s_0}(y)$ in the residual identity.
			Away from the wall, positivity and smallness of $\|h\|_{\mathcal E}$ give
			$|y_h(\psi)-\widehat y(\psi)|\le C_{s_0}\|h\|_{\mathcal E}(1+\psi)$;
			the two inverse coordinates are comparable, so the exponential decay of
			the reference derivatives absorbs this coordinate
			shift. Thus, with $\widehat\vartheta_0=\Psi_1^0(s_0)$ and its residual
			$\widehat{\mathcal F}_0$,
			\begin{equation}\label{perturbed-initial-profile-residual-bound}
				|\vartheta_0^h-\widehat\vartheta_0|
				+|\mathcal F_0^h-\widehat{\mathcal F}_0|
				\le C_{s_0}\|h\|_{\mathcal E}\min\{\xi,1\}.
			\end{equation}
			The projection kernels in \eqref{projection-kernel-parameter-bound} and
			the implicit-function argument of Lemma~\ref{lemma-normalization-propagation-flows}
			give a unique nearby normalized pair, with
			\begin{equation}\label{perturbed-initial-normalization-bound}
				\left|\frac{b_0^h}{\widehat b_0}-1\right|
				+|\beta_0^h-\widehat\beta_0|\le C_{s_0}\|h\|_{\mathcal E}.
			\end{equation}
			Indeed, \eqref{perturbed-initial-profile-residual-bound} controls the change
			of the two constraint functionals, and their parameter Jacobian is
			invertible at the reference datum.
			
			\emph{Comparison with Lemma~\ref{lemma-initial-data-l=1}.}
			Set $\varepsilon_0^h=\vartheta_0^h-Q_{b_0^h}$ and let
			$\varepsilon_{-,0}^h$ be its stable remainder.
			The parameter bounds in Proposition~\ref{diagonalisation-L-b}, including
			the cancellation in the adapted direction, give, after enlarging $C_{s_0}$,
			\begin{equation}\label{perturbed-initial-remainder-difference}
				|\varepsilon_{-,0}^h-\widehat\varepsilon_{-,0}|
				\le d_h\begin{cases}\xi,&0<\xi\le1,\\ \xi^{14/3},&\xi\ge1,
				\end{cases}
				\qquad d_h:=C_{s_0}\|h\|_{\mathcal E}.
			\end{equation}
			For the fixed projection $P_-$ of Section~\ref{section-Pointwise-Duhamel},
			the same estimates give
			$|P_-(\varepsilon_0^h-\widehat\varepsilon_0)|\le d_h\mathfrak f_1$,
			after a further enlargement of $C_{s_0}$.
			Choose $\gamma(s_0)$ within the normalization neighborhood and so small that
			\begin{equation}\label{choice-perturbation-radius-stable}
				C_{s_0}\gamma(s_0)\le
				\exp\!\left(-\widetilde b_0^{-\delta}/32\right).
			\end{equation}
			On $[0,R_\delta/8]$, all polynomial losses in
			\eqref{perturbed-initial-remainder-difference} are absorbed by weakening
			this exponential to $e^{-\widetilde b_0^{-\delta}/64}$.
			Consequently the normalization and remainder estimates
			\eqref{distance-b-s-0-tildeb-0}--\eqref{initial-epsilon-minus-exponential}
			hold for $(b_0^h,\beta_0^h,\varepsilon_{-,0}^h)$ with this weaker exponential.
			The residual estimates in Lemma~\ref{lemma-initial-data-l=1}(i) also
			persist, with a fixed enlarged constant, by
			\eqref{perturbed-initial-profile-residual-bound}.
			In particular, these bounds and the initial profile comparisons below put
			the normalized datum strictly in $\mathcal V_1[A,\eta](s_0)$.
			
			\emph{Initial terms in Section~\ref{section-Pointwise-Duhamel}.}
			For each fixed $L$, Lemma~\ref{lem-simigroup-L-infty-beta=2} gives
			\begin{equation*}
				|S_1(t)P_-(\varepsilon_0^h-\widehat\varepsilon_0)|
				\lesssim_L d_h\mathfrak f_1\qquad(0\le t\le L).
			\end{equation*}
			Thus the additional initial term in \eqref{Duhamel-fixed-beta-one} is
			smaller than every power of $b_0^h$, including after division by
			$(b_0^h)^{1/2+\eta}$. For the residual, put $q=1/2-\eta$.
			Equation~\eqref{perturbed-initial-profile-residual-bound} gives
			\begin{equation*}
				(b_0^h)^{-q}\|\mathcal F_0^h-\widehat{\mathcal F}_0\|_{L^2_{\rho_{0,1}}}
				\lesssim d_h(b_0^h)^{-q}\longrightarrow0\quad(s_0\to\infty),
			\end{equation*}
			uniformly over the chosen perturbation balls; the same holds locally
			uniformly in $\xi$. Hence the zero initial trace in the compactness proof
			of Lemma~\ref{lem-priori-F-K} remains valid even when $s_0\to\infty$
			and $h$ varies with $s_0$. Its weighted tails, the restart estimate, and
			the forcing bounds use only the shrinking-set bounds and the modulation
			equations. Propositions~\ref{pro-finite-projections} and
			\ref{propo-priori-b-eta-K} and Lemma~\ref{lem-priori-F-K} therefore apply
			with the same uniform improvements at level $2A$.
			
			\emph{Initial comparison inequalities in Sections~8--9.}
			All barriers in this paragraph are taken at level $2A$.
			For Lemma~\ref{lem-comparison-vartheta}, the reference datum is increasing
			and $\widehat\vartheta_0((b_0^h)^{-\eta})\gtrsim(b_0^h)^{-4\eta/3}$.
			Equation~\eqref{perturbed-initial-profile-residual-bound} consequently gives
			\begin{equation*}
				\vartheta_0^h\ge2(b_0^h)^{-\eta/2}\quad(\xi\ge(b_0^h)^{-\eta}),
				\qquad \sup_\xi\vartheta_0^h\lesssim\widetilde b_0^{-1}
				\ll2(b_0^h)^{-1-\eta}.
			\end{equation*}
			This supplies the initial orderings without a monotonicity assumption on
			the perturbed datum. The inner profile ordering follows from the preceding
			remainder estimate and \eqref{outer-estimates-epsilon-plus}, exactly as in
			Proposition~\ref{propo-priori-near-0}.
			For the outer upper ordering of Proposition~\ref{pro-priori-epsilon-xi-large},
			the initial interval extends to
			$\Xi_+(s_0)=C_*(b_0^h)^{-\nu_+}$, $\nu_+=(9+12\eta)/20$.
			The change in the remainder, divided by its barrier
			$(2A)^{-3}(b_0^h)^{1/2+\eta}\phi_{\beta_0^h,2}$, is at most
			\begin{equation*}
				C_{A,K}d_h(b_0^h)^{-1/2-\eta-4\nu_+/3}=o(1).
			\end{equation*}
			The estimates proving the reference initial upper ordering hold on a fixed
			enlargement of this interval, including the gluing region, with a fixed
			margin; changing its parameters by
			\eqref{perturbed-initial-normalization-bound} preserves that margin.
			The lower ordering holds up to $\Xi_-(s_0)<R_\delta/8$ by the exponential
			remainder estimate.
			Finally, the residual bounds from Lemma~\ref{lemma-initial-data-l=1}(i)
			give the initial orderings for Lemmas~\ref{lem-esti-F-le-A},
			\ref{lem-esti-F-0-1/K}, and~\ref{lem-esti-F-outer-region}: the global
			bound is uniform, the inner bound is smaller than
			$(b_0^h)^{1-\eta}\mathcal Z_{1/2}$ by a factor $O((b_0^h)^\eta)$,
			and the outer bound is smaller than
			$(b_0^h)^{1/2-\eta}\phi_{\beta_0^h,2}$ by the same factor on
			$[K,(b_0^h)^{-3/20}]\subset[0,R_\delta]$.
			The wall compatibility follows directly from
			\eqref{perturbed-initial-von-mises} and $u_0''(y)-1=O(y^2)$.
			
			These are all uses of the prepared datum in the initial terms and
			comparison orderings of Sections~7--9. The subsequent estimates depend
			only on these inequalities, the modulation equations, and trapping in
			$\mathcal V_1[2A,\eta]$. Repeating the proof of
			Proposition~\ref{propo-priori-ell-1} therefore gives its strict improvements
			at level $2A$, uniformly in $h$ and independently of $\bar s$.
		\end{proof}
		
		\subsection{Completion of the stable construction and proof of Theorem~\ref{theo-main-separate}}\label{section-prove-theorem-stab}
		
		Fix $\delta=3/8$, $M_0\ge M_1(x_0,\delta)$, and $A\ge A_{21}(M_0,x_0)$.
		Choose $0<\eta\le\eta_{21}(A)$ and
		$s_0\ge s_{21}(A,\eta,M_0,x_0)$ as in
		Proposition~\ref{propo-stable-perturbed-data}.
		These choices include the estimates at both $A$ and $2A$ and, since
		$\eta\le1/16$, the bootstrap exponent restriction in
		Lemma~\ref{lemma-physical-reconstruction} with $\ell=1$.
		The level $2A$ is used for stability. Set $\widetilde b_0=\mu_0=e^{-s_0}$.
		The initial datum is $w_{\rm in}(\psi)=\mu_0\Psi_1^0(s_0,\psi/\mu_0^{3/4})$;
		its inverse von Mises transform is the admissible profile of Section~4.3:
		\begin{equation}\label{defi-u-IN-final-proof}
			u_{\rm in}(y_{\rm in}(\psi))=\sqrt{w_{\rm in}(\psi)},\qquad
			y_{\rm in}(\psi)=\int_0^\psi w_{\rm in}(t)^{-1/2}\,dt.
		\end{equation}
		Lemma~\ref{lemma-initial-data-l=1}, the tail profile \eqref{defi-U-infty-1},
		and the inverse transformation give
		\begin{equation*}
			u_{\rm in}(y)=
			\begin{cases}
				\frac12\widetilde b_0^{1/2}y+\frac12y^2+O(y^4),
				&0\le y\le\widetilde b_0^{5/8},\\
				\frac{a^3}{9}y^2+o_{\widetilde b_0\to0}(y^2),
				&\widetilde b_0^{1/2}\ll y\ll\widetilde b_0^{1/8},\\
				2a\widetilde b_0^{1/4}+x_0\widetilde b_0^{3/4}y
				+o_{\widetilde b_0\to0}(\widetilde b_0^{3/4}y),
				&M_0^{-1}\le y\le1,\\
				\sqrt{2x_0}+o_{y\to\infty}(y^{-z}),
				&y\to\infty,\quad\text{for every }z>0.
			\end{cases}
		\end{equation*}
		In the first range, $\xi\lesssim\widetilde b_0^{-1/4}y^2$, so the local
		residual bound gives $|u_{\rm in}''-1|\lesssim y^2$ and hence the stated
		$O(y^4)$ remainder. These ranges describe the limit $\widetilde b_0\downarrow0$.
		For every fixed $y_0>0$, the wall compatibility
		\eqref{condition-in-u-locally-exsist}, verified in Section~4.3, also gives
		\begin{equation*}
			u_{\rm in}(y)=\tfrac12\widetilde b_0^{1/2}y+\tfrac12y^2
			+O_{u_{\rm in},y_0}(y^4),\qquad 0\le y\le y_0.
		\end{equation*}
		The asymptotics \eqref{asymptotics-mu-s-final-proof}--\eqref{asymptotics-b-s-final-proof}
		below yield $\widehat{\mathbf C}_1\sqrt{\widehat x_1^*}
		=\tfrac12\widetilde b_0^{1/2}+O(\widetilde b_0^{3/4})$ and
		$\widehat x_1^*\asymp\widetilde b_0$. Thus
		\eqref{pointwise-physical-remainder-1} holds at $x=0$ for $\eta<1/4$;
		Lemma~\ref{lemma-physical-reconstruction} extends it uniformly in $x$.
		The constants may depend on the fixed datum and $y_0$, hence on $s_0$.
		
		\medskip\noindent
		\emph{Global continuation in the renormalized time.}
		The strict initial bounds and
		Corollary~\ref{cor-local-trapping-continuity} give $\widetilde s>s_0$ such that
		$(\vartheta,b,\beta)\in\mathcal V_1[A,\eta](s)$ for $s\in[s_0,\widetilde s]$.
		Define the maximal trapping time by
		\begin{equation*}
			s_{\max}:=\sup\bigl\{S>s_0:
			(\vartheta,b,\beta)\in\mathcal V_1[A,\eta](s)
			\text{ for every }s\in[s_0,S)\bigr\}.
		\end{equation*}
		If $s_{\max}<\infty$, Proposition~\ref{propo-priori-ell-1} gives strict
		margins in every bootstrap inequality up to $s_{\max}$.
		Lemma~\ref{lem-modulation} supplies an admissible trace and extends the
		solution and its modulation parameters beyond this time.
		Corollary~\ref{cor-local-trapping-continuity} preserves the strict bounds
		on a short extension, contradicting maximality. Hence
		$(\vartheta,b,\beta)\in\mathcal V_1[A,\eta](s)$ for every $s\ge s_0$.
		
		\medskip\noindent
		\emph{Reconstruction of the physical solution.}
		Set
		\begin{equation*}
			x(s)=\int_{s_0}^s\mu(\sigma)\,d\sigma,\qquad
			w(x(s),\psi)=\mu(s)\vartheta\!\left(s,\frac{\psi}{\mu(s)^{3/4}}\right).
		\end{equation*}
		Since $x_s=\mu>0$ and $w>0$ for $\psi>0$, the inverse von Mises map is
		\begin{equation*}
			y(x,\psi)=\int_0^\psi\frac{d\psi'}{\sqrt{w(x,\psi')}},\qquad
			u(x,y(x,\psi))=\sqrt{w(x,\psi)}.
		\end{equation*}
		The wall bounds make the integral finite at zero, and boundedness of $w$
		on each finite time strip gives $y(x,\psi)\to\infty$ as $\psi\to\infty$.
		These transformations recover the continued Oleinik solution, including
		its prescribed far-field value $u_E(x)$.
		
		\medskip\noindent
		\emph{Existence of a finite separation point.}
		Write $r_b=b_s/b+\beta$. Proposition~\ref{pro-finite-projections} gives
		$|\beta_s|+|r_b|\lesssim b^{1/4}$; our choice of $s_0$ ensures
		$\beta\in[3/4,5/4]$ and $b_s/b\le-1/2$. Thus
		$b(\sigma)\le b(s)e^{-(\sigma-s)/2}$ for $\sigma\ge s$, and
		\begin{equation*}
			\beta_\infty:=\beta(s_0)+\int_{s_0}^\infty\beta_s(\sigma)\,d\sigma
			=1+O(b(s_0)^{1/4})>0.
		\end{equation*}
		With $r_\beta=\beta-\beta_\infty$, this yields
		\begin{equation*}
			|r_\beta(s)|+|r_b(s)|+
			\int_s^\infty(|r_\beta|+|r_b|)\,d\sigma\lesssim b(s)^{1/4}.
		\end{equation*}
		Integrating $\mu_s/\mu=-\beta$ and $b_s/b=-\beta+r_b$, we obtain
		\begin{align}
			\mu(s)&=\mu_\infty e^{-\beta_\infty s}(1+O(b(s)^{1/4})),&
			\mu_\infty&=\mu_0e^{\beta_\infty s_0}
			\exp\!\left(-\int_{s_0}^\infty r_\beta\,d\sigma\right)>0,
			\label{asymptotics-mu-s-final-proof}\\*
			b(s)&=b_\infty e^{-\beta_\infty s}(1+O(b(s)^{1/4})),&
			b_\infty&=b(s_0)e^{\beta_\infty s_0}
			\exp\!\left(\int_{s_0}^\infty(r_b-r_\beta)\,d\sigma\right)>0.
			\label{asymptotics-b-s-final-proof}
		\end{align}
		The separation point satisfies
		$0<x^*:=\int_{s_0}^\infty\mu(s)\,ds\le\frac43e^{-s_0}<x_0/2$
		by the choice of $s_{21}$. For $t=x^*-x(s)$,
		\begin{equation*}
			t=\frac{\mu_\infty}{\beta_\infty}e^{-\beta_\infty s}
			(1+O(b(s)^{1/4})),\qquad b(s)\asymp t.
		\end{equation*}
		Consequently,
		\begin{equation*}
			\mu(x)=\beta_\infty t(1+O(t^{1/4})),\qquad
			b(x)=\frac{\beta_\infty b_\infty}{\mu_\infty}t(1+O(t^{1/4})),\qquad
			\lambda_{\rm mod}:=\mu b=
			\frac{\beta_\infty^2b_\infty}{\mu_\infty}t^2(1+O(t^{1/4})).
		\end{equation*}
		The near-wall bounds and cancellation in the adapted direction give
		$\vartheta_\xi(s,0)=b^{1/4}(1+O_A(b^\eta))$. Since
		$u_y(x,0)=\frac12\mu^{1/4}\vartheta_\xi(s,0)$,
		\begin{equation*}
			u_y(x,0)=\mathbf C_{\rm sep}t^{1/2}(1+O(t^\eta)),\qquad
			\mathbf C_{\rm sep}:=\frac12
			\left(\frac{\beta_\infty^2b_\infty}{\mu_\infty}\right)^{1/4}>0.
		\end{equation*}
		The wall shear stays positive for $x<x^*$ and vanishes as $x\uparrow x^*$,
		so separation occurs at $x^*$. The exact wall scale satisfies
		$(2u_y(x,0))^4=\lambda_{\rm mod}(1+O(b^\eta))$.
		Lemma~\ref{lemma-physical-reconstruction} gives the local profile and both
		physical remainder bounds, with a smaller exponent in the theorem if necessary.
		
		\medskip\noindent
		\paragraph{Stability.}
		Use $\mathcal O_\alpha$ and the norm \eqref{definition-norm-E}.
		It suffices to treat $0<\alpha<1$: the endpoint $\alpha=1$ follows from
		$C_b^{3,1}\hookrightarrow C_b^{3,1/2}$ and the corresponding norm bound.
		Write $\widehat u_0=u_{\rm in}$ for the reference datum, with separation
		point $\widehat x^*$, coefficient $\widehat{\mathbf C}_{\rm sep}$, and
		initial time $\widehat s_0=s_0$. For $\gamma>0$, set
		\begin{equation*}
			\mathfrak B_{\alpha,\gamma}:=
			\{h\in\mathcal O_\alpha:\|h\|_{\mathcal E}<\gamma,\ 
			\widehat u_0+h\in\mathcal K_\alpha\}.
		\end{equation*}
		Fix $\mu_0=e^{-\widehat s_0}$ and take $\gamma$ from
		Proposition~\ref{propo-stable-perturbed-data}, with reference datum
		$\widehat u_0$. Its proof controls the inverse von Mises transformation,
		the initial normalization, and every initial estimate from
		Lemma~\ref{lemma-initial-data-l=1} needed in Sections~7--9.
		The normalized data lie strictly in $\mathcal V_1[A,\eta](\widehat s_0)$,
		and the bootstrap improvements hold at level $2A$, uniformly in
		$h\in\mathfrak B_{\alpha,\gamma}$ and independently of the terminal time.
		
		Corollary~\ref{cor-local-trapping-continuity} and the strict initial estimates give
		$\widetilde s_h>\widehat s_0$ with
		$(\vartheta_h,b_h,\beta_h)\in\mathcal V_1[A,\eta](s)$ for
		$\widehat s_0\le s\le\widetilde s_h$. Here $\vartheta_h$ is the
		renormalized profile associated with $u_h$, and its perturbation has the
		unique decomposition
		\begin{equation*}
			\vartheta_h=Q_{b_h}+\varepsilon_h,\qquad
			\varepsilon_h=\mathbf n_1(\beta_h)\sqrt{b_h}
			\left(\varphi_{1,b_h,\beta_h}
			-\frac{c_{1,0}}{c_{0,0}}\varphi_{0,b_h,\beta_h}\right)+\varepsilon_{h,-}.
		\end{equation*}
		The initial choice of $\eta,\widehat s_0$ covers the thresholds at both
		$A$ and $2A$. Since $\mathcal V_1[A,\eta]\subset\mathcal V_1[2A,\eta]$,
		Proposition~\ref{propo-stable-perturbed-data} and the preceding continuation argument yield
		\begin{equation*}
			(\vartheta_h,b_h,\beta_h)\in\mathcal V_1[2A,\eta](s),\qquad
			s\ge\widehat s_0,\quad h\in\mathfrak B_{\alpha,\gamma}.
		\end{equation*}
		The scale-ratio bounds, unchanged by replacing $A$ with $2A$, retain the
		strict margins supplied by the modulation estimates.
		Consequently $u_h$ separates at
		$x_h^*=\int_{\widehat s_0}^\infty\mu_h(s)\,ds\in(0,x_0)$.
		The same integration and physical reconstruction give the asserted profile and
		$u_{h,y}(x,0)=\mathbf C_{\rm sep}(h)(x_h^*-x)^{1/2}
		(1+O((x_h^*-x)^\eta))$, where
		\begin{equation}\label{defi-separation-constant-stability}
			\mathbf C_{\rm sep}(h)=\frac12
			\left(\frac{\beta_{\infty,h}^2b_{\infty,h}}{\mu_{\infty,h}}\right)^{1/4}>0.
		\end{equation}
		
		\medskip\noindent
		\paragraph{Continuity of the separation parameters.}
		Let $h_n,h\in\mathfrak B_{\alpha,\gamma}$ with
		$\|h_n-h\|_{\mathcal E}\to0$.
		Lemma~\ref{lem-local-weighted-dependence} and
		Corollary~\ref{cor-local-trapping-continuity} give, for each fixed $\mathfrak s>\widehat s_0$,
		\begin{equation}\label{finite-time-continuity-stability}
			\sup_{\widehat s_0\le s\le \mathfrak s}
			\bigl(|\mu_{h_n}(s)-\mu_h(s)|+|b_{h_n}(s)-b_h(s)|
			+|\beta_{h_n}(s)-\beta_h(s)|\bigr)\longrightarrow0.
		\end{equation}
		Global trapping gives $\mu_h(s)\le Ce^{-(s-\widehat s_0)/2}$ uniformly
		in $h$, hence
		\begin{equation}\label{uniform-tail-mu-stability}
			\sup_{h\in\mathfrak B_{\alpha,\gamma}}\int_{\mathfrak s}^\infty\mu_h(s)\,ds
			\longrightarrow0\qquad(\mathfrak s\to\infty).
		\end{equation}
		Combining \eqref{finite-time-continuity-stability} and
		\eqref{uniform-tail-mu-stability} yields $x_{h_n}^*\to x_h^*$.
		Moreover, $|\beta_h'|+|r_{b,h}|\le Ce^{-(s-\widehat s_0)/8}$ uniformly
		in the ball. Thus $\beta_{\infty,h_n}\to\beta_{\infty,h}$, and the tails
		of $r_{b,h}$ and $r_{\beta,h}:=\beta_h-\beta_{\infty,h}$ are uniformly
		integrable. On finite intervals, use
		\begin{equation*}
			\int_{\widehat s_0}^{\mathfrak s} r_{b,h}(s)\,ds
			=\log\frac{b_h(\mathfrak s)}{b_h(\widehat s_0)}
			+\int_{\widehat s_0}^{\mathfrak s}\beta_h(s)\,ds.
		\end{equation*}
		Finite-time continuity and these tail bounds give convergence of both
		integrals in \eqref{asymptotics-mu-s-final-proof}--\eqref{asymptotics-b-s-final-proof}.
		Hence $(\mu_{\infty,h_n},b_{\infty,h_n})\to(\mu_{\infty,h},b_{\infty,h})$,
		and \eqref{defi-separation-constant-stability} gives
		$\mathbf C_{\rm sep}(h_n)\to\mathbf C_{\rm sep}(h)$.
		Taking $h\to0$, we obtain
		\begin{equation*}
			x_h^*\to\widehat x^*,\qquad
			\mathbf C_{\rm sep}(h)\to\widehat{\mathbf C}_{\rm sep}.
		\end{equation*}
		This completes the proof of Theorem~\ref{theo-main-separate}.
		
		\begin{remark}[Prescribing a small separation point]\label{remark-prescribed-stable-separation}
			Fix the preparation parameters $\delta=3/8$, $M_0$, $A$, and $\eta$ as above,
			and set $s_{\rm p}:=s_{21}(A,\eta,M_0,x_0)$, so that the construction
			applies for every $s_0\ge s_{\rm p}$.
			For $s_0\ge s_{\rm p}$, denote the datum \eqref{defi-u-IN-final-proof} by
			$u_{\rm in}^{[s_0]}$ and its separation point by $X(s_0)$.
			We first verify that, for every finite $\sigma\ge s_{\rm p}$ and
			$0<\alpha\le1$,
			\begin{equation}\label{continuity-prepared-physical-data}
				u_{\rm in}^{[s_0]}-u_{\rm in}^{[\sigma]}\in\mathcal O_\alpha,
				\qquad
				\|u_{\rm in}^{[s_0]}-u_{\rm in}^{[\sigma]}\|_{\mathcal E}\longrightarrow0
				\quad(s_0\to\sigma).
			\end{equation}
			Work on a compact interval $I\subset[s_{\rm p},\infty)$ of preparation times.
			Here $\widetilde b_0=e^{-s_0}$ stays bounded away from zero.
			The parameter dependence in Proposition~\ref{diagonalisation-L-b}, the
			eigenfunction ODE, and the smooth cutoff give continuous dependence in
			the $C^{3,\alpha}$ norm on compact intervals away from the wall. At the wall, use the
			regular analytic branches described in Section~4.3. The branch normalized
			by its wall derivative depends continuously on the parameters by the
			convergent power series for the regular singular ODE; the remaining scalar
			factor is continuous by evaluating at a fixed positive point and using
			Proposition~\ref{diagonalisation-L-b}. With $r=\sqrt\psi$, write
			\begin{equation*}
				w_{\rm in}^{[s_0]}(r^2)=r^2a_{s_0}(r),\qquad
				a_{s_0}(0)=e^{-s_0/2}>0,\qquad
				y_{\rm in}^{[s_0]}(r^2)=\int_0^r\frac{2\,dt}{\sqrt{a_{s_0}(t)}}.
			\end{equation*}
			The last map has a nonzero derivative at zero, uniformly on $I$;
			its inverse therefore gives continuous dependence of
			$u_{\rm in}^{[s_0]}$ in $C^{3,\alpha}$ on compact $y$-intervals, including
			the wall and the endpoint $\alpha=1$.
			The exponential tail \eqref{defi-U-infty-1} and the inverse von Mises
			transformation give constants $C_I,c_I,Y_I>0$ such that
			\begin{equation*}
				\bigl\|u_{\rm in}^{[s_0]}-\sqrt{2x_0}\bigr\|_{C_b^{3,\alpha}([R,\infty))}
				\le C_Ie^{-c_IR}\qquad(s_0\in I,\ R\ge Y_I).
			\end{equation*}
			Combining these bounds gives
			$\|u_{\rm in}^{[s_0]}-u_{\rm in}^{[\sigma]}\|_{C_b^{3,\alpha}}\to0$.
			The analytic wall expansion and the common compatibility conditions give
			$\partial_y^2u_{\rm in}^{[s_0]}(y)=1+y^2\mathcal R(s_0,y)$ for $0\le y\le1$,
			with $\mathcal R$ continuous on $I\times[0,1]$.
			For $h=u_{\rm in}^{[s_0]}-u_{\rm in}^{[\sigma]}$, we have $h(0)=0$,
			$h(y)\to0$ as $y\to\infty$, and
			$h''(y)=y^2(\mathcal R(s_0,y)-\mathcal R(\sigma,y))$.
			Thus $h\in\mathcal O_\alpha$. Since $|h(y)|/y\le\|h'\|_\infty$,
			the definition \eqref{definition-norm-E} yields
			\begin{equation*}
				\|h\|_{\mathcal E}\lesssim
				\|h\|_{C_b^{3,\alpha}}+
				\sup_{0\le y\le1}|\mathcal R(s_0,y)-\mathcal R(\sigma,y)|
				\longrightarrow0\qquad(s_0\to\sigma),
			\end{equation*}
			which proves \eqref{continuity-prepared-physical-data}.
			
			Apply the continuity of the separation point proved above with
			$u_{\rm in}^{[\sigma]}$ as reference datum. Since every prepared datum is
			admissible, \eqref{continuity-prepared-physical-data} gives
			$X(s_0)\to X(\sigma)$; hence $X$ is continuous on $[s_{\rm p},\infty)$.
			Moreover, $\mu(s_0)=e^{-s_0}$, $\mu_s=-\beta\mu$, and
			$3/4\le\beta\le5/4$ imply
			\begin{equation}\label{bounds-prepared-separation-point}
				\frac45e^{-s_0}\le X(s_0)=\int_{s_0}^\infty\mu(s)\,ds
				\le\frac43e^{-s_0}<\frac{x_0}{2}.
			\end{equation}
			In particular, $X$ is positive and tends to zero as $s_0\to\infty$.
			Set $\widehat{\mathfrak x}_1(x_0)=X(s_{\rm p})$.
			For any $x^*\in(0,\widehat{\mathfrak x}_1(x_0))$, choose
			$\sigma>s_{\rm p}$ with $X(\sigma)<x^*$ and apply the intermediate value
			theorem on $[s_{\rm p},\sigma]$. This gives a prepared datum separating
			at $x^*$ with the behavior in Theorem~\ref{theo-main-separate}.
			No monotonicity of $X$ is required.
		\end{remark}
		
		\section{Existence of solutions in the unstable regimes}
		We construct the higher-order regimes by combining the unstable shrinking set with a finite-dimensional exit argument for $(\mathbf P_{\rm unstab})$ in \eqref{defi-unstable-problem}.
		Fix $\ell\ge2$ and $\beta=1$. Since $\mu_s=-\mu$ and $x_s=\mu$,
		the prospective separation point satisfies $x+\mu=x^*$; after translating $s$,
		\begin{equation}\label{similarity_variable_ell_gener}
			w(x,\psi)=(x^*-x)\vartheta(s,\xi),\qquad
			\xi=\psi/(x^*-x)^{3/4},\qquad s=-\log(x^*-x).
		\end{equation}
		We use \eqref{notation-beta-one} and the adapted direction
		\eqref{definition-adapted-eigendirection-strategy} to write
		\begin{equation}\label{decomposition-unstable-regime}
			\vartheta=Q_b+\varepsilon,\qquad
			\varepsilon=\varepsilon_\ell\psi_{\ell,b}
			+\sum_{j=1}^{\ell-1}\varepsilon_j\varphi_{j,b}+\varepsilon_-.
		\end{equation}
		The geometric constraint is encoded by this decomposition and
		\begin{equation}\label{orthogonality-epsilon-minus-unstable}
			\langle\varepsilon_-,\varphi_{j,b}\rangle_{L^2_{\rho_b}}=0,
			\qquad 0\le j\le\ell.
		\end{equation}
		The shooting variables control the lower amplitudes and the deviation
		$\varepsilon_\ell-\mathbf n_\ell\sqrt b$, as in Section~3.
		\begin{remark}[Choice of the renormalization rate]\label{remark-reason-choice-beta-1}
			As in Section~\ref{subsection-sketch-unstable-regime}, we fix $\beta=1$ and
			use the scalar geometric constraint to select $b$. The lower amplitudes and
			the defect $\varepsilon_\ell-\mathbf n_\ell\sqrt b$ are tuned by shooting,
			since their normalized dynamics are unstable.
		\end{remark}
		
		\subsection{Definition of the unstable shrinking set}
		\begin{definition}\label{defi-shinkingset-V-ell}
			Let $A\ge1$, $0<\eta<1$, $s\ge1$, and $b>0$. Consider a bounded profile
			$\vartheta\in C^1([0,\infty))\cap C^2((0,\infty))$ with
			$\vartheta(0)=0$, $\vartheta>0$ on $(0,\infty)$, and
			\eqref{decomposition-unstable-regime}--\eqref{orthogonality-epsilon-minus-unstable}.
			We write $(\vartheta,b)\in\mathcal V_\ell[A,\eta](s)$ if the following bounds hold,
			using $\mathfrak I_\ell$, $\mathfrak X_{b,\eta}^{(\ell)}$, and $\mathbf n_\ell$ from
			\eqref{defi-I-ell}, \eqref{definition-weighted-remainder-norm}, and \eqref{notation-beta-one}.
			
			\noindent\textup{(i)} \emph{Scale and discrete modes:}
			$\frac12\le b\mathfrak I_\ell^{-1}(s)\le2$, and
			\begin{equation}\label{compatibility-condi-epsilon-ell}
				|\varepsilon_\ell-\mathbf n_\ell b^{1/2}|\le Ab^{1/2+2\eta},\qquad
				|\varepsilon_j|\le Ab^{1/2+\eta}\quad(1\le j<\ell).
			\end{equation}
			\noindent\textup{(ii)} \emph{Profile and remainder:}
			\begin{gather}
				A^{-1}b^{-\eta/2}\le\vartheta(\xi)\le Ab^{-1/(2\ell-1)-\eta}
				\quad(\xi\ge b^{-\eta}),\label{far-field-vartheta-V-ell}\\*
				\|\varepsilon_-\|_{\mathfrak X_{b,\eta}^{(\ell)}}\le A^2b^{1/2+\eta}.\notag
			\end{gather}
			\noindent\textup{(iii)} \emph{Residual:}
			$-A\le\mathcal F(\xi)\le Ab^{-\eta/8}$ for $\xi\in\mathbb R_+$, with
			$\mathcal F$ defined in \eqref{defi-mathcal-F}, and
			\begin{equation}\label{esti-V-ell-F}
				\|\mathcal F\|_{\mathfrak X_{b,\eta}^{(\ell)}}\le Ab^{1/2-\eta}.
			\end{equation}
		\end{definition}

		\subsection{Preparation of the initial data}\label{subsec-initial-data-ell}
		Fix $\ell\ge2$, $A\ge1$, $x_0>0$, $0<\eta\ll\delta<3/[4(2\ell-1)]$,
		and $\mathbf d\in[-2,2]^\ell$. Set
		$b_0^\dagger=\mathfrak I_\ell(s_0)=e^{-(2\ell-1)s_0}$,
		$\mu_0^\dagger=e^{-s_0}=(b_0^\dagger)^{1/(2\ell-1)}$.
		Thus $x^*=e^{-s_0}<x_0$ once $s_0>\log(1/x_0)$. Write $\kappa_\ell=c_{\ell,0}/c_{0,0}$ and
		\begin{equation*}
			\alpha_\ell^\dagger=\mathbf n_\ell\sqrt{b_0^\dagger}
			+Ad_\ell(b_0^\dagger)^{1/2+2\eta},\qquad
			\alpha_j^\dagger=Ad_j(b_0^\dagger)^{1/2+\eta}\quad(1\le j<\ell).
		\end{equation*}
		To preserve wall compatibility for arbitrary lower-mode parameters, we use
		a nonlinear correction from the ground-state family, with the normalization
		\eqref{notation-scaling-normalization}. Define
		\begin{equation}\label{unstable-wall-correction-definition}
			\begin{gathered}
				\omega_{\mathbf d}=\cscale(b_0^\dagger)^{-1/2}\sum_{j=1}^{\ell-1}c_{j,0}\alpha_j^\dagger,\qquad
				\mathscr C_{\mathbf d}=Q_{b_0^\sharp}-Q_{b_0^\dagger}
				-\omega_{\mathbf d}\Lambda Q_{b_0^\dagger},\\
				(\mathfrak F_{0,b_0^\dagger}+\kappa_\ell\mathfrak F_{\ell,b_0^\dagger})
				(\mathscr C_{\mathbf d})=0.
			\end{gathered}
		\end{equation}
		Here $b_0^\sharp$ is the locally unique root near $b_0^\dagger$; its existence
		and estimates are proved below. Set
		\begin{equation*}
			\Psi_{\ell,\mathbf d}=
			\alpha_\ell^\dagger\psi_{\ell,b_0^\dagger}
			+\sum_{j=1}^{\ell-1}\alpha_j^\dagger\varphi_{j,b_0^\dagger}
			+\mathscr C_{\mathbf d}.
		\end{equation*}
		The first two terms use the spectral notation \eqref{notation-beta-one}.
		Define the outer scales and profile by
		\begin{gather*}
			R_0=(b_0^\dagger)^{-\delta},\qquad
			\Xi_\ell=(b_0^\dagger)^{-1/(2\ell-1)-4\delta/3},\qquad
			U_\infty^\ell=2x_0(\mu_0^\dagger)^{-1}=2x_0e^{s_0},\\*
			\Psi_\infty^\ell(\xi)=U_\infty^\ell-
			\bigl(U_\infty^\ell-4a^2R_0^{4/3}\bigr)e^{-(\xi-R_0)/\Xi_\ell}.
		\end{gather*}
		Finally, prescribe
		\begin{equation}\label{defi-initi-vartheta-s-0-ell}
			\Psi_{\ell,\mathbf d}^0=
			\chi_{\ell,\delta}(Q_{b_0^\dagger}+\Psi_{\ell,\mathbf d})
			+(1-\chi_{\ell,\delta})\Psi_\infty^\ell,
		\end{equation}
		where
		\begin{equation}\label{defi_chi_ell-}
			\chi_{\ell,\delta}(\xi):=\chi_0\!\left(M_0(b_0^\dagger)^\delta\xi\right),
		\end{equation}
		and $\chi_0$ is defined in \eqref{defi-chi-0}.
		Thus, with $R_{\ell,\delta}:=M_0^{-1}(b_0^\dagger)^{-\delta}$,
		the cutoff equals one on $[0,R_{\ell,\delta}]$ and vanishes
		for $\xi\ge2R_{\ell,\delta}$. Moreover, for $j=1,2$,
		\[
		|\partial_\xi^j\chi_{\ell,\delta}|\le C_jR_{\ell,\delta}^{-j}
		=C_jM_0^j(b_0^\dagger)^{j\delta},\qquad
		\operatorname{supp}\partial_\xi^j\chi_{\ell,\delta}
		\subset[R_{\ell,\delta},2R_{\ell,\delta}].
		\]
		We prove the following result.
		\begin{lemma}[Preparation of the initial data]\label{lemma-initial-data-ell}
			Fix $\ell\ge2$, $x_0>0$, and
			$\delta=3(\ell+1)/[2(2\ell-1)(2\ell+3)]$.
			There is $M_5=M_5(\ell,x_0)\ge1$ such that, for each fixed $M_0\ge M_5$,
			there are thresholds $A_5=A_5(\ell,M_0,x_0)\ge1$,
			$\eta_5=\eta_5(\ell)\in(0,\delta/16]$, and
			$s_5=s_5(A,\eta,\ell,M_0,x_0)\ge1$ with the following property.
			For $A\ge A_5$, $0<\eta\le\eta_5$, and $s_0\ge s_5$, the datum
			\eqref{defi-initi-vartheta-s-0-ell} is strictly increasing, vanishes at zero,
			and belongs to $C^\infty((0,\infty))\cap C^1([0,\infty))$, uniformly for
			$\mathbf d\in[-2,2]^\ell$.
			
			\noindent\textup{(i)} Its residual
			\begin{equation*}
				\mathcal F_\ell(s_0,\xi)=\sqrt{\vartheta(s_0,\xi)}\,\partial_\xi^2\vartheta(s_0,\xi)-2
			\end{equation*}
			extends continuously by zero at the wall and satisfies
			\begin{equation}\label{global-bound-F-ell-initial}
				\|\mathcal F_\ell(s_0)\|_\infty\le C_{\rm in},
			\end{equation}
			where $C_{\rm in}=C_{\rm in}(\ell,M_0,x_0)\ge2$ is independent of
			$A,\eta,s_0,\mathbf d$, with the cutoff $\chi_0$ fixed.
			For $0<\xi\le R_0$, it also satisfies
			\begin{equation*}
				|\mathcal F_\ell(s_0,\xi)|\lesssim_\ell
				(b_0^\dagger)^{1/4}\xi\,\mathbf1_{\{\xi\le(b_0^\dagger)^{3/4}\}}
				+(b_0^\dagger)^{1/2}\mathfrak f_\ell(\xi)
				\mathbf1_{\{\xi\ge(b_0^\dagger)^{3/4}\}}.
			\end{equation*}
			
			\noindent\textup{(ii)} There is a locally unique $b_0$ near $b_0^\dagger$,
			with $|b_0-b_0^\dagger|\lesssim_N(b_0^\dagger)^N$ for every fixed $N$,
			such that the initial profile has decomposition
			\eqref{decomposition-unstable-regime}--\eqref{orthogonality-epsilon-minus-unstable}.
			Its coefficients satisfy
			\begin{equation*}
				\varepsilon_j(s_0)=Ad_jb_0^{1/2+\eta}+E_j^0\quad(j<\ell),\qquad
				\varepsilon_\ell(s_0)=\mathbf n_\ell\sqrt b_0+Ad_\ell b_0^{1/2+2\eta}+E_\ell^0,
			\end{equation*}
			where $\|E_j^0\|_{C^1([-2,2]^\ell)}\lesssim_{A,\ell}(b_0^\dagger)^{3/4+2\eta}$.
			The remainder obeys
			\begin{equation}\label{unstable-initial-remainder-corrected}
				|\varepsilon_-(s_0,\xi)|\lesssim_{A,\ell}
				\begin{cases}
					b_0^{1/4+2\eta}\xi,&0<\xi\le b_0^{3/4},\\
					b_0^{3/4+2\eta}\bigl(\xi^{1/3}+\mathfrak f_\ell(\xi)\bigr),
					&b_0^{3/4}\le\xi\le R_0/(2M_0).
				\end{cases}
			\end{equation}
			Moreover, $0<\vartheta(s_0,\xi)\le2U_\infty^\ell$ for $\xi>0$, and
			$\vartheta(s_0,\xi)\ge(a^2/8)b_0^{-4\eta/3}$ for $\xi\ge b_0^{-\eta}/4$.
			
			\noindent\textup{(iii)} Define
			\begin{equation*}
				\widehat{\mathcal V}_A(s)=
				\prod_{j=1}^{\ell-1}[-Ab^{1/2+\eta},Ab^{1/2+\eta}]
				\times[\mathbf n_\ell\sqrt b-Ab^{1/2+2\eta},\mathbf n_\ell\sqrt b+Ab^{1/2+2\eta}].
			\end{equation*}
			For this choice of $A_5$, all non-discrete bounds hold strictly, so
			\begin{equation}\label{equivalence-initial-shrinking-box}
				(\vartheta,b)(s_0)\in\mathcal V_\ell[A,\eta](s_0)
				\quad\Longleftrightarrow\quad
				(\varepsilon_1,\ldots,\varepsilon_\ell)(s_0)\in\widehat{\mathcal V}_A(s_0).
			\end{equation}
			There is a set $\mathcal D_{s_0}\subset[-2,2]^\ell$, homeomorphic to
			$[-1,1]^\ell$, such that
			\begin{equation}\label{equivalence-d-in-D-initial}
				(\vartheta,b)(s_0)\in\mathcal V_\ell[A,\eta](s_0)
				\quad\Longleftrightarrow\quad \mathbf d\in\mathcal D_{s_0}.
			\end{equation}
		\end{lemma}
		\begin{proof}
			\emph{Compatibility correction.} Abbreviate $b_*=b_0^\dagger$,
			$\omega=\omega_{\mathbf d}$, and
			$\mathfrak J_*:=\mathfrak F_{0,b_*}+\kappa_\ell\mathfrak F_{\ell,b_*}$.
			The scalar equation in \eqref{unstable-wall-correction-definition} is
			$\mathfrak J_*[Q_{b_*r}-Q_{b_*}-\omega\Lambda Q_{b_*}]=0$, where
			$r=b_0^\sharp/b_*$. At $(r,\omega)=(1,0)$ its $r$ derivative is
			$\mathfrak J_*[\Lambda Q_{b_*}]=\mathbf m_0\sqrt{b_*}+O_\ell(b_*^{3/4})$,
			by \eqref{projections-Lambda-Q}. The implicit function theorem therefore
			gives a unique root for $|\omega|\ll1$, uniformly for small $b_*$.
			Taylor expansion in $r$, using Lemma~\ref{lemma-Q} and
			$\partial_rQ_{b_*r}|_{r=1}=\Lambda Q_{b_*}$, gives $r=1+\omega+O(\omega^2)$.
			In the outer expansion, the coefficient of $\xi^{2/3}$ is
			$(a/4)\sqrt{b_*}(\sqrt r-1-\omega/2)$; all other terms in the correction
			are $O_\ell(b_*^{3/4}\omega^2\xi^{1/3})$ for $\xi\ge b_*^{3/4}$.
			Since $\mathfrak J_*[\xi^{2/3}]$ tends to a nonzero constant, the scalar
			constraint improves this coefficient to $O_\ell(b_*^{3/4}\omega^2)$.
			Consequently, for $D=\xi\partial_\xi$ and $0\le m\le2$,
			\begin{equation}\label{unstable-wall-correction-bounds}
				|D^m\mathscr C_{\mathbf d}|+
				\sum_{k=1}^\ell|D^m\partial_{d_k}\mathscr C_{\mathbf d}|
				\lesssim_{A,\ell}
				\begin{cases}
					b_*^{1/4+2\eta}\xi,&0<\xi\le b_*^{3/4},\\
					b_*^{3/4+2\eta}(\xi^{1/3}+\xi^{2/3}),&\xi\ge b_*^{3/4}.
				\end{cases}
			\end{equation}
			The inner estimate follows directly from Taylor's formula in $r$;
			the differentiated estimates follow by implicit differentiation and
			$|\omega|+|\partial_{d_k}\omega|\lesssim_{A,\ell}b_*^\eta$.
			In particular, every spectral moment of this correction, including its
			$d_k$ derivatives, is $O_{A,\ell}(b_*^{3/4+2\eta})$.
			
			The corrected profile satisfies wall compatibility exactly.
			Indeed, near zero the uncut profile equals $Q_{b_0^\sharp}+g_{\mathbf d}$, where
			\begin{equation*}
				g_{\mathbf d}=
				\alpha_\ell^\dagger\psi_{\ell,b_*}
				+\sum_{j=1}^{\ell-1}\alpha_j^\dagger
				\bigl(\varphi_{j,b_*}-\cscale c_{j,0}b_*^{-1/2}\Lambda Q_{b_*}\bigr).
			\end{equation*}
			The $T_0$ term cancels in each bracket. The regular Dirichlet eigenfunction
			ODE, or \eqref{phi-i-int}, then gives
			$|D^m g_{\mathbf d}|\lesssim_\ell b_*^2\zeta^{5/2}$ for $m\le2$,
			$\zeta=\xi/b_*^{3/4}\le1$, after taking $Ab_*^\eta\le1$.
			Since $\sqrt{Q_{b_0^\sharp}}Q_{b_0^\sharp}''=2$, rationalizing the residual
			gives $|\mathcal F_\ell|\lesssim_\ell b_*\zeta=b_*^{1/4}\xi$ there.
			For $b_*^{3/4}\le\xi\le R_0/M_0$, the differentiated spectral bounds give
			$|\mathcal F_\ell|\lesssim_\ell b_*^{1/2}\mathfrak f_\ell$.
			They also give $g_{\mathbf d}/Q_{b_0^\sharp}=o(1)$ and the corresponding
			small derivative error before the cutoff, because
			$1/2-2(2\ell-1)\delta/3>0$.
			
			\emph{Gluing and initial estimates.} Thus the inner profile is positive
			and increasing. On the support of $\chi_{\ell,\delta}'$, it is smaller than
			$\Psi_\infty^\ell$, whose derivative is positive. Since
			$\chi_{\ell,\delta}'\le0$, the gluing is increasing as well.
			Write $Z=Q_{b_*}+\Psi_{\ell,\mathbf d}$, $V=\Psi_\infty^\ell$,
			and $\chi=\chi_{\ell,\delta}$. On the transition interval
			$[R_0/M_0,2R_0/M_0]$, the profile estimates and the cutoff derivative bounds give
			\[
			(\Psi_{\ell,\mathbf d}^0)''
			=\chi Z''+(1-\chi)V''+2\chi'(Z'-V')+\chi''(Z-V),\qquad
			|(\Psi_{\ell,\mathbf d}^0)''|\lesssim_{\ell,M_0,x_0}R_0^{-2/3}.
			\]
			Since $\Psi_{\ell,\mathbf d}^0\asymp_{\ell,M_0,x_0}R_0^{4/3}$ there,
			the residual is uniformly bounded. Beyond the transition,
			$\Psi_{\ell,\mathbf d}^0=V\le U_\infty^\ell$ and
			$|V''|\lesssim_{x_0}b_*^{1/(2\ell-1)+8\delta/3}$, which give the same bound
			for $s_0\ge s_5$, after increasing $s_5$.
			On the transition and beyond, $b_*^{1/2}\mathfrak f_\ell(\xi)\gg1$ for
			$\xi\ge R_0/M_0$, since $\delta(2+4\ell/3)>1/2$.
			This proves (i), including its bound up to $R_0$.
			The profile comparisons also give (ii)'s upper and lower bounds;
			at $\xi=b_0^{-\eta}/4$, the leading coefficient is
			$a^2 4^{-4/3}>a^2/8$.
			
			\emph{Scalar normalization.} Before the cutoff,
			$\mathfrak J_*[\Psi_{\ell,\mathbf d}]=0$ by
			\eqref{unstable-wall-correction-definition} and orthogonality.
			The cutoff error $\widetilde\varepsilon_{\ell,\mathbf d}^0
			:=(1-\chi_{\ell,\delta})(\Psi_\infty^\ell-Q_{b_*}-\Psi_{\ell,\mathbf d})$ satisfies
			\begin{equation}\label{support-varep-ell-d-s-0}
				\operatorname{supp}\widetilde\varepsilon_{\ell,\mathbf d}^0
				\subset\{\xi\ge b_*^{-\delta}/M_0\}.
			\end{equation}
			For trial $b$ near $b_*$, define
			\begin{equation*}
				G(b,\mathbf d):=(\mathfrak F_{0,b}+\kappa_\ell\mathfrak F_{\ell,b})
				[\Psi_{\ell,\mathbf d}^0-Q_b].
			\end{equation*}
			The kernel estimates \eqref{projection-kernel-parameter-bound} apply to
			these finitely many modes. They give $G\in C^1$ and, uniformly in the cube,
			\begin{equation*}
				\partial_bG(b_*,\mathbf d)=-\mathbf m_0b_*^{-1/2}
				\bigl(1+O_{A,\ell}(b_*^\eta+b_*^{1/4})\bigr).
			\end{equation*}
			Here the lower modes contribute $O_{A,\ell}(b_*^\eta)$ relatively;
			\eqref{unstable-wall-correction-bounds} contributes a smaller term.
			The cutoff depends on the fixed preparation scale $b_*$, not on
			$\mathbf d$ or the trial scale $b$. Thus the Gaussian kernel bounds and
			\eqref{support-varep-ell-d-s-0} give
			$\|G(b_*,\cdot)\|_{C^1}\lesssim_{A,\ell,M_0,x_0}b_*^{-C}e^{-c_{M_0} b_*^{-4\delta/3}}$.
			The parameter-dependent implicit function theorem, with this uniform
			Jacobian bound, therefore gives
			\begin{equation}\label{distance-b0-bdagger-quantitative}
				\|b_0(\cdot)-b_*\|_{C^1([-2,2]^\ell)}\lesssim_N b_*^N
				\qquad\text{for every fixed }N>0.
			\end{equation}
			Projecting at $b_0$ now yields the required decomposition, and
			\eqref{unstable-wall-correction-bounds} gives
			\begin{equation}\label{C1-bound-Ej-d}
				\|E_j^0(\cdot,s_0)\|_{C^1([-2,2]^\ell)}
				\lesssim_{A,\ell}b_*^{3/4+2\eta},\qquad1\le j\le\ell.
			\end{equation}
			Subtracting these finite projections from the correction proves
			\eqref{unstable-initial-remainder-corrected}; the cutoff and modulation
			errors are smaller than any power before $R_0/(2M_0)$.
			At the wall, the finite projections cost at most
			$C_{A,\ell}b_*^{1/2+2\eta}\xi$, which is smaller than the stated bound.
			This proves (ii).
			
			\emph{The initial parameter box.} Choose $A_5>2C_{\rm in}$ for the global residual bound.
			The other normalized initial bounds tend to zero as $s_0\to\infty$ by
			\eqref{unstable-initial-remainder-corrected}, with a gain of at least $b_0^\eta$.
			Thus every non-discrete component lies strictly inside its bootstrap bounds, proving
			\eqref{equivalence-initial-shrinking-box}. Define
			\begin{equation*}
				\Gamma_{s_0}(\mathbf d)=
				\left(\frac{\varepsilon_1}{Ab_0^{1/2+\eta}},\ldots,
				\frac{\varepsilon_{\ell-1}}{Ab_0^{1/2+\eta}},
				\frac{\varepsilon_\ell-\mathbf n_\ell\sqrt{b_0}}{Ab_0^{1/2+2\eta}}\right).
			\end{equation*}
			Equations \eqref{distance-b0-bdagger-quantitative} and \eqref{C1-bound-Ej-d}
			give $\Gamma_{s_0}=\operatorname{Id}+\mathbf r_{s_0}$ with
			$\|\mathbf r_{s_0}\|_{C^1}\lesssim_{A,\ell}b_*^{1/4}\to0$.
			Take the derivative norm induced by $|\cdot|_\infty$ and increase $s_5$
			so that both $\|\mathbf r_{s_0}\|_\infty$ and
			$\|D\mathbf r_{s_0}\|_\infty$ are at most $1/4$.
			Then $\Gamma_{s_0}$ is injective, and for each $\mathbf y\in[-1,1]^\ell$
			the contraction $\mathbf d\mapsto\mathbf y-\mathbf r_{s_0}(\mathbf d)$
			has a unique fixed point in $[-3/2,3/2]^\ell$.
			Its inverse $\mathcal H_{s_0}$ is Lipschitz, and
			$\mathcal D_{s_0}:=\mathcal H_{s_0}([-1,1]^\ell) =\Gamma_{s_0}^{-1}([-1,1]^\ell)$
			is homeomorphic to the cube. This proves \eqref{equivalence-d-in-D-initial}.
		\end{proof}
		
		\begin{remark}[Propagation of the solution and the geometric decomposition]\label{remark-unstable-continuation}
			Fix the preparation parameters of Lemma~\ref{lemma-initial-data-ell}.
			There are thresholds $A_{22}=A_{22}(\ell,M_0,x_0)\ge A_5$,
			$\eta_{22}(A,\ell)\in(0,\eta_5]$, and
			$s_{22}=s_{22}(A,\eta,\ell,M_0,x_0)\ge s_5$ such that the following
			holds for $A\ge A_{22}$, $0<\eta\le\eta_{22}(A,\ell)$, and $s_0\ge s_{22}$.
			For $\mathbf d\in[-2,2]^\ell$, define the physical initial datum by
			\[
			\begin{gathered}
				w_{0,\mathbf d}(\psi):=\mu_0^\dagger\Psi_{\ell,\mathbf d}^0
				\!\left(s_0,\frac{\psi}{(\mu_0^\dagger)^{3/4}}\right),\qquad
				y_{0,\mathbf d}(\psi):=\int_0^\psi\frac{d\psi'}{\sqrt{w_{0,\mathbf d}(\psi')}},\\
				u_{\rm in,\mathbf d}(y_{0,\mathbf d}(\psi)):=\sqrt{w_{0,\mathbf d}(\psi)}.
			\end{gathered}
			\]
			Lemma~\ref{lemma-initial-data-ell} gives positivity and the prescribed Euler
			trace. Moreover, $\mathcal F_\ell(s_0,\xi)=O(\xi)$ and $\psi\asymp y^2$
			yield $u_{\rm in,\mathbf d}''(y)-1=O(y^2)$. Together with the regular-branch
			argument of Section~4.3, this proves $u_{\rm in,\mathbf d}\in\mathcal K_1$
			for every $\mathbf d$. Theorem~\ref{Olenik-thm-2.1.1} therefore gives a
			local solution on $[0,\mathfrak x^*)$ for some $\mathfrak x^*>0$.
			
			For fixed $s_0$, the family $\mathbf d\mapsto u_{\rm in,\mathbf d}$
			is continuous in \eqref{definition-norm-E}. Near the wall, the regular
			branch and the inverse map give
			$u_{\rm in,\mathbf d}''(y)-1=y^2\mathcal R(\mathbf d,y)$ with
			$\mathcal R$ smooth in both variables. On compact intervals away from
			zero this follows from the smooth gluing formulas, and the common
			exponential tail gives uniform convergence of derivatives at infinity.
			Lemma~\ref{lem-local-weighted-dependence} therefore applies to this family.
			
			With $\beta=1$, set $\mu(s)=e^{-s}$ and $x(s)=e^{-s_0}-e^{-s}$.
			As in Lemma~\ref{lemma-normalization-propagation-flows}, select $b$ by
			$G_\ell(b):=(\mathfrak F_{0,b}+\kappa_\ell\mathfrak F_{\ell,b})[\vartheta-Q_b]=0$.
			At the prepared datum, and along the trapped flow, the spectral and
			projection-kernel estimates give, with $\vartheta$ fixed,
			\[
			\partial_bG_\ell=-\mathbf m_0b^{-1/2}
			\bigl(1+O_{A,\ell}(b^\eta+b^{1/4})\bigr)\ne0.
			\]
			Here $s_{22}$ is chosen so that the relative error is at most $1/2$
			throughout the shrinking set and $e^{-s_{22}}<x_0$.
			The implicit function theorem thus gives a locally unique $C^1$ parameter
			$b(s)$ and the decomposition \eqref{decomposition-unstable-regime}--\eqref{orthogonality-epsilon-minus-unstable}
			on $[s_0,\widetilde s)$ for some $\widetilde s>s_0$.
			
			The continuation argument of Lemma~\ref{lem-modulation} applies as long as
			$(\vartheta(s),b(s))\in\mathcal V_\ell[A,\eta](s)$.
			Indeed, if this holds on $[s_0,\mathfrak s^*)$ with $\mathfrak s^*<\infty$,
			then, with the preceding parameter choices,
			\begin{align*}
				\mu(s)&=e^{-s},\qquad b(s)\asymp e^{-(2\ell-1)s},\qquad
				x(s)<e^{-s_0}<x_0,\\
				u_y(x(s),0)&=\tfrac12(\mu b)^{1/4}
				\bigl(1+O_{A,\ell}(b^\eta)\bigr)\ge c_{\mathfrak s^*}>0,\\
				u_{yy}(x(s),y)-1&=O_{\mathfrak s^*}(y^2)\quad(y\to0^+),\qquad
				\|u_{yy}(x(s))\|_\infty\le1+\tfrac A2(1+b^{-\eta/8}),\\
				\|u(x(s))\|_\infty&\le C_{\mathfrak s^*},\qquad
				\inf_{s_0\le s<\mathfrak s^*,\,y\ge r}u(x(s),y)\ge c_{\mathfrak s^*,r}>0\quad(r>0).
			\end{align*}
			The propagated compatibility follows from the near-wall residual bound
			and $\psi(x,y)\asymp_{\mathfrak s^*}y^2$.
			These estimates continue the physical solution through
			$x(\mathfrak s^*)=e^{-s_0}-e^{-\mathfrak s^*}$; the nonzero Jacobian then
			continues its modulation decomposition to $[s_0,\mathfrak s^*+\epsilon)$
			for some $\epsilon>0$. Thus no finite breakdown occurs while the solution
			remains in the shrinking set. It remains to choose $\mathbf d$ so that
			$(\vartheta,b)(s)\in\mathcal V_\ell[A,\eta](s)$ for every $s\ge s_0$;
			the finite-dimensional shooting argument below provides this choice and
			hence the separation regime of Theorem~\ref{theo-unstab}.
		\end{remark}
		
		\subsection{Reduction to a finite-dimensional problem}\label{section-reduction-finite-ell=2}
		The following bootstrap improvements are proved in Section~\ref{section-prioei-ell-2}.
		\begin{proposition}[A priori estimates]\label{propo-priori-ell-2}
			Fix the preparation parameters of Lemma~\ref{lemma-initial-data-ell}.
			There are thresholds $A_6=A_6(\ell,M_0,x_0)\ge A_5$,
			$\eta_6(A,\ell)\in(0,\eta_5]$, and
			$s_6=s_6(A,\eta,\ell,M_0,x_0)\ge\max\{s_5,s_{19}\}$ with the following property.
			Let $A\ge A_6$, $0<\eta\le\eta_6$, $s_0\ge s_6$, and
			$\mathbf d\in[-2,2]^\ell$. Suppose that a solution $\vartheta$ of
			\eqref{equa-vartheta}, with initial datum
			$\vartheta(s_0)=\Psi_{\ell,\mathbf d}^0(s_0)$ from
			\eqref{defi-initi-vartheta-s-0-ell}, satisfies
			$(\vartheta,b)(s)\in\mathcal V_\ell[A,\eta](s)$ on $[s_0,\bar s]$, where
			$\bar s>s_0$. Then, at every $s\in[s_0,\bar s]$, the following bounds hold
			with $b=b(s)$ and all profiles evaluated at time $s$.
			
			\noindent\textup{(i)} \emph{Scale and discrete dynamics:}
			\[
			\frac34\le b\mathfrak I_\ell(s)^{-1}\le\frac32,\qquad
			\varepsilon_j'=(\tfrac12-j)\varepsilon_j+O_\ell(b^{3/4})
			\quad(1\le j\le\ell).
			\]
			\noindent\textup{(ii)} \emph{Far-field profile:}
			\[
			\frac2A b^{-\eta/2}\le\vartheta(\xi)\le\frac A2 b^{-1/(2\ell-1)-\eta}
			\quad(\xi\ge b^{-\eta}).
			\]
			\noindent\textup{(iii)} \emph{Remainder:}
			\[
			\|\varepsilon_-\|_{\mathfrak X_{b,\eta}^{(\ell)}}\le\frac{A^2}{2}b^{1/2+\eta}.
			\]
			\noindent\textup{(iv)} \emph{Residual:}
			$-A/2\le\mathcal F(\xi)\le(A/2)b^{-\eta/8}$ for $\xi\in\mathbb R_+$, and
			\[
			\|\mathcal F\|_{\mathfrak X_{b,\eta}^{(\ell)}}\le\frac A2 b^{1/2-\eta}.
			\]
			The constant in $O_\ell(\cdot)$ is uniform in $s_0$, $\bar s$, and $\mathbf d$.
		\end{proposition}
		The next proposition shows that, under the a priori estimates of Proposition~\ref{propo-priori-ell-2}, a trapped solution can exit only through a discrete-mode face, and that every such crossing is outward.
		
		\begin{proposition}[Transversality]
			\label{lemma-transversality-ell}
			Fix the preparation parameters of Lemma~\ref{lemma-initial-data-ell}.
			There exists
			$
			A_7=A_7(\ell,M_0,x_0)\ge\max\{A_6,A_{22}\}
			$
			such that, for every \(A\geq A_7\), there exists
			$
			\eta_7=\eta_7(A,\ell)\in(0,\min\{\eta_6(A,\ell),\eta_{22}(A,\ell),1/16\}]
			$
			such  that  for every
			\(\eta\in(0,\eta_7]\), there exists
			$
			s_7=s_7(A,\eta,\ell,M_0,x_0)\ge\max\{s_6,s_{22}\}
			$
			such that, for every \(s_0\geq s_7\) and every
			\(\mathbf d\in[-2,2]^\ell\), the following statement holds:
			Assume that $\vartheta$ is a solution of \eqref{equa-vartheta} with initial datum
			$
			\vartheta(s_0)
			=
			\Psi_{\ell,\mathbf d}^0(s_0)
			$, where $\Psi_{\ell,\mathbf d}^0(s_0)$ is defined in \eqref{defi-initi-vartheta-s-0-ell}. Assume further that there exists $s^*\ge s_0$ such that
			$(\vartheta,b)(s) \in \mathcal V_\ell[A,\eta](s), \qquad s\in[s_0,s^*]$,
			and
			\[
			(\vartheta,b)(s^*)
			\in
			\partial\mathcal V_\ell[A,\eta](s^*).
			\]
			Then,  we have 
			\[
			\bigl(
			\varepsilon_1,\ldots,\varepsilon_\ell
			\bigr)(s^*)
			\in
			\partial\widehat{\mathcal V}_A(s^*).
			\]
			When $s^*=s_0$, derivatives at the crossing are understood as right derivatives.
			Moreover, the crossing is transverse: there exists \(\kappa>0\)
			such that, for every continuation of the solution to
			\([s^*,s^*+\kappa]\),
			$(\vartheta,b)(s) \notin \mathcal V_\ell[A,\eta](s), \qquad s\in(s^*,s^*+\kappa]$.
		\end{proposition}
		\begin{proof}
			If $s^*>s_0$, Proposition~\ref{propo-priori-ell-2} gives strict
			improvements of all non-discrete bounds at $s^*$; if $s^*=s_0$,
			their strictness follows from Lemma~\ref{lemma-initial-data-ell}\textup{(iii)}.
			Thus a boundary contact can occur only in the finite-dimensional coordinates.
			Use the projected equations in Section~\ref{subsection-proof-item-i-ell},
			with right derivatives at $s_0$.
			Put $\gamma_\ell=2\ell-1$, $e_\ell=\varepsilon_\ell-\mathbf n_\ell\sqrt b$,
			and $\alpha_\ell=\kappa_\ell^{-1}\mathbf m_0=\mathbf n_\ell/2$.
			The two equations for $\varepsilon_\ell$ give
			\[
			\frac{b_s}{b}=1-\frac{\ell\varepsilon_\ell}{\alpha_\ell\sqrt b}
			+O_{A,\ell}(b^{1/4}),\qquad
			e_\ell'=\frac12e_\ell+O_{A,\ell}(b^{3/4}).
			\]
			On the shrinking set $b_s/b=-\gamma_\ell+o(1)$. Thus the normalized
			coordinates $z_j=\varepsilon_j/(A b^{1/2+\eta})$, $1\le j<\ell$, and
			$z_\ell=e_\ell/(A b^{1/2+2\eta})$ satisfy
			\[
			z_j'=(\ell-j+\gamma_\ell\eta+o(1))z_j
			+O_{A,\ell}(b^{1/4-\eta}),\qquad 1\le j<\ell,
			\]
			\[
			z_\ell'=(\ell+2\gamma_\ell\eta+o(1))z_\ell
			+O_{A,\ell}(b^{1/4-2\eta}).
			\]
			Since $\eta\le\eta_7\le1/16$, both error powers are positive.
			Enlarge $s_7$ so that these errors are small for $s_0\ge s_7$. At a face $z_j=1$
			we have $z_j'>0$, and at $z_j=-1$ we have $z_j'<0$. This also applies at
			corners and at the initial time using right derivatives. Hence every contact
			with the boundary is strictly outward. The other bounds cannot be saturated
			by their strict improvements. This proves the stated transversality, including at $s^*=s_0$.
		\end{proof}
		
		\subsection{Completion of the unstable construction and proof of Theorem~\ref{theo-unstab}}\label{subsection-proof-existen-solu-ell=2}
		
		We first choose $\mathbf d^*\in\mathcal D_{s_0}$ so that the solution of \eqref{equa-vartheta} with initial datum $\vartheta_{\mathbf d^*}(s_0)=\Psi_{\ell,\mathbf d^*}^0(s_0)$ remains trapped for all renormalized times:
		\begin{equation}\label{trapping-V-ell-proof}
			\bigl(\vartheta_{\mathbf d^*},b_{\mathbf d^*}\bigr)(s)
			\in
			\mathcal V_\ell[A,\eta](s),
			\qquad
			s\geq s_0.
		\end{equation}
		
		\medskip
		\noindent
		\emph{Proof of \eqref{trapping-V-ell-proof}.}
		Fix $M_0\ge M_5(\ell,x_0)$ and $A\ge A_7$ as in
		Proposition~\ref{lemma-transversality-ell}. Choose
		$0<\eta\le\min\{\eta_7(A,\ell),\eta_{23}(\ell)/2\}$, with
		$\eta_{23}$ from Lemma~\ref{lemma-physical-reconstruction}, and take
		$s_0\ge s_7(A,\eta,\ell,M_0,x_0)$. Then
		Lemma~\ref{lemma-initial-data-ell},
		Proposition~\ref{propo-priori-ell-2}, and
		Proposition~\ref{lemma-transversality-ell} apply, as does the continuation
		statement of Remark~\ref{remark-unstable-continuation}. Recall from
		Lemma~\ref{lemma-initial-data-ell} that
		$\mathcal D_{s_0}\subset[-2,2]^\ell$
		is homeomorphic to the closed cube \([-1,1]^\ell\), and that
		\begin{equation*}
			\bigl(\vartheta_{\mathbf d},b_{\mathbf d}\bigr)(s_0)
			\in
			\mathcal V_\ell[A,\eta](s_0)
			\quad\Longleftrightarrow\quad
			\mathbf d\in\mathcal D_{s_0}.
		\end{equation*}
		
		For \(\mathbf d\in\mathcal D_{s_0}\), let
		\(\vartheta_{\mathbf d}\) denote the corresponding solution and define
		its maximal trapping time by
		\begin{equation*}
			\begin{aligned}
				s_*(\mathbf d)
				:=
				\sup\Bigl\{
				S\geq s_0:
				\bigl(\vartheta_{\mathbf d},b_{\mathbf d}\bigr)(s)
				\in
				\mathcal V_\ell[A,\eta](s)
				\text{ for every }s\in[s_0,S]
				\Bigr\}.
			\end{aligned}
		\end{equation*}
		Corollary~\ref{cor-local-trapping-continuity} ensures that
		\(s_*(\mathbf d)>s_0\) whenever
		\(\mathbf d\) belongs to the interior of \(\mathcal D_{s_0}\).
		
		Assume, by contradiction, that
		\begin{equation}\label{all-exit-finite-assumption}
			s_*(\mathbf d)<+\infty,
			\qquad
			\text{for every }\mathbf d\in\mathcal D_{s_0}.
		\end{equation}
		By continuity of the trajectory and the definition of the exit time,
		\begin{equation*}
			\bigl(
			\vartheta_{\mathbf d},
			b_{\mathbf d}
			\bigr)\bigl(s_*(\mathbf d)\bigr)
			\in
			\partial\mathcal V_\ell[A,\eta]
			\bigl(s_*(\mathbf d)\bigr).
		\end{equation*}
		None of the non-discrete bounds can be saturated at the exit time:
		this follows from Proposition~\ref{propo-priori-ell-2} when $s_*(\mathbf d)>s_0$,
		and from Lemma~\ref{lemma-initial-data-ell}\textup{(iii)} when $s_*(\mathbf d)=s_0$.
		Consequently, Proposition~\ref{lemma-transversality-ell} yields
		\begin{equation}\label{finite-dimensional-boundary-at-exit}
			\bigl(
			\varepsilon_{1,\mathbf d},
			\ldots,
			\varepsilon_{\ell,\mathbf d}
			\bigr)
			\bigl(s_*(\mathbf d)\bigr)
			\in
			\partial\widehat{\mathcal V}_A
			\bigl(s_*(\mathbf d)\bigr).
		\end{equation}
		
		We therefore define the normalized exit map
		$\Pi: \mathcal D_{s_0} \longrightarrow \partial[-1,1]^\ell$
		by
		\begin{equation*}
			\begin{aligned}
				&\Pi(\mathbf d) := \Bigg( \frac{ \varepsilon_{1,\mathbf d}(s_*(\mathbf d)) }{ A b_{\mathbf d}(s_*(\mathbf d))^{1/2+\eta} }, \ldots, \frac{ \varepsilon_{\ell-1,\mathbf d}(s_*(\mathbf d)) }{ A b_{\mathbf d}(s_*(\mathbf d))^{1/2+\eta} },\qquad \frac{ \varepsilon_{\ell,\mathbf d}(s_*(\mathbf d)) - \mathbf n_\ell b_{\mathbf d}(s_*(\mathbf d))^{1/2} }{ A b_{\mathbf d}(s_*(\mathbf d))^{1/2+2\eta} } \Bigg).
			\end{aligned}
		\end{equation*}
		The map \(\Pi\) is well defined by
		\eqref{finite-dimensional-boundary-at-exit}. Moreover, it is
		continuous. Let $\mathbf d_n\to\mathbf d$ and $t_*=s_*(\mathbf d)<\infty$.
		Remark~\ref{remark-unstable-continuation} and the nonzero normalization
		Jacobian extend the reference normalized flow slightly beyond $t_*$.
		Corollary~\ref{cor-local-trapping-continuity} gives uniform convergence
		of its finite-mode coordinates on this larger interval. On each compact
		subinterval before $t_*$, the reference coordinates are strictly inside
		their box: a boundary contact would force immediate exit by
		Proposition~\ref{lemma-transversality-ell}. Nearby trajectories cannot
		exit there, since every exit must occur through a discrete face.
		Thus $\liminf_n s_*(\mathbf d_n)\ge t_*$; for $t_*=s_0$ this is automatic.
		An active coordinate crosses strictly outward at $t_*$, and hence lies
		outside its box at $t_*+\epsilon$ for every sufficiently small
		$\epsilon>0$. The nearby coordinate does too, so
		$\limsup_n s_*(\mathbf d_n)\le t_*+\epsilon$.
		Letting $\epsilon\downarrow0$ proves continuity of the exit time,
		including at $s_0$, and hence of $\Pi$.
		
		Let
		$\mathcal H_{s_0}: [-1,1]^\ell \longrightarrow \mathcal D_{s_0}$
		be the homeomorphism constructed in the proof of
		Lemma~\ref{lemma-initial-data-ell}. It satisfies
		\begin{equation}\label{Gamma-H-identity-s0}
			\Gamma_{s_0}
			\bigl(
			\mathcal H_{s_0}(\mathbf y)
			\bigr)
			=
			\mathbf y,
			\qquad
			\mathbf y\in[-1,1]^\ell.
		\end{equation}
		If \(\mathbf y\in\partial[-1,1]^\ell\), then
		\(\mathcal H_{s_0}(\mathbf y)\in\partial\mathcal D_{s_0}\), and the
		corresponding initial finite-dimensional coordinates already lie on
		the boundary of \(\widehat{\mathcal V}_A(s_0)\), while all non-discrete
		bounds are strict by Lemma~\ref{lemma-initial-data-ell}\textup{(iii)}.
		Applying Proposition~\ref{lemma-transversality-ell} with $s^*=s_0$,
		the trajectory leaves the shrinking set immediately after \(s_0\). Hence
		\[
		s_*
		\bigl(
		\mathcal H_{s_0}(\mathbf y)
		\bigr)
		=
		s_0,
		\qquad
		\mathbf y\in\partial[-1,1]^\ell.
		\]
		Combining this identity with \eqref{Gamma-H-identity-s0}, we obtain
		\begin{equation*}
			\Pi
			\bigl(
			\mathcal H_{s_0}(\mathbf y)
			\bigr)
			=
			\mathbf y,
			\qquad
			\mathbf y\in\partial[-1,1]^\ell.
		\end{equation*}
		
		Thus, the continuous map
		\[
		\widetilde\Pi
		:=
		\Pi\circ\mathcal H_{s_0}
		:
		[-1,1]^\ell
		\longrightarrow
		\partial[-1,1]^\ell
		\]
		restricts to the identity on \(\partial[-1,1]^\ell\). This gives a
		continuous retraction of the closed cube onto its boundary, in
		contradiction with the no-retraction theorem; see
		\cite[Theorem~12.8, p.~136]{Smoller1994}.
		
		Therefore, assumption \eqref{all-exit-finite-assumption} is false.
		There exists
		$\mathbf d^*\in\mathcal D_{s_0}$
		such that
		$s_*(\mathbf d^*)=+\infty$.
		The propagation lemma excludes a finite maximal existence time as
		long as the trajectory remains in the shrinking set. Consequently,
		the solution exists for every \(s\geq s_0\) and satisfies
		\eqref{trapping-V-ell-proof}.
		
		\medskip
		\noindent
		\emph{Construction of the physical initial datum.}
		
		From now on, we fix the parameter \(\mathbf d^*\) selected above and
		write, for simplicity,
		$\vartheta=\vartheta_{\mathbf d^*}, \qquad b=b_{\mathbf d^*}$.
		Since \(\beta\equiv1\), the scaling parameter satisfies
		\[
		\frac{\mu_s}{\mu}=-1,
		\qquad
		\mu(s_0)=\mu_0^\dagger=e^{-s_0}.
		\]
		Therefore,
		\begin{equation}\label{exact-mu-unstable}
			\mu(s)=e^{-s},
			\qquad
			s\geq s_0.
		\end{equation}
		
		The initial datum in the physical von Mises variables is defined by
		\begin{equation}\label{defi-w-IN-ell-final}
			w_{\rm in}(\psi)
			:=
			w(0,\psi)
			=
			\mu_0^\dagger
			\Psi_{\ell,\mathbf d^*}^0
			\left(
			s_0,
			\frac{\psi}{(\mu_0^\dagger)^{3/4}}
			\right).
		\end{equation}
		We then set
		\begin{equation*}
			y_{\rm in}(\psi)
			:=
			\int_0^\psi
			\frac{\mathrm d\psi'}{
				\sqrt{w_{\rm in}(\psi')}
			},
		\end{equation*}
		and define the corresponding physical initial velocity
		\(\widehat u_{0,\ell}\) by
		\begin{equation}\label{defi-u-IN-ell-final}
			\widehat u_{0,\ell}
			\bigl(
			y_{\rm in}(\psi)
			\bigr)
			=
			\sqrt{w_{\rm in}(\psi)}.
		\end{equation}
		
		By Lemma~\ref{lemma-initial-data-ell},
		\(\Psi_{\ell,\mathbf d^*}^0\) is positive and increasing. Moreover,
		\[
		\lim_{\xi\to+\infty}
		\Psi_{\ell,\mathbf d^*}^0(s_0,\xi)
		=
		U_\infty^\ell(s_0)
		=
		2x_0(\mu_0^\dagger)^{-1}.
		\]
		It follows that
		$\lim_{\psi\to+\infty} w_{\rm in}(\psi) = 2x_0$.
		Consequently,
		\[
		y_{\rm in}(\psi)
		\sim
		\frac{\psi}{\sqrt{2x_0}}
		\qquad
		\text{as }\psi\to+\infty,
		\]
		and the map
		$\psi\longmapsto y_{\rm in}(\psi)$
		is a \(C^1\)-diffeomorphism from \((0,+\infty)\) onto itself.
		
		Let
		$b_0:=b(s_0)$.
		The boundary estimate contained in the shrinking set gives
		\begin{equation*}
			\partial_\xi\vartheta(s_0,0)
			=
			b_0^{1/4}
			\left(
			1+
			O_A(b_0^\eta)
			\right).
		\end{equation*}
		Hence
		\begin{equation*}
			\mathfrak c_{0,\ell}
			:=
			\partial_\psi w_{\rm in}(0)
			=
			(\mu_0^\dagger)^{1/4}
			\partial_\xi\vartheta(s_0,0)
			=
			(\mu_0^\dagger b_0)^{1/4}
			\left(
			1+
			O_A(b_0^\eta)
			\right)
			>0.
		\end{equation*}
		In particular,
		$w_{\rm in}(\psi) = \mathfrak c_{0,\ell}\psi + O(\psi^{3/2}) \qquad \text{as }\psi\to0^+$,
		and therefore
		\[
		y_{\rm in}(\psi)
		=
		\frac{2}{\sqrt{\mathfrak c_{0,\ell}}}
		\psi^{1/2}
		\bigl(1+o(1)\bigr).
		\]
		Equivalently,
		\begin{equation}\label{psi-y-origin-ell}
			\psi
			=
			\frac{\mathfrak c_{0,\ell}}{4}
			y^2
			\bigl(1+o(1)\bigr)
			\qquad
			\text{as }y\to0^+.
		\end{equation}
		
		Differentiating \eqref{defi-u-IN-ell-final}, we obtain
		\begin{align*}
			\partial_y\widehat u_{0,\ell} \bigl( y_{\rm in}(\psi) \bigr)&= \frac12 \partial_\psi w_{\rm in}(\psi),\\*
			\partial_y^2\widehat u_{0,\ell} \bigl( y_{\rm in}(\psi) \bigr)&= \frac12 \sqrt{w_{\rm in}(\psi)} \,\partial_{\psi\psi}^2w_{\rm in}(\psi).
		\end{align*}
		Since the defect is invariant under the scaling
		\eqref{defi-w-IN-ell-final},
		\begin{equation*}
			\sqrt{w_{\rm in}(\psi)}
			\,\partial_{\psi\psi}^2w_{\rm in}(\psi)
			-2
			=
			\mathcal F_\ell
			\left(
			s_0,
			\frac{\psi}{(\mu_0^\dagger)^{3/4}}
			\right).
		\end{equation*}
		The near-origin estimate in item~{\rm(i)} of
		Lemma~\ref{lemma-initial-data-ell}, together with
		\eqref{psi-y-origin-ell}, gives
		\[
		\partial_y^2\widehat u_{0,\ell}(y)-1
		=
		O(y^2)
		\qquad
		\text{as }y\to0^+.
		\]
		Thus,
		\begin{equation}\label{initial-u-expansion-ell}
			\widehat u_{0,\ell}(y)
			=
			\frac{\mathfrak c_{0,\ell}}{2}y
			+
			\frac12y^2
			+
			O(y^4)
			\qquad
			\text{as }y\to0^+.
		\end{equation}
		Moreover,
		$\widehat u_{0,\ell}(y) = \sqrt{2x_0} + o_{y\to+\infty}(y^{-z}) \qquad \text{for every }z>0$.
		The analytic regular-branch argument used for the stable initial datum
		applies to the corrected profile as well. Thus
		$\widehat u_{0,\ell}\in C_b^{3,1}\cap C^\infty\cap\mathcal K_1$;
		\eqref{initial-u-expansion-ell} verifies its wall compatibility.
		
		\medskip
		\noindent
		\emph{Reconstruction of the physical solution.}
		
		Define
		\begin{equation*}
			x(s)
			:=
			\int_{s_0}^s
			\mu(\sigma)\,\mathrm d\sigma.
		\end{equation*}
		Using \eqref{exact-mu-unstable}, we obtain the exact formula
		$x(s) = e^{-s_0}-e^{-s}$.
		We therefore set
		$x_\ell^* := e^{-s_0} = \mu_0^\dagger$.
		Then
		\begin{equation}\label{exact-distance-x-star-ell}
			x_\ell^*-x(s)
			=
			e^{-s}
			=
			\mu(s).
		\end{equation}
		After increasing $s_7$, we have
		$0<x_\ell^*<x_0$.
		
		For \(s\geq s_0\), define
		\begin{equation}\label{reconstruction-w-ell-final}
			w\bigl(x(s),\psi\bigr)
			=
			\mu(s)
			\vartheta
			\left(
			s,
			\frac{\psi}{\mu(s)^{3/4}}
			\right).
		\end{equation}
		The global trapping property \eqref{trapping-V-ell-proof} implies
		$w(x,\psi)>0, \qquad 0\leq x<x_\ell^*, \quad \psi>0$.
		We may therefore introduce the inverse von Mises transformation
		\begin{equation}\label{inverse-von-Mises-ell-final}
			y(x,\psi)
			=
			\int_0^\psi
			\frac{\mathrm d\psi'}{\sqrt{w(x,\psi')}},
			\qquad
			u\bigl(x,y(x,\psi)\bigr)
			=
			\sqrt{w(x,\psi)}.
		\end{equation}
		It follows that \(u\) solves the original Prandtl equation on
		\[
		[0,x_\ell^*)\times\mathbb R_+,
		\]
		with initial datum
		$u(0,y)=\widehat u_{0,\ell}(y)$.
		
		\medskip
		\noindent
		\emph{Asymptotic behavior at the separation point.}
		
		Write
		$\mathfrak I_\ell(s) = e^{-\gamma_\ell s}$.
		The modulation equation for \(b\), combined with the improved
		estimates of Proposition~\ref{propo-priori-ell-2}, gives
		\begin{equation}\label{modulation-b-final-ell}
			\left|
			\frac{b_s(s)}{b(s)}
			+
			\gamma_\ell
			\right|
			\lesssim
			b(s)^{2\eta}.
		\end{equation}
		Since
		\[
		b(s)\asymp e^{-\gamma_\ell s},
		\]
		the right-hand side of \eqref{modulation-b-final-ell} is integrable.
		Set
		\[
		r_b(s)
		:=
		\frac{b_s(s)}{b(s)}
		+
		\gamma_\ell.
		\]
		Then
		$|r_b(s)| \lesssim e^{-2\eta\gamma_\ell s}$.
		Consequently, the limit
		\begin{equation*}
			b_{\ell,\infty}
			:=
			b(s_0)e^{\gamma_\ell s_0}
			\exp\left(
			\int_{s_0}^{+\infty}
			r_b(\sigma)\,\mathrm d\sigma
			\right)
			>0
		\end{equation*}
		is well defined, and
		\begin{equation*}
			b(s)
			=
			b_{\ell,\infty}
			e^{-\gamma_\ell s}
			\left(
			1+
			O\left(
			e^{-2\eta\gamma_\ell s}
			\right)
			\right).
		\end{equation*}
		Using \eqref{exact-distance-x-star-ell}, we obtain
		\begin{equation*}
			b(x)
			=
			b_{\ell,\infty}
			(x_\ell^*-x)^{\gamma_\ell}
			\left(
			1+
			O\left(
			(x_\ell^*-x)^{2\eta\gamma_\ell}
			\right)
			\right).
		\end{equation*}
		
		Define
		$\lambda(x) := \mu(x)b(x)$.
		Since
		$\mu(x)=x_\ell^*-x$,
		it follows that
		\begin{equation*}
			\lambda(x)
			=
			b_{\ell,\infty}
			(x_\ell^*-x)^{1+\gamma_\ell}
			\left(
			1+
			O\left(
			(x_\ell^*-x)^{2\eta\gamma_\ell}
			\right)
			\right).
		\end{equation*}
		
		It remains to translate this estimate into the behavior of the wall
		shear. From \eqref{inverse-von-Mises-ell-final},
		\begin{equation*}
			\partial_yu(x,0)
			=
			\frac12
			\partial_\psi w(x,0).
		\end{equation*}
		Moreover, \eqref{reconstruction-w-ell-final} gives
		\begin{equation}\label{wall-w-derivative-ell}
			\partial_\psi w(x(s),0)
			=
			\mu(s)^{1/4}
			\partial_\xi\vartheta(s,0).
		\end{equation}
		The boundary estimate in the shrinking set yields
		\begin{equation}\label{boundary-vartheta-derivative-ell}
			\partial_\xi\vartheta(s,0)
			=
			b(s)^{1/4}
			\left(
			1+
			O_A\left(
			b(s)^\eta
			\right)
			\right).
		\end{equation}
		Combining
		\eqref{wall-w-derivative-ell} and
		\eqref{boundary-vartheta-derivative-ell}, we obtain
		\begin{equation*}
			\partial_\psi w(x,0)
			=
			\lambda(x)^{1/4}
			\left(
			1+
			O\left(
			(x_\ell^*-x)^{\eta\gamma_\ell}
			\right)
			\right).
		\end{equation*}
		Therefore,
		\begin{equation*}
			\partial_yu(x,0)
			=
			\mathbf C_\ell
			(x_\ell^*-x)^{(1+\gamma_\ell)/4}
			\left(
			1+
			O\left(
			(x_\ell^*-x)^{\eta\gamma_\ell}
			\right)
			\right),
			\qquad
			x\to(x_\ell^*)_-,
		\end{equation*}
		where
		\begin{equation*}
			\mathbf C_\ell
			:=
			\frac12
			b_{\ell,\infty}^{1/4}
			>0.
		\end{equation*}
		In particular,
		\[
		\partial_yu(x,0)>0,
		\qquad
		0\leq x<x_\ell^*,
		\]
		and
		\[
		\lim_{x\to(x_\ell^*)_-}
		\partial_yu(x,0)
		=
		0.
		\]
		Thus, the solution separates at the finite point \(x_\ell^*\).
		The local expansion and both spatial remainder bounds follow from Lemma~\ref{lemma-physical-reconstruction}. This completes the proof of Theorem~\ref{theo-unstab}.
		
		\subsection{Reconstruction of the physical separation profile}
		\label{revision-physical-reconstruction}
		\begin{lemma}[Reconstruction of the physical profile]\label{lemma-physical-reconstruction}
			Let $\ell\ge1$, $\gamma=2\ell-1$, $p=2+4\ell/3$, and set
			$\eta_{23}(\ell):=\min\{1/4,1/(2\gamma)\}$.
			Let $u=u_\ell$ be a globally trapped solution constructed above, with
			bootstrap exponent $0<\eta_0<\eta_{23}(\ell)$.
			Choose $0<\eta_*<\eta_0$ with $\eta_0+(p+4/3)\eta_*<1/4$. For $t=x_\ell^*-x$, the remainder $\widetilde u=u-\mathbf C_\ell t^{\ell/2}y-y^2/2$ in \eqref{decom-u_1-theo} and \eqref{decom-u-ell-theo} satisfies, for some $t_0>0$ and every fixed sufficiently small $y_0>0$,
			\begin{align}
				|\widetilde u(x,y)|&\lesssim t^{\ell/2+\gamma\eta_*}y,
				&&0<t<t_0,\quad 0\le y\le t^{1/4-\gamma\eta_*/3},\label{shrinking-physical-remainder}\\*
				|\widetilde u(x,y)|&\lesssim t^{\ell/2+\gamma\eta_*}y+y^{2\ell+1},
				&&0\le x\le x_\ell^*,\quad 0\le y\le y_0.\notag
			\end{align}
			The cutoff $y_0$ may be prescribed independently of $s_0$; the implicit constants may depend on the initial datum and $y_0$, but not on $x$ or $y$.
			The terminal profile is continuous on every bounded $y$-interval and satisfies $u(x_\ell^*,y)=y^2/2+O(y^{2\ell+1})$ as $y\downarrow0$.
		\end{lemma}
		\begin{proof}
			For $Z=\mathcal Y_Q(r)$, set $U_Q(Z)=\sqrt{Q(r)}$.
			The ground-state equation gives $U_Q'(Z)=Q'(r)/2$ and
			$U_Q''(Z)=\sqrt{Q(r)}Q''(r)/2=1$, with
			$U_Q(0)=0$ and $U_Q'(0)=1/2$. Hence
			\[
			U_Q(Z)=\frac Z2+\frac{Z^2}{2},\qquad
			\mathcal Q_{\rm sep}(Y)=4U_Q(Y/2)=Y+\frac{Y^2}{2}.
			\]
			Thus the remainder in the statement agrees with that in
			\eqref{decom-u_1-theo} and \eqref{decom-u-ell-theo}.
			Choose $0<t_0<\min\{1,x_\ell^*\}$, to be decreased below.
			
			Equations \eqref{asymptotics-mu-s-final-proof}--\eqref{asymptotics-b-s-final-proof} and \eqref{modulation-b-final-ell}, together with $dx/ds=\mu$, give
			\begin{equation*}
				\mu\asymp t,\qquad b\asymp t^\gamma,\qquad
				\alpha_{\rm mod}:=\tfrac12(\mu b)^{1/4}
				=\mathbf C_\ell t^{\ell/2}+O(t^{\ell/2+\gamma\eta_0}).
			\end{equation*}
			Since $\varepsilon(s,0)=0$ and $\varepsilon(s,\cdot)$ is $C^1$ at the wall, the bound $|\varepsilon(s,\xi)|\lesssim b^{1/4+\eta_0}\xi$ for $\xi\le b^{3/4}$ yields
			\begin{equation*}
				|\partial_\xi\varepsilon(s,0)|
				=\left|\lim_{\xi\downarrow0}\frac{\varepsilon(s,\xi)}{\xi}\right|
				\lesssim b^{1/4+\eta_0}.
			\end{equation*}
			Using $u_y=\tfrac12\mu^{1/4}\partial_\xi\vartheta$ and $Q_b'(0)=b^{1/4}$, we obtain $\alpha:=u_y(x,0)=\alpha_{\rm mod}+O(t^{\ell/2+\gamma\eta_0})$.
			
			\emph{The shrinking interval.} For each fixed large $M$, the regional residual estimates imply
			\begin{equation*}
				|\mathcal F(s,\xi)|\lesssim b^{1/2-\eta_0}(1+\xi)^p
				\lesssim t^{\gamma(1/2-\eta_0-p\eta_*)}=o(1),
				\qquad 0\le\xi\le Mt^{-\gamma\eta_*},
			\end{equation*}
			since $Mt^{-\gamma\eta_*}<b^{-\eta_0}$ for small $t$. While $\xi(x,y)$ stays in this range, $u_{yy}=1+\mathcal F/2$ gives $\alpha y+y^2/4\le u\le\alpha y+3y^2/4$ and hence $\psi\le\alpha y^2/2+y^3/4$. Thus, for $y\le t^{1/4-\gamma\eta_*/3}$,
			\begin{equation*}
				\xi(x,y)\lesssim t^{\ell/2-3/4}y^2+t^{-3/4}y^3
				\lesssim t^{\gamma/4-2\gamma\eta_*/3}+t^{-\gamma\eta_*}
				\le M_*t^{-\gamma\eta_*},
			\end{equation*}
			where $M_*$ is independent of $M$. Fix $M>2M_*$ and then take $t_0$ sufficiently small. A first point at which $\xi=Mt^{-\gamma\eta_*}$ would contradict the preceding bound, so the whole physical interval lies in the controlled region. The identity
			\begin{equation*}
				u(x,y)=\alpha y+y^2/2+\tfrac12\int_0^y(y-z)\mathcal F(s,\xi(x,z))\,dz
			\end{equation*}
			gives the first estimate: after division by $t^{\ell/2+\gamma\eta_*}y$, the integral term is $O(t^{\gamma[1/4-\eta_0-(p+4/3)\eta_*]})$, and the wall-shear error gains $t^{\gamma(\eta_0-\eta_*)}$.
			
			\emph{Extension to a fixed interval.} Write $w(t,\psi)=w(x_\ell^*-t,\psi)$, $\xi=\psi/\mu^{3/4}$, and $w_Q(t,\psi)=\mu Q_b(\xi)$. The spectral expansion \eqref{eigenfuc-L-b}, the weight \eqref{defi-f-ell}, and the comparison bounds give
			\begin{equation*}
				|\varepsilon|\lesssim
				\begin{cases}
					b^{1/4+\eta_0}\xi,&0<\xi\le b^{3/4},\\
					b^{1/2}\xi^2+b^{1/2+\eta_0}\xi^{2/3},&b^{3/4}\le\xi\le1,\\
					b^{1/2}\xi^{p-4/3}+b^{1/2+\eta_0}\xi^p,&1\le\xi\le cb^{-r_{\rm out}},
				\end{cases}
				\qquad r_{\rm out}=\frac{3(1+2\eta_0)}{4(2\ell+1)}.
			\end{equation*}
			Indeed, the $T_0$ terms cancel in the adapted direction. For $j\ge1$ and $\xi\le b^{3/4}$, Lemma~\ref{lemma-T-i} gives
			\begin{equation*}
				b^{j+1/2}\left|T_j(\xi/b^{3/4})\right|
				\lesssim b^{-1/4-j/8}\xi^{1+3j/2}\le b^{j-1}\xi^2\le\xi^2.
			\end{equation*}
			The eigenfunction remainders are also $O(\xi^2)$ there. The adapted contribution is therefore $O(b^{1/2}\xi^2)$, whereas the lower modes and $\varepsilon_-$ are bounded by $Cb^{1/4+\eta_0}\xi$. For $\xi\ge b^{3/4}$, the adapted finite sum has powers $\xi^{2/3+4j/3}$, $1\le j\le\ell$. The lower modes have coefficients $O(b^{1/2+\eta_0})$; their finite sums are $O(\xi^{2/3})$ for $\xi\le1$ and $O(\xi^{p-4/3})$ for $\xi\ge1$. The spectral remainder $b^{3/4}\mathfrak f_\ell$ is absorbed into $b^{1/2+\eta_0}\mathfrak f_\ell$. The last regional estimate uses the extended outer comparisons \eqref{outer-trapping-epsilon-minus} and \eqref{ell-outer-trapping-extended}, which hold up to $cb^{-r_{\rm out}}$. For the perturbed stable data, the same comparisons follow from the strict initial orderings propagated in Section~\ref{section-prove-theorem-stab}.
			
			The global bounds \eqref{definition-m0-F} and Section~\ref{subsection-unstable-residual-estimates} yield $|\mathcal F|\lesssim b^{-\eta_0/8}$. At fixed $\psi$, the von Mises equation is $\partial_tw=-\mathcal F$, so, with $q=\gamma\eta_0/8<1$,
			\begin{equation*}
				|w(t,\psi)-w(\tau,\psi)|\lesssim\int_t^\tau v^{-q}\,dv
				\le\frac{\tau^{1-q}}{1-q}\quad(0<t\le\tau),\qquad
				m:=p/2,\quad \sigma:=(1-q)/m.
			\end{equation*}
			Set $r_{\rm in}:=(3/4-\sigma)/\gamma$. Since
			\begin{equation*}
				0<\sigma<\frac34,\qquad
				r_{\rm in}=\frac{3(2+\eta_0)}{8(2\ell+3)},\qquad
				r_{\rm out}-r_{\rm in}
				=\frac{3[4+(6\ell+11)\eta_0]}{8(2\ell+1)(2\ell+3)}>0,
			\end{equation*}
			we have $\xi b^{r_{\rm out}}\lesssim t^{\gamma(r_{\rm out}-r_{\rm in})}\to0$ for $\mu^{3/4}\le\psi\le t^\sigma$. Hence the extended comparison applies on this interval. Substituting $\xi=\psi/\mu^{3/4}$ in the last regional estimate gives
			\begin{align*}
				|w-w_Q|&\lesssim b^{1/2}\mu^{1/2-\ell}\psi^{p-4/3}
				+b^{1/2+\eta_0}\mu^{-1/2-\ell}\psi^{2m}\\*
				&\lesssim\psi^{p-4/3}+t^{-1+\gamma\eta_0}\psi^{2m}
				\lesssim\psi^m.
			\end{align*}
			Here $\gamma/2+1/2-\ell=0$, $p-4/3-m=\gamma/3>0$, and
			\begin{equation*}
				t^{-1+\gamma\eta_0}\psi^m
				\le t^{-1+\gamma\eta_0+\sigma m}
				=t^{7\gamma\eta_0/8}\le1.
			\end{equation*}
			Write $\lambda=\mu b\asymp t^{2\ell}$. Lemma~\ref{lemma-Q} yields
			\begin{equation*}
				|w_Q(v,\psi)-a^2\psi^{4/3}|
				\lesssim\lambda(v)^{1/2}\psi^{2/3}
				\lesssim v^\ell\psi^{2/3},\qquad \psi\ge\lambda(v)^{3/4}.
			\end{equation*}
			Choose $\psi_0<1$ sufficiently small, with $\psi_0^{1/\sigma}<t_0$. If $t^\sigma<\psi\le\psi_0$, set $\tau=\psi^{1/\sigma}>t$. Since $\sigma<3/4$, the preceding perturbation estimate applies at time $\tau$, with $\mu(\tau)^{3/4}\le\psi=\tau^\sigma$. Moreover, $\psi\ge\lambda(v)^{3/4}$ for $v=t,\tau$. At this fixed $\psi$,
			\begin{align*}
				|w(t,\psi)-w_Q(t,\psi)|
				&\le |w(t,\psi)-w(\tau,\psi)|
				+|w(\tau,\psi)-w_Q(\tau,\psi)|\\*
				&\quad+|w_Q(\tau,\psi)-a^2\psi^{4/3}|
				+|w_Q(t,\psi)-a^2\psi^{4/3}|\\*
				&\lesssim\tau^{1-q}+\psi^m+(\tau^\ell+t^\ell)\psi^{2/3}
				\lesssim\psi^m,
			\end{align*}
			because $\tau^{1-q}=\psi^m$, $t^\ell\le\tau^\ell$, and
			\begin{equation*}
				\frac\ell\sigma+\frac23-m=\frac{m(\ell-1+q)}{1-q}+\frac23>0.
			\end{equation*}
			We have thus proved $|w-w_Q|\lesssim\psi^m$ on $\mu^{3/4}\le\psi\le\psi_0$.
			Since $m>4/3$, taking $\psi_0$ smaller gives $w\asymp w_Q\asymp\psi^{4/3}$ there. For $\xi\le b^{3/4}$, the first regional bound gives $|\varepsilon|/Q_b\lesssim b^{\eta_0}$; for $b^{3/4}\le\xi\le1$, the second gives
			\begin{equation*}
				\frac{|\varepsilon|}{Q_b}
				\lesssim b^{1/2}\xi^{2/3}+b^{1/2+\eta_0}\xi^{-2/3}
				\lesssim b^{\eta_0}.
			\end{equation*}
			Consequently $w\asymp w_Q$ throughout $0<\psi\le\psi_0$.
			
			\emph{The inverse transformation.} Set
			\begin{equation*}
				Y(t,\psi)=\int_0^\psi\frac{dz}{\sqrt{w(t,z)}},\qquad
				Y_Q(t,\psi)=\int_0^\psi\frac{dz}{\sqrt{w_Q(t,z)}}.
			\end{equation*}
			Then $y=Y\asymp Y_Q$ and $u(x,y)=\sqrt w$. The ground-state equation \eqref{ground-state-Q} gives
			\begin{equation*}
				\sqrt{w_Q}=\alpha_{\rm mod}Y_Q+Y_Q^2/2,
			\end{equation*}
			since this physical profile has second derivative $1$, value zero at the wall, and wall slope $\alpha_{\rm mod}$. Hence
			\begin{equation*}
				u(x,y)-\alpha_{\rm mod}y-y^2/2
				=\sqrt w-\sqrt{w_Q}+(Y_Q-y)\bigl(\alpha_{\rm mod}+(Y_Q+y)/2\bigr).
			\end{equation*}
			We estimate the two differences using
			\begin{align*}
				|Y-Y_Q|&\lesssim\mu^{1/4}\int_0^\xi
				\frac{|\varepsilon(s,z)|}{Q_b(z)^{3/2}}\,dz,\\*
				|\sqrt w-\sqrt{w_Q}|&\lesssim\mu^{1/2}
				\frac{|\varepsilon(s,\xi)|}{\sqrt{Q_b(\xi)}}.
			\end{align*}
			Put $D(t):=\mu^{1/4}b^{1/4+\eta_0}\asymp t^{\ell/2+\gamma\eta_0}$.
			For $\xi\le b^{3/4}$, we have $Q_b\asymp b^{1/4}\xi$ and
			\begin{align*}
				|Y-Y_Q|&\lesssim\mu^{1/4}b^{-1/8+\eta_0}
				\int_0^\xi z^{-1/2}\,dz
				\lesssim\mu^{1/4}b^{-1/8+\eta_0}\xi^{1/2},\\*
				|\sqrt w-\sqrt{w_Q}|&\lesssim\mu^{1/2}b^{1/8+\eta_0}\xi^{1/2}.
			\end{align*}
			Since $Y\asymp Y_Q\asymp\mu^{1/4}b^{-1/8}\xi^{1/2}\lesssim\alpha_{\rm mod}$, both contributions to the remainder are bounded by $CD(t)Y$.
			
			For $b^{3/4}\le\xi\le1$, use $Q_b\asymp\xi^{4/3}$ and split the integral at $b^{3/4}$:
			\begin{align*}
				|Y-Y_Q|&\lesssim\mu^{1/4}\left(b^{1/4+\eta_0}
				+b^{1/2}\int_{b^{3/4}}^\xi dz
				+b^{1/2+\eta_0}\int_{b^{3/4}}^\xi z^{-4/3}\,dz\right)\\*
				&\lesssim\mu^{1/4}(b^{1/4+\eta_0}+b^{1/2}\xi),\\*
				|\sqrt w-\sqrt{w_Q}|&\lesssim\mu^{1/2}
				(b^{1/2}\xi^{4/3}+b^{1/2+\eta_0}).
			\end{align*}
			Here $\int_{b^{3/4}}^\xi z^{-4/3}\,dz
			=3(b^{-1/4}-\xi^{-1/3})\le3b^{-1/4}$, and
			$Y\asymp Y_Q\asymp\mu^{1/4}\xi^{1/3}$, with $\alpha_{\rm mod}\lesssim Y$.
			Dividing the square-root error by $D(t)Y$ gives the bounded factors
			$b^{1/4-\eta_0}\xi\le1$ and $b^{1/4}\xi^{-1/3}\le1$.
			The inverse-map contribution is at most $C D(t)Y(1+b^{1/4-\eta_0}\xi)$.
			Thus
			\begin{equation*}
				|u(x,Y)-\alpha_{\rm mod}Y-Y^2/2|\lesssim D(t)Y\quad(\psi\le\mu^{3/4}),
				\qquad |Y-Y_Q|_{\psi=\mu^{3/4}}\lesssim D(t).
			\end{equation*}
			For $\mu^{3/4}\le\psi\le\psi_0$, we use $|w-w_Q|\lesssim\psi^m$ and $w\asymp w_Q\asymp\psi^{4/3}$ to obtain
			\begin{equation*}
				|Y-Y_Q|\lesssim D(t)+\int_{\mu^{3/4}}^\psi z^{m-2}\,dz
				\lesssim D(t)+\psi^{m-1}.
			\end{equation*}
			Here $m>1$ and $\int_{\mu^{3/4}}^\psi z^{m-2}\,dz
			=(\psi^{m-1}-\mu^{3(m-1)/4})/(m-1)$. Moreover,
			$|\sqrt w-\sqrt{w_Q}|\lesssim\psi^{m-2/3}$ and
			$Y\asymp Y_Q\asymp\psi^{1/3}$: the contribution of $[0,\mu^{3/4}]$ is comparable to $\mu^{1/4}$, and the remaining integral is comparable to $\int_{\mu^{3/4}}^\psi z^{-2/3}\,dz$.
			Since $\alpha_{\rm mod}\lesssim y$, the preceding decomposition gives
			\begin{equation*}
				|u-\alpha_{\rm mod}y-y^2/2|
				\lesssim\psi^{m-2/3}+y\bigl(D(t)+\psi^{m-1}\bigr)
				\lesssim D(t)y+y^{2\ell+1},
			\end{equation*}
			where $3m-2=2\ell+1$.
			Since $\sigma<3/4$ and $\mu(t)\asymp t$, we may decrease $\psi_0$ and then
			$t_0$, keeping $\psi_0^{1/\sigma}<t_0$, so that
			$\mu(t)^{3/4}<\psi_0/2$ for $0<t<t_0$.
			Then $Y(t,\psi_0)\gtrsim\psi_0^{1/3}$ uniformly in $t$, so the estimate holds on $0\le y\le r$ for some $r=r(u_{\rm in})>0$, possibly depending on $s_0$. Replacing $\alpha_{\rm mod}$ by $\mathbf C_\ell t^{\ell/2}$ and using $\eta_*<\eta_0$ gives the claimed remainder there.
			
			\emph{Trace at separation.} Since $q<1$, the global bound $|\partial_tw|\lesssim t^{-q}$ gives a bounded continuous limit $w_*$, with $w(t,\cdot)\to w_*$ uniformly as $t\downarrow0$. The far-field lower bound in the shrinking set and $\mathcal F\ge-A$ imply, for every fixed $\psi>0$ and sufficiently small $\tau>0$,
			\begin{equation*}
				w_*(\psi)\ge w(\tau,\psi)-A\tau
				\ge c\tau^{1-\gamma\eta_0/2}-A\tau>0.
			\end{equation*}
			Indeed, $\mu(\tau)^{3/4}b(\tau)^{-\eta_0}\asymp\tau^{3/4-\gamma\eta_0}\to0$, so the far-field bound applies at this fixed $\psi$. Choosing $\tau$ uniformly for $\psi\in[\delta,L]$, with $0<\delta<L<\infty$, shows that $w_*$ and $w(t,\cdot)$ are bounded below there for small $t$.
			For $r>0$, the first integral in Lemma~\ref{lemma-Q} gives
			\[
			\partial_r(Q^{3/4})
			=\tfrac34Q^{-1/4}\sqrt{1+8\sqrt Q}\ge\frac3{\sqrt2},
			\qquad Q(r)\ge a^2r^{4/3},
			\]
			where the second inequality follows by integration from $Q(0)=0$.
			Consequently, $w_Q(t,\psi)=\lambda Q(\psi/\lambda^{3/4})\ge a^2\psi^{4/3}$,
			with $\lambda=\mu b$.
			Since $w\asymp w_Q$ on $0<\psi\le\psi_0$, this gives the uniform integrable
			bound $w^{-1/2}\lesssim\psi^{-2/3}$ near zero.
			Dominated convergence gives, for every $L<\infty$,
			\begin{equation*}
				\sup_{0\le\psi\le L}|Y(t,\psi)-Y_*(\psi)|
				\le\int_0^L\left|w(t,z)^{-1/2}-w_*(z)^{-1/2}\right|\,dz\longrightarrow0,
				\qquad Y_*(\psi):=\int_0^\psi\frac{dz}{\sqrt{w_*(z)}}.
			\end{equation*}
			The function $Y_*$ is continuous and strictly increasing, with $Y_*(0)=0$ and
			$Y_*(\psi)\ge\psi/\sqrt{\|w_*\|_\infty}\to\infty$.
			Its inverse is therefore defined on $\mathbb R_+$.
			For a fixed $y_0$, choose $L$ such that $Y_*(L)>y_0+1$.
			The preceding uniform convergence and strict monotonicity imply
			$Y(t,\cdot)^{-1}\to Y_*^{-1}$ uniformly on $[0,y_0]$.
			Together with the uniform convergence of $w$, this yields
			\begin{equation*}
				u(x_\ell^*-t,y)\longrightarrow
				u(x_\ell^*,y):=\sqrt{w_*(Y_*^{-1}(y))}
				\quad\text{uniformly on }[0,y_0].
			\end{equation*}
			Passing to the limit in the local remainder estimate gives
			$u(x_\ell^*,y)=y^2/2+O(y^{2\ell+1})$ as $y\downarrow0$.
			
			Now prescribe $y_0>0$ independently of $s_0$ and reduce $r$ so that $r\le y_0$. Comparison in \eqref{wvariable:1} gives $0\le u\le\max\{\|u_{\rm in}\|_\infty,\sqrt{2x_0}\}$, hence
			\begin{equation*}
				|\widetilde u(x,y)|\lesssim_{u_{\rm in},y_0}r^{-(2\ell+1)}y^{2\ell+1},
				\qquad r\le y\le y_0,\quad 0\le x\le x_\ell^*.
			\end{equation*}
			On $[0,r]$, the local estimate covers $0<t<t_0$. For the remaining times, fix $0<t_1<\min\{t_0,x_\ell^*\}$; Theorem~\ref{Olenik-thm-2.1.1} bounds $u_y$ on $[0,x_\ell^*-t_1]\times[0,r]$, giving $|\widetilde u(x,y)|\lesssim y\lesssim_{t_1}t^{\ell/2+\gamma\eta_*}y$. Combining these bounds proves the estimate on the prescribed interval, with constants depending on the datum and $y_0$, uniformly in $x$ and $y$. In particular, their dependence on $s_0$ includes that of $r^{-(2\ell+1)}$. Renaming $\eta_*$ as the theorem exponent completes the proof.
		\end{proof}
		\section{Modulation equations}
		We derive the stable modulation laws by differentiating the two constraints in \eqref{defi-stable-problem} and projecting \eqref{equa-varepsilon-appen-} onto the first two eigenfunctions.
		
		\begin{proposition}[Dynamics of the modulation parameters]\label{pro-finite-projections}
			There exist $A_8\ge1$, $\eta_8(A)>0$ and $s_8(A,\eta)\ge1$ such that the following holds for $A\ge A_8$, $0<\eta\le\eta_8(A)$ and $s_0\ge s_8(A,\eta)$. Let $(\vartheta,b,\beta)$ be a classical solution in the stable regime satisfying \eqref{equa-vartheta} and \eqref{defi-stable-problem}, with $b,\beta\in C^1$ and $(\vartheta,b,\beta)(s)\in\mathcal V_1[A,\eta](s)$ on $[s_0,\bar s]$. Then, for $s_0<s<\bar s$,
			\begin{equation}\label{ode-beta-b}
				|\beta'(s)|+\left|\frac{b'(s)}{b(s)}+\beta(s)\right|\lesssim b(s)^{1/4}.
			\end{equation}
		\end{proposition}
		\begin{proof}
			All estimates are uniform on $(s_0,\bar s)$. Write $\kappa=c_{1,0}/c_{0,0}$, $\rho=\rho_{b,\beta}$, $\varphi_i=\varphi_{i,b,\beta}$ and $h_i=\|\varphi_i\|_{L^2_\rho}^2$ for $i=0,1$, and let $\psi=\psi_{1,b,\beta}$ be the adapted direction in \eqref{definition-adapted-eigendirection-strategy}. Proposition~\ref{diagonalisation-L-b} gives
			\begin{equation}\label{uniform-bounds-norm-eigenfunctions}
				c_*\le h_i\le C_*,\qquad \langle\varphi_0,\varphi_1\rangle_{L^2_\rho}=0,
			\end{equation}
			for small $b$ and $\beta\in[1/2,3/2]$. The stable constraints and \eqref{defi-m-0}--\eqref{defi-n-ell-1} give
			\begin{align}
				\varepsilon&=\varepsilon_1\psi+\varepsilon_-,\qquad
				\langle\varepsilon_-,\varphi_i\rangle_{L^2_\rho}=0,\label{decomposition-epsilon-proof-modulation}\\*
				\varepsilon_1&=\mathbf n_1(\beta)b^{1/2}=2\kappa^{-1}\mathbf m_0(\beta)b^{1/2},\label{compatibility-epsilon-one}\\*
				\frac{\varepsilon_1'}{\varepsilon_1}&=\frac12\frac{b'}b-\frac58\frac{\beta'}\beta,\label{derivative-epsilon-one}\\*
				|\mathbf n_1(\beta)|&\asymp|\mathbf m_0(\beta)|\asymp1.\label{n-one-uniform-nonzero}
			\end{align}
			
			\emph{Moving spectral projections.} Equations \eqref{uniform-bounds-norm-eigenfunctions} and \eqref{decomposition-epsilon-proof-modulation} imply
			\begin{equation}\label{projection-identities-epsilon-one}
				\varepsilon_1=h_1^{-1}\langle\varepsilon,\varphi_1\rangle_{L^2_\rho}
				=-\kappa^{-1}h_0^{-1}\langle\varepsilon,\varphi_0\rangle_{L^2_\rho}.
			\end{equation}
			Set $D_s^\rho f=\partial_sf+f\partial_s\log\rho$. Differentiating \eqref{projection-identities-epsilon-one}, using the equation and self-adjointness, gives
			\begin{align*}
				\varepsilon_1'&=\mu_{1,b,\beta}\varepsilon_1+h_1^{-1}\langle\mathscr B_b(\varepsilon)+\Phi_{b,\beta},\varphi_1\rangle_{L^2_\rho}+\mathcal K_1,\\*
				\varepsilon_1'&=\mu_{0,b,\beta}\varepsilon_1-\kappa^{-1}h_0^{-1}\langle\mathscr B_b(\varepsilon)+\Phi_{b,\beta},\varphi_0\rangle_{L^2_\rho}+\mathcal K_0,\\*
				h_1\mathcal K_1&=\langle\varepsilon_-,D_s^\rho\varphi_1\rangle_{L^2_\rho}-\varepsilon_1\langle\partial_s\psi,\varphi_1\rangle_{L^2_\rho},\\*
				\kappa h_0\mathcal K_0&=\varepsilon_1\langle\partial_s\psi,\varphi_0\rangle_{L^2_\rho}-\langle\varepsilon_-,D_s^\rho\varphi_0\rangle_{L^2_\rho}.
			\end{align*}
			Here $\Phi_{b,\beta}$ and $\mathscr B_b$ are defined in \eqref{Phi-simple-} and \eqref{defi-B-quadratic-appendix-}. The eigenfunction expansion \eqref{eigenfuc-L-b} gives $\Lambda_\xi Q_b=\mathbf m_0(\beta)b^{1/2}(\varphi_0-\widetilde\varphi_0)$, hence
			\begin{equation}\label{projections-Lambda-Q}
				\langle\Lambda_\xi Q_b,\varphi_1\rangle_{L^2_\rho}=O(b^{3/4}),\qquad
				\langle\Lambda_\xi Q_b,\varphi_0\rangle_{L^2_\rho}=\mathbf m_0(\beta)b^{1/2}h_0+O(b^{3/4}).
			\end{equation}
			We estimate the resulting system in terms of
			\begin{equation}\label{definition-X-one-X-two}
				X_1=\frac{b'}b+\beta,\qquad X_2=\frac{\beta'}\beta;
			\end{equation}
			in particular, $\beta-b'/b=2\beta-X_1$ and $|b'/b|\lesssim1+|X_1|$.
			
			\emph{Nonlinear projections.} We first prove
			\begin{equation}\label{nonlinear-projection-bound}
				|\langle\mathscr B_b(\varepsilon),\varphi_i\rangle_{L^2_\rho}|\lesssim b^{3/4},\qquad i=0,1.
			\end{equation}
			Put $z=\varepsilon/Q_b$. Since $\sqrt{Q_b}\,Q_b''=2$ and $z\ge-1$,
			\begin{equation}\label{quadratic-first-nonlinearity}
				\left|\left(\sqrt\vartheta-\sqrt{Q_b}-\frac{\varepsilon}{2\sqrt{Q_b}}\right)Q_b''\right|
				=2|\sqrt{1+z}-1-z/2|\lesssim z^2.
			\end{equation}
			The residual \eqref{defi-mathcal-F} also gives
			\begin{equation}\label{second-nonlinearity-via-F}
				(\sqrt\vartheta-\sqrt{Q_b})\varepsilon''
				=\bigl(1-(1+z)^{-1/2}\bigr)\bigl[\mathcal F-2(\sqrt{1+z}-1)\bigr].
			\end{equation}
			Split the half-line into $\Omega_1=(0,b^{3/4})$, $\Omega_2=(b^{3/4},b^{-\eta})$ and $\Omega_3=(b^{-\eta},\infty)$. The $T_0$ cancellation in $\psi$, the profile estimates and the shrinking bounds give $|z|\lesssim A^2b^\eta=o(1)$ on $\Omega_1\cup\Omega_2$, for $\eta\le1/10$. Thus \eqref{quadratic-first-nonlinearity}--\eqref{second-nonlinearity-via-F} imply $|\mathscr B_b(\varepsilon)|\lesssim z^2+|z||\mathcal F|$ there. On $\Omega_1$ we use $|\mathcal F|\le Ab^{1/4-\eta}\xi$, $|\varphi_i|\lesssim b^{-1/4}\xi$ and $\rho\lesssim b^{-1/8}\xi^{-1/2}$. On $\Omega_2$, after taking $A^2b^\eta\le1$, the bounds give $|z|\lesssim b^{1/2}\xi^{-2/3}\langle\xi\rangle^{8/3}$ and $|\varphi_i|\rho\lesssim\langle\xi\rangle^{4/3}e^{-c\xi^{4/3}}$. Consequently, for some fixed $M$,
			\begin{align*}
				\int_{\Omega_1}|\mathscr B_b(\varepsilon)||\varphi_i|\rho\,d\xi
				&\lesssim A^4b^{3/4+2\eta}+A^3b^{7/4},\\*
				\int_{\Omega_2}|\mathscr B_b(\varepsilon)||\varphi_i|\rho\,d\xi
				&\lesssim b\int_{b^{3/4}}^\infty\xi^{-4/3}\langle\xi\rangle^Me^{-c\xi^{4/3}}\,d\xi+Ab^{1-\eta}
				\lesssim b^{3/4}+Ab^{1-\eta}.
			\end{align*}
			Both are $O(b^{3/4})$ for $s_0\ge s_8$. On $\Omega_3$, the lower bound for $\vartheta$, the global bounds for $\vartheta,\mathcal F$, and \eqref{second-nonlinearity-via-F} give polynomial growth in $\xi$ and $b^{-1}$; the density then bounds the remaining integral by $A^Cb^{-C}e^{-cb^{-4\eta/3}}=O(b^2)$. This proves \eqref{nonlinear-projection-bound}.
			
			\emph{Motion of the basis and weight.} We use $P_b$ and $\rho$ from \eqref{defi-P-b}--\eqref{defi-rho-b-beta}. Since $\partial_sQ_b=(b'/b)\Lambda_\xi Q_b$, differentiation gives
			\begin{equation}\label{correct-derivative-log-weight}
				\partial_s\log\rho=-X_2P_b+\frac{b'}{2b}
				\left[\frac{3\beta}{4}\int_0^\xi yQ_b(y)^{-3/2}\Lambda_yQ_b(y)\,dy
				-\frac{\Lambda_\xi Q_b}{Q_b}\right].
			\end{equation}
			The profile asymptotics imply
			\begin{equation}\label{bound-D-b}
				\left|\frac{\Lambda_\xi Q_b}{Q_b}\right|\lesssim\min\{1,b^{1/2}\xi^{-2/3}\}.
			\end{equation}
			The spectral bounds, \eqref{correct-derivative-log-weight}--\eqref{bound-D-b}, and the shrinking bounds yield the following estimates, with all norms taken in $L^2_\rho$:
			\begin{alignat*}{2}
				\|b\partial_b\psi\|+\|(\beta\partial_\beta-13/8)\psi\|&\lesssim b^{1/4},&\qquad \|\varepsilon_-\|&\lesssim A^2b^{1/2+\eta},\\*
				\|b\partial_b\varphi_i+\varphi_i b\partial_b\log\rho\|&\lesssim b^{1/4},&\qquad \|\beta\partial_\beta\varphi_i-P_b\varphi_i\|&\lesssim1.
			\end{alignat*}
			Indeed, \eqref{eigenfuc-L-b} gives $\psi=c_{1,1}\beta^{13/8}b^{3/2}T_1(\xi/b^{3/4})+\widetilde\varphi_1-\kappa\widetilde\varphi_0$. The operator $b\partial_b$ cancels the leading outer powers of each $b^{j+1/2}T_j(\xi/b^{3/4})$, while $\beta\partial_\beta-13/8$ cancels the displayed $T_1$ term. Lemma~\ref{lemma-T-i} and Proposition~\ref{diagonalisation-L-b} bound the remaining terms by $O(b^{1/4})$ in $L^2_\rho$. For the weight terms, $P_b\lesssim\xi^{4/3}$ and the integral in \eqref{correct-derivative-log-weight} is $O(b^{1/2}\xi^{2/3})$ by \eqref{bound-D-b}. The estimate for $\varepsilon_-$ follows by integration over $\Omega_1,\Omega_2$; the tail is exponentially small.
			
			Since $\beta\partial_\beta\log\rho=-P_b$, Cauchy--Schwarz now gives
			\begin{equation*}
				|\langle\varepsilon_-,D_s^\rho\varphi_i\rangle_{L^2_\rho}|
				\lesssim A^2b^{1/2+\eta}\left(|X_2|+b^{1/4}|b'/b|\right).
			\end{equation*}
			Using $\partial_s\psi=(b'/b)b\partial_b\psi+X_2\beta\partial_\beta\psi$ in the exact expressions for $\mathcal K_i$, and taking $A^2b^\eta\le1$, we obtain
			\begin{equation*}
				\left|\mathcal K_i+\frac{13}{8}X_2\varepsilon_1\right|
				\lesssim b^{3/4}|b'/b|+(b^{3/4}+A^2b^{1/2+\eta})|X_2|.
			\end{equation*}
			
			\emph{Closing the modulation system.} Insert these bounds, \eqref{projections-Lambda-Q}, \eqref{nonlinear-projection-bound}, and $\mu_{i,b,\beta}=\beta(1/2-i)+O(b^{1/4})$ into the projected equations. This gives
			\begin{equation}\label{system-b_beta}
				\begin{aligned}
					\varepsilon_1'&=-\frac\beta2\varepsilon_1-\frac{13}{8}X_2\varepsilon_1+\mathcal R_1,\\
					\varepsilon_1'&=\frac\beta2\varepsilon_1-\frac{13}{8}X_2\varepsilon_1
					+\kappa^{-1}\mathbf m_0(\beta)b^{1/2}(b'/b-\beta)+\mathcal R_0,
				\end{aligned}
			\end{equation}
			where the total remainders satisfy
			\begin{equation}\label{final-remainders-modulation}
				|\mathcal R_0|+|\mathcal R_1|\lesssim b^{3/4}
				+(b^{3/4}+A^2b^{1/2+\eta})(|X_1|+|X_2|).
			\end{equation}
			By \eqref{derivative-epsilon-one}, the first equation in \eqref{system-b_beta} becomes $\frac12\varepsilon_1X_1+\varepsilon_1X_2=\mathcal R_1$, and the second becomes
			\begin{equation}\label{second-modulation-equation-before-cancellation}
				\left(\frac{\varepsilon_1}{2}-\kappa^{-1}\mathbf m_0(\beta)b^{1/2}\right)(b'/b-\beta)
				+\varepsilon_1X_2=\mathcal R_0.
			\end{equation}
			The first coefficient in \eqref{second-modulation-equation-before-cancellation} vanishes by \eqref{compatibility-epsilon-one}, leaving $\varepsilon_1X_2=\mathcal R_0$. All derivative terms in the remainders are linear in $(X_1,X_2)$, so \eqref{final-remainders-modulation} yields
			\begin{equation}\label{matrix-modulation-system-correct}
				b^{1/2}\left[\mathbf n_1(\beta)\begin{pmatrix}1/2&1\\0&1\end{pmatrix}+\mathcal E(s)\right]
				\binom{X_1}{X_2}=O(b^{3/4}),\qquad
				\|\mathcal E(s)\|\lesssim b^{1/4}+A^2b^\eta.
			\end{equation}
			The principal determinant is $\mathbf n_1(\beta)^2/2$, bounded away from zero by \eqref{n-one-uniform-nonzero}. Choose $\eta_8\le1/10$ and then $s_8(A,\eta)$ so large that the preceding estimates hold and $\sup_{[s_0,\bar s]}(b^{1/4}+A^2b^\eta)\ll1$. The matrix in \eqref{matrix-modulation-system-correct} is uniformly invertible, giving $|X_1|+|X_2|\lesssim b^{1/4}$. Together with \eqref{definition-X-one-X-two} and $\beta\sim1$, this proves \eqref{ode-beta-b}.
		\end{proof}
		
		\section{Pointwise estimates via the Duhamel representation}\label{section-Pointwise-Duhamel}
		We next implement the intermediate-region argument from Section~\ref{subsection-sketch-stable-regime}.  A fixed spectral projection controls $\varepsilon_-$ on $[b^\eta,K]$, while the spectral gap of $\widehat{\mathscr L}_{0,1}$ controls $\mathcal F$ on compact subsets of $(0,\infty)$. We use $\mathfrak f_1$, $\rho_{0,1}$, and $S_1$ from \eqref{defi-f-ell}, \eqref{defi-rho-infty}, and \eqref{defi-S-beta-0-t}.
		
		\begin{proposition}[Pointwise improvement for the stable remainder]\label{propo-priori-b-eta-K}
			Fix $K\ge1$ and the initial-data parameters of Lemma~\ref{lemma-initial-data-l=1}.
			There exist $A_9\ge A_8$, $\eta_9(A,\delta)\in(0,\delta/16]$, and
			$s_9=s_9(A,K,\eta,\delta,M_0,x_0)\ge\max\{s_1,s_8,s_{18}\}$
			such that the following holds for $A\ge A_9$, $0<\eta\le\eta_9(A,\delta)$,
			and $s_0\ge s_9$. Let $(\vartheta,b,\beta)$ be a classical solution of \eqref{equa-vartheta} and \eqref{defi-stable-problem}, with datum \eqref{defi-initi-varthe-s-0}, remaining in $\mathcal V_1[A,\eta](s)$ for $s\in[s_0,\bar s]$. Then
			\begin{equation*}
				|\varepsilon_-(s,\xi)|\le A^{-26}b(s)^{1/2+\eta}\mathfrak f_1(\xi),\qquad s_0\le s\le\bar s,\quad b(s)^\eta\le\xi\le K.
			\end{equation*}
			The threshold $s_9$ is independent of $\bar s$.
		\end{proposition}
		
		\begin{proof}
			Set $q=1/2+\eta$ and choose
			$\eta_9(A,\delta)\le\min\{\delta/16,1/16,\eta_8(A),\eta_{18}(A)\}$.
			With the stated choice of $s_9$, Proposition~\ref{pro-finite-projections}
			and Lemma~\ref{lem-esti-tilde-R-1-ep-plus-stable} apply.
			Estimate \eqref{ode-beta-b} and the prepared initial parameters give, with $\omega_0:=|\beta(s_0)-1|+b(s_0)^{1/4}$,
			\begin{equation}\label{uniform-beta-minus-one}
				\sup_{[s_0,\bar s]}|\beta-1|\lesssim\omega_0\longrightarrow0\quad(s_0\to\infty).
			\end{equation}
			Choose $\omega_*>0$ with $(1/2+\eta)\omega_*<\eta/2$. Increasing $s_9$ gives $|b'/b+1|\le\omega_*$ and hence, for either exponent $p=1/2\pm\eta$,
			\begin{equation}\label{comparison-b-beta-one}
				b(\tau)^p\le b(s)^p e^{p(1+\omega_*)(s-\tau)},\qquad s_0\le\tau\le s\le\bar s.
			\end{equation}
			
			\emph{Equations and fixed projection.} From \eqref{equa-varepsilon-appen-},
			\begin{equation*}
				\partial_s\varepsilon_-=\mathscr L_{0,1}\varepsilon_-+(\beta-1)\Lambda_\xi\varepsilon_-+\widetilde{\mathcal B}_1(\varepsilon)+\widetilde{\mathcal R}_1(\varepsilon_+),
			\end{equation*}
			where, using \eqref{defi-mathscr-L-}, \eqref{defi-L_0_beta}, and \eqref{defi-B-quadratic-appendix-},
			\begin{align}
				\widetilde{\mathcal B}_1(\varepsilon)&=(\mathscr L_{b,\beta}-\mathscr L_{0,\beta})\varepsilon+\mathscr B_b(\varepsilon),\label{definition-B-one-beta-one}\\*
				\widetilde{\mathcal R}_1(\varepsilon_+)&=-\partial_s\varepsilon_++\mathscr L_{0,\beta}\varepsilon_++(\beta-b'/b)\Lambda_\xi Q_b.\label{definition-R-one-beta-one}
			\end{align}
			In particular, \eqref{definition-R-one-beta-one} equals $\mathcal R_{1,b}+(\mathscr L_{0,\beta}-\mathscr L_{b,\beta})\varepsilon_+$, with $\mathcal R_{1,b}$ from \eqref{definition-R-one-b-stable}. We instead project the full perturbation equation, which cancels its leading forcing exactly.
			
			Write $\phi_j:=\phi_{1,j}$, with \eqref{defi-phi-beta-i}, and let $P_-$ be the orthogonal projection in $L^2_{\rho_{0,1}}$ onto $E_{\rm fd}^{\perp}$, where $E_{\rm fd}=\operatorname{span}\{\phi_0,\phi_1\}=\operatorname{span}\{\xi^{2/3},\xi^2\}$. Set $\widehat\varepsilon_-:=P_-\varepsilon$. Since $\Lambda_\xi E_{\rm fd}\subset E_{\rm fd}$, both $\mathscr L_{0,\beta}$ and $\Lambda_\xi$ preserve this space, so
			\begin{equation}\label{equation-fixed-stable-beta-one}
				\partial_s\widehat\varepsilon_-=\mathscr L_{0,1}\widehat\varepsilon_-+P_-\bigl[(\beta-1)\Lambda_\xi\widehat\varepsilon_-+\widetilde{\mathcal B}_1(\varepsilon)+\Phi_{b,\beta}\bigr],
			\end{equation}
			with $\Phi_{b,\beta}$ defined in \eqref{Phi-simple-}. We use the weak Duhamel formula of Lemma~\ref{lemma-weak-duhamel-singular-source}, with the weighted space $H=L^2_{\rho_{0,1}}$ and form domain $V$ defined there. On every finite trapping interval, $b$ is bounded away from zero. Near the wall, the profile bounds and $\vartheta''=(2+\mathcal F)/\sqrt\vartheta$ give $\varepsilon=O(\xi)$ and $\varepsilon'=O(1)$, uniformly on that interval. Since $\phi_j=O(\xi^{2/3})$ and $\phi_j'=O(\xi^{-1/3})$, the projected unknown $u:=\widehat\varepsilon_-$ satisfies $u=O(\xi^{2/3})$ and $u'=O(\xi^{-1/3})$. At infinity, the global bounds for $\vartheta,\mathcal F$ bound $\vartheta''$ and, by interpolation, $\vartheta'$; subtracting the profiles leaves polynomial bounds for $u,u'$. Thus $u\in L^\infty([s_*,s];V)\cap C([s_*,s];H)$, with its classical initial trace. The source bound in Lemma~\ref{lem-esti-tilde-R-1-ep-plus-stable}(i) places $\widetilde{\mathcal B}_1+\Phi_{b,\beta}$ in $L^\infty([s_*,s];V')$, and $\Lambda_\xi:V\to V'$ is bounded. Density therefore extends the perturbation equation from interior tests to $V$, and projection gives \eqref{equation-fixed-stable-beta-one} in $V'$. The cited lemma yields
			\begin{equation}\label{Duhamel-fixed-beta-one}
				\widehat\varepsilon_-(s)=S_1(s-s_*)\widehat\varepsilon_-(s_*)+\int_{s_*}^s S_1(s-\tau)P_-\bigl[(\beta-1)\Lambda_\xi\widehat\varepsilon_-+\widetilde{\mathcal B}_1+\Phi_{b,\beta}\bigr](\tau)\,d\tau.
			\end{equation}
			The integral is defined using the extension $S_1(t):V'\to H$. Its pointwise kernel representation follows by spatial cutoff approximation as in Lemma~\ref{lemma-weak-duhamel-singular-source}; the bounds below provide the required integrable majorants on $b(s)^\eta\le\xi\le K$.
			
			The expansion \eqref{eigenfuc-L-b}, the $T_0$ cancellation in \eqref{definition-adapted-eigendirection-strategy}, and Lemma~\ref{lemma-T-i} show that the leading part of $\varepsilon_+$ is a multiple of $\xi^2$, with error $O(b^{3/4}\mathfrak f_1)$. Moving orthogonality then yields
			\begin{align}
				\|\widehat\varepsilon_-(s)\|_{L^2_{\rho_{0,1}}}&\lesssim_A b(s)^q,\label{rough-L2-fixed-beta-one}\\*
				|\widehat\varepsilon_--\varepsilon_-|&\lesssim_{A,K}\bigl(|\beta-1|b^q+b^{3/4}\bigr)\mathfrak f_1,
				\qquad b^\eta\le\xi\le K.\label{comparison-moving-fixed-beta-one}
			\end{align}
			Indeed, $\widehat\varepsilon_--\varepsilon_-=P_-\varepsilon_+-(I-P_-)\varepsilon_-$. Pairing the eigenfunction and density expansions against the shrinking-set bounds gives $|\langle\varepsilon_-,\phi_j\rangle_{\rho_{0,1}}|\lesssim_A(|\beta-1|+b^{1/4})b^q$ for $j=0,1$. On $\xi\le b^{-\eta}$, these bounds also give $|\widehat\varepsilon_-|\lesssim_A b^q\mathfrak f_1$; beyond this cutoff, the coarse profile bounds give only polynomial growth. Their Gaussian-weighted tails are $O(b^N)$ for every fixed $N$, which proves \eqref{rough-L2-fixed-beta-one} without a global pointwise bootstrap bound.
			
			\emph{Bounds on a fixed time interval.} Fix $L\ge1$ and put $\sigma_\eta:=1/4-8\eta/3>0$. For $0<s-\tau\le L$ and $b(s)^\eta\le\xi\le K$, the kernel bounds \eqref{Semi-hit-Derivative} and Lemma~\ref{lem-semi-tle-K-0-xi-le-K} give
			\begin{align}
				|S_1(s-\tau)P_-\Lambda_\xi\widehat\varepsilon_-(\tau)(\xi)|
				&\lesssim_{A,K,L}(s-\tau)^{-1/2}b(s)^q\mathfrak f_1(\xi),\label{short-time-Lambda-beta-one}\\*
				|S_1(s-\tau)P_-[\widetilde{\mathcal B}_1+\Phi_{b,\beta}](\tau)(\xi)|
				&\lesssim_{A,K,L}b(s)^{q+\sigma_\eta}\mathfrak f_1(\xi).\label{short-time-forcing-beta-one}
			\end{align}
			Here and below constants may also depend on the fixed initial-data parameters and $\eta$. We give the details needed to apply these estimates. For \eqref{short-time-Lambda-beta-one}, split $\widehat\varepsilon_-$ at $b(\tau)^{-\eta}$: the inner part has the polynomial envelope just proved, and integration by parts in the differentiated kernel makes the outer contribution $O(b(\tau)^N)\xi^{2/3}$. The two subtracted finite-mode coefficients are bounded by transferring $\Lambda_\xi$ to its weighted adjoint on $\phi_0,\phi_1$, whose growth is polynomial. No derivative bound on $\varepsilon_-$ is required.
			
			For \eqref{short-time-forcing-beta-one}, Lemma~\ref{lem-esti-tilde-R-1-ep-plus-stable}(i) and the envelope lemma bound $S_1(t)|\widetilde{\mathcal B}_1|$ by $Cb^{3/4-5\eta/3}\xi^{2/3}$. The same inner, intermediate, and far-field integrals give $|\langle\widetilde{\mathcal B}_1,\phi_j\rangle_{\rho_{0,1}}|\lesssim_A b^{3/4-\eta}$, so subtracting the first two modes preserves this estimate. These moments are well defined even though the inner envelope need not lie in $L^2_{\rho_{0,1}}$.
			
			For the remaining forcing, Lemma~\ref{lemma-Q} gives $|\Lambda_\xi Q_b-\frac a8b^{1/2}\xi^{2/3}|\lesssim b^{3/4}\xi^{1/3}$ on $(0,\infty)$. The leading monomial belongs to $E_{\rm fd}$ and $|\beta-b'/b|\lesssim1$, hence
			\begin{equation*}
				|P_-\Phi_{b,\beta}|\lesssim b^{3/4}(\xi^{1/3}+\mathfrak f_1),\qquad
				\|P_-\Phi_{b,\beta}\|_{L^2_{\rho_{0,1}}}\lesssim b^{3/4}.
			\end{equation*}
			The positive kernel satisfies $S_1(t)(\xi^{1/3})\le e^{3t/4}\xi^{1/3}$: apply comparison on truncated intervals to the positive solution $e^{3t/4}\xi^{1/3}$ and pass to the Friedrichs semigroup. Thus its contribution is $O_{K,L}(b^{3/4-\eta/3}\xi^{2/3})$ on $[b^\eta,K]$, which is smaller than the bound in \eqref{short-time-forcing-beta-one}. Finally, \eqref{comparison-b-beta-one} makes $b(\tau)$ and $b(s)$ comparable on this interval; replacing the lower cutoff by a fixed multiple changes only the constants in the kernel estimates.
			
			By \eqref{uniform-beta-minus-one}, integration of \eqref{short-time-Lambda-beta-one}--\eqref{short-time-forcing-beta-one} gives
			\begin{equation}\label{recent-forcing-beta-one-term}
				\left|\int_{s_*}^s S_1(s-\tau)P_-\bigl[(\beta-1)\Lambda_\xi\widehat\varepsilon_-+\widetilde{\mathcal B}_1+\Phi_{b,\beta}\bigr](\tau)\,d\tau\right|
				\lesssim_{A,K,L}(\omega_0+b(s_0)^{\sigma_\eta})b(s)^q\mathfrak f_1,
			\end{equation}
			whenever $0\le s-s_*\le L$. For $s_*=s_0$, \eqref{initial-epsilon-minus-exponential} and the preceding estimate of $P_-\varepsilon_+$ also give $|S_1(s-s_0)\widehat\varepsilon_-(s_0)|\lesssim_{A,K,L}b(s_0)^{3/4}\mathfrak f_1$. The initial tail starts at $R_\delta/8\to\infty$ and is absorbed by the Gaussian kernel. Consequently, \eqref{Duhamel-fixed-beta-one} and \eqref{comparison-moving-fixed-beta-one} prove the desired improvement for $s-s_0\le L$, after increasing $s_9$ with $L$ fixed.
			
			\emph{Restart and spectral decay.} Lemma~\ref{lemma-spectrum-L0-beta-xi} gives eigenvalues $1/2-j$. The kernel at time $1$, followed by the spectral gap on $E_{\rm fd}^{\perp}$, therefore yields
			\begin{equation}\label{stable-semigroup-decay-beta-one}
				|S_1(t)P_-f(\xi)|\lesssim_K e^{-3t/2}\xi^{2/3}\|f\|_{L^2_{\rho_{0,1}}},\qquad t\ge1,\quad0<\xi\le K.
			\end{equation}
			For $s-s_0\ge L$, take $s_*=s-L$ in \eqref{Duhamel-fixed-beta-one}. Equations \eqref{rough-L2-fixed-beta-one}, \eqref{comparison-b-beta-one}, and \eqref{stable-semigroup-decay-beta-one} bound the initial term by $C_{A,K}e^{-(3/2-q(1+\omega_*))L}b(s)^q\mathfrak f_1$. Together with \eqref{recent-forcing-beta-one-term} and \eqref{comparison-moving-fixed-beta-one}, this gives
			\begin{equation*}
				\frac{|\varepsilon_-(s,\xi)|}{b(s)^q\mathfrak f_1(\xi)}
				\le C_{A,K}e^{-(3/2-q(1+\omega_*))L}
				+C_{A,K,L}(\omega_0+b(s_0)^{\sigma_\eta}).
			\end{equation*}
			The first constant is independent of $L$ and $3/2-q(1+\omega_*)>0$. Choose $L$ so that the first term is at most $A^{-26}/2$, then $s_9$ so that the second term and the short-time bound are at most $A^{-26}/2$. This proves the proposition.
		\end{proof}
		
		We next estimate the residual \eqref{defi-mathcal-F}. Differentiating its definition and using \eqref{equa-vartheta} gives the exact equation
		\begin{equation}\label{equa-mathcal-F}
			\partial_s\mathcal F=\sqrt\vartheta\,\partial_\xi^2\mathcal F-\tfrac34\beta\xi\partial_\xi\mathcal F
			+\frac{\mathcal F}{\vartheta}+\frac{\mathcal F^2}{2\vartheta}.
		\end{equation}
		The principal error relative to \eqref{definition-Lhat-zero-one} is $(\sqrt\vartheta-a\xi^{2/3})\partial_\xi^2\mathcal F$. We use local parabolic compactness to pass to the limiting equation, since $\mathcal V_1$ controls $\mathcal F$ but not its second derivative.
		
		\begin{lemma}[Improved estimate for $\mathcal F$ on compact sets]\label{lem-priori-F-K}
			Under the assumptions of Proposition~\ref{propo-priori-b-eta-K}, there exists
			$s_{10}=s_{10}(A,K,\eta,\delta,M_0,x_0)\ge s_9$ such that, for $s_0\ge s_{10}$,
			the same solution satisfies
			\begin{equation}\label{improved-F-compact}
				|\mathcal F(s,\xi)|\le A^{-26}b(s)^{1/2-\eta},\qquad s_0\le s\le\bar s,\quad K^{-1}\le\xi\le K.
			\end{equation}
		\end{lemma}
		
		\begin{proof}
			The inherited hypotheses give $A\ge A_9\ge A_8$ and
			$0<\eta\le\eta_9(A,\delta)$. It suffices to treat $K>1$;
			for $K=1$, enlarge the spatial interval and include its threshold in $s_{10}$.
			Set $q=1/2-\eta>0$ and $I_K=[K^{-1},K]$, and fix $\omega_*$ as in \eqref{comparison-b-beta-one}. The shrinking-set bounds and Gaussian decay of \eqref{defi-rho-infty} imply
			\begin{equation}\label{rough-L2-F-beta-one}
				\|\mathcal F(s)\|_{L^2_{\rho_{0,1}}}\lesssim_A b(s)^q.
			\end{equation}
			Indeed, $b^{-1/4}\xi\le\xi^{2/3}$ for $\xi\le b^{3/4}$, so $|\mathcal F|\le Ab^q\mathfrak f_1$ on $(0,b^{-\eta})$. The far-field bound contributes $O(b^N)$ for every fixed $N$. We will also use the uniform envelope
			\begin{equation*}
				b^{-q}|\mathcal F(s,\xi)|\le A\left[\mathfrak f_1(\xi)
				+\mathbf1_{\{\xi\ge1\}}\xi^{q/\eta+1/8}\right],
			\end{equation*}
			whose right-hand side belongs to $L^2_{\rho_{0,1}}$: on $\xi\ge b^{-\eta}$, use $b^{-q-\eta/8}\le\xi^{q/\eta+1/8}$. For $\widehat S_1$ in \eqref{notation-semigroup-F}, the eigenvalues are $-1/2-j$, hence its $L^2_{\rho_{0,1}}$ norm is at most $e^{-t/2}$. Combining this with the kernel bound from $L^2_{\rho_{0,1}}$ to $L^\infty(I_K)$ at time $1$ gives
			\begin{equation}\label{semigroup-decay-F-beta-one}
				\|\widehat S_1(t)f\|_{L^\infty(I_K)}\lesssim_K e^{-t/2}\|f\|_{L^2_{\rho_{0,1}}},\qquad t\ge1.
			\end{equation}
			
			Suppose that \eqref{improved-F-compact} fails for arbitrarily large $s_0$. Choose solutions $(\vartheta_n,b_n,\beta_n)$ on $[s_{0,n},\bar s_n]$, with $s_{0,n}\to\infty$, and points $(\sigma_n,\xi_n)$ of failure. Passing to a subsequence, assume $\xi_n\to\xi_*\in I_K$, and set
			\begin{equation*}
				T_n=\sigma_n-s_{0,n},\qquad B_n=b_n(\sigma_n),\qquad
				G_n(t,\xi)=B_n^{-q}\mathcal F_n(\sigma_n+t,\xi),\quad -T_n\le t\le0.
			\end{equation*}
			Thus $|G_n(0,\xi_n)|>A^{-26}$. On $[-\min\{L,T_n\},0]$, \eqref{comparison-b-beta-one} and the preceding envelope give local bounds and uniformly small $L^2_{\rho_{0,1}}$ tails at zero and infinity, for each fixed $L$.
			
			\emph{Local compactness.} Equation \eqref{equa-mathcal-F} gives the exact normalized equation
			\begin{equation*}
				\partial_tG_n=a_n\partial_\xi^2G_n+d_n\partial_\xi G_n+c_nG_n,
				\qquad
				a_n=\sqrt{\vartheta_n},\quad d_n=-\tfrac34\beta_n\xi,\quad
				c_n=\frac{1+\mathcal F_n/2}{\vartheta_n},
			\end{equation*}
			where the coefficients are evaluated at $\sigma_n+t$. On every fixed $J\Subset(0,\infty)$, the decomposition and shrinking-set bounds give $\vartheta_n\to a^2\xi^{4/3}$ and $\mathcal F_n\to0$ uniformly in time. Consequently,
			\begin{equation*}
				\partial_\xi^2\vartheta_n=\frac{2+\mathcal F_n}{\sqrt{\vartheta_n}}
				\longrightarrow\frac{2}{a\xi^{2/3}}
				=\partial_\xi^2(a^2\xi^{4/3}).
			\end{equation*}
			Interpolation on a slightly larger spatial interval also gives convergence of the first derivatives. Thus $\vartheta_n\to a^2\xi^{4/3}$ in $C^2(J)$, uniformly in time, and \eqref{uniform-beta-minus-one} gives $\beta_n\to1$. In particular, for constants independent of $n$,
			\begin{equation*}
				0<\lambda_J\le a_n\le\Lambda_J,\qquad
				|\partial_\xi a_n|+|d_n|+|c_n|\le C_J.
			\end{equation*}
			Writing $a_nG_{n,\xi\xi}=(a_nG_{n,\xi})_\xi-(\partial_\xi a_n)G_{n,\xi}$, interior parabolic H\"older estimates apply with uniform constants; see \cite[Section~1]{Aronson1968}. On cylinders separated from the initial time, they give local uniform compactness, including the terminal slice $t=0$ by estimates on backward cylinders. Every resulting limit, denoted by $\mathcal F_*$, satisfies
			\begin{equation*}
				\partial_t\mathcal F_*=\widehat{\mathscr L}_{0,1}\mathcal F_*,\qquad
				|\mathcal F_*(t,\xi)|\lesssim_{A,L}\mathfrak f_1(\xi)
			\end{equation*}
			on each time interval of length at most $L$ under consideration. Indeed, the $C^2$ convergence of $a_n$ permits transfer of the two spatial derivatives to $a_n$ times a test function, while $d_n\to-3\xi/4$ and $c_n\to a^{-2}\xi^{-4/3}$. The pointwise bound follows because every fixed spatial interval eventually lies below $b_n^{-\eta}$. The uniform weighted tails also give strong $L^2_{\rho_{0,1}}$ convergence at each time where local uniform convergence holds.
			
			\emph{The Friedrichs evolution.} To identify the realization at the degenerate endpoint, set
			\begin{equation*}
				z=a\xi^{2/3},\qquad \mathcal F_*(t,\xi)=zv(t,z),\qquad
				m(z)=z^{3/2}e^{-z^2/8}.
			\end{equation*}
			The conjugation in Lemma~\ref{lemma-spectrum-L-z} gives
			\begin{equation}\label{residual-limit-ground-state-equation}
				\partial_tv=2\partial_z^2v+\left(\frac3z-\frac z2\right)\partial_zv-\frac12v
				=2m^{-1}\partial_z(m\partial_zv)-\frac12v,
			\end{equation}
			and the change of variables gives
			\begin{equation*}
				\|\mathcal F_*(t)\|_{L^2_{\rho_{0,1}}}^2
				=\frac{3}{2\sqrt a}\int_0^\infty |v(t,z)|^2m(z)\,dz.
			\end{equation*}
			Define
			\begin{equation*}
				H_m=L^2(m\,dz),\qquad
				V_m=\overline{C_c^\infty(0,\infty)}^{\,\|\cdot\|_{V_m}},\qquad
				\|v\|_{V_m}^2=\int_0^\infty(|v|^2+|v_z|^2)m\,dz.
			\end{equation*}
			Interior regularity for the limiting equation justifies the following local energy computation. For $\chi\in C_c^\infty(0,\infty)$, multiplication of
			\eqref{residual-limit-ground-state-equation} by $\chi^2vm$ and integration by parts yield
			\begin{equation}\label{residual-limit-ground-state-energy}
				\frac12\frac d{dt}\int\chi^2v^2m
				+2\int\chi^2|\partial_zv|^2m
				+\frac12\int\chi^2v^2m
				=-4\int\chi\chi'v\partial_zv\,m.
			\end{equation}
			The right-hand side is absorbed by Young's inequality, up to
			$C\int|\chi'|^2v^2m$. Since $\mathcal F_*=O(\xi^{2/3})=O(z)$, $v$ is bounded at zero. For a cutoff at $z=r$, the error is therefore
			$O(r^{-2}\int_0^{2r}z^{3/2}\,dz)=O(r^{1/2})$; at infinity it tends to zero because $v$ has at most polynomial growth whereas $m$ has Gaussian decay. Integrating over any compact time interval $[t_1,t_2]$ away from the initial endpoint and removing both cutoffs gives $v_z\in L^2((t_1,t_2)\times\mathbb R_+,m\,dt\,dz)$. The same cutoffs approximate $v$ in the form norm after time integration, so $v\in L^2(t_1,t_2;V_m)$. By density, its equation extends to tests $w\in V_m$ as
			\begin{equation*}
				\langle\partial_tv,w\rangle_{V_m',V_m}
				=-2\int v_zw_zm\,dz-\tfrac12\int vwm\,dz.
			\end{equation*}
			Thus $\partial_tv\in L^2(t_1,t_2;V_m')$ and $v\in C([t_1,t_2];H_m)$. Testing the equation for the difference of two such solutions by that difference proves uniqueness. The closed form is exactly the ground-state conjugate of the Friedrichs form in Lemma~\ref{lemma-spectrum-L-z}; hence
			\begin{equation*}
				\mathcal F_*(t)=\widehat S_1(t-t_1)\mathcal F_*(t_1),\qquad t_1\le t\le t_2.
			\end{equation*}
			
			\emph{Control at the initial time.} Suppose $T_n\le L$ and put $h_n(t,\xi)=G_n(t-T_n,\xi)$ for $0\le t\le T_n$. Lemma~\ref{lemma-initial-data-l=1} and \eqref{comparison-b-beta-one} give $h_n(0)\to0$ locally uniformly and strongly in $L^2_{\rho_{0,1}}$: before the cutoff the normalized bound gains $b_n(s_{0,n})^\eta$, and beyond $R_\delta\to\infty$ the Gaussian weight absorbs the normalization.
			
			Fix $I\Subset J=(j_-,j_+)\Subset(0,\infty)$ and let $d_0=\operatorname{dist}(I,\partial J)>0$. On $[0,T_n]\times J$, the translated coefficients satisfy $a_n\le\Lambda$, $|d_n|\le D$, and $c_n\le C_0$, with $C_0\ge0$, while $|h_n|\le M$ uniformly in $n$. Set $\delta_n=\|h_n(0)\|_{L^\infty(J)}\to0$. For each $\lambda>0$, the function
			\begin{equation*}
				W_n(t,\xi)=e^{C_0t}\left[\delta_n+
				M e^{(\Lambda\lambda^2+D\lambda)t}
				\left(e^{-\lambda(\xi-j_-)}+e^{-\lambda(j_+-\xi)}\right)\right]
			\end{equation*}
			is a supersolution of the translated linear equation and dominates $\pm h_n$ on its parabolic boundary. Indeed, for either exponential the remaining coefficient is $(\Lambda-a_n)\lambda^2+D\lambda\pm d_n\lambda+C_0-c_n\ge0$. Comparison and the choice $\lambda=d_0/(2\Lambda t)$ at each fixed target time give
			\begin{equation}\label{residual-uniform-initial-trace}
				\|h_n(t)\|_{L^\infty(I)}
				\le e^{C_0t}\delta_n+C_1e^{-c_1/t},\qquad 0<t\le\min\{1,T_n\},
			\end{equation}
			where $C_1,c_1>0$ are independent of $n$. For a limiting interval of positive length, this estimate on arbitrary compact spatial intervals, together with the uniform weighted tails, gives a zero initial trace in $L^2_{\rho_{0,1}}$.
			
			\emph{Contradiction and decay.} After extraction, there are three cases. If $T_n\to0$, \eqref{residual-uniform-initial-trace} at $t=T_n$, or the prepared bound when $T_n=0$, gives $\|G_n(0)\|_{L^\infty(I_K)}\to0$, contradicting the choice of $(\sigma_n,\xi_n)$. If $T_n\to T\in(0,\infty)$, compactness gives a limit on $(-T,0]$ with $|\mathcal F_*(0,\xi_*)|\ge A^{-26}$. Its initial trace at $-T$ is zero by \eqref{residual-uniform-initial-trace}; letting $t_1\downarrow-T$ in the semigroup identity proves $\mathcal F_*\equiv0$, again a contradiction.
			
			Finally, if $T_n\to\infty$, diagonal extraction on bounded backward time intervals gives a limit on $(-\infty,0]$, with the same nonzero value at $(0,\xi_*)$. For every $L\ge1$, \eqref{comparison-b-beta-one} and \eqref{rough-L2-F-beta-one} give $\|\mathcal F_*(-L)\|_{L^2_{\rho_{0,1}}}\le C_Ae^{q(1+\omega_*)L}$. Therefore \eqref{semigroup-decay-F-beta-one} yields
			\begin{equation*}
				\|\mathcal F_*(0)\|_{L^\infty(I_K)}
				=\|\widehat S_1(L)\mathcal F_*(-L)\|_{L^\infty(I_K)}
				\le C_{A,K}e^{-[1/2-q(1+\omega_*)]L}
				\le C_{A,K}e^{-\eta L/2}.
			\end{equation*}
			Here $C_{A,K}$ is independent of $L$. Letting $L\to\infty$ contradicts the same lower bound at $(0,\xi_*)$, proving \eqref{improved-F-compact} uniformly on the trapping interval.
		\end{proof}
		\section{Comparison principles}\label{Section-comparison-varep}
		We construct sub- and supersolutions to obtain pointwise estimates for
		$\vartheta$ and $\varepsilon_-$ using parabolic comparison; see
		\cite[Appendix~F]{Souplet19-book} for the general principle.
		The degeneracy at $\xi=0$ requires barriers compatible with the singular wall behavior and comparison arguments on truncated cylinders, as justified below.
		
		\subsection{Comparison principle for the far-field estimates}
		For bounded continuous $\beta$, the operator associated with \eqref{equa-vartheta}, using \eqref{Lambda:op}, is
		\begin{equation}\label{definition-operator-P-beta}
			\mathscr P_\beta(f):=\partial_s f-\sqrt f\,\partial_\xi^2f-\beta(s)\Lambda_\xi f+2
			=\partial_s f-\sqrt f\,\partial_\xi^2f+\tfrac34\beta\xi\partial_\xi f-\beta f+2.
		\end{equation}
		
		For the rest of this section, fix the stable initial-data choice $\delta=3/8$, and let $(\vartheta,b,\beta)$ be a classical solution of \eqref{equa-vartheta} and \eqref{defi-stable-problem}, with datum \eqref{defi-initi-varthe-s-0}, remaining in $\mathcal V_1[A,\eta](s)$ on $[s_0,\bar s]$.
		Take $A\ge A_8$ and $0<\eta\le\min\{1/32,\eta_8(A)\}$.
		Assume $s_0\ge\max\{s_1,s_8\}$, with $\delta=3/8$; each result below
		states its additional threshold, independently of $\bar s$. By \eqref{ode-beta-b},
		\begin{equation}\label{preliminary-beta-b-comparison}
			r_b:=b'/b+\beta=O(b^{1/4}),\qquad |\beta'|\lesssim b^{1/4},\qquad
			\tfrac12\le\beta\le\tfrac32,\qquad b'/b\le-\tfrac14.
		\end{equation}
		
		\begin{lemma}[Pointwise comparison bounds]\label{lem-comparison-vartheta}
			There exist $A_{11}\ge\max\{4,A_8\}$ and
			$\eta_{11}(A)\in(0,\min\{1/32,\eta_8(A)\}]$ such that, under the
			preceding assumptions, for $A\ge A_{11}$ and $0<\eta\le\eta_{11}(A)$,
			there is a threshold
			$s_{11}=s_{11}(A,\eta,M_0,x_0)\ge\max\{s_1,s_8\}$ such that,
			for $s_0\ge s_{11}$ and $s_0\le s\le\bar s$,
			\begin{alignat}{2}
				\vartheta(s,\xi)&\ge b(s)^{-\eta/2}
				&\qquad&\xi\ge b(s)^{-\eta},
				\label{lower-comparison-vartheta}\\*
				\vartheta(s,\xi)&\le2b(s)^{-1-\eta}
				&\qquad&\xi\ge0.\nonumber
			\end{alignat}
		\end{lemma}
		
		\begin{proof}
			Set $\xi_b(s)=b(s)^{-\eta}$ and consider the barriers
			$h_-=b^{-\eta/2}$ and $h_+=2b^{-1-\eta}$ on the respective
			parabolic domains
			\[
			\mathcal D_-:=\{(s,\xi):s_0<s<\bar s,\ \xi>\xi_b(s)\},
			\qquad
			\mathcal D_+:=(s_0,\bar s)\times(0,\infty).
			\]
			By \eqref{preliminary-beta-b-comparison}, increasing $s_{11}$
			if necessary,
			\begin{align*}
				\mathscr P_\beta(h_-)
				&=\bigl[-(1-\eta/2)\beta-(\eta/2)r_b\bigr]h_-+2<0
				&&\text{in }\mathcal D_-,\\
				\mathscr P_\beta(h_+)
				&=\bigl[\eta\beta-(1+\eta)r_b\bigr]h_++2>0
				&&\text{in }\mathcal D_+.
			\end{align*}
			
			We next verify the ordering on the parabolic boundaries.
			On the lateral boundary $\xi=\xi_b(s)$ of $\mathcal D_-$,
			the ground-state asymptotics and shrinking bounds give
			\[
			\vartheta(s,\xi_b(s))
			\ge\tfrac12a^2b(s)^{-4\eta/3}\ge2h_-(s).
			\]
			Monotonicity of the prepared datum in
			Lemma~\ref{lemma-initial-data-l=1} then gives the initial lower
			ordering on $[\xi_b(s_0),\infty)$. The same lemma yields
			$\sup_\xi\vartheta(s_0,\xi)\lesssim\widetilde b_0^{-1}
			\ll h_+(s_0)$, while $\vartheta(s,0)=0<h_+(s)$ gives the
			upper ordering at $\xi=0$.
			Finally, the comparison principle yields the desired bounds,
			completing the proof.
		\end{proof}
		
		\subsection{Comparison in the inner region}
		We use $\mathfrak f_1$ from \eqref{defi-f-ell}. The adapted direction in \eqref{definition-adapted-eigendirection-strategy}, the expansion \eqref{eigenfuc-L-b}, and Lemma~\ref{lemma-T-i} give
		\begin{equation}\label{outer-estimates-epsilon-plus}
			\sum_{k=0}^2| (\xi\partial_\xi)^k\varepsilon_+|
			\lesssim
			\begin{cases}
				b^{3/4}\xi,&0\le\xi\le b^{3/4},\\
				b^{1/2}\xi^2+b^{3/4}\mathfrak f_1(\xi),&\xi\ge b^{3/4}.
			\end{cases}
		\end{equation}
		In particular, on $[b^{3/4},b^\eta]$ its ratio to $b^{1/2+\eta}\xi^{2/3}$ is $O(b^{\eta/3}+b^{1/4-\eta})$. Thus, it is enough to estimate for  $\varepsilon$.
		\begin{proposition}[Pointwise estimate in the inner region]\label{propo-priori-near-0}
			There exist $A_{12}\ge\max\{A_9,A_{11}\}$, $\eta_{12}(A)\in(0,1/32]$, and
			$s_{12}=s_{12}(A,\eta,M_0,x_0)\ge\max\{s_9,s_{11},s_{18}\}$,
			with $s_9$ evaluated at $K=1$ and $\delta=3/8$, independent of $\bar s$,
			such that, under the standing assumptions, for
			$A\ge A_{12}$, $0<\eta\le\eta_{12}(A)$ and $s_0\ge s_{12}$,
			the following estimate holds for $s_0\le s\le\bar s$:
			\begin{equation}\label{a-priori-estimate-epsilon-near-zero}
				|\varepsilon(s,\xi)|\le\frac A4
				\begin{cases}
					b^{1/4+\eta}\xi,&0\le\xi\le b^{3/4},\\
					b^{1/2+\eta}\xi^{2/3},&b^{3/4}\le\xi\le b^\eta.
				\end{cases}
			\end{equation}
		\end{proposition}
		\begin{proof}
			Choose $A_{12}\ge\max\{A_9,A_{11}\}$ sufficiently large and
			$\eta_{12}(A)\le\min\{\eta_9(A,3/8),\eta_{11}(A),\eta_{18}(A)\}$.
			All subsequent increases are absorbed into these thresholds.
			
			\emph{Inner equation and profiles.} Use the variables \eqref{definition-inner-theta-zeta}, put $\widetilde\theta=\theta-Q$, and set $R(s)=b^{\eta-3/4}$. The exact transformed equation is
			\begin{equation}\label{equa-theta}
				b\partial_s\theta=\sqrt\theta\,\partial_\zeta^2\theta-2+\gamma\Lambda_\zeta\theta,
				\qquad \gamma=b\beta-b'=2\beta b+O(b^{5/4}).
			\end{equation}
			Since $\varepsilon(s,\xi)=b\widetilde\theta(s,\xi/b^{3/4})$, the shrinking-set bounds, \eqref{outer-estimates-epsilon-plus}, and Lemma~\ref{lemma-Q} imply
			\begin{equation}\label{tilde-theta-bootstrap-inner}
				|\widetilde\theta|\lesssim A^2b^\eta\Lambda Q,
				\qquad |\widetilde\theta/Q|\lesssim A^2b^\eta,
				\qquad \tfrac12Q\le Q+\widetilde\theta\le\tfrac32Q
				\quad(0<\zeta\le R).
			\end{equation}
			
			Fix $\sigma=A^{-24}$ and use $Q_\sigma$ and $\mathscr H_\sigma$ from \eqref{definition-Q-sigma-general-p} and \eqref{notation-H-sigma}. Define
			\begin{equation}\label{equations-Z-Y-general-p}
				\mathcal Z_\sigma:=\Lambda Q_\sigma>0,\qquad
				\mathcal Y_\sigma:=\mathscr H_\sigma^{-1}(-\Lambda Q/\sqrt Q),\qquad
				\mathscr H_\sigma\mathcal Z_\sigma=0,
			\end{equation}
			where the inverse has zero value and derivative at $\zeta=0$. Scaling and the ground-state expansions yield
			\begin{alignat*}{3}
				\mathcal Z_\sigma&=\tfrac14\sigma^{1/4}\zeta+O_\sigma(\zeta^{3/2}),&\qquad \mathcal Y_\sigma&=-\tfrac1{15}\zeta^{5/2}+O_\sigma(\zeta^3)&\qquad & (\zeta\to0),\\*
				\mathcal Z_\sigma&=\tfrac a8\sigma^{1/2}\zeta^{2/3}+O_\sigma(\zeta^{1/3}),&\qquad \mathcal Y_\sigma&=-\tfrac9{160}\zeta^2+O_\sigma(\zeta^{5/3})&\qquad & (\zeta\to\infty).
			\end{alignat*}
			For $\mathcal Y_\sigma$, the leading forcing in \eqref{equations-Z-Y-general-p} is $-\zeta^{1/2}/4$ at zero and $-1/8$ at infinity. Substitution gives the displayed coefficients, while the zero-data inverse
			\begin{equation*}
				\mathcal Y_\sigma(\zeta)=-\mathcal Z_\sigma(\zeta)\int_0^\zeta\frac1{\mathcal Z_\sigma(t)^2}
				\int_0^t\frac{\mathcal Z_\sigma(r)\Lambda Q(r)}{\sqrt{Q(r)}}\,dr\,dt
			\end{equation*}
			controls the remainders and their first two Euler derivatives. In particular,
			\begin{equation}\label{global-bounds-Z-general-p}
				c\sigma^{1/2}\Lambda Q\le\mathcal Z_\sigma\le C\sigma^{1/4}\Lambda Q,
				\qquad
				|\mathcal Y_\sigma|+|\Lambda\mathcal Y_\sigma|
				\lesssim_A\zeta^{5/2}\mathbf1_{\{\zeta\le1\}}+\zeta^2\mathbf1_{\{\zeta\ge1\}}.
			\end{equation}
			
			Set $c_\pm=2\beta b\pm b^{1+\eta/2}$ and $f_\pm=c_\pm\mathcal Y_\sigma\pm b^\eta\mathcal Z_\sigma$. These will be the upper and lower barriers for $\widetilde\theta$. Since
			\begin{equation*}
				\sup_{0<\zeta\le R}\frac{|c_\pm\mathcal Y_\sigma|}{b^\eta\mathcal Z_\sigma}
				\lesssim_A b^{1-\eta}+b^{1-\eta}R^{4/3}
				=O_A(b^{\eta/3})\longrightarrow0,
			\end{equation*}
			we have, after increasing $s_{12}$,
			\begin{equation}\label{sign-barriers-general-p}
				\tfrac34b^\eta\mathcal Z_\sigma\le f_+\le\tfrac54b^\eta\mathcal Z_\sigma,
				\qquad
				-\tfrac54b^\eta\mathcal Z_\sigma\le f_-\le-\tfrac34b^\eta\mathcal Z_\sigma,
				\qquad Q+f_\pm\ge Q/2.
			\end{equation}
			
			\emph{Residual signs.} To avoid a singular comparison coefficient near $\zeta=0$, freeze the diffusion there and use the nonlinear operator away from this boundary:
			\begin{align}
				\mathcal G_1[f]&:=\sqrt{Q+\widetilde\theta}\,(Q+f)''-2+\gamma\Lambda(Q+f)-b\partial_sf,\label{definition-G1-general-p}\\*
				\mathcal G_2[f]&:=\sqrt{Q+f}\,(Q+f)''-2+\gamma\Lambda(Q+f)-b\partial_sf.\label{definition-G2-general-p}
			\end{align}
			Both vanish on $\widetilde\theta$ by \eqref{equa-theta}. For $f=c\mathcal Y_\sigma+m\mathcal Z_\sigma$, direct expansion gives
			\begin{align}
				\mathcal G_1[f]&=(\gamma-c)\Lambda Q+\widetilde\theta/Q-\sqrt Q\,Q_\sigma^{-3/2}f
				+\mathcal N_1(\widetilde\theta,f)+\gamma\Lambda f-b\partial_sf,\label{decomposition-G1-general-p}\\*
				\mathcal G_2[f]&=(\gamma-c)\Lambda Q+D_\sigma f+\mathcal N_2(f)+\gamma\Lambda f-b\partial_sf,\label{decomposition-G2-general-p}
			\end{align}
			where $D_\sigma=\sqrt Q(Q^{-3/2}-Q_\sigma^{-3/2})$ and $\mathcal N_2(f)=\mathcal N_1(f,f)$, with the exact identity
			\begin{equation}\label{definition-N1-general-p}
				\mathcal N_1(h,f)=\frac{hf''}{\sqrt{Q+h}+\sqrt Q}
				-\frac{h^2}{Q(\sqrt{Q+h}+\sqrt Q)^2}.
			\end{equation}
			We prove, for a fixed sufficiently small $\zeta_{\rm in}=\zeta_{\rm in}(A)>0$,
			\begin{alignat}{2}
				\mathcal G_1[f_-]&>0>\mathcal G_1[f_+]&\qquad & (0<\zeta\le\zeta_{\rm in}),\label{G1-sign-general-p}\\*
				\mathcal G_2[f_-]&>0>\mathcal G_2[f_+]&\qquad & (\zeta_{\rm in}\le\zeta\le R).\label{G2-sign-general-p}
			\end{alignat}
			Near zero, the profile expansions give
			\begin{equation}\label{near-origin-coercivity-general-p}
				\sqrt Q\,Q_\sigma^{-3/2}\mathcal Z_\sigma
				=\tfrac14\sigma^{-1/8}+O_A(\zeta^{1/2})
				=\tfrac14A^3+O_A(\zeta^{1/2}).
			\end{equation}
			Choose $\zeta_{\rm in}<\sigma^{3/4}$ so that this quantity is at least $A^3/8$. The term $-\sqrt Q\,Q_\sigma^{-3/2}f_\pm$ then has the required sign and magnitude at least $3A^3b^\eta/32$, whereas \eqref{tilde-theta-bootstrap-inner} gives $|\widetilde\theta/Q|\lesssim A^2b^\eta$. From \eqref{definition-N1-general-p}, the profile bounds, and \eqref{preliminary-beta-b-comparison}, the remaining nonlinear and time-dependent errors are bounded by $C_A(b^{2\eta}+b^{1+\eta}+b^2)$. Finally, $\gamma-c_+\le-\frac12b^{1+\eta/2}$ and $\gamma-c_-\ge\frac12b^{1+\eta/2}$, since $\eta<1/2$. Choosing $A$ first and then $s_0$ proves \eqref{G1-sign-general-p} by \eqref{decomposition-G1-general-p}.
			
			For the second region, $\partial_\sigma Q_\sigma=\Lambda Q(\zeta/\sigma^{3/4})>0$ gives $Q_\sigma<Q$ and $D_\sigma<0$. At infinity,
			\begin{equation*}
				-D_\sigma\mathcal Z_\sigma
				=\frac{3\sigma^{1/2}(1-\sigma^{1/2})}{64a^2}\zeta^{-4/3}
				\bigl(1+O_A(\zeta^{-1/3})\bigr).
			\end{equation*}
			Continuity and positivity therefore imply $-D_\sigma\mathcal Z_\sigma\ge c_A\langle\zeta\rangle^{-4/3}$ for $\zeta\ge\zeta_{\rm in}$. Equations \eqref{sign-barriers-general-p} and \eqref{decomposition-G2-general-p} give a signed contribution comparable to $b^\eta\langle\zeta\rangle^{-4/3}+b^{1+\eta/2}\Lambda Q$. The error satisfies
			\begin{equation}\label{G2-errors-general-p}
				|\mathcal N_2(f_\pm)|+|\gamma\Lambda f_\pm-b\partial_sf_\pm|
				\lesssim_A b^{2\eta}\langle\zeta\rangle^{-4/3}+b^{1+\eta}\Lambda Q+b^2\langle\zeta\rangle^2.
			\end{equation}
			Indeed, \eqref{definition-N1-general-p} bounds the quadratic terms by $C_A(b^{2\eta}\zeta^{-4/3}+b^{1+\eta}+b^2\zeta^{4/3})$ for $\zeta\ge1$, and differentiation of $c_\pm,b^\eta$ gives the remaining terms. On the fixed interval $[\zeta_{\rm in},1]$ the same estimate follows with constants depending on $A$. The first two terms in \eqref{G2-errors-general-p} gain $b^\eta$ and $b^{\eta/2}$; the last gains $b^{5\eta/6}$ for $1\le\zeta\le R$, since $b^{1-\eta/2}R^{4/3}=b^{5\eta/6}$. This proves \eqref{G2-sign-general-p}.
			
			\emph{Boundary ordering and comparison.} All three functions vanish at $\zeta=0$. At $s=s_0$, Lemma~\ref{lemma-initial-data-l=1}, \eqref{initial-epsilon-minus-exponential}, and \eqref{outer-estimates-epsilon-plus} give
			\begin{equation*}
				\sup_{0<\zeta\le R(s_0)}\frac{|\widetilde\theta(s_0,\zeta)|}{b(s_0)^\eta\mathcal Z_\sigma(\zeta)}
				\lesssim_A b(s_0)^{1/2-\eta}+b(s_0)^{\eta/3}\longrightarrow0.
			\end{equation*}
			Apply Proposition~\ref{propo-priori-b-eta-K} with $K=1$;
			its threshold is included in $s_{12}$.
			At $\zeta=R(s)$, corresponding to $\xi=b(s)^\eta$, that proposition
			and \eqref{outer-estimates-epsilon-plus} give:
			\begin{equation}\label{Duhamel-arbitrary-L-general-p}
				\frac{|\widetilde\theta(s,R(s))|}{b^\eta R(s)^{2/3}}
				\lesssim A^{-26}+b^{\eta/3}+b^{1/4-\eta}.
			\end{equation}
			Since $\mathcal Z_\sigma(R)\gtrsim A^{-12}R^{2/3}$, choose $A$ large and then $s_{12}$ large so that \eqref{Duhamel-arbitrary-L-general-p} and the initial estimate give $f_-\le\widetilde\theta\le f_+$ on the full moving parabolic boundary.
			
			For completeness, take the maximum of $e^{-M\tau(s)}(\widetilde\theta-f_+)$ on a truncated moving cylinder, where $\tau(s)=\int_{s_0}^s b(t)^{-1}\,dt$. At a positive maximum, the spatial first derivative vanishes, the second derivative is nonpositive, and $b\partial_s(\widetilde\theta-f_+)\ge M(\widetilde\theta-f_+)$. If the maximum lies below $\zeta_{\rm in}$, subtraction in \eqref{definition-G1-general-p} bounds the positive residual difference by $(\gamma-M)(\widetilde\theta-f_+)$. Above $\zeta_{\rm in}$, subtraction in \eqref{definition-G2-general-p} adds the coefficient
			\begin{equation*}
				\frac{(Q+f_+)''}{\sqrt{Q+\widetilde\theta}+\sqrt{Q+f_+}},
			\end{equation*}
			which is uniformly bounded there by the profile bounds and $Q+\widetilde\theta,Q+f_+\ge Q/2$. Both alternatives contradict \eqref{G1-sign-general-p}--\eqref{G2-sign-general-p} when $M>C_A+2$. For the lower bound, use $f_--\widetilde\theta$ and the same actual diffusion $\sqrt{Q+\widetilde\theta}$; the corresponding quotient contains $(Q+f_-)''$ and has the same bound. Thus $f_-\le\widetilde\theta\le f_+$ throughout $0\le\zeta\le R(s)$.
			
			The comparison principle, together with \eqref{sign-barriers-general-p}
			and \eqref{global-bounds-Z-general-p}, yields
			$|\widetilde\theta|\le CA^{-6}b^\eta\Lambda Q$.
			Finally, returning to $\xi$ and using $\Lambda Q\asymp\zeta$ for
			$0<\zeta\le1$ and $\Lambda Q\asymp\zeta^{2/3}$ for $\zeta\ge1$,
			we obtain \eqref{a-priori-estimate-epsilon-near-zero} after increasing
			$A_{12}$ if necessary, completing the proof.
		\end{proof}
		
		\subsection{Comparison in the outer region}
		We estimate $\varepsilon_-$ on $[K,cb^{-1/4-\eta/2}]$, which contains $[K,b^{-\eta}]$ for $b$ small. The constant $K$ is the one used in Proposition~\ref{propo-priori-b-eta-K}, which controls $[b^\eta,K]$.
		Let $\phi_{\beta,2}$ be the normalized Laguerre eigenfunction in \eqref{defi-phi-beta-i}--\eqref{defi-N-beta-i}. Lemma~\ref{lemma-spectrum-L0-beta-xi} gives $\mathscr L_{0,\beta}\phi_{\beta,2}=-\frac32\beta\phi_{\beta,2}$. Its leading coefficient is positive, so for all sufficiently large $K$, uniformly in $\beta\in[1/2,3/2]$,
		\begin{equation}\label{bounds-phi-2-beta-outer}
			\phi_{\beta,2}(\xi)\asymp\xi^{10/3},\qquad
			|\partial_\beta\phi_{\beta,2}|+\sum_{k=1}^2| (\xi\partial_\xi)^k\phi_{\beta,2}|\lesssim\phi_{\beta,2}
			\quad(\xi\ge K).
		\end{equation}
		Set $p=1/2+\eta$ and $F(s,\xi)=A^{-3}b(s)^p\phi_{\beta(s),2}(\xi)$.
		We use the outer barriers $\varepsilon_-^{\rm sup}=F$ and
		$\varepsilon_-^{\rm sub}=-F$.
		
		\begin{proposition}[Pointwise estimate in the outer region]\label{pro-priori-epsilon-xi-large}
			There exist $K_{13}\ge1$, $A_{13}(K)\ge\max\{A_9,A_{11}\}$,
			$\eta_{13}(A,K)\in(0,1/32]$, and
			$s_{13}=s_{13}(A,K,\eta,M_0,x_0)\ge\max\{s_9,s_{11},s_{18}\}$,
			independent of $\bar s$, with $s_9$ evaluated at $\delta=3/8$,
			such that, under the standing assumptions, for
			$K\ge K_{13}$, $A\ge A_{13}(K)$, $0<\eta\le\eta_{13}(A,K)$
			and $s_0\ge s_{13}$, there is a constant $c=c(A,K)>0$, independent of $s_0$ and $\bar s$, such that
			\begin{equation}\label{outer-trapping-epsilon-minus}
				-F(s,\xi)\le\varepsilon_-(s,\xi)\le F(s,\xi),\qquad
				s_0\le s\le\bar s,\quad K\le\xi\le cb(s)^{-1/4-\eta/2}.
			\end{equation}
			In particular, the estimate holds on $K\le\xi\le b(s)^{-\eta}$.
		\end{proposition}
		\begin{proof}
			Choose $K_{13}$ large enough for \eqref{bounds-phi-2-beta-outer}
			and the positivity argument below, and $A_{13}(K)\ge\max\{A_9,A_{11}\}$
			large enough for the ordering at $\xi=K$.
			Take
			\[
			\eta_{13}(A,K)\le\min\{\eta_9(A,3/8),\eta_{11}(A),\eta_{18}(A)\}.
			\]
			All subsequent increases are absorbed into these thresholds.
			
			\emph{Residual inequalities.} For $h=\varepsilon_++f$, define
			\begin{equation}\label{definition-P-outer}
				\mathscr P(f):=\partial_sf-\mathscr L_{0,\beta}f-\widetilde{\mathcal B}_1(\varepsilon_++f)-\widetilde{\mathcal R}_1(\varepsilon_+)
				=\mathscr P_\beta(Q_b+\varepsilon_++f).
			\end{equation}
			Here the two remainders are defined in \eqref{definition-B-one-beta-one}--\eqref{definition-R-one-beta-one}; the last identity follows from $\partial_sQ_b=(b'/b)\Lambda Q_b$ and $\sqrt{Q_b}Q_b''=2$. In particular, $\mathscr P(\varepsilon_-)=0$. The eigenvalue equation, including the time dependence of $\beta$, gives
			\begin{equation}\label{coercive-linear-F-outer}
				(\partial_s-\mathscr L_{0,\beta})F
				=\bigl[(1-\eta)\beta+pr_b\bigr]F+A^{-3}b^p\beta'\partial_\beta\phi_{\beta,2}
				\ge\tfrac14F.
			\end{equation}
			For the complete modulation remainder, use its exact relation to \eqref{definition-R-one-b-stable}:
			\begin{equation}\label{estimate-Rtilde-outer}
				\widetilde{\mathcal R}_1=\mathcal R_{1,b}+(\mathscr L_{0,\beta}-\mathscr L_{b,\beta})\varepsilon_+,
				\qquad |\widetilde{\mathcal R}_1|\lesssim b^{3/4}\xi^{10/3}\quad(\xi\ge K).
			\end{equation}
			Indeed, \eqref{global-estimate-R-one-b-stable} controls the first term; the scaling \eqref{spectral-beta-scaling} and bounds \eqref{prof-estimate-tilde-varphi} control the $\beta$ derivatives without an additional spatial weight. Lemma~\ref{lemma-Q} gives $|Q_b^{-1}-a^{-2}\xi^{-4/3}|\lesssim b^{1/2}\xi^{-2}$ and $|\sqrt{Q_b}-a\xi^{2/3}|\lesssim b^{1/2}$; combined with \eqref{outer-estimates-epsilon-plus}, these bound the second term by $C(b+b^{5/4}\xi^{4/3})$. This verifies \eqref{estimate-Rtilde-outer} for the actual remainder used here.
			
			If $Q_b+h>0$, rationalization gives the exact quadratic remainder $-h^2/[Q_b(\sqrt{Q_b+h}+\sqrt{Q_b})^2]$. Thus \eqref{definition-B-one-beta-one} implies
			\begin{equation}\label{general-Btilde-bound-outer}
				|\widetilde{\mathcal B}_1(h)|\lesssim
				b^{1/2}\xi^{-2}|h|+\bigl(b^{1/2}+\xi^{-2/3}|h|\bigr)|h''|+\xi^{-8/3}|h|^2.
			\end{equation}
			Choose $\nu_+=(9+12\eta)/20$ and $\Xi_+(s)=C_*b(s)^{-\nu_+}$, with $C_*>1$ fixed below. Since $\eta\le1/32$, we have $\eta<\nu_+<1/2$. For $h_\pm=\varepsilon_+\pm F$, \eqref{outer-estimates-epsilon-plus} and \eqref{bounds-phi-2-beta-outer} give $|h_\pm|+\xi^2|h_\pm''|\lesssim b^{1/2}\xi^2+F$, because $b^{3/4}\xi^{10/3}\lesssim A^3b^{1/4-\eta}F$. Dividing \eqref{general-Btilde-bound-outer} by $F$ therefore gives, on $K\le\xi\le\Xi_+$ wherever $Q_b+h_\pm>0$,
			\begin{equation*}
				\frac{|\widetilde{\mathcal B}_1(h_\pm)|}{F}
				\lesssim_{A,K}b^{1/2-\eta}+b^{1/2}\xi^{-2/3}+b^p\xi^{2/3}
				\lesssim_{A,K,C_*}b^{1/2-\eta}+b^{1/5+3\eta/5}.
			\end{equation*}
			The ratio in \eqref{estimate-Rtilde-outer} is $O(A^3b^{1/4-\eta})$. These positive powers, together with \eqref{coercive-linear-F-outer}, imply
			\begin{equation}\label{residual-sign-sub-outer}
				\mathscr P(F)\ge F/8>0,\qquad \mathscr P(-F)\le-F/8<0
			\end{equation}
			on the respective positivity domains, after increasing $s_{13}$.
			
			\emph{Upper comparison and initial ordering.} On $K\le\xi\le\Xi_+$, the leading finite-mode term satisfies $b^{1/2}\xi^2=o(Q_b)$ and its remainder is $o(F)$, since
			$b^{1/2}\Xi_+^{2/3}=O(b^{1/5-2\eta/5})$ and $b^{1/4-\eta}\to0$. Hence $Q_b+\varepsilon_++F\ge(Q_b+F)/4>0$. At $\xi=K$, Proposition~\ref{propo-priori-b-eta-K} gives $|\varepsilon_-(s,K)|\le A^{-26}b^p\mathfrak f_1(K)\le F(s,K)/2$, using \eqref{bounds-phi-2-beta-outer} and $A$ large.
			
			We verify the initial ordering across the gluing region, with $R_\delta$ from \eqref{definition-initial-cutoff-radius}. On $K\le\xi\le R_\delta$, \eqref{defi-initi-varthe-s-0} has no cutoff error; the differences between $(b_0,\beta_0)$ and $(\widetilde b_0,1)$ are smaller than every power by \eqref{distance-b-s-0-tildeb-0}. The parameter derivatives in Proposition~\ref{diagonalisation-L-b} grow at most polynomially in this range, so $|\varepsilon_-(s_0,\xi)|\le F(s_0,\xi)/2$ for $s_0\ge s_{13}$.
			
			For $R_\delta\le\xi\le\Xi_+(s_0)$, the choice $\delta=3/8$ and $\nu_+<1/2$ give, directly from \eqref{defi-U-infty-1},
			\begin{equation*}
				0<\Psi_\infty^1(\xi)\lesssim b_0^{-1/2}+b_0^{1/2}\xi\lesssim b_0^{-1/2}.
			\end{equation*}
			The inner profile is also $O(b_0^{-1/2})$ on $[R_\delta,2R_\delta]$, by \eqref{outer-estimates-epsilon-plus} and $Q_{b_0}\asymp\xi^{4/3}$. Thus the convex gluing formula gives $\Psi_1^0\lesssim b_0^{-1/2}$ throughout this outer range, and
			\begin{equation*}
				\frac{\varepsilon_-(s_0,\xi)}{F(s_0,\xi)}
				\lesssim A^3\bigl(b_0^{-1-\eta}\xi^{-10/3}+b_0^{-\eta}\xi^{-4/3}+b_0^{1/4-\eta}\bigr)
				\lesssim A^3\bigl(b_0^{1/4-\eta}+b_0^{1/2-\eta}\bigr)\to0.
			\end{equation*}
			Here constants may depend on $M_0,x_0,C_*$. This proves $\varepsilon_-(s_0,\xi)\le F(s_0,\xi)/2$ on the entire initial interval $[K,\Xi_+(s_0)]$.
			
			At the moving right boundary, the global shrinking-set bound and \eqref{outer-estimates-epsilon-plus} yield
			\begin{equation*}
				|\varepsilon_-|\lesssim Ab^{-1-\eta}+\xi^{4/3}+b^{1/2}\xi^2+b^{3/4}\xi^{10/3}.
			\end{equation*}
			At $\xi=\Xi_+$, the ratio of the first term to $F$ is $O(A^4C_*^{-10/3})$, because $(10/3)\nu_+=3/2+2\eta$. The other ratios tend to zero, with powers $2\nu_+-p$, $4\nu_+/3-\eta$, and $1/4-\eta$. Choose $C_*$ large and then $s_{13}$ large, obtaining $\varepsilon_-(s,\Xi_+(s))\le F(s,\Xi_+(s))/2$. On every finite time interval the two positive profiles have bounded comparison coefficients in this bounded moving cylinder. Applying the comparison principle through \eqref{definition-P-outer} and \eqref{residual-sign-sub-outer} gives
			\begin{equation}\label{upper-bound-epsilon-minus-outer}
				\varepsilon_-(s,\xi)\le F(s,\xi),\qquad K\le\xi\le\Xi_+(s).
			\end{equation}
			
			\emph{Positivity domain and lower comparison.} Set $g=Q_b+\varepsilon_+-F$. We construct a $C^1$ endpoint $\Xi_-(s)$ such that
			\begin{equation}\label{properties-Xi-minus-outer}
				g(s,\xi)\ge1\quad(K\le\xi\le\Xi_-(s)),\qquad g(s,\Xi_-(s))=1,
				\qquad 2b^{-\eta}\le\Xi_-\le\Xi_+/2.
			\end{equation}
			The leading coefficient of $\phi_{\beta,2}$ is $d_2(\beta)=\mathcal N_2(\beta)a(\beta a^2/8)^2/2>0$, using the normalization in \eqref{defi-N-beta-i}. With $\xi=b^{-p/2}y$, the profile and eigenfunction estimates give, locally uniformly in $C^1((0,\infty))$,
			\begin{equation}\label{limit-normalized-g-outer}
				b^{2p/3}g(s,b^{-p/2}y)=a^2y^{4/3}-A^{-3}d_2(\beta)y^{10/3}+O_A(b^{1/4-\eta}).
			\end{equation}
			For example, the term $b^{1/2}\xi^2$ becomes $b^{1/3-\eta/3}y^2$, and the finite-mode remainder becomes $O(b^{1/4-\eta}y^{10/3})$. The limiting polynomial has the unique positive zero $y_0(\beta)=aA^{3/2}/\sqrt{d_2(\beta)}$, with derivative $-2a^2y_0^{1/3}<0$. The parameter-dependent implicit function theorem, applied to $b^{2p/3}g=b^{2p/3}$, gives a $C^1$ root $y_*(s)$ near $y_0(\beta(s))$; put $\Xi_-=b^{-p/2}y_*$.
			
			To justify positivity on the full interval, first choose a small fixed $c_A>0$. On $K\le\xi\le c_Ab^{-p/2}$, $F\le Q_b/4$ and $|\varepsilon_+|=o(Q_b)$, so $g\ge Q_b/2\ge2$ for $K$ large. On the remaining scaled interval, \eqref{limit-normalized-g-outer} is positive away from the root and strictly decreasing near it; hence $g\ge1$ up to $\Xi_-$. Finally, $p/2-\eta=1/4-\eta/2>0$ and $\nu_+-p/2=1/5+\eta/10>0$, proving \eqref{properties-Xi-minus-outer}.
			
			On $\xi=K$, the Duhamel estimate gives $\varepsilon_-\ge-F/2$. Since $p/2<\delta=3/8$, we have $\Xi_-(s_0)<R_\delta/8$ for $s_0\ge s_{13}$, so \eqref{initial-epsilon-minus-exponential} gives the same lower ordering throughout $[K,\Xi_-(s_0)]$. At the moving boundary, \eqref{lower-comparison-vartheta} and \eqref{properties-Xi-minus-outer} imply
			\begin{equation*}
				\varepsilon_-(s,\Xi_-)+F(s,\Xi_-)=\vartheta(s,\Xi_-)-1\ge b^{-\eta/2}-1>0.
			\end{equation*}
			The lower profile $g$ is at least $1$, so comparison applies on $K<\xi<\Xi_-(s)$ and gives $-F\le\varepsilon_-$. By \eqref{limit-normalized-g-outer}, $\Xi_-(s)\ge cb(s)^{-1/4-\eta/2}$ for a fixed $c=c(A,K)>0$. Together with \eqref{upper-bound-epsilon-minus-outer} and $\Xi_-<\Xi_+$, this proves \eqref{outer-trapping-epsilon-minus}. Since $1/4-\eta/2>0$, its interval contains $[K,b^{-\eta}]$ after increasing $s_{13}$.
		\end{proof}
		
		\begin{proposition}[A priori estimates for $\varepsilon_-$]\label{lem-improved-bootstrap-epsilon-negative-pointwise}
			There exist $K_{14}\ge1$, $A_{14}(K)\ge1$,
			$\eta_{14}(A,K)\in(0,1/32]$, and
			$s_{14}=s_{14}(A,K,\eta,M_0,x_0)\ge\max\{s_{12},s_{13}\}$,
			independent of $\bar s$, such that, under the standing assumptions, for
			$K\ge K_{14}$, $A\ge A_{14}(K)$, $0<\eta\le\eta_{14}(A,K)$
			and $s_0\ge s_{14}$,
			\begin{equation*}
				\|\varepsilon_-(s)\|_{\mathfrak X_{b(s),\eta}}\le Ab(s)^{1/2+\eta},\qquad s_0\le s\le\bar s.
			\end{equation*}
		\end{proposition}
		\begin{proof}
			Take $K_{14}=K_{13}$, $A_{14}(K)\ge\max\{A_{12},A_{13}(K)\}$,
			and $\eta_{14}(A,K)\le\min\{\eta_{12}(A),\eta_{13}(A,K)\}$.
			On $[0,b^\eta]$, subtract \eqref{outer-estimates-epsilon-plus} from \eqref{a-priori-estimate-epsilon-near-zero}; its normalized contribution is $O(b^{1/2-\eta}+b^{\eta/3}+b^{1/4-\eta})$, giving constant $A/2$ in each region. Proposition~\ref{propo-priori-b-eta-K} controls $[b^\eta,K]$, while \eqref{outer-trapping-epsilon-minus} and \eqref{bounds-phi-2-beta-outer} control $[K,b^{-\eta}]$ by $CA^{-3}b^{1/2+\eta}\mathfrak f_1$. Increase $A_{14}(K)$ so that $CA^{-3}\le A/2$ for $A\ge A_{14}(K)$,
			and then increase $s_{14}$. Each term in \eqref{definition-weighted-remainder-norm} is at most $(A/2)b^{1/2+\eta}$; their sum gives the claim.
		\end{proof}
		\section{Pointwise a priori estimates for \texorpdfstring{$\mathcal F$}{F}}\label{section-pointwise-estimates-F}
		We improve the residual bounds in Proposition~\ref{propo-priori-ell-1}(iv), using constant barriers globally and profile barriers near the wall and in the outer region. Throughout, assume the stable hypotheses of Section~\ref{Section-comparison-varep}, including the prepared datum with $\delta=3/8$ and a solution remaining in $\mathcal V_1[A,\eta]$ on $[s_0,\bar s]$. We use $\mathfrak f_1$ from \eqref{defi-f-ell} and the modulation bounds \eqref{preliminary-beta-b-comparison}.
		
		\subsection{Evolution equation and global bounds}
		Set $\mathbbm F(s,\zeta)=\mathcal F(s,b^{3/4}\zeta)$. The rescalings \eqref{chang-varialbe} and \eqref{definition-inner-theta-zeta} preserve the residual:
		\begin{equation*}
			\sqrt w\,\partial_\psi^2w-2=\mathcal F(s,\xi)=\mathbbm F(s,\zeta)
			=\sqrt\theta\,\partial_\zeta^2\theta-2.
		\end{equation*}
		Since $\partial_s\vartheta=\mathcal F+\beta\Lambda_\xi\vartheta$, differentiating \eqref{defi-mathcal-F} and using
		\begin{equation*}
			\frac{\vartheta''}{2\sqrt\vartheta}\Lambda_\xi\vartheta
			+\sqrt\vartheta\,(\Lambda_\xi\vartheta)''=-\tfrac34\xi\partial_\xi\mathcal F
		\end{equation*}
		gives the first equation below; the same calculation in the inner variables gives the second, with $\gamma=b\beta-b'$ from \eqref{equa-theta}:
		\begin{align}
			\partial_s\mathcal F&=\sqrt\vartheta\,\partial_\xi^2\mathcal F
			-\tfrac34\beta\xi\partial_\xi\mathcal F
			+\frac{\mathcal F(\mathcal F+2)}{2\vartheta},\label{equation-F-outer-exact}\\*
			b\partial_s\mathbbm F&=\sqrt\theta\,\partial_\zeta^2\mathbbm F
			-\tfrac34\gamma\zeta\partial_\zeta\mathbbm F
			+\frac{\mathbbm F(\mathbbm F+2)}{2\theta}.\label{equa-mathbbm-F}
		\end{align}
		These identities hold in the positive interior, where parabolic regularity justifies the differentiations. The wall value $\mathcal F(s,0)=0$ follows from the shrinking-set bound.
		
		\begin{lemma}[Global residual bounds]\label{lem-esti-F-le-A}
			There exist $A_{15}\ge A_{11}$, $\eta_{15}(A)\in(0,\eta_{11}(A)]$, and
			$s_{15}=s_{15}(A,\eta,M_0,x_0)\ge s_{11}$, independent of $\bar s$,
			such that, under the standing assumptions, for $A\ge A_{15}$,
			$0<\eta\le\eta_{15}(A)$ and $s_0\ge s_{15}$, uniformly on
			$[s_0,\bar s]\times\mathbb R_+$,
			\begin{equation}\label{definition-m0-F}
				\mathfrak m_0:=\min\bigl\{-2,\inf_{\xi\ge0}\mathcal F(s_0,\xi)\bigr\},
				\qquad -C_0\le\mathfrak m_0\le\mathcal F(s,\xi)\le b(s)^{-\eta/8},
			\end{equation}
			where $C_0=C_0(M_0,x_0)\ge2$ is independent of $A,\eta,s_0$ and $\bar s$.
		\end{lemma}
		\begin{proof}
			Take $A_{15}\ge A_{11}$ and $\eta_{15}(A)\le\eta_{11}(A)$;
			all further lower bounds on $s_0$ are absorbed into $s_{15}$.
			For comparison in \eqref{equation-F-outer-exact}, write
			\begin{equation}\label{definition-comparison-operator-F}
				\mathscr P_{\mathcal F}[f]:=\partial_sf-\sqrt\vartheta\,\partial_\xi^2f
				+\tfrac34\beta\xi\partial_\xi f-\frac{f(f+2)}{2\vartheta},
				\qquad \mathscr P_{\mathcal F}[\mathcal F]=0.
			\end{equation}
			For a constant $h<\mathfrak m_0\le-2$, we have $\mathscr P_{\mathcal F}[h]=-h(h+2)/(2\vartheta)<0$, with strict initial and wall ordering. The difference $h-\mathcal F$ has zeroth-order coefficient $-(h+\mathcal F+2)/(2\vartheta)$. Its negative part is bounded on each finite time strip: near the wall the numerator is negative, and away from the wall the solution is uniformly positive, using the far-field lower bound in $\mathcal V_1$. The diffusion and the difference are bounded there as well. The comparison argument recalled in Section~\ref{Section-comparison-varep}, with a growing spatial penalty on the half-line, gives $h\le\mathcal F$. Letting $h\uparrow\mathfrak m_0$ proves the lower bound. Lemma~\ref{lemma-initial-data-l=1} supplies $\|\mathcal F(s_0)\|_\infty\le C_0$, hence $\mathfrak m_0\ge-C_0$.
			
			For the upper bound, the shrinking-set estimates give
			\begin{equation*}
				\sup_{0\le\xi\le b^{-\eta}}|\mathcal F(s,\xi)|
				\lesssim_A b^{1/2-13\eta/3}=o(1),
				\qquad \eta\le1/32.
			\end{equation*}
			On $\xi\ge b^{-\eta}$, use $H(s)=b^{-\eta/8}$ and $r_b=b'/b+\beta$. Lemma~\ref{lem-comparison-vartheta} gives $\vartheta\ge b^{-\eta/2}$, so
			\begin{equation*}
				\frac{\mathscr P_{\mathcal F}[H]}{H}
				=\tfrac\eta8(\beta-r_b)-\frac{H+2}{2\vartheta}
				\ge\tfrac\eta{32}-Cb^{3\eta/8}>0.
			\end{equation*}
			The initial bound $C_0$ and the preceding estimate at the moving left boundary are at most $H/2$ for $s_0\ge s_{15}$. On this exterior domain the linear difference coefficient has bounded negative part, by the lower bound for $\vartheta$ and the global bootstrap bound for $\mathcal F$. Comparison on finite time strips, again with a spatial penalty at infinity, yields $\mathcal F\le H$ and proves \eqref{definition-m0-F}.
		\end{proof}
		
		It remains to improve the weighted bounds on $[0,b^{-\eta}]$. Lemma~\ref{lem-priori-F-K} already controls $[K^{-1},K]$ through \eqref{improved-F-compact}. We therefore treat the two complementary regions.
		
		\subsection{Comparison near the wall}
		Fix $\sigma=1/2$ and use $\mathcal Z_\sigma=\Lambda Q_\sigma$ from \eqref{equations-Z-Y-general-p}. Put $R_K(s)=K^{-1}b^{-3/4}$ and $m(s)=b^{1-\eta}$.
		
		\begin{lemma}[Inner residual bound]\label{lem-esti-F-0-1/K}
			There exist $K_{16}\ge1$, $A_{16}(K)\ge1$,
			$\eta_{16}(A,K)\in(0,1/32]$, and
			$s_{16}=s_{16}(A,K,\eta,M_0,x_0)\ge\max\{s_{10},s_{12},s_{15}\}$,
			independent of $\bar s$, with $s_{10}$ evaluated at $\delta=3/8$,
			such that, under the standing assumptions, for
			$K\ge K_{16}$, $A\ge A_{16}(K)$, $0<\eta\le\eta_{16}(A,K)$
			and $s_0\ge s_{16}$,
			\begin{equation}\label{trapping-F-inner}
				-m(s)\mathcal Z_\sigma(\zeta)\le\mathbbm F(s,\zeta)
				\le m(s)\mathcal Z_\sigma(\zeta),
				\qquad s_0\le s\le\bar s,\quad 0\le\zeta\le R_K(s).
			\end{equation}
			Consequently,
			\begin{equation}\label{inner-estimate-F-final}
				|\mathcal F(s,\xi)|\lesssim
				\begin{cases}
					b^{1/4-\eta}\xi,&0\le\xi\le b^{3/4},\\
					b^{1/2-\eta}\xi^{2/3},&b^{3/4}\le\xi\le K^{-1}.
				\end{cases}
			\end{equation}
		\end{lemma}
		\begin{proof}
			Take $A_{16}(K)\ge\max\{A_{12},A_{15}\}$ and
			$\eta_{16}(A,K)\le\min\{\eta_{12}(A),\eta_{15}(A)\}$.
			The choices of $K_{16}$ and the remaining increases are specified below.
			
			\emph{Coercivity and coefficient errors.} The profile identities used in Proposition~\ref{propo-priori-near-0}, with $D_\sigma$ from \eqref{decomposition-G2-general-p}, give
			\begin{equation}\label{coercivity-D-sigma-F}
				\mathscr H_\sigma\mathcal Z_\sigma=0,\qquad
				\mathcal Z_\sigma\asymp\Lambda Q,\qquad
				D_\sigma\mathcal Z_\sigma\le-c_\sigma\langle\zeta\rangle^{-4/3}.
			\end{equation}
			Indeed, $Q_\sigma<Q$, the product tends to $-\sigma^{1/4}(\sigma^{-3/8}-1)/4<0$ at zero, and its expansion at infinity in that proof is a negative multiple of $\zeta^{-4/3}$. Continuity gives the global inequality.
			
			Write $\theta=Q+\widetilde\theta$. The shrinking-set bound for $\varepsilon_-$ and \eqref{outer-estimates-epsilon-plus} imply
			\begin{equation}\label{tilde-theta-estimate-F-inner}
				|\widetilde\theta|\lesssim A^2b^\eta\Lambda Q+b\zeta^2,
				\qquad \tfrac12Q\le\theta\le\tfrac32Q
				\quad(0<\zeta\le R_K).
			\end{equation}
			For example, the relative error is $O(A^2b^\eta)$ below $1$ and $O(A^2b^\eta\zeta^{-2/3}+b\zeta^{2/3})$ above $1$; its supremum tends to zero. Define
			\begin{equation*}
				\mathcal R_\theta[f]:=(\sqrt\theta-\sqrt Q)f''+(\theta^{-1}-Q^{-1})f.
			\end{equation*}
			Since $\mathcal Z_\sigma''=-Q_\sigma^{-3/2}\mathcal Z_\sigma$, rationalization yields
			\begin{equation}\label{R-theta-A-estimate}
				|\mathcal R_\theta[\mathcal Z_\sigma]|
				\lesssim_\sigma (A^2b^\eta+K^{-4/3})\langle\zeta\rangle^{-4/3}.
			\end{equation}
			For $\zeta\ge1$, multiply the relative error above by $O(\zeta^{-2/3})$ and use $b\zeta^{4/3}\le bR_K^{4/3}=K^{-4/3}$. For $\zeta\le1$, the two factors multiplying the relative error are bounded, giving the same estimate.
			
			\emph{Barrier residuals.} Factoring $m$ out of the time and drift terms gives
			\begin{equation*}
				\mathcal T_\sigma:=-\tfrac34\gamma\zeta\mathcal Z_\sigma'-(1-\eta)b'\mathcal Z_\sigma
				=-\eta\beta b\mathcal Z_\sigma
				+O_\sigma(b\zeta^{1/3}+b^{5/4}\mathcal Z_\sigma)
				\quad(\zeta\ge1).
			\end{equation*}
			Here $\zeta\mathcal Z_\sigma'=\frac23\mathcal Z_\sigma+O_\sigma(\zeta^{1/3})$ and \eqref{preliminary-beta-b-comparison} were used. Thus $\mathcal T_\sigma\le-c_{\eta,\sigma}b\mathcal Z_\sigma$ above a fixed threshold depending on $\eta$; on the remaining bounded interval, its $O_{\eta,\sigma}(b)$ contribution is absorbed by \eqref{coercivity-D-sigma-F}. Choose $K_{16}$ so that the $K^{-4/3}$ term in
			\eqref{R-theta-A-estimate} is absorbed by \eqref{coercivity-D-sigma-F},
			and then increase $s_{16}$, obtaining
			\begin{equation}\label{linear-coercivity-F-barrier}
				D_\sigma\mathcal Z_\sigma+\mathcal R_\theta[\mathcal Z_\sigma]+\mathcal T_\sigma
				\le-c_{\eta,\sigma}\bigl(\langle\zeta\rangle^{-4/3}+b\Lambda Q\bigr).
			\end{equation}
			Set $\mathbbm F^\pm=\pm m\mathcal Z_\sigma$ and write the residual of \eqref{equa-mathbbm-F} as
			\begin{equation*}
				\mathbbm G[f]:=\sqrt\theta\,f''+\frac{f(f+2)}{2\theta}
				-\tfrac34\gamma\zeta f'-b\partial_sf.
			\end{equation*}
			The exact computation is
			\begin{equation}\label{residual-barriers-F}
				\mathbbm G[\mathbbm F^\pm]
				=\pm m\bigl(D_\sigma\mathcal Z_\sigma+\mathcal R_\theta[\mathcal Z_\sigma]+\mathcal T_\sigma\bigr)
				+\frac{m^2\mathcal Z_\sigma^2}{2\theta}.
			\end{equation}
			The last term is $O(m^2)$ by \eqref{tilde-theta-estimate-F-inner} and the profile bounds. Since $\langle\zeta\rangle^{-4/3}+b\Lambda Q\gtrsim b^{2/3}$ and $m/b^{2/3}=b^{1/3-\eta}\to0$, \eqref{linear-coercivity-F-barrier}--\eqref{residual-barriers-F} give
			\begin{equation}\label{barrier-signs-F}
				\mathbbm G[\mathbbm F^-]>0>\mathbbm G[\mathbbm F^+]
				\qquad(0<\zeta<R_K).
			\end{equation}
			
			\emph{Boundary ordering and the degenerate wall.} With $b_0=b(s_0)$, Lemma~\ref{lemma-initial-data-l=1} gives $|\mathbbm F(s_0,\zeta)|\lesssim b_0\Lambda Q=o(m(s_0)\mathcal Z_\sigma)$ on $[0,R_K(s_0)]$. At the right boundary, \eqref{improved-F-compact} gives
			\begin{equation*}
				|\mathbbm F(s,R_K)|\le A^{-26}b^{1/2-\eta},\qquad
				m\mathcal Z_\sigma(R_K)\gtrsim K^{-2/3}b^{1/2-\eta}.
			\end{equation*}
			Thus the ordering holds there after increasing $A_{16}(K)$. At the wall, all three functions vanish and are $O(\zeta)$ on each finite time strip.
			
			To justify comparison despite the singular reaction coefficient, apply the maximum argument to $(\mathbbm F-\mathbbm F^+)/\sqrt\zeta$ and $(\mathbbm F^--\mathbbm F)/\sqrt\zeta$. Both extend continuously by zero at the wall. Their evolution coefficients of order zero are
			\begin{equation*}
				\frac{\mathbbm F+\mathbbm F^\pm+2}{2\theta}
				-\frac{\sqrt\theta}{4\zeta^2}-\frac38\gamma.
			\end{equation*}
			These are bounded above: near zero the negative term has order $-\zeta^{-3/2}$ and dominates the $O(\zeta^{-1})$ reaction term; away from zero, \eqref{tilde-theta-estimate-F-inner} bounds the coefficients on every finite moving cylinder. An exponential factor in $\tau(s)=\int_{s_0}^s b(t)^{-1}\,dt$ therefore excludes a positive interior maximum, using \eqref{barrier-signs-F}. This proves \eqref{trapping-F-inner}. Finally, $\mathcal Z_\sigma\asymp\zeta$ below $1$ and $\mathcal Z_\sigma\asymp\zeta^{2/3}$ above $1$, so the change $\xi=b^{3/4}\zeta$ gives \eqref{inner-estimate-F-final}.
		\end{proof}
		
		\subsection{Comparison in the outer region}
		Use the normalized eigenfunction $\phi_{\beta,2}$ from \eqref{defi-phi-beta-i}, with the uniform bounds \eqref{bounds-phi-2-beta-outer}. We keep its dependence on $\beta$; Lemma~\ref{lemma-spectrum-L0-beta-xi} gives $(\mathscr L_{0,\beta}-\beta\operatorname{Id})\phi_{\beta,2}=-\frac52\beta\phi_{\beta,2}$.
		
		\begin{lemma}[Outer residual bound]\label{lem-esti-F-outer-region}
			There exist $K_{17}\ge1$, $A_{17}(K)\ge1$,
			$\eta_{17}(A,K)\in(0,1/32]$, and
			$s_{17}=s_{17}(A,K,\eta,M_0,x_0)\ge\max\{s_{10},s_{13},s_{15}\}$,
			independent of $\bar s$, with $s_{10}$ evaluated at $\delta=3/8$,
			such that, under the standing assumptions, for
			$K\ge K_{17}$, $A\ge A_{17}(K)$, $0<\eta\le\eta_{17}(A,K)$
			and $s_0\ge s_{17}$,
			\begin{equation}\label{outer-pointwise-estimate-F}
				|\mathcal F(s,\xi)|\le\frac A4b(s)^{1/2-\eta}\mathfrak f_1(\xi),
				\qquad s_0\le s\le\bar s,\quad K\le\xi\le b(s)^{-\eta}.
			\end{equation}
		\end{lemma}
		\begin{proof}
			Take $K_{17}=K_{13}$, $A_{17}(K)\ge\max\{A_{13}(K),A_{15}\}$,
			and $\eta_{17}(A,K)\le\min\{\eta_{13}(A,K),\eta_{15}(A)\}$.
			All subsequent increases are absorbed into these thresholds.
			Set $q=1/2-\eta$, $R(s)=b^{-3/20}$, and $H(s,\xi)=b^q\phi_{\beta(s),2}(\xi)$. Since $\eta<3/20$, the comparison interval $[K,R(s)]$ contains the target interval.
			
			\emph{Control on the enlarged interval.} The upper bound \eqref{upper-bound-epsilon-minus-outer} and the lower comparison in the proof of Proposition~\ref{pro-priori-epsilon-xi-large} hold up to $\Xi_-\asymp b^{-1/4-\eta/2}$, by \eqref{limit-normalized-g-outer}. Thus they both apply up to $R(s)$ for $s_0\ge s_{17}$. Together with \eqref{outer-estimates-epsilon-plus}, they give
			\begin{equation*}
				|\varepsilon|\lesssim b^{1/2}\xi^2+A^{-3}b^{1/2+\eta}\xi^{10/3}
				\qquad(K\le\xi\le R),
			\end{equation*}
			where the $b^{3/4}\mathfrak f_1$ term is absorbed because $\eta<1/4$. Lemma~\ref{lemma-Q} then implies
			\begin{equation}\label{relative-coefficient-error-simplified}
				\frac{|\vartheta-a^2\xi^{4/3}|}{\xi^{4/3}}
				\lesssim b^{1/2}\xi^{-2/3}+b^{1/2}\xi^{2/3}+A^{-3}b^{1/2+\eta}\xi^2
				\lesssim_A b^{1/5}.
			\end{equation}
			In particular, $\vartheta\asymp\xi^{4/3}$ on the enlarged interval. For
			\begin{equation*}
				\mathcal E_b[f]:=(\sqrt\vartheta-a\xi^{2/3})f''
				+(\vartheta^{-1}-a^{-2}\xi^{-4/3})f,
			\end{equation*}
			the derivative bounds in \eqref{bounds-phi-2-beta-outer} and \eqref{relative-coefficient-error-simplified} yield
			\begin{equation}\label{linear-coefficient-error-F}
				|\mathcal E_b[H]|\lesssim_A b^{1/5}H,\qquad
				\frac{H}{\vartheta}\lesssim b^q\xi^2\le Cb^{1/5-\eta}.
			\end{equation}
			
			\emph{Residual signs.} The operator \eqref{definition-comparison-operator-F} gives the exact formula
			\begin{equation*}
				\mathscr P_{\mathcal F}[\pm H]
				=\pm\bigl([(2+\eta)\beta+qr_b]H
				+b^q\beta'\partial_\beta\phi_{\beta,2}-\mathcal E_b[H]\bigr)
				-\frac{H^2}{2\vartheta}.
			\end{equation*}
			The principal coefficient is bounded below by $3/4$ for $s_0\ge s_{17}$. Equations \eqref{preliminary-beta-b-comparison}, \eqref{bounds-phi-2-beta-outer}, and \eqref{linear-coefficient-error-F} bound the remaining terms relative to $H$ by $C_A(b^{1/4}+b^{1/5}+b^{1/5-\eta})=o(1)$. Hence
			\begin{equation}\label{sign-residual-outer-F}
				\mathscr P_{\mathcal F}[-H]<0<\mathscr P_{\mathcal F}[H]
				\qquad(K<\xi<R(s)).
			\end{equation}
			
			\emph{Boundary ordering.} At $\xi=K$, \eqref{improved-F-compact} and $\phi_{\beta,2}(K)\gtrsim\mathfrak f_1(K)$ give $|\mathcal F|\le H/2$ after increasing $A_{17}(K)$. At the initial time, $R(s_0)<R_\delta$ from \eqref{definition-initial-cutoff-radius}, since $3/20<\delta=3/8$. Lemma~\ref{lemma-initial-data-l=1} therefore gives $|\mathcal F(s_0,\xi)|\lesssim b_0^{1/2}\mathfrak f_1(\xi)=o(H(s_0,\xi))$ on $[K,R(s_0)]$. Finally,
			\begin{equation*}
				H(s,R(s))\gtrsim b^qR^{10/3}=b^{-\eta}
				\gg b^{-\eta/8}+C_0\ge|\mathcal F(s,R(s))|,
			\end{equation*}
			by \eqref{definition-m0-F}. All coefficients in the linear difference equations are bounded on each finite moving cylinder, since $\vartheta\asymp\xi^{4/3}$ and $\xi\ge K$. Comparison using \eqref{sign-residual-outer-F} gives $|\mathcal F|\le H$ throughout $[K,R(s)]$. Since $\phi_{\beta,2}\lesssim\mathfrak f_1$, increasing $A_{17}(K)$
			if necessary proves \eqref{outer-pointwise-estimate-F}.
		\end{proof}
		
		Equations \eqref{inner-estimate-F-final}, \eqref{improved-F-compact}, and \eqref{outer-pointwise-estimate-F} bound each term of $\|\mathcal F\|_{\mathfrak X_{b,\eta}}$ by $(A/4)b^{1/2-\eta}$ after increasing $A_{16}(K)$ and $A_{17}(K)$ if necessary; on $[K^{-1},K]$, use $\mathfrak f_1\gtrsim K^{-2/3}$. Their sum gives the norm bound in Proposition~\ref{propo-priori-ell-1}(iv), while $A\ge2C_0$ in \eqref{definition-m0-F} gives the global improvements.
		\section{Derivation of a priori estimates in the unstable regimes}\label{section-prioei-ell-2}
		We prove Proposition~\ref{propo-priori-ell-2}, using the stable arguments of Sections~\ref{Section-comparison-varep}--\ref{section-pointwise-estimates-F} and recording the changes caused by the lower modes and the faster modulation. Throughout, $\ell\ge2$, $\beta=1$, and the solution has the prepared datum of Lemma~\ref{lemma-initial-data-ell} and remains in $\mathcal V_\ell[A,\eta]$ on $[s_0,\bar s]$. Use \eqref{notation-beta-one}, \eqref{defi-f-ell}, and the decomposition \eqref{definition-epsilon-plus-general-ell}. Write $\gamma_\ell=2\ell-1$ and $p_\ell=2+4\ell/3$.
		Fix the auxiliary cutoff $K=K_6(\ell,M_0,x_0)$ large enough for the
		comparison arguments below. All subsequent amplitude requirements are
		included in $A_6\ge A_5$, and we choose
		\[
		\eta_6(A,\ell)\le\min\{\eta_5(\ell),\eta_{19}(A,\ell),1/(32\ell)\},
		\qquad s_6\ge\max\{s_5,s_{19}\}.
		\]
		Further restrictions on $\eta$ and $s_0$ in this section are absorbed into
		$\eta_6$ and $s_6$, respectively. The fixed $K$ is suppressed in these
		thresholds, which are independent of $\bar s$.
		
		\subsection{Modulation equations: item \emph{(i)}}\label{subsection-proof-item-i-ell}
		Put $\kappa_\ell=c_{\ell,0}/c_{0,0}$, $h_j=\|\varphi_{j,b}\|_{L^2_{\rho_b}}^2\asymp_\ell1$, and use the adapted direction \eqref{definition-adapted-eigendirection-strategy}. Orthogonality gives $\varepsilon_j=h_j^{-1}\langle\varepsilon,\varphi_{j,b}\rangle_{\rho_b}$ for $1\le j\le\ell$ and $\varepsilon_\ell=-\kappa_\ell^{-1}h_0^{-1}\langle\varepsilon,\varphi_{0,b}\rangle_{\rho_b}$. Set $\widehat\varphi_k=\varphi_{k,b}$ for $k<\ell$, $\widehat\varphi_\ell=\psi_{\ell,b}$, and $D_s^\rho f=\partial_sf+f\partial_s\log\rho_b$. Differentiating the projections in \eqref{equa-varepsilon-appen-} gives
		\begin{align*}
			\varepsilon_j'&=\mu_{j,b}\varepsilon_j+h_j^{-1}\langle\mathscr B_b(\varepsilon)+\Phi_{b,1},\varphi_{j,b}\rangle_{\rho_b}+\mathcal K_j,\quad 1\le j\le\ell,\\*
			\varepsilon_\ell'&=\mu_{0,b}\varepsilon_\ell-\kappa_\ell^{-1}h_0^{-1}\langle\mathscr B_b(\varepsilon)+\Phi_{b,1},\varphi_{0,b}\rangle_{\rho_b}+\mathcal K_0,\\*
			h_j\mathcal K_j&=-\sum_{k=1}^\ell\varepsilon_k\langle\partial_s\widehat\varphi_k,\varphi_{j,b}\rangle_{\rho_b}
			+\langle\varepsilon_-,D_s^\rho\varphi_{j,b}\rangle_{\rho_b},\quad 1\le j\le\ell,\\*
			\kappa_\ell h_0\mathcal K_0&=\sum_{k=1}^\ell\varepsilon_k\langle\partial_s\widehat\varphi_k,\varphi_{0,b}\rangle_{\rho_b}
			-\langle\varepsilon_-,D_s^\rho\varphi_{0,b}\rangle_{\rho_b}.
		\end{align*}
		The nonlinear calculation \eqref{quadratic-first-nonlinearity}--\eqref{second-nonlinearity-via-F} applies unchanged. Split at $b^{3/4}$ and $b^{-\eta}$: the inner integral is $O_\ell(A^4b^{3/4+2\eta}+A^3b^{7/4})$, the intermediate integral is $O_\ell(b^{3/4}+Ab^{1-\eta})$, and the Gaussian-weighted tail is smaller than every power of $b$. Here the lower modes contribute $O_\ell(Ab^{1/2+\eta}\xi^{2/3})$ for $\xi\le1$, which is covered by the same relative-error bound after taking $Ab^\eta\le1$. Thus the nonlinear projections are $O_\ell(b^{3/4})$.
		
		The differentiated expansions in Proposition~\ref{diagonalisation-L-b} and \eqref{correct-derivative-log-weight}, with $\beta'=0$, give
		\begin{equation*}
			\|b\partial_b\widehat\varphi_k\|_{L^2_{\rho_b}}
			+\|b\partial_b\varphi_{j,b}+\varphi_{j,b}b\partial_b\log\rho_b\|_{L^2_{\rho_b}}
			\lesssim_\ell b^{1/4},\qquad
			\|\varepsilon_-\|_{L^2_{\rho_b}}\lesssim_\ell A^2b^{1/2+\eta}.
		\end{equation*}
		Indeed, $b\partial_b$ cancels the leading outer powers of the kernel iterates; the next powers and the inner contribution are $O_\ell(b^{1/4})$ in the weighted norm. Consequently $|\mathcal K_j|\lesssim_\ell b^{3/4}|b'/b|$ once $A^2b^\eta\le1$. Moreover, \eqref{projections-Lambda-Q} holds for every $1\le j\le\ell$, with $O_\ell(b^{3/4})$ in the nonzero modes. Using \eqref{eigen-value-L-b}, we obtain
		\begin{equation}\label{ODE-general-item-i}
			\begin{aligned}
				\varepsilon_\ell'&=(\tfrac12-\ell)\varepsilon_\ell+\mathcal R_\ell,
				&\varepsilon_j'&=(\tfrac12-j)\varepsilon_j+\mathcal R_j\quad(1\le j<\ell),\\
				\varepsilon_\ell'&=\tfrac12\varepsilon_\ell-\kappa_\ell^{-1}\mathbf m_0\sqrt b\,(1-b'/b)+\mathcal R_0,
			\end{aligned}
		\end{equation}
		where
		\begin{equation}\label{remainder-ODE-item-i}
			|\mathcal R_j|\lesssim_\ell b^{3/4}(1+|b'/b|),\qquad 0\le j\le\ell.
		\end{equation}
		Since $\mathbf n_\ell=2\kappa_\ell^{-1}\mathbf m_0\ne0$, subtracting the two equations for $\varepsilon_\ell$ gives the exact relation
		\begin{equation}\label{equation-b-item-i}
			\kappa_\ell^{-1}\mathbf m_0\sqrt b\,[b'/b+\gamma_\ell]
			=-\ell(\varepsilon_\ell-\mathbf n_\ell\sqrt b)+\mathcal R_\ell-\mathcal R_0.
		\end{equation}
		The compatibility bound \eqref{compatibility-condi-epsilon-ell} has exponent $1/2+2\eta$. Dividing by $\sqrt b$, taking $2\eta<1/4$, and absorbing the $O_\ell(b^{1/4}|b'/b|)$ term therefore yield
		\begin{align}
			|b'/b+\gamma_\ell|&\lesssim_\ell A b^{2\eta},\label{dynamics-b-item-i}\\*
			\varepsilon_j'&=(\tfrac12-j)\varepsilon_j+O_\ell(b^{3/4}),\qquad1\le j\le\ell.\label{dynamics-modes-item-i}
		\end{align}
		Finally, $\partial_s\log(b/\mathfrak I_\ell)=O_\ell(Ab^{2\eta})$. Since $b\asymp\mathfrak I_\ell$ in the shrinking set, its integral from $s_0$ is $O_{\ell,\eta}(Ab_0^{2\eta})$. Lemma~\ref{lemma-initial-data-ell} gives $b_0/\mathfrak I_\ell(s_0)=1+O(b_0^9)$; hence $3/4\le b/\mathfrak I_\ell\le5/4$ for $s_0\ge s_6$, proving item~(i).
		
		\subsection{Profile bounds: item \emph{(ii)}}
		Write $r_b=b'/b+\gamma_\ell$. For a spatially constant $h=cb^{-p}$, the operator \eqref{definition-operator-P-beta} satisfies $-\mathscr P_1[h]=(1-p\gamma_\ell+pr_b)h-2$. Thus
		\begin{equation*}
			h_-:=\tfrac2A b^{-\eta/2},\qquad h_+:=\tfrac A2b^{-1/\gamma_\ell-\eta},
			\qquad \mathscr P_1[h_-]<0<\mathscr P_1[h_+]
		\end{equation*}
		for $\eta\gamma_\ell$ small and $s_0\ge s_6$. At $\xi=b^{-\eta}$, the spectral and shrinking-set bounds give $\vartheta\sim a^2b^{-4\eta/3}$, which lies between $h_-$ and $h_+$ when $\eta<3/\gamma_\ell$. The prepared datum is increasing, is bounded above by $C_{\ell,x_0}b_0^{-1/\gamma_\ell}$, and is bounded below by $c_\ell b_0^{-4\eta/3}$ on $\xi\ge b_0^{-\eta}$. These estimates give the initial ordering. The comparison argument of Section~\ref{Section-comparison-varep}, with a spatial penalty at infinity, applies on each finite moving cylinder: the differences and diffusion are bounded there, and the comparison profiles have zero second derivative. It gives $h_-\le\vartheta\le h_+$ on $\xi\ge b^{-\eta}$, proving item~(ii).
		
		\subsection{The infinite-dimensional remainder: item \emph{(iii)}}\label{subsection-unstable-remainder-estimates}
		\emph{Intermediate region.} Set $q=1/2+\eta$ and let $P_\ell$ be the orthogonal projection in $L^2_{\rho_{0,1}}$ onto the complement of $\operatorname{span}\{\phi_{1,0},\ldots,\phi_{1,\ell}\}$. Apply the fixed projection to the full perturbation, writing $\widehat\varepsilon_-=P_\ell\varepsilon$. With $S_1$ from \eqref{defi-S-beta-0-t}, its exact equation is
		\begin{equation*}
			\partial_s\widehat\varepsilon_-=\mathscr L_{0,1}\widehat\varepsilon_-+
			P_\ell[\widetilde{\mathcal B}_\ell(\varepsilon)+\Phi_{b,1}].
		\end{equation*}
		Here $\widetilde{\mathcal B}_\ell$ is given by \eqref{definition-B-one-stable-corrected} with $\beta=1$. The finite-interval form regularity is the same as in Proposition~\ref{propo-priori-b-eta-K}, and Lemma~\ref{lem-unstable-forcing-estimates} gives the required singular-source envelope. Lemma~\ref{lemma-weak-duhamel-singular-source}, with $P=P_\ell$ and $c=0$, therefore justifies the weak Duhamel formula and its pointwise kernel representation. The eigenfunction and density expansions, paired with the moving orthogonality conditions, give
		\begin{equation*}
			\|\widehat\varepsilon_-\|_{L^2_{\rho_{0,1}}}\lesssim_{A,\ell}b^q,
			\qquad |\widehat\varepsilon_--\varepsilon_-|\lesssim_{A,K,\ell}b^{3/4-\eta/3}\mathfrak f_\ell
			\quad(b^\eta\le\xi\le K).
		\end{equation*}
		To see the second bound, the leading outer polynomials of every finite mode belong to the range of $I-P_\ell$; the next $T_0$ term contributes $O_{A,\ell}(b^{3/4}\xi^{1/3})$, and the remaining errors are $O_{A,\ell}(b^{3/4}\mathfrak f_\ell)$. The change of orthogonality contributes $O_{A,\ell}(b^{q+1/4}\mathfrak f_\ell)$. Integration of these bounds proves the weighted estimate; the tail beyond $b^{-\eta}$ is exponentially small.
		
		For a fixed $L\ge1$, Lemmas~\ref{lem-unstable-forcing-estimates} and~\ref{lem-semi-tle-K-0-xi-le-K} give, as in \eqref{short-time-forcing-beta-one},
		\begin{equation*}
			|S_1(t)P_\ell[\widetilde{\mathcal B}_\ell+\Phi_{b,1}](\tau,\xi)|
			\lesssim_{A,K,L,\ell,\eta}b(s)^{q+\sigma_\eta}\mathfrak f_\ell(\xi),
			\quad 0<t=s-\tau\le L,\qquad \sigma_\eta=\tfrac14-\tfrac83\eta>0.
		\end{equation*}
		The finitely many subtracted nonlinear moments are $O_{A,\ell,\eta}(b^{3/4-\eta})$ by the same three-region integration. For $\Phi_{b,1}$, $P_\ell$ removes the leading multiple of $\sqrt b\,\xi^{2/3}$; Lemma~\ref{lemma-Q} leaves $O_\ell(b^{3/4}(\xi^{1/3}+\mathfrak f_\ell))$, whose kernel image has the stated bound. The comparability of $b(\tau)$ and $b(s)$ on fixed time intervals follows from \eqref{dynamics-b-item-i}.
		
		On the range of $P_\ell$, Lemma~\ref{lemma-spectrum-L0-beta-xi} gives decay $e^{-(\ell+1/2)t}$; smoothing for one unit of time therefore gives
		\begin{equation*}
			|S_1(t)P_\ell f(\xi)|\lesssim_{K,\ell}e^{-(\ell+1/2)t}\xi^{2/3}\|f\|_{L^2_{\rho_{0,1}}},\qquad t\ge1,\quad0<\xi\le K.
		\end{equation*}
		Choose $\omega_*>0$ with $q\omega_*<1/4$ and enlarge $s_6$ so that $|r_b|\le\omega_*$. If $\gamma_\ell\eta<1/4$, the decay rate after division by $b^q$ is at least $1-\gamma_\ell\eta-q\omega_*>1/2$. Duhamel's formula restarted at $s-L$ consequently bounds the normalized remainder by $C_{A,K,\ell}e^{-L/2}+C_{A,K,L,\ell}b(s_0)^{\sigma_\eta}$. For $s-s_0\le L$, \eqref{unstable-initial-remainder-corrected} gives a prepared remainder
		$O_{A,\ell}(b_0^{3/4+2\eta}(\xi^{1/3}+\mathfrak f_\ell))$ on
		$b_0^\eta\le\xi$ before the cutoff; its inner part has $L^1$ norm
		$O_{A,\ell}(b_0^{7/4+2\eta})$ and is controlled by the separated kernel bound.
		The finite-mode projection error is $O_{A,\ell}(b_0^{3/4}(\xi^{1/3}+\mathfrak f_\ell))$, and the remote initial tail is suppressed by the Gaussian kernel. Choosing $L$ first and then $s_0$ proves
		\begin{equation}\label{ell-estimate-epsilon-minus-b-eta-K}
			|\varepsilon_-(s,\xi)|\le A^{-26}b^q\mathfrak f_\ell(\xi),\qquad b^\eta\le\xi\le K.
		\end{equation}
		
		\emph{Inner region.} Use \eqref{definition-inner-theta-zeta} with $\theta=Q+\widetilde\theta$, $R=b^{\eta-3/4}$, and $\gamma=b-b'=2\ell b+O_\ell(Ab^{1+2\eta})$. The $T_0$ cancellation holds in $\psi_{\ell,b}$, but the lower modes retain a contribution of size $Ab^\eta\Lambda Q$ to $\widetilde\theta$. Thus \eqref{tilde-theta-bootstrap-inner} remains valid, whereas the boundary estimate in \eqref{Duhamel-arbitrary-L-general-p} has an additional $O_\ell(A)$ term. To accommodate it, use the profiles \eqref{equations-Z-Y-general-p} with
		\begin{equation*}
			\sigma=A^{-2},\qquad m=A^{9/4}b^\eta,\qquad
			c_\pm=2\ell b\pm b^{1+\eta/2},\qquad
			f_\pm=c_\pm\mathcal Y_\sigma\pm m\mathcal Z_\sigma.
		\end{equation*}
		By \eqref{global-bounds-Z-general-p}, $m\mathcal Z_\sigma\gtrsim A^{5/4}b^\eta\Lambda Q$ and $m\mathcal Z_\sigma\lesssim A^{7/4}b^\eta\Lambda Q$, while $|c_\pm\mathcal Y_\sigma|/(m\mathcal Z_\sigma)=O_{A,\ell}(b^{\eta/3})$. The two operator decompositions \eqref{decomposition-G1-general-p}--\eqref{decomposition-G2-general-p} now give the barrier signs as follows. Near zero, the signed frozen-diffusion term has size
		\begin{equation*}
			m\sqrt Q\,Q_\sigma^{-3/2}\mathcal Z_\sigma\sim\tfrac14A^{5/2}b^\eta,
		\end{equation*}
		which dominates the $O_\ell(A^2b^\eta)$ actual-profile error. Away from zero, $D_\sigma\mathcal Z_\sigma<0$ and $\pm(\gamma-c_\pm)\le-\frac12b^{1+\eta/2}$ give coercivity comparable to $b^\eta\langle\zeta\rangle^{-4/3}+b^{1+\eta/2}\Lambda Q$, with constants depending on $A,\ell$. The remaining errors are bounded by
		\begin{equation*}
			C_{A,\ell}\bigl(b^{2\eta}\langle\zeta\rangle^{-4/3}+b^{1+\eta}\Lambda Q+b^2\langle\zeta\rangle^2\bigr);
		\end{equation*}
		relative to that coercivity, they gain $b^\eta$, $b^{\eta/2}$, and $b^{5\eta/6}$, respectively. These are precisely the estimates in \eqref{G2-errors-general-p}, with the modified coefficients.
		
		At $s=s_0$ and at $\zeta=R(s)$, the prepared data, \eqref{ell-estimate-epsilon-minus-b-eta-K}, and the finite-mode expansion give $|\widetilde\theta|\le C_\ell(A+o(1))b^\eta\Lambda Q$. This lies between $f_-$ and $f_+$ when $A$ is large. All functions vanish at the wall. The frozen/nonlinear comparison argument in Proposition~\ref{propo-priori-near-0} therefore yields $|\widetilde\theta|\lesssim_\ell A^{7/4}b^\eta\Lambda Q$ throughout $[0,R]$. On returning to $\xi$ and subtracting $\varepsilon_+$, its lower-mode part costs $O_\ell(A)$ in the same normalized bounds, while its adapted part costs $o(1)$. Taking $A$ large so that $C_\ell(A^{7/4}+A)<A^2/4$ bounds each term of the norm, restricted to $0<\xi\le b^\eta$, by $(A^2/4)b^{1/2+\eta}$.
		
		\emph{Outer region.} Let $\phi=(-1)^{\ell+1}\phi_{1,\ell+1}$, using \eqref{defi-phi-beta-i}. For $K$ sufficiently large, $\phi\asymp_\ell\mathfrak f_\ell\asymp_\ell\xi^{p_\ell}>0$ on $\xi\ge K$, its first two Euler derivatives are $O_\ell(\phi)$, and $\mathscr L_{0,1}\phi=-(\ell+1/2)\phi$. Set $F=b^q\phi$. Then
		\begin{equation*}
			(\partial_s-\mathscr L_{0,1})F=(1-\gamma_\ell\eta+qr_b)F\ge\tfrac12F.
		\end{equation*}
		Use \eqref{definition-P-outer} with $\beta=1$ and the unstable finite-mode part. The exact remainder is $\widetilde{\mathcal R}_\ell=\mathcal R_{\ell,b}+(\mathscr L_{0,1}-\mathscr L_{b,1})\varepsilon_+$. By \eqref{global-estimate-R-general-ell} and the differentiated spectral bounds,
		\begin{equation*}
			|\widetilde{\mathcal R}_\ell|\lesssim_{A,\ell}b^{3/4}\mathfrak f_\ell,
			\qquad |\varepsilon_+|+\xi^2|\varepsilon_+''|
			\lesssim_\ell b^{1/2}\xi^{p_\ell-4/3}+C_{A,\ell}b^{3/4}\mathfrak f_\ell
			\quad(\xi\ge K).
		\end{equation*}
		Choose the upper comparison endpoint $\Xi_+=C_*b^{-k_\ell}$, where
		\begin{equation*}
			k_\ell=\frac{3(2\ell+1)}{4\gamma_\ell(2\ell+3)}+\frac{3\eta}{2\ell+3},
			\qquad q-k_\ell p_\ell=-\frac1{\gamma_\ell}-\eta.
		\end{equation*}
		The rationalized estimate \eqref{general-Btilde-bound-outer} applies to $h_\pm=\varepsilon_+\pm F$ wherever $Q_b+h_\pm>0$. Dividing it by $F$, the quadratic terms have sizes $b^q\xi^{(4\ell-2)/3}$, $b^{1/2}\xi^{(4\ell-6)/3}$, and $b^{1/2-\eta}\xi^{(4\ell-10)/3}$. Their maxima on $[K,\Xi_+]$, together with the linear errors, give
		\begin{equation*}
			\frac{|\widetilde{\mathcal B}_\ell(h_\pm)|+|\widetilde{\mathcal R}_\ell|}{F}
			\lesssim_{A,K,\ell,C_*}b^{1/4-\eta}+b^{1/2-\eta}
			+b^{[1+(5-2\ell)\eta]/(2\ell+3)}=o(1).
		\end{equation*}
		Thus $\mathscr P(-F)<0<\mathscr P(F)$ on the respective positive-profile domains. The expansion also gives $Q_b+\varepsilon_++F>0$ throughout $[K,\Xi_+]$ for $\eta$ small.
		
		At $\xi=K$, \eqref{ell-estimate-epsilon-minus-b-eta-K} gives the boundary ordering. At $\Xi_+$, choosing $C_*$ large makes $F\ge\vartheta+|\varepsilon_+|$, by item~(ii), the identity for $q-k_\ell p_\ell$, and the preceding finite-mode bound. The initial ordering holds on the entire interval: the choice of $\delta$ in Lemma~\ref{lemma-initial-data-ell} satisfies
		\begin{equation*}
			\frac{3(2\ell+1)}{4\gamma_\ell(2\ell+3)}<\delta
			=\frac{3(\ell+1)}{2\gamma_\ell(2\ell+3)}<\frac{3}{4\gamma_\ell},
		\end{equation*}
		More precisely,
		$\delta-k_\ell=\frac{3}{2\ell+3}\bigl(\frac{1}{4\gamma_\ell}-\eta\bigr)>0$.
		Since $b_0/b_0^\dagger\to1$, increasing $s_6$ gives
		$\Xi_+(s_0)<R_{\ell,\delta}/2=(b_0^\dagger)^{-\delta}/(2M_0)$.
		Thus \eqref{unstable-initial-remainder-corrected} applies throughout
		$[K,\Xi_+(s_0)]$, where there is no cutoff error, and yields
		$|\varepsilon_-(s_0)|/F\lesssim_{A,K,\ell}b_0^{1/4+\eta}=o(1)$. Comparison proves $\varepsilon_-\le F$ up to $\Xi_+$.
		
		For the lower bound, put $g=Q_b+\varepsilon_+-F$ and $r_\ell=3(1+2\eta)/[4(2\ell+1)]$. If $c_\phi>0$ is the leading coefficient of $\phi$, then
		\begin{equation*}
			b^{4r_\ell/3}g(s,b^{-r_\ell}y)\longrightarrow a^2y^{4/3}-c_\phi y^{p_\ell}
		\end{equation*}
		in $C^1$ on compact subsets of $y>0$. Indeed, the finite-mode relative error gains $b^{[1-\gamma_\ell\eta]/(2\ell+1)}$, and its remainder gains $b^{1/4-\eta}$. The positive zero of the limit is simple. The same compact/large-scale argument as in the lower comparison of Proposition~\ref{pro-priori-epsilon-xi-large} therefore gives a $C^1$ endpoint $\Xi_-\asymp b^{-r_\ell}$ with $g(s,\Xi_-)=1$ and $g\ge1$ on $[K,\Xi_-]$. For small $\eta$, $\eta<r_\ell<k_\ell<\delta$. Item~(ii) gives $\vartheta(s,\Xi_-)>1$, while the initial and left boundary ordering were just proved. Comparison yields $\varepsilon_-\ge-F$ up to $\Xi_-$. Consequently, for a fixed $c_\ell>0$ independent of $s_0$ and $\bar s$,
		\begin{equation}\label{ell-outer-trapping-extended}
			|\varepsilon_-(s,\xi)|\le b(s)^{1/2+\eta}\phi(\xi)
			\lesssim_\ell b(s)^{1/2+\eta}\xi^{p_\ell},\qquad
			K\le\xi\le c_\ell b(s)^{-r_\ell},\quad s_0\le s\le\bar s.
		\end{equation}
		Since $\Xi_->b^{-\eta}$ and $\phi\lesssim_\ell\mathfrak f_\ell$, the inner, intermediate, and outer bounds make each term of $\|\varepsilon_-\|_{\mathfrak X_{b,\eta}^{(\ell)}}$ at most $(A^2/4)b^{1/2+\eta}$ for $A$ large. Summing proves item~(iii).
		
		\subsection{The residual: item \emph{(iv)}}\label{subsection-unstable-residual-estimates}
		\emph{Global and compact bounds.} Use the exact equations \eqref{equation-F-outer-exact}--\eqref{equa-mathbbm-F}. Since $C_{\rm in}\ge2$, the constant lower barriers in Lemma~\ref{lem-esti-F-le-A}, together with \eqref{global-bound-F-ell-initial}, give $\mathcal F\ge-C_{\rm in}$. On $\xi\le b^{-\eta}$, the shrinking-set bound is $|\mathcal F|\lesssim_A b^{1/2-(p_\ell+1)\eta}=o(1)$. On the complementary region, $H=b^{-\eta/8}$ satisfies
		\begin{equation*}
			\frac{\mathscr P_{\mathcal F}[H]}H
			=\tfrac\eta8(\gamma_\ell-r_b)-\frac{H+2}{2\vartheta}
			\ge\tfrac{\eta\gamma_\ell}{16}-C_A b^{3\eta/8}>0,
		\end{equation*}
		using item~(ii). The same initial and moving-boundary comparisons give $-C_{\rm in}\le\mathcal F\le b^{-\eta/8}$ globally. Since $A\ge A_6\ge A_5>2C_{\rm in}$, these bounds imply the global improvement in item~(iv).
		
		For every fixed $K$, we also have $|\mathcal F|\le A^{-26}b^{1/2-\eta}$ on $[K^{-1},K]$ after increasing $s_6$. Here one cannot use the stable residual decay rate alone, since $b$ decays faster. Instead, on every fixed compact spatial interval, the decomposition and shrinking bounds give $\varepsilon/b^{1/2-\eta}\to0$ uniformly as $s_0\to\infty$. Normalize on time intervals of bounded length by the terminal value of $b^{1/2-\eta}$. As in Lemma~\ref{lem-priori-F-K}, the residual equation is uniformly parabolic there, $\vartheta\to a^2\xi^{4/3}$ in $C^2$, and the normalized residuals are locally H\"older compact away from the initial time. The exact identity
		\begin{equation*}
			\varepsilon''=\frac{\mathcal F}{\sqrt\vartheta}
			+2\left(\frac1{\sqrt\vartheta}-\frac1{\sqrt{Q_b}}\right)
		\end{equation*}
		identifies every such limit as zero: its normalized left side tends to zero in distributions, and the second term does so locally uniformly.
		To include times approaching the initial time, let $T_n=\sigma_n-s_{0,n}\to0$ and set
		$h_n(t,\xi)=b_n(\sigma_n)^{-(1/2-\eta)}\mathcal F_n(s_{0,n}+t,\xi)$ for $0\le t\le T_n$.
		For $I\Subset J\Subset(0,\infty)$, Lemma~\ref{lemma-initial-data-ell} and
		$b_n(\sigma_n)\asymp b_n(s_{0,n})$ give
		$\delta_n:=\|h_n(0)\|_{L^\infty(J)}\to0$.
		The shrinking bounds give $\|h_n\|_{L^\infty([0,T_n]\times J)}\le C_J$, and the normalized equations have the uniform local coefficient bounds used in Lemma~\ref{lem-priori-F-K}. Thus the comparison argument proving \eqref{residual-uniform-initial-trace} gives
		\begin{equation*}
			\|h_n(t)\|_{L^\infty(I)}
			\le e^{C_0t}\delta_n+C_1e^{-c_1/t},\qquad 0<t\le\min\{1,T_n\},
		\end{equation*}
		with constants independent of $n$. At $t=T_n$, this tends to zero; if $T_n=0$, the initial estimate suffices.
		Thus a sequence of points violating the asserted compact bound would have a nonzero limit, a contradiction.
		
		\emph{Inner residual barrier.} Fix $\sigma=1/2$ and abbreviate $Z=\mathcal Z_\sigma$, $Y=\mathcal Y_\sigma$ from \eqref{equations-Z-Y-general-p}. On $0\le\zeta\le R_K=K^{-1}b^{-3/4}$, the spectral and shrinking bounds give
		\begin{equation*}
			|\theta-Q|\lesssim_\ell A^2b^\eta\Lambda Q+b\zeta^2,\qquad \theta\asymp Q.
		\end{equation*}
		The stable barrier $b^{1-\eta}Z$ requires a correction: its leading time/drift contribution is $[(\ell-1)-\gamma_\ell\eta]bZ$, which is positive for small $\eta$. Choose a fixed $D_\ell>0$ sufficiently large, and use
		\begin{equation*}
			m=b^{1-\eta},\qquad J=Z+D_\ell bY,\qquad \mathbbm F^\pm=\pm mJ.
		\end{equation*}
		Since $|Y|/Z\lesssim\langle\zeta\rangle^{4/3}$ and $bR_K^{4/3}=K^{-4/3}$, taking $K$ large gives $Z/2\le J\le Z$. Write $\mathcal R_\theta[f]=(\sqrt\theta-\sqrt Q)f''+(\theta^{-1}-Q^{-1})f$, as in Section~\ref{section-pointwise-estimates-F}. After division by $m$, the linear residual of the upper barrier is exactly
		\begin{align*}
			\mathcal A_\ell&={}D_\sigma Z+\mathcal R_\theta[Z]-\tfrac34\gamma\zeta Z'-(1-\eta)b'Z\\*
			&\quad{}+D_\ell b\bigl[-\Lambda Q+D_\sigma Y+\mathcal R_\theta[Y]-\tfrac34\gamma\zeta Y'-(2-\eta)b'Y\bigr].
		\end{align*}
		The first two terms are at most $-c\langle\zeta\rangle^{-4/3}$ by \eqref{coercivity-D-sigma-F} and \eqref{R-theta-A-estimate}, after taking $K,s_6$ large. For large $\zeta$, $\zeta Z'=2Z/3+O(\zeta^{1/3})$, so the remaining terms in the first line are bounded above by $C_\ell b\Lambda Q$. In the second line, $D_\sigma Y=O(1)=o(\Lambda Q)$, while the coefficient and time/drift errors, after multiplication by $b$, are bounded by $C_\ell(A^2b^\eta+K^{-4/3})b\Lambda Q$. Thus $-D_\ell b\Lambda Q$ dominates the positive drift when $D_\ell$ is chosen first; on any remaining bounded interval all $O_\ell(b)$ errors are absorbed by the first coercive term. It follows that
		\begin{equation*}
			\mathcal A_\ell\le-c_\ell(\langle\zeta\rangle^{-4/3}+b\Lambda Q),\qquad
			\mathbbm G[\mathbbm F^\pm]=\pm m\mathcal A_\ell+\frac{m^2J^2}{2\theta},
		\end{equation*}
		where $\mathbbm G$ is the residual operator used in \eqref{residual-barriers-F}. Since $J^2/\theta\lesssim1$, $\langle\zeta\rangle^{-4/3}+b\Lambda Q\gtrsim b^{2/3}$, and $m=o(b^{2/3})$ for $\eta<1/3$, these residuals have the required opposite signs.
		
		The initial residual is $O_\ell(b_0\Lambda Q)=o(mJ)$ before the cutoff. At $R_K$, the compact estimate just proved is smaller than $mJ\asymp K^{-2/3}b^{1/2-\eta}$ for $A$ large. At zero all three functions vanish linearly. Dividing the comparison differences by $\sqrt\zeta$, exactly as in Lemma~\ref{lem-esti-F-0-1/K}, bounds the singular zeroth-order coefficient above and justifies comparison at the wall. Hence $|\mathbbm F|\le mJ\lesssim b^{1-\eta}\Lambda Q$, which gives
		\begin{equation*}
			|\mathcal F|\lesssim_\ell b^{1/4-\eta}\xi\quad(0\le\xi\le b^{3/4}),\qquad
			|\mathcal F|\lesssim_\ell b^{1/2-\eta}\xi^{2/3}\quad(b^{3/4}\le\xi\le K^{-1}).
		\end{equation*}
		
		\emph{Outer residual barrier.} Use the same positive tail $\phi=(-1)^{\ell+1}\phi_{1,\ell+1}$ as above, and put $H=b^{1/2-\eta}\phi$, $R=b^{-3/[4(2\ell+3)]}$. For small $\eta$, $b^{-\eta}<R<\Xi_-$.
		Since $\delta-3/[4(2\ell+3)]=9/[4(2\ell-1)(2\ell+3)]>0$,
		we also have $R(s_0)<R_{\ell,\delta}/2$ after increasing $s_6$. The two-sided outer estimate for $\varepsilon_-$ therefore applies throughout $[K,R]$. Together with the finite-mode expansion, it gives
		\begin{equation*}
			\frac{|\vartheta-a^2\xi^{4/3}|}{\xi^{4/3}}\lesssim_{A,\ell}b^{1/(2\ell+3)},
			\qquad \vartheta\asymp\xi^{4/3},\qquad H/\vartheta\lesssim_\ell b^{1/(2\ell+3)-\eta}.
		\end{equation*}
		For $\mathcal E_b[f]=(\sqrt\vartheta-a\xi^{2/3})f''+(\vartheta^{-1}-a^{-2}\xi^{-4/3})f$, the first estimate gives $|\mathcal E_b[H]|\lesssim_{A,\ell}b^{1/(2\ell+3)} H$. Since $(\mathscr L_{0,1}-\operatorname{Id})\phi=-(\ell+3/2)\phi$, the exact residuals are
		\begin{equation*}
			\mathscr P_{\mathcal F}[\pm H]
			=\pm\bigl([2+\gamma_\ell\eta+(\tfrac12-\eta)r_b]H-\mathcal E_b[H]\bigr)-\frac{H^2}{2\vartheta}.
		\end{equation*}
		They have the desired signs for $\eta<1/(2\ell+3)$ and $s_0\ge s_6$. The compact estimate orders the barriers at $K$; the prepared residual is $O_\ell(b_0^{1/2}\mathfrak f_\ell)=o(H)$ at $s_0$; and at $R$, $H\asymp_\ell b^{-\eta}\gg C_{\rm in}+b^{-\eta/8}\ge|\mathcal F|$. Comparison yields $|\mathcal F|\le H\lesssim_\ell b^{1/2-\eta}\mathfrak f_\ell$ on $[K,R]$. Taking $A$ large, the inner, compact, and outer bounds make each term of $\|\mathcal F\|_{\mathfrak X_{b,\eta}^{(\ell)}}$ at most $(A/4)b^{1/2-\eta}$. Their sum and the global bounds prove item~(iv).
		
		\section{Spectral analysis of the linearized operator}\label{section-proof-L-b diagonal}
		We prove Proposition~\ref{diagonalisation-L-b} by an inner--outer matching argument of the type used in~\cite[Section~2]{Collot-et-alJAMS-2020}. The kernel estimates give an inner expansion around $Q$, while the limiting eigenfunctions generate the outer expansion; matching their values and derivatives determines the eigenvalue correction. Set $N=\ell+1$ and carry out the construction for $0\le i\le N$. We work on a slightly larger compact $\beta$-interval than $[C_0^{-1},C_0]$, so that parameter derivatives are defined at its endpoints. All constants and smallness thresholds may depend on $\ell,C_0$. We use the definitions of $Q_b$, $\Lambda$, and the operators in \eqref{defi-Q-b}, \eqref{Lambda:op}, \eqref{defi-mathscr-L-}--\eqref{defi-L_0_beta}.
		
		The dependence on $\beta$ can be recovered from the case $\beta=1$. Indeed, with $\mathcal D_\beta f(\xi)=f(\beta^{3/4}\xi)$,
		\begin{equation}\label{spectral-beta-scaling}
			\mathscr L_{b,\beta}\mathcal D_\beta=\beta\mathcal D_\beta\mathscr L_{\beta b,1},\qquad
			\mu_{i,b,\beta}=\beta\mu_{i,\beta b,1},\qquad
			\varphi_{i,b,\beta}=\beta^{1/8}\mathcal D_\beta\varphi_{i,\beta b,1}.
		\end{equation}
		Indeed, $Q_{\beta b}(\beta^{3/4}\xi)=\beta Q_b(\xi)$ and, with $y=\beta^{3/4}\xi$, $d\rho_{b,\beta}(\xi)=\beta^{-1/4}d\rho_{\beta b,1}(y)$. Hence $\beta^{1/8}\mathcal D_\beta$ is unitary between the corresponding weighted spaces and maps their closed form domains onto one another. The operator identity therefore holds also for the Friedrichs realizations. We use a single dimensionless correction $\nu$, so the eigenvalue is $\beta(\frac12-i+\nu)$; its correction is $\widetilde\mu=\beta\nu$.
		
		\subsection{The ground state and the inner kernel}\label{section-app-T-i}
		\begin{lemma}[Ground state]\label{lemma-Q}
			The problem \eqref{ground-state-Q} has a unique positive solution $Q\in C^1([0,\infty))\cap C^\infty((0,\infty))$. With $a$ as in \eqref{notation-constant-a},
			\begin{alignat*}{2}
				Q(r)&=r+\tfrac83r^{3/2}+O(r^2) &\qquad &(r\to0),\\*
				Q(r)&=a^2r^{4/3}+\tfrac a4r^{2/3}-\tfrac{a^2}{18}r^{1/3}+O(1) &\qquad &(r\to\infty).
			\end{alignat*}
			Moreover, $\Lambda Q>0$. These expansions remain valid after any fixed number of Euler derivatives $D_r=r\partial_r$.
		\end{lemma}
		\begin{proof}
			The first integral is $(Q')^2=1+8\sqrt Q$. Its positive branch gives the strictly increasing relation
			\begin{align}
				r&=\int_0^{Q(r)}\frac{dt}{\sqrt{1+8\sqrt t}},\notag\\*
				48r&=Z^3-3Z+2,\qquad Z=\sqrt{1+8\sqrt Q}.\label{implicit-Q-Z}
			\end{align}
			The integral maps $[0,\infty)$ onto itself, proving existence and uniqueness. Expanding \eqref{implicit-Q-Z} at $Z=1$ gives the first expansion; at infinity it gives
			\begin{equation*}
				r=\frac{\sqrt2}{3}Q^{3/4}-\frac{\sqrt2}{16}Q^{1/4}+\frac1{24}+O(r^{-1/3}),
			\end{equation*}
			whose inversion gives the second. The analytic implicit function theorem in $r^{1/2}$ and $r^{-1/3}$ justifies the differentiated remainders. Finally, $(\Lambda Q)'=(Z-1)/(16\sqrt Q)>0$ and $\Lambda Q(0)=0$.
		\end{proof}
		In particular, for $k\ge2$ at zero and $k\ge0$ at infinity,
		\begin{alignat*}{2}
			Q^{(k)}(r)&=\frac83\frac{\Gamma(5/2)}{\Gamma(5/2-k)}r^{3/2-k}+O(r^{2-k})&\qquad &(r\to0),\\*
			Q^{(k)}(r)&=a^2\frac{\Gamma(7/3)}{\Gamma(7/3-k)}r^{4/3-k}+O(r^{2/3-k})&\qquad &(r\to\infty).
		\end{alignat*}
		
		Set $\mathscr H_0=\mathcal L_{b,\beta}-b\beta\Lambda=\sqrt Q(\partial_r^2+Q^{-3/2})$, using \eqref{defi-mathcal-beta}. Differentiating the ground-state equation along its scaling gives $\mathscr H_0(\Lambda Q)=0$. A second kernel element and its Wronskian are
		\begin{align}
			\widetilde Q_0(r)&=-\Lambda Q(r)\int_r^\infty\frac{dt}{(\Lambda Q(t))^2},\notag\\*
			\mathcal W[\Lambda Q,\widetilde Q_0]&=(\Lambda Q)\widetilde Q_0'-(\Lambda Q)'\widetilde Q_0=1.\label{wroskian-u-1-2}
		\end{align}
		Lemma~\ref{lemma-Q} gives
		\begin{alignat*}{3}
			\Lambda Q&=\tfrac14r+O(r^{3/2}),&\qquad \widetilde Q_0&=-4+O(r^{1/2})&\qquad &(r\to0),\\*
			\Lambda Q&=\tfrac a8r^{2/3}-\tfrac{a^2}{24}r^{1/3}+O(1),&\qquad \widetilde Q_0&=-\tfrac{24}a r^{1/3}+O(1)&\qquad &(r\to\infty).
		\end{alignat*}
		The inverse with zero value and derivative at the origin is
		\begin{align}
			\mathscr H_0^{-1}f(r)&=\widetilde Q_0(r)\int_0^r\frac{\Lambda Q\,f}{\sqrt Q}\,dt-\Lambda Q(r)\int_0^r\frac{\widetilde Q_0 f}{\sqrt Q}\,dt\label{inverse-H-0}\\*
			&=\Lambda Q(r)\int_0^r\frac1{(\Lambda Q(t))^2}\int_0^t\frac{\Lambda Q(s)f(s)}{\sqrt{Q(s)}}\,ds\,dt.\label{secon-defi-H_0-inverse}
		\end{align}
		The Wronskian \eqref{wroskian-u-1-2} verifies $\mathscr H_0\mathscr H_0^{-1}f=f$, and \eqref{secon-defi-H_0-inverse} proves positivity preservation.
		
		\begin{lemma}[Iterated kernel]\label{lemma-T-i}
			Let $T_0=\cscale\Lambda Q$, with \eqref{notation-scaling-normalization}, and $T_{j+1}=\mathscr H_0^{-1}T_j$. Then $T_j>0$, with the following expansions at zero and infinity, respectively:
			\begin{equation}\label{asym-T}
				T_j(r)=t_jr^{1+3j/2}(1+O(r^{1/2})),\qquad
				T_j(r)=\mathsf C_jr^{2/3+4j/3}(1+O(r^{-1/3})),
			\end{equation}
			Here the constants are
			\begin{equation}\label{defi-C_i}
				\mathsf C_j=\frac{(9/(16a))^j}{j!\Gamma(j+5/4)},\qquad
				t_j=\frac{2\Gamma(5/3)}{a\Gamma(5/4)}\frac{(4/9)^j}{j!\Gamma(j+5/3)},\qquad j\ge0.
			\end{equation}
			Both remainders satisfy the same bounds after any fixed number of Euler derivatives. In particular, $\Theta_j:=\Lambda T_j-(\frac12-j)T_j$ is $O(r^{1+3j/2})$ at zero and $O(r^{1/3+4j/3})$ at infinity.
		\end{lemma}
		\begin{proof}
			Insert the endpoint expansions of $\Lambda Q,\widetilde Q_0,Q$ into \eqref{inverse-H-0}. The leading coefficients obey
			\begin{alignat*}{2}
				t_{j+1}&=\frac{4t_j}{(3j+3)(3j+5)},&\qquad t_0&=\frac{\cscale}{4},\\*
				\mathsf C_{j+1}&=\frac{9\mathsf C_j}{16a(j+1)(j+5/4)},&\qquad \mathsf C_0&=\frac1{\Gamma(5/4)}.
			\end{alignat*}
			These recurrences give \eqref{defi-C_i} and \eqref{asym-T}. Differentiate the integral once and then use $\mathscr H_0T_{j+1}=T_j$ for the derivative bounds; further derivatives follow from the differentiated equation. Positivity follows from \eqref{secon-defi-H_0-inverse}, and the leading term of $\Theta_j$ cancels by homogeneity.
		\end{proof}
		
		\subsection{Inner expansion}\label{sub-inner expansion}
		Write $\zeta=\xi/b^{3/4}$ and $\zeta_0=\xi_0/b^{3/4}$, where $0<\xi_0\ll1$. For functions in $C^1([0,\zeta_0])\cap C^2((0,\zeta_0])$, use
		\begin{equation}\label{defi-norm-X-a-1-2}
			\|f\|_{X^{p,q}_{\zeta_0}}:=\sup_{0<\zeta\le\zeta_0}
			\frac{\sum_{m=0}^2|D_\zeta^mf(\zeta)|}{\zeta^p\langle\zeta\rangle^{q-p}}.
		\end{equation}
		These are Banach spaces for $p\ge1$; the bracket is defined in \eqref{notation-japanese-bracket}.
		\begin{lemma}[Inner inverse]\label{lemma_continued-H_0_on_X-a-12}
			For $p\ge1$ and $q>4/3$, uniformly in $\zeta_0$,
			\begin{align}
				\|f\|_{X^{p,q-4/3}}&\le\langle\zeta_0\rangle^{4/3}\|f\|_{X^{p,q}},\label{X-a-to-a-2-f}\\*
				\|\mathscr H_0^{-1}f\|_{X^{p,q}}+\|\mathscr H_0^{-1}\Lambda f\|_{X^{p,q}}&\lesssim\|f\|_{X^{p,q-4/3}}.\label{norm-H-1-f-Lambda-f}
			\end{align}
		\end{lemma}
		\begin{proof}
			The embedding follows from \eqref{defi-norm-X-a-1-2}. For $g=\mathscr H_0^{-1}f$, split \eqref{inverse-H-0} at $1$: $|g|+|\zeta g'|$ is bounded by $\zeta^{p+3/2}$ near zero and $\zeta^q$ at infinity, times the source norm. The equation $g''=f/\sqrt Q-g/Q^{3/2}$ then gives
			\begin{equation*}
				|g(\zeta)|+|\zeta g'(\zeta)|+|\zeta^2g''(\zeta)|\lesssim
				\zeta^p\langle\zeta\rangle^{q-p}\|f\|_{X^{p,q-4/3}}.
			\end{equation*}
			For $\Lambda f$, only $f$ and $D_\zeta f$ enter the same argument, so no derivative is lost.
		\end{proof}
		\begin{example}\label{examp-T-j+1-norm-X}
			Lemma~\ref{lemma-T-i}, inserted into \eqref{defi-norm-X-a-1-2}, gives, for $q_j=2+4j/3$,
			\begin{equation*}
				\|T_{j+1}\|_{X^{2,q_j}}\lesssim1,\qquad \|\Theta_j\|_{X^{1,1/3+4j/3}}\lesssim1.
			\end{equation*}
		\end{example}
		
		\begin{proposition}[Inner eigenfunction]\label{propo-esti-S-j-R-i}
			There exist $\xi_{\rm int},\nu_{\rm int}\in(0,1)$, depending only on
			$\ell$ and the fixed compact $\beta$-range, such that, for
			$0<\xi_0\le\xi_{\rm int}$, $0<b\le\xi_0^{4/3}$ and
			$|\nu|\le\nu_{\rm int}$, there is an inner eigenfunction,
			smooth for $\zeta>0$ and $C^1$ at zero, satisfying
			\begin{equation}\label{eigen-equa-inner}
				\mathcal L_{b,\beta}\phi_{i,\mathrm{int}}=\beta b(\tfrac12-i+\nu)\phi_{i,\mathrm{int}},
			\end{equation}
			with the expansion
			\begin{equation}\label{phi-i-int}
				\phi_{i,\mathrm{int}}
				=\sum_{j=0}^i c_{i,j}\beta^{j+5/8}b^{j+1/2}T_j
				+\nu\sum_{j=0}^i b^{j+3/2}(c_{i,j}\beta^{j+13/8}T_{j+1}+S_j)+b^{3/2}R_i.
			\end{equation}
			Here $c_{i,j}$ is given by \eqref{defi-c-i-j}. With $q_j=2+4j/3$, its corrections satisfy
			\begin{alignat}{2}
				\|S_j\|_{X^{2,q_j}}+\|\partial_\nu S_j\|_{X^{2,q_j}}+\|\partial_\beta S_j\|_{X^{2,q_j}}
				&\lesssim\xi_0^{4/3},&\qquad \|\partial_bS_j\|_{X^{2,q_j+4/3}}&\lesssim1,\label{prop12-S-bounds}\\*
				\|R_i\|_{X^{1,5/3}}+\|\partial_\beta R_i\|_{X^{1,5/3}}
				&\lesssim1,&\qquad \|\partial_bR_i\|_{X^{1,3}}+b^{-1}\|\partial_\nu R_i\|_{X^{1,3}}&\lesssim1.\label{prop12-R-bounds}
			\end{alignat}
			All parameter derivatives are at fixed $\zeta$. The solutions depend $C^1$ on the parameters on every common compact subinterval of their domains. On $0<\zeta<1$, the local refinement is
			\begin{equation}\label{impro-beha-R_i}
				\sum_{m=0}^2\bigl(|D_\zeta^mR_i|+|D_\zeta^m\partial_bR_i|+|D_\zeta^m\partial_\beta R_i|
				+b^{-1}|D_\zeta^m\partial_\nu R_i|\bigr)\lesssim\zeta^{5/2}.
			\end{equation}
			The pointwise bounds underlying \eqref{prop12-S-bounds}--\eqref{impro-beha-R_i} also hold with three Euler derivatives.
		\end{proposition}
		\begin{proof}
			Set $\mathscr D_{i,\nu}=\frac12-i+\nu-\Lambda$, $K=\beta b\mathscr H_0^{-1}\mathscr D_{i,\nu}$, $d_j=c_{i,j}\beta^{j+13/8}$, and $F_i=\sum_{j=0}^id_jb^j\Theta_j$. For $\mathcal E_i=\mathscr H_0-\beta b\mathscr D_{i,\nu}$, substitution of \eqref{phi-i-int} gives
			\begin{align*}
				\mathcal E_i\phi_{i,\mathrm{int}}
				&={}\nu\sum_{j=0}^ib^{j+3/2}\bigl[\mathscr H_0S_j-\beta b\mathscr D_{i,\nu}(S_j+d_jT_{j+1})\bigr]\\*
				&\quad{}+b^{3/2}\bigl[\mathscr H_0R_i-\beta b\mathscr D_{i,\nu}R_i+F_i\bigr].
			\end{align*}
			Indeed, $\mathscr H_0T_{j+1}=T_j$ and $c_{i,j+1}=-(i-j)c_{i,j}$ cancel the consecutive kernel terms, and the $\nu T_j$ terms cancel against the $T_{j+1}$ correction. Thus \eqref{eigen-equa-inner} follows once
			\begin{align}
				\mathscr H_0S_j&=\beta b\mathscr D_{i,\nu}(S_j+d_jT_{j+1}),\notag\\*
				\mathscr H_0R_i&=\beta b\mathscr D_{i,\nu}R_i-F_i.\label{equa-R_i}
			\end{align}
			Using \eqref{inverse-H-0}, these equations become
			\begin{align}
				S_j&=K(S_j+d_jT_{j+1}),\label{equa-S_j-H-0-inverse}\\*
				R_i&=KR_i-\mathscr H_0^{-1}F_i.\label{R-fixed-point-corrected}
			\end{align}
			Put $\chi=b\langle\zeta_0\rangle^{4/3}\lesssim b+\xi_0^{4/3}$. By \eqref{X-a-to-a-2-f}--\eqref{norm-H-1-f-Lambda-f}, on each space used below,
			\begin{equation*}
				\|Kf\|_{X^{p,q}}\le C\chi\|f\|_{X^{p,q}}\le\tfrac12\|f\|_{X^{p,q}},
			\end{equation*}
			by fixing $\nu_{\rm int}<1$ and choosing $\xi_{\rm int}$ small enough,
			since $b\le\xi_0^{4/3}$ and $\xi_0\le\xi_{\rm int}$. Example~\ref{examp-T-j+1-norm-X} also gives
			\begin{equation*}
				\|F_i\|_{X^{1,1/3}}+\|\partial_\beta F_i\|_{X^{1,1/3}}\lesssim\sum_{j=0}^i\chi^j\lesssim1,
				\qquad \|\partial_bF_i\|_{X^{1,5/3}}\lesssim\sum_{j=1}^ij\chi^{j-1}\lesssim1.
			\end{equation*}
			Hence the fixed points exist uniquely and satisfy
			\begin{align}
				\|S_j\|_{X^{2,q_j}}&\lesssim\chi\|T_{j+1}\|_{X^{2,q_j}}\lesssim\xi_0^{4/3},\notag\\*
				\|R_i\|_{X^{1,5/3}}&\lesssim\|F_i\|_{X^{1,1/3}}\lesssim1.\label{R-bound-corrected}
			\end{align}
			The parameter estimates follow from the identities
			\begin{align*}
				(I-K)\partial_bS_j&=b^{-1}K(S_j+d_jT_{j+1}),&
				(I-K)\partial_bR_i&=b^{-1}KR_i-\mathscr H_0^{-1}\partial_bF_i,\\*
				(I-K)\partial_\nu S_j&=\beta b\mathscr H_0^{-1}(S_j+d_jT_{j+1}),&
				(I-K)\partial_\nu R_i&=\beta b\mathscr H_0^{-1}R_i,\\*
				(I-K)\partial_\beta S_j&=\beta^{-1}K(S_j+d_jT_{j+1})+K\bigl((\partial_\beta d_j)T_{j+1}\bigr),&
				(I-K)\partial_\beta R_i&=\beta^{-1}KR_i-\mathscr H_0^{-1}\partial_\beta F_i.
			\end{align*}
			The map $b^{-1}K$ sends $X^{p,q}$ to $X^{p,q+4/3}$ with bounded norm, while $\beta b\mathscr H_0^{-1}$ has norm $O(b)$ between these spaces and norm $O(\chi)$ on $X^{p,q}$. Applying $(I-K)^{-1}$ on the indicated target spaces proves \eqref{prop12-S-bounds}--\eqref{prop12-R-bounds}. Uniform contraction on a slightly larger common interval, followed by restriction, justifies $C^1$ parameter dependence despite the moving endpoint $\zeta_0$.
			
			Near zero, \eqref{R-bound-corrected} makes the forcing in \eqref{equa-R_i} $O(\zeta)$. Formula \eqref{inverse-H-0} gives $R_i=O(\zeta^{5/2})$, $R_i'=O(\zeta^{3/2})$, and \eqref{equa-R_i} gives $R_i''=O(\zeta^{1/2})$. The same argument for the differentiated equations proves \eqref{impro-beha-R_i}. To obtain three Euler derivatives, write each equation as $g''=Q^{-1/2}(\beta b\mathscr D_{i,\nu}g+f)-Q^{-3/2}g$ and differentiate once. The new right-hand side involves at most two derivatives of $g$ and one derivative of the already controlled forcing. The coefficient bounds from Lemma~\ref{lemma-Q} and $b\zeta^{4/3}\lesssim\chi$ for $\zeta\ge1$ preserve the same weights. This applies also to the parameter equations above.
		\end{proof}
		
		\subsection{Outer expansion}\label{sub-outer expansion}
		By \eqref{spectral-beta-scaling}, take $\beta=1$. Put $h_i=\frac12-i$, $p_i=\frac23+\frac43i$, $\phi_i=\phi_{1,i}$ from \eqref{defi-phi-beta-i}, and $\mathcal A_i=\mathscr L_{0,1}-h_i$. The constants \eqref{defi-C_i} give $\phi_i=\sum_{j=0}^ic_{i,j}\mathsf C_j\xi^{2/3+4j/3}$; substitution, or the Laguerre equation, verifies $\mathcal A_i\phi_i=0$. A second solution $\check\phi_i$ has $\mathcal W[\phi_i,\check\phi_i]=e^{P_0}$, where $P_0=a^2\xi^{4/3}/8$, and
		\begin{equation*}
			\check\phi_i=-\frac3{A_i^0}\xi^{1/3}+O(\xi^{2/3})\quad(\xi\to0),\qquad
			\check\phi_i=\frac6{a^2A_i^\infty}e^{P_0}\xi^{-p_i-1/3}(1+O(\xi^{-4/3}))\quad(\xi\to\infty),
		\end{equation*}
		with $A_i^0=c_{i,0}\mathsf C_0$ and $A_i^\infty=c_{i,i}\mathsf C_i$. Fix $\check\phi_i$ independently of $b,\nu,\xi_0$, for instance by setting its value to zero at a fixed point where $\phi_i\ne0$ and fixing its derivative by the Wronskian. Reduction of order gives the endpoint expansions; continuation by the regular ODE defines $\check\phi_i$ across the zeros of $\phi_i$. The integral of $e^{P_0}/\phi_i^2$ is not taken across a zero.
		
		The right inverse selecting polynomial growth at infinity is
		\begin{equation}\label{resol-L-ext}
			G_if(\xi)=-\phi_i(\xi)\int_1^\xi\frac{\check\phi_i(t)f(t)}{at^{2/3}e^{P_0(t)}}\,dt
			-\check\phi_i(\xi)\int_\xi^\infty\frac{\phi_i(t)f(t)}{at^{2/3}e^{P_0(t)}}\,dt.
		\end{equation}
		In particular, $\widetilde\phi_i=G_i\phi_i$ satisfies $\mathcal A_i\widetilde\phi_i=\phi_i$. Set
		$I_i:=\int_0^\infty\phi_i^2/(a\xi^{2/3}e^{P_0})\,d\xi=1/a$, using the normalization of $\phi_i$. Then
		\begin{equation*}
			\widetilde\phi_i=\frac{3I_i}{A_i^0}\xi^{1/3}+O(\xi^{2/3})\quad(\xi\to0),\qquad
			\widetilde\phi_i=-\frac43A_i^\infty\xi^{p_i}\log\xi+O(\xi^{p_i})\quad(\xi\to\infty).
		\end{equation*}
		These expansions and their derivatives follow directly from \eqref{resol-L-ext}.
		
		For $q>p_i$, let $Z_q$ consist of continuous functions and $Y_q$ of $C^2$ functions on $[\xi_0,\infty)$ with the respective finite norms
		\begin{align*}
			\|f\|_{Z_q}&=\sup_{\xi_0\le\xi\le1}\xi^{-1/3}|f|+\sup_{\xi\ge1}\xi^{-q}|f|,\\*
			\|g\|_{Y_q}&=\sup_{\xi_0\le\xi\le1}\xi^{-1/3}(|g|+|\xi g'|+|\xi^2g''|)
			+\sup_{\xi\ge1}\xi^{-q}(|g|+|\xi g'|+|\xi^{2/3}g''|).
		\end{align*}
		\begin{lemma}[Outer inverse]\label{lem-continuity-L-ext-inverse}
			For $q>p_i$, $G_i$ satisfies $\mathcal A_iG_if=f$ and $\|G_if\|_{Y_q}\lesssim\|f\|_{Z_q}$, uniformly for $0<\xi_0\le1$.
		\end{lemma}
		\begin{proof}
			Let $M=\|f\|_{Z_q}$ and denote the two integrals in \eqref{resol-L-ext} by $I_1(\xi)$ and $I_2(\xi)$, in that order. For $\xi\ge1$, the fundamental-solution bounds and $q>p_i$ give
			\begin{equation*}
				|I_1(\xi)|\lesssim M\int_1^\xi t^{q-p_i-1}\,dt\lesssim M\xi^{q-p_i},\qquad
				|I_2(\xi)|\lesssim M e^{-P_0(\xi)}\xi^{p_i+q-1}.
			\end{equation*}
			Writing $g=G_if$, the endpoint terms cancel in $g'=-\phi_i'I_1-\check\phi_i'I_2$. Since $|\phi_i'|\lesssim\xi^{p_i-1}$ and $|\check\phi_i'|\lesssim e^{P_0}\xi^{-p_i}$, it follows that $|g|+|\xi g'|\lesssim M\xi^q$. Differentiating once more, the Wronskian gives exactly $\mathcal A_i g=f$, hence $|g''|\lesssim M\xi^{q-2/3}$. On $[\xi_0,1]$, the endpoint expansions give bounded $I_1,I_2$, uniformly in $\xi_0$, and therefore $|g|+|\xi g'|\lesssim M\xi^{1/3}$; the equation gives $|\xi^2g''|\lesssim M\xi^{1/3}$. These are the asserted $Y_q$ bounds. Only continuity and the weighted bound on $f$ have been used, so the term $L_bg$ causes no derivative loss.
		\end{proof}
		
		\begin{proposition}[Outer eigenfunction]\label{propo-outer-eigenfunction}
			Fix $0<\xi_0\le1$, $0<\epsilon<1/3$, and put $q=p_i+\epsilon$.
			There exist $\nu_{\rm out}=\nu_{\rm out}(\ell,\epsilon)\in(0,1)$ and
			$b_{\rm out}=b_{\rm out}(\ell,\epsilon,\xi_0)\in(0,1)$ such that, for
			$|\nu|\le\nu_{\rm out}$ and $0<b\le b_{\rm out}$, there is a solution of
			$(\mathscr L_{b,1}-h_i-\nu)\phi_{i,\mathrm{out}}=0$ on $[\xi_0,\infty)$ with
			\begin{equation}\label{phi-out-expansion}
				\phi_{i,\mathrm{out}}=\phi_i+\nu(\widetilde\phi_i+\widetilde R_{i,1})+\widetilde R_{i,2},\qquad
				\|\widetilde R_{i,1}\|_{Y_q}\lesssim|\nu|,\quad
				\|\widetilde R_{i,2}\|_{Y_q}\lesssim_{\xi_0}\sqrt b.
			\end{equation}
			The solution is $C^1$ in $(b,\nu)$ with values in $Y_q$. Moreover, $\partial_b\widetilde R_{i,1}=0$, $\|\partial_\nu\widetilde R_{i,1}\|_{Y_q}\lesssim1$, and
			$\|\partial_\nu\widetilde R_{i,2}\|_{Y_q}+\|b\partial_b\widetilde R_{i,2}\|_{Y_q}\lesssim_{\xi_0}\sqrt b$.
		\end{proposition}
		\begin{proof}
			Write $L_b=\mathscr L_{b,1}-\mathscr L_{0,1}=(\sqrt{Q_b}-a\xi^{2/3})\partial_\xi^2+(Q_b^{-1}-a^{-2}\xi^{-4/3})$. Lemma~\ref{lemma-Q} gives, for $k=0,1$,
			\begin{equation*}
				|(b\partial_b)^k(\sqrt{Q_b}-a\xi^{2/3})|+|(b\partial_b)^k(Q_b^{-1}-a^{-2}\xi^{-4/3})|
				\lesssim_{\xi_0}\sqrt b,\qquad \xi\ge\xi_0.
			\end{equation*}
			For $d_b=\sqrt{Q_b}-a\xi^{2/3}$ and $e_b=Q_b^{-1}-a^{-2}\xi^{-4/3}$, the differentiated expansions also give, for $0\le m\le3$, $k=0,1$, and $\xi\ge1$,
			\begin{equation}\label{outer-coefficient-symbol-bounds}
				|\partial_\xi^m(b\partial_b)^kd_b|\lesssim\sqrt b\,\xi^{-m},\qquad
				|\partial_\xi^m(b\partial_b)^ke_b|\lesssim\sqrt b\,\xi^{-2-m}.
			\end{equation}
			On $[\xi_0,1]$ these derivatives are bounded by $C_{\xi_0}\sqrt b$. In particular, $\|L_bg\|_{Z_q}+\|(b\partial_bL_b)g\|_{Z_q}\lesssim_{\xi_0}\sqrt b\|g\|_{Y_q}$. We solve
			\begin{align}
				\widetilde R_{i,1}&=\nu G_i(\widetilde\phi_i+\widetilde R_{i,1}),\label{equa-R-1-tilde}\\*
				\widetilde R_{i,2}&=G_i\bigl[(\nu-L_b)\widetilde R_{i,2}-L_b\{\phi_i+\nu(\widetilde\phi_i+\widetilde R_{i,1})\}\bigr].\label{equa-R-2-tilde}
			\end{align}
			Lemma~\ref{lem-continuity-L-ext-inverse} bounds the two operator norms by
			$C|\nu|$ and $C|\nu|+C_{\xi_0}\sqrt b$, respectively, with constants
			depending also on $\ell,\epsilon$. Choose $\nu_{\rm out}$ so that
			$C\nu_{\rm out}\le1/4$, then $b_{\rm out}$ so that
			$C_{\xi_0}\sqrt{b_{\rm out}}\le1/4$.
			Both maps are contractions, giving the fixed points and their differentiated estimates. To verify the eigenvalue equation explicitly, put $\mathcal E_i^{\rm out}=\mathcal A_i+L_b-\nu$ and use $\mathcal A_i\widetilde\phi_i=\phi_i$:
			\begin{align*}
				\mathcal E_i^{\rm out}\phi_{i,\mathrm{out}}
				&={}\nu[\mathcal A_i\widetilde R_{i,1}-\nu(\widetilde\phi_i+\widetilde R_{i,1})]\\*
				&\quad{}+(\mathcal A_i-\nu)\widetilde R_{i,2}+L_b\phi_{i,\mathrm{out}}=0,
			\end{align*}
			where the two cancellations are exactly \eqref{equa-R-1-tilde}--\eqref{equa-R-2-tilde} after applying $\mathcal A_i$.
			
			We need sharper spatial derivatives at infinity. Write $E=\phi_{i,\mathrm{out}}-\phi_i$ and $\kappa=|\nu|+\sqrt b$. The preceding bounds give $|E|+|\xi E'|\lesssim\kappa\xi^q$. Its equation is $E''-P_b'E'=H$, where
			\begin{equation*}
				H=Q_b^{-1/2}\bigl[(h_i+\nu-1-Q_b^{-1})E+(\nu-L_b)\phi_i\bigr].
			\end{equation*}
			Equation \eqref{outer-coefficient-symbol-bounds} gives $|H|\lesssim\kappa\xi^{q-2/3}$ and $|H'|\lesssim\kappa\xi^{q-5/3}$. Put $\mathcal J_bH=e^{P_b(\xi)}\int_\xi^\infty e^{-P_b(t)}H(t)\,dt$. Since $P_b'\asymp\xi^{1/3}$, a source bounded by $M\xi^\gamma$ has $|\mathcal J_bH|\lesssim_\gamma M\xi^{\gamma-1/3}$ on $\xi\ge1$. The polynomial bound on $E'$ implies $E'=-\mathcal J_bH$. Integration by parts gives
			\begin{equation*}
				\mathcal J_bH=\frac{H}{P_b'}+\mathcal J_b\left(\frac{H}{P_b'}\right)',\qquad
				E''=-P_b'\mathcal J_b J,\qquad J:=\left(\frac{H}{P_b'}\right)'.
			\end{equation*}
			Thus $J=O(\kappa\xi^{q-2})$ and $E''=O(\kappa\xi^{q-2})$. Substituting this improved estimate into $H''$ gives $H''=O(\kappa\xi^{q-8/3})$ and $J'=O(\kappa\xi^{q-3})$. A second integration by parts yields
			\begin{equation*}
				E'''=-P_b''\mathcal J_bJ-(P_b')^2\mathcal J_b\left(\frac J{P_b'}\right)'=O(\kappa\xi^{q-3}).
			\end{equation*}
			For parameter derivatives, with $\mathcal E_i^{\rm out}=\mathscr L_{b,1}-h_i-\nu$, differentiation gives
			\begin{equation*}
				\mathcal E_i^{\rm out}\partial_\nu E=\phi_i+E,\qquad
				\mathcal E_i^{\rm out}(b\partial_bE)=-(b\partial_bL_b)(\phi_i+E).
			\end{equation*}
			Their $Y_q$ bounds are already available. The preceding argument for the second derivative applies to each equation: differentiating its source once uses at most $E'''$, just proved. Consequently, for $m\le2$ and $\xi\ge1$,
			\begin{equation}\label{outer-sharp-derivatives}
				|D_\xi^mE|\lesssim\kappa\xi^q,\qquad
				|D_\xi^m\partial_\nu E|\lesssim\xi^q,\qquad
				|D_\xi^m b\partial_bE|\lesssim\sqrt b\,\xi^q.
			\end{equation}
			Here $P_b$ is the weight potential from \eqref{defi-P-b}, at $\beta=1$.
		\end{proof}
		
		\subsection{Matching and self-adjoint realization}\label{subsection_matching_asymptotic}
		Continue with $\beta=1$, fix $\epsilon=1/6$ in the outer construction, and put $U(b,\nu;\xi)=\phi_{i,\mathrm{int}}(\xi/b^{3/4})$. Choose $\xi_0$ so small that every $\phi_i$, $0\le i\le N$, is nonzero on $(0,2\xi_0]$. The matching equation is the scalar determinant
		\begin{equation}\label{matching-condition}
			M_i(b,\nu)=U'(\xi_0)\phi_{i,\mathrm{out}}(\xi_0)-U(\xi_0)\phi_{i,\mathrm{out}}'(\xi_0)=0.
		\end{equation}
		The leading sum in \eqref{phi-i-int} converges to $\phi_i$ in $C^1$ near $\xi_0$, with error $O_{\xi_0}(b^{1/4})$, by \eqref{asym-T}. Also, $\partial_\nu U=O(\xi_0^2)+O_{\xi_0}(b^{1/4})$ and $\partial_\nu U'=O(\xi_0)+O_{\xi_0}(b^{1/4})$. The Wronskian identity
		\begin{equation*}
			\phi_i'\widetilde\phi_i-\phi_i\widetilde\phi_i'
			=e^{P_0(\xi)}\int_\xi^\infty\frac{\phi_i(t)^2}{at^{2/3}e^{P_0(t)}}\,dt
			=I_i+O(\xi^{4/3})
		\end{equation*}
		therefore makes the derivative of the matching determinant nonzero. For the $b$ derivative at fixed $\xi$, the leading terms satisfy the exact cancellation
		\begin{equation}\label{leading-mode-b-derivative}
			b\partial_b\left[b^{j+1/2}T_j(\xi/b^{3/4})\right]
			=b^{j+1/2}\Theta_j(\xi/b^{3/4}).
		\end{equation}
		At $\xi=\xi_0$, this is $O_{\xi_0}(b^{1/4})$, also after one spatial derivative. The $S_j$ terms contribute $O_{\xi_0}(|\nu|)$; the bounds for $R_i$ and $b\partial_bR_i$ contribute $O_{\xi_0}(b^{1/4})$. Thus \eqref{prop12-S-bounds}--\eqref{prop12-R-bounds} and the outer estimates give
		\begin{equation}\label{matching-estimates}
			\begin{aligned}
				M_i(b,0)&=O_{\xi_0}(b^{1/4}),&
				\partial_\nu M_i(b,\nu)&=I_i+O(\xi_0^{4/3})+O_{\xi_0}(|\nu|+b^{1/4}),\\
				b\partial_bM_i(b,\nu)&=O_{\xi_0}(b^{1/4}+|\nu|).&&
			\end{aligned}
		\end{equation}
		Choose $\xi_0\le\xi_{\rm int}$ first. Since there are only finitely many indices $i\le N$ and $I_i>0$, one can then fix $\delta\in(0,\min\{\nu_{\rm int},\nu_{\rm out}\})$ and a common upper bound on $b\le\min\{\xi_0^{4/3},b_{\rm out}\}$ such that $\partial_\nu M_i\ge I_i/2$ for $|\nu|\le\delta$. Decrease that bound on $b$ until $|M_i(b,0)|<I_i\delta/4$. Then $M_i(b,-\delta)<0<M_i(b,\delta)$, so there is exactly one root $\nu_i(b)$ in this interval, with $|\nu_i(b)|\le2|M_i(b,0)|/I_i\lesssim b^{1/4}$. The implicit function theorem and \eqref{matching-estimates} give
		\begin{equation*}
			b\nu_i'(b)=-\frac{b\partial_bM_i(b,\nu_i(b))}{\partial_\nu M_i(b,\nu_i(b))}=O(b^{1/4}).
		\end{equation*}
		
		At this zero, set $A_i(b)=U(\xi_0)/\phi_{i,\mathrm{out}}(\xi_0)=1+O_{\xi_0}(b^{1/4})$ and define
		\begin{equation}\label{defi-varphi}
			\varphi_{i,b,1}(\xi)=
			\begin{cases}U(b,\nu_i(b);\xi),&0\le\xi\le\xi_0,\\
				A_i(b)\phi_{i,\mathrm{out}}(b,\nu_i(b);\xi),&\xi\ge\xi_0.
			\end{cases}
		\end{equation}
		On the inner piece, the change of variables gives $\mathscr L_{b,1}\varphi=b^{-1}\mathcal L_{b,1}\phi_{i,\mathrm{int}}=(h_i+\nu_i)\varphi$ by \eqref{eigen-equa-inner}; on the outer piece the verified outer equation gives the same identity. Equation \eqref{matching-condition} makes both the value and derivative continuous, so the ODE extends it smoothly across $\xi_0$:
		\begin{equation}\label{equa-eigen-}
			\mathscr L_{b,1}\varphi_{i,b,1}=(h_i+\nu_i(b))\varphi_{i,b,1}\quad\hbox{on }(0,\infty).
		\end{equation}
		The denominator defining $A_i$ stays bounded away from zero. Differentiating its quotient along $\nu_i(b)$, using \eqref{leading-mode-b-derivative}, the outer parameter bounds, and $|\nu_i|+|b\nu_i'|\lesssim b^{1/4}$, gives $bA_i'(b)=O(b^{1/4})$. The fixed points and the implicit function theorem give $C^1$ dependence on $b$ in $C^2$ on compact intervals away from zero. At the fixed interface this follows also for parameter derivatives, since the matched values and first derivatives agree for every $b$ and the ODE determines the second derivative. All corrections in \eqref{phi-i-int} have zero derivative at the wall. Since $T_0'(0)=\cscale/4$ and $T_j'(0)=0$ for $j\ge1$, the resulting eigenfunction has precisely the normalization \eqref{spectral-wall-normalization}.
		
		\medskip\noindent\emph{Pointwise bounds.}
		Write $\widetilde\varphi=\varphi_{i,b,1}-\varphi_i^{\rm lead}$, where $\varphi_i^{\rm lead}(\xi)=\sum_{j=0}^ic_{i,j}b^{j+1/2}T_j(\xi/b^{3/4})$. We specify which estimates control each region; all the following bounds hold after $m\le2$ Euler derivatives.
		
		For $0<\xi\le b^{3/4}$, use \eqref{asym-T} for $T_{j+1}$, \eqref{prop12-S-bounds} for $S_j$, and \eqref{impro-beha-R_i} for $R_i$ in \eqref{phi-i-int}. Since $\zeta=\xi/b^{3/4}\le1$ and $b^{3/2}\zeta^2=\xi^2$, they give
		\begin{equation*}
			|D_\xi^m\widetilde\varphi|\lesssim |\nu_i|\xi^2+b^{3/2}\zeta^{5/2}\lesssim\xi^2.
		\end{equation*}
		For $b^{3/4}\le\xi\le\xi_0$, the large-$\zeta$ bounds in \eqref{asym-T}, \eqref{prop12-S-bounds}, and \eqref{prop12-R-bounds} give, respectively,
		\begin{equation*}
			|D_\xi^m\widetilde\varphi|
			\lesssim |\nu_i|(1+\xi_0^{4/3})\sum_{j=0}^i\xi^{2+4j/3}+b^{1/4}\xi^{5/3}
			\lesssim b^{1/4}\xi^{2/3},
		\end{equation*}
		using $|\nu_i|\lesssim b^{1/4}$ from \eqref{matching-estimates} and $\xi_0<1$.
		For $\xi\ge\xi_0$, decompose \eqref{defi-varphi} as
		\begin{equation*}
			\widetilde\varphi=(A_i-1)\phi_{i,\mathrm{out}}+(\phi_{i,\mathrm{out}}-\phi_i)+(\phi_i-\varphi_i^{\rm lead}).
		\end{equation*}
		The first term uses $|A_i-1|\lesssim b^{1/4}$ and Proposition~\ref{propo-outer-eigenfunction}; the second uses \eqref{phi-out-expansion} on $[\xi_0,1]$ and \eqref{outer-sharp-derivatives} on $[1,\infty)$. For the last, \eqref{asym-T} gives $|D_\xi^m(\phi_i-\varphi_i^{\rm lead})|\lesssim b^{1/4}\sum_{j=0}^i\xi^{1/3+4j/3}$. Each term is therefore bounded by $C_{\xi_0}b^{1/4}\xi^{2/3}\langle\xi\rangle^{4/3+4i/3}$.
		
		For the $b$ derivative, subtract the leading sum first and differentiate its remainder along the matched branch. On a function of $(b,\nu,\zeta)$, the required operator is
		\begin{equation}\label{matched-branch-b-derivative}
			b\frac{d}{db}\bigg|_\xi
			=b\partial_b\bigg|_{\nu,\zeta}+b\nu_i'(b)\partial_\nu-\frac34D_\zeta.
		\end{equation}
		For $\zeta\le1$, \eqref{prop12-S-bounds}--\eqref{impro-beha-R_i} bound the differentiated remainder by
		$C(|\nu_i|+|b\nu_i'|)\xi^2+Cb^{3/2}\zeta^{5/2}\lesssim\xi^2$.
		For $1\le\zeta\le\zeta_0$, use $b\zeta^{4/3}\lesssim\xi_0^{4/3}$ to absorb the extra weight in $b\partial_bS_j$, $b\partial_bR_i$, and $\partial_\nu R_i$. Their total contribution is bounded by
		$C(|\nu_i|+|b\nu_i'|)\sum_{j=0}^i\xi^{2+4j/3}+Cb^{1/4}\xi^{5/3}\lesssim b^{1/4}\xi^{2/3}$.
		These estimates hold after two Euler derivatives, using the three-derivative inner bounds for the last term of \eqref{matched-branch-b-derivative}. On the outer region use $|bA_i'|+|b\nu_i'|\lesssim b^{1/4}$ and \eqref{outer-sharp-derivatives}. Combining the three regions yields
		\begin{equation}\label{prof-estimate-tilde-varphi}
			|D_\xi^m\widetilde\varphi_{i,b,1}|+|D_\xi^m b\partial_b\widetilde\varphi_{i,b,1}|
			\lesssim\xi^2\mathbf1_{\{\xi\le b^{3/4}\}}+
			b^{1/4}\xi^{2/3}\langle\xi\rangle^{4/3+4i/3}\mathbf1_{\{\xi\ge b^{3/4}\}}.
		\end{equation}
		The third Euler derivative of $\widetilde\varphi$ obeys the same endpoint bounds, by the inner estimates and the bound for $E'''$ in the outer proof. Finally, \eqref{spectral-beta-scaling} gives
		\begin{equation*}
			\partial_\beta\widetilde\varphi=\beta^{-1}\left(\frac18+b\partial_b+\frac34D_\xi\right)\widetilde\varphi,\qquad
			\widetilde\mu_{i,b,\beta}=\beta\nu_i(\beta b),\qquad
			\partial_\beta\widetilde\mu_{i,b,\beta}=\nu_i(\beta b)+\beta b\nu_i'(\beta b).
		\end{equation*}
		Together with \eqref{prof-estimate-tilde-varphi}, these identities give all the parameter bounds in Proposition~\ref{diagonalisation-L-b}, uniformly on the prescribed compact $\beta$-interval.
		
		The weighted norms converge to those of the limiting eigenfunctions. Indeed, the integral formula for $Q$ implies
		$r\le\int_0^{Q(r)}(8\sqrt t)^{-1/2}\,dt=(\sqrt2/3)Q(r)^{3/4}$, hence $Q_b(\xi)\ge a^2\xi^{4/3}$. For $0<b\le1$ and $\xi\ge1$, the opposite comparison up to a constant follows from Lemma~\ref{lemma-Q}. Therefore, at $\beta=1$,
		\begin{equation*}
			\rho_{b,1}(\xi)\lesssim\xi^{-2/3}
			\begin{cases}1,&0<\xi\le1,\\e^{-c\xi^{4/3}},&\xi\ge1,\end{cases}
			\qquad
			|\varphi_{i,b,1}(\xi)|\lesssim_N\xi^{2/3}\langle\xi\rangle^{4/3+4N/3}.
		\end{equation*}
		These bounds give an integrable majorant for $|\varphi_{i,b,1}|^2\rho_{b,1}$. Since $\varphi_{i,b,1}\to\phi_{1,i}$ and $\rho_{b,1}\to\rho_{0,1}$ locally away from zero, dominated convergence gives $\|\varphi_{i,b,1}\|_{L^2_{b,1}}\to1$. The unitary scaling \eqref{spectral-beta-scaling} makes this convergence uniform in $\beta$ on the fixed compact interval.
		
		\medskip\noindent\emph{Self-adjoint realization and compact resolvent.}
		Fix $b>0$ and $\beta>0$. With $\rho=\rho_{b,\beta}$ from \eqref{defi-rho-b-beta}, write $p=\rho\sqrt{Q_b}=ae^{-P_b}$ and $V=\beta+Q_b^{-1}$. The divergence form gives, for $f,g\in C_c^\infty(0,\infty)$,
		\begin{equation*}
			\langle-\mathscr L_{b,\beta}f,g\rangle_{L^2_\rho}
			=\mathfrak q[f,g]:=\int_0^\infty p f'\overline{g'}\,d\xi-\int_0^\infty V f\overline g\,\rho\,d\xi.
		\end{equation*}
		Thus the operator on $C_c^\infty(0,\infty)$ is densely defined and symmetric. Near zero, $Q_b\asymp_b\xi$, $p\asymp_b1$, and $V\rho\lesssim_b\xi^{-3/2}$. Since $f(0)=0$, $|f(\xi)|^2\le\xi\int_0^\xi|f'|^2$ implies
		$\int_0^{r_*} \xi^{-3/2}|f|^2\,d\xi\le2\sqrt{r_*}\int_0^{r_*}|f'|^2\,d\xi$.
		On $[{r_*},\infty)$, $V$ is bounded. Hence, for every $\epsilon>0$,
		\begin{equation*}
			\int_0^\infty V|f|^2\rho\,d\xi\le\epsilon\int_0^\infty p|f'|^2\,d\xi+C_{\epsilon,b,\beta}\|f\|_{L^2_\rho}^2.
		\end{equation*}
		The form is therefore bounded below, and its closure has domain $H^1_{0;b,\beta}$ from \eqref{definition-H1-zero-b-beta}; the shifted form norm is equivalent to \eqref{definition-H1-norm-b-beta}. The representation theorem defines the self-adjoint Friedrichs extension $A_{b,\beta}^{\rm F}$ and the required realization $\mathscr L_{b,\beta}=-A_{b,\beta}^{\rm F}$.
		
		For compactness, use the Liouville coordinate $t=\int_0^\xi Q_b(s)^{-1/4}\,ds$ and the unitary map $\mathcal U f(t)=r_b(\xi)f(\xi)$, where $r_b=(p\rho)^{1/4}=\sqrt a\,Q_b^{-1/8}e^{-P_b/2}$. On compactly supported smooth functions, direct differentiation gives
		\begin{equation*}
			\mathcal U(-\mathscr L_{b,\beta})\mathcal U^{-1}
			=-\partial_t^2+W_{b,\beta}(t),\qquad
			W_{b,\beta}=\frac{\partial_t^2r_b}{r_b}-V.
		\end{equation*}
		Using $(Q_b')^2=b^{1/2}+8\sqrt{Q_b}$ and $Q_b''=2/\sqrt{Q_b}$, this potential is
		\begin{equation}\label{liouville-potential-finite-b}
			W_{b,\beta}(t(\xi))
			=\frac{9\beta^2\xi^2}{64\sqrt{Q_b}}
			+\frac{3\beta\xi Q_b'}{16Q_b}
			-\frac{11\beta}{8}-\frac{3}{8Q_b}
			+\frac{7b^{1/2}}{64Q_b^{3/2}}.
		\end{equation}
		Lemma~\ref{lemma-Q} gives $t\sim(4/3)b^{-1/16}\xi^{3/4}$ at zero and $t\sim3\xi^{2/3}/(2\sqrt a)$ at infinity. Consequently,
		\begin{equation*}
			W_{b,\beta}(t)=\frac7{36t^2}+O_{b,\beta}(t^{-4/3})\quad(t\downarrow0),\qquad
			W_{b,\beta}(t)\sim\frac{\beta^2t^2}{16}\quad(t\to\infty).
		\end{equation*}
		In particular, $W_{b,\beta}(t)\ge c_{b,\beta}t^2-C_{b,\beta}$ on $(0,\infty)$. The transformed form, shifted by a sufficiently large multiple of the norm, therefore controls $\int_0^\infty(|F'|^2+(1+t^2)|F|^2)\,dt$. This holds on the closed form domain by density. Its unit ball has uniformly small $L^2$ tails, and Rellich compactness applies on each bounded interval. Thus the form-domain embedding is compact and so is the resolvent.
		
		The endpoint estimates give $\varphi(\xi)=O_b(\xi)$ and $\varphi'(\xi)=O_b(1)$ at zero, and polynomial growth at infinity. Cutting off at $\xi\sim r$ costs $O_{b,\beta}(r)$ in the squared gradient norm, while the cutoff cost at infinity tends to zero by the exponential weight. Smooth approximation away from the endpoints therefore gives $\varphi\in H^1_{0;b,\beta}$. Integrating \eqref{equa-eigen-} and its scaled version against compactly supported tests, with cancellation of the interface terms, and extending by density gives $\mathfrak q[\varphi,g]=-\mu_{i,b,\beta}\langle\varphi,g\rangle_{L^2_\rho}$ for every form-domain test $g$. The representation theorem yields $\varphi\in D(\mathscr L_{b,\beta})$ with eigenvalue $\mu_{i,b,\beta}$.
		
		\subsection{Ordered spectrum and spectral gap}\label{subsection-spectral-ordering}
		We identify the eigenvalues constructed for $0\le i\le N=\ell+1$ with the first ordered eigenvalues. By \eqref{spectral-beta-scaling}, it suffices to count zeros for $\beta=1$.
		
		\medskip\noindent\emph{Absence of zeros near the wall.}
		Put $\zeta=\xi/b^{3/4}$. Positivity and the endpoint expansions in Lemma~\ref{lemma-T-i} imply
		\begin{equation*}
			\frac{T_j(\zeta)}{T_0(\zeta)}\lesssim_j
			\begin{cases}\zeta^{3j/2},&0<\zeta\le1,\\
				\zeta^{4j/3},&\zeta\ge1.
			\end{cases}
		\end{equation*}
		Moreover, $c_{i,0}>0$ and $b^{1/2}T_0(\xi/b^{3/4})\asymp b^{-1/4}\xi$ for $\xi\le b^{3/4}$, and $\asymp\xi^{2/3}$ for $\xi\ge b^{3/4}$. Thus \eqref{eigenfuc-L-b} and the already proved remainder estimate \eqref{prof-estimate-tilde-varphi} give, uniformly for $0\le i\le N$ and $b^{3/4}\le r_0<1$,
		\begin{equation}\label{no-wall-zeros-finite-b}
			\sup_{0<\xi\le r_0}
			\left|\frac{\varphi_{i,b,1}(\xi)}{c_{i,0}b^{1/2}T_0(\xi/b^{3/4})}-1\right|
			\le C_N\bigl(r_0^{4/3}+b^{1/4}\bigr).
		\end{equation}
		Indeed, on $\xi\le b^{3/4}$ the terms with $j\ge1$ and the remainder have relative size $O_N(b)$; on $b^{3/4}\le\xi\le r_0$ their relative sizes are $O_N(\xi^{4/3})$ and $O_N(b^{1/4})$. Choose $r_0$ small and then $b$ small so that the right-hand side of \eqref{no-wall-zeros-finite-b} is less than $1/2$. None of these eigenfunctions vanishes on $(0,r_0]$.
		
		\medskip\noindent\emph{Absence of zeros in the far field.}
		Since $Q_b'>0$ and $Q_b\asymp\xi^{4/3}$ uniformly for $\xi\ge1$, $0<b\le1$, formula \eqref{liouville-potential-finite-b} implies
		$W_{b,\beta}(t(\xi))\ge c\xi^{4/3}-C$, uniformly for $\beta\in[C_0^{-1},C_0]$.
		The constructed eigenvalues are bounded for $i\le N$. Hence there is $R>1$, independent of small $b$, such that $W_{b,1}+\mu_{i,b,1}\ge1$ on $\xi\ge R$ for every $i\le N$. If $\varphi_{i,b,1}(\xi_1)=0$ for some $\xi_1\ge R$, let $F=\mathcal U\varphi_{i,b,1}$ and $t_1=t(\xi_1)$. The function $F\mathbf1_{[t_1,\infty)}$ belongs to the transformed form domain, since $F(t_1)=0$. Testing $-F''+W_{b,1}F=-\mu_{i,b,1}F$ with this function gives
		\begin{equation*}
			\int_{t_1}^\infty\left(|F'|^2+(W_{b,1}+\mu_{i,b,1})|F|^2\right)dt=0.
		\end{equation*}
		This forces $F$ to vanish on the tail and then everywhere by uniqueness for the ODE, a contradiction. Thus there are no zeros on $[R,\infty)$.
		
		\medskip\noindent\emph{Zeros on a compact interval.}
		The differentiated expansions \eqref{asym-T} and \eqref{prof-estimate-tilde-varphi} imply
		$\varphi_{i,b,1}\to\phi_{1,i}$ in $C^1([r_0,R])$ as $b\downarrow0$.
		By \eqref{defi-phi-beta-i}, $\phi_{1,i}$ is a positive power of $\xi$ times $L_i^{(1/4)}$, and therefore has exactly $i$ simple positive zeros. Decrease $r_0$ and increase $R$ so that all these zeros, for $i\le N$, lie in $(r_0,R)$. On a sufficiently small neighborhood of each zero, the derivative of $\phi_{1,i}$ has a fixed nonzero sign; on the remaining compact set, $|\phi_{1,i}|$ is bounded below. The $C^1$ convergence therefore gives exactly one zero in each such neighborhood and none elsewhere. Together with the endpoint arguments, this proves that $\varphi_{i,b,1}$ has exactly $i$ zeros on $(0,\infty)$.
		
		\medskip\noindent\emph{Identification of the spectral indices.}
		We verify the endpoint conditions needed for the nodal count. Fix $b,\beta>0$ and a spectral parameter $\mu$. With $p=\rho\sqrt{Q_b}$ and $V=\beta+Q_b^{-1}$, the solution of $\mathscr L_{b,\beta}f=\mu f$ normalized by $f(0)=0$, $f'(0)=1$ solves
		\begin{equation*}
			f(\xi)=p(0)\int_0^\xi\frac{dt}{p(t)}
			+\int_0^\xi\frac1{p(t)}\int_0^t(\mu-V(s))\rho(s)f(s)\,ds\,dt.
		\end{equation*}
		On $0<\xi\le r$, the integral operator has norm $O_{b,\beta,\mu}(r^{1/2})$ in $\sup|f(\xi)|/\xi$. Hence it gives a unique solution $f_D=\xi(1+O(\xi^{1/2}))$, $f_D'=1+O(\xi^{1/2})$. Reduction of order yields an independent solution $f_N=1+O(\xi^{1/2})$. Since $p$ is bounded above and below near zero, the form-domain condition imposes zero trace and selects $f_D$. Thus the weighted Wronskian vanishes at zero for any two eigenfunctions. At infinity, the transformed form controls $F,F'\in L^2$ for every transformed eigenfunction. For two such functions, both their Wronskian and its derivative are integrable on the tail, by Cauchy--Schwarz and the eigenvalue equations. Its limit is therefore zero. For equal eigenvalues, constancy of the Wronskian implies linear dependence; all eigenvalues are therefore simple. Every eigenfunction has finitely many zeros: the local expansion excludes zeros near zero, the tail test above applies to any fixed eigenvalue, and zeros cannot accumulate in a compact interior interval by uniqueness for the ODE.
		
		List the eigenvalues of $A_{b,\beta}^{\rm F}$ as $\lambda_0<\lambda_1<\cdots$. Sturm comparison, in the form of \cite[Theorem~9.39]{Teschl2014}, shows that an eigenfunction for $\lambda_{n+1}$ has a zero in each nodal interval of an eigenfunction for $\lambda_n$. Induction gives at least $n$ zeros for the latter. Conversely, if it has $m$ zeros, its restrictions to the $m+1$ nodal intervals, extended by zero, belong to the form domain. Their disjoint supports and the eigenvalue equation give Rayleigh quotient $\lambda_n$ for every nonzero linear combination. The min--max principle gives $\lambda_m\le\lambda_n$, hence $m\le n$. Thus the eigenfunction for $\lambda_n$ has exactly $n$ zeros. The zero count above identifies the constructed branch as $\mu_{i,b,\beta}=-\lambda_i$ for $i\le N$; completeness follows from the compact self-adjoint resolvent. This proves \eqref{ordered-spectrum-finite-b} and excludes any additional eigenvalues above $\mu_{N,b,\beta}$.
		
		Finally, the construction at index $N=\ell+1$ gives
		$\mu_{\ell,b,\beta}-\mu_{\ell+1,b,\beta}=\beta+O_{\ell,C_0}(b^{1/4})$.
		Since $\beta\ge C_0^{-1}$, take $b^*$ so that both $b$ and $\beta b$ lie below the construction thresholds, and small enough that the last error is at most $1/(2C_0)$. This proves \eqref{next-eigenvalue-finite-b} for the single parameter range stated in the proposition. Expand $f$ in the normalized eigenbasis. If its first $\ell+1$ coefficients vanish, the spectral representation of the closed form gives
		$-\mathfrak q[f,f]\le\mu_{\ell+1,b,\beta}\|f\|_{L^2_{b,\beta}}^2$ for every such $f\in H^1_{0;b,\beta}$. This is \eqref{finite-b-spectral-gap} and completes the proof of Proposition~\ref{diagonalisation-L-b}.
		
		\section{Pointwise bounds for the nonlinear forcing}
		We use $\mathfrak f_j$ from \eqref{defi-f-ell}, $\mathcal F$ from \eqref{defi-mathcal-F}, and, for $j\ge1$,
		\begin{equation}\label{definition-Z-general-ell}
			Z_j(s,\xi):=\left[b^{1/2}\langle\xi\rangle^{-8/3}
			+\bigl(b^{3/4}+A^2b^{1/2+\eta}\bigr)\xi^{-4/3}\right]\mathfrak f_j(\xi),\qquad b=b(s).
		\end{equation}
		
		\subsection{Pointwise estimates under the bounds of \texorpdfstring{$\mathcal V_1$}{V1}}
		Assume $(\vartheta,b,\beta)(s)\in\mathcal V_1[A,\eta](s)$ on $[s_0,\bar s]$, with \eqref{equa-vartheta}. Use \eqref{decomposition-stable-sketch} and \eqref{exact-stable-finite-mode}: $\vartheta=Q_b+\varepsilon$, $\varepsilon=\varepsilon_++\varepsilon_-$, $\varepsilon_+=\varepsilon_1\psi_{1,b,\beta}$, and $\varepsilon_1=\mathbf n_1(\beta)b^{1/2}$. With the operators and remainder from \eqref{defi-mathscr-L-}, \eqref{defi-L_0_beta}, and \eqref{defi-B-quadratic-appendix-}, set
		\begin{align}
			\widetilde{\mathcal B}_1(\varepsilon)&:=(\mathscr L_{b,\beta}-\mathscr L_{0,\beta})\varepsilon+\mathscr B_b(\varepsilon),
			\label{definition-B-one-stable-corrected}\\*
			\mathcal R_{1,b}(\varepsilon_+)&:=-\partial_s\varepsilon_++\mathscr L_{b,\beta}\varepsilon_++(\beta-b'/b)\Lambda_\xi Q_b.
			\label{definition-R-one-b-stable}
		\end{align}
		
		\begin{lemma}[Stable forcing estimates]\label{lem-esti-tilde-R-1-ep-plus-stable}
			For every $A\ge1$, there exist $\eta_{18}(A)\in(0,1/16]$ and
			$s_{18}=s_{18}(A,\eta)\ge s_8(\max\{A,A_8\},\eta)$,
			independent of $\bar s$, such that, for $0<\eta\le\eta_{18}(A)$
			and $s_0\ge s_{18}$, the following bounds hold on $[s_0,\bar s]$:
			\begin{gather}
				\textup{(i)}\quad |\widetilde{\mathcal B}_1(\varepsilon)(s,\xi)|\lesssim_{A,\eta}
				\begin{cases}
					b^{1/4+\eta}\xi^{-1/3},&0<\xi\le b^{3/4},\\
					(b^{1/2}\xi^{-2/3}+Z_1)(|\mathcal F|+Z_1),&b^{3/4}\le\xi\le b^{-\eta},\\
					\xi^{2/\eta},&\xi\ge b^{-\eta}.
				\end{cases}\notag\\*
				\textup{(ii)}\quad |\mathcal R_{1,b}(\varepsilon_+)|\lesssim
				b^{1/2}\xi\mathbf1_{\{\xi\le b^{3/4}\}}
				+b^{3/4}\mathfrak f_1\mathbf1_{\{\xi\ge b^{3/4}\}}.
				\label{global-estimate-R-one-b-stable}
			\end{gather}
		\end{lemma}
		
		\begin{proof}
			Take $\eta_{18}(A)\le\min\{1/16,\eta_8(\max\{A,A_8\})\}$.
			Proposition~\ref{pro-finite-projections} then applies with $A$ replaced by
			$\max\{A,A_8\}$; all further lower bounds on $s_0$ are absorbed into $s_{18}$.
			Put $q=Q_b$, $p=a\xi^{2/3}$, $z=\varepsilon/q$, and $r=\sqrt{1+z}$.
			Since $q''=2/\sqrt q$, the residual identity gives
			$\varepsilon''=[\mathcal F-2(r-1)]/(\sqrt q\,r)$, hence
			\begin{equation}\label{exact-identity-B-one-stable}
				\widetilde{\mathcal B}_1(\varepsilon)=\left(1-\frac p{\sqrt q\,r}\right)\mathcal F
				+z\left[\frac{2p}{\sqrt q\,r(1+r)}-\frac q{p^2}\right].
			\end{equation}
			
			\emph{For $0<\xi\le b^{3/4}$}, Lemma~\ref{lemma-Q}, the eigenfunction bounds, and $\mathcal V_1$ give $q\asymp b^{1/4}\xi$, $|\varepsilon_+|\lesssim b^{3/4}\xi$, $|\varepsilon_-|\le A^2b^{1/4+\eta}\xi$, and $|\mathcal F|\lesssim_A b^{1/4-\eta}\xi$. Hence $|z|\lesssim_A b^\eta\le1/2$, $p/\sqrt q\lesssim1$, and $q/p^2\lesssim b^{1/4}\xi^{-1/3}$. Equation \eqref{exact-identity-B-one-stable} gives
			\begin{equation*}
				|\widetilde{\mathcal B}_1(\varepsilon)|\lesssim_A
				b^{1/4-\eta}\xi+b^\eta+b^{1/4+\eta}\xi^{-1/3}
				\lesssim_A b^{1/4+\eta}\xi^{-1/3},
			\end{equation*}
			using $\eta<1/2$ and $\xi\le b^{3/4}$.
			
			\emph{For $b^{3/4}\le\xi\le b^{-\eta}$}, the profile estimates and the bounds $|\varepsilon_+|\lesssim b^{1/2}\xi^2+b^{3/4}\mathfrak f_1$, $|\varepsilon_-|\le A^2b^{1/2+\eta}\mathfrak f_1$ give
			\begin{equation}\label{intermediate-profile-relative-error}
				q\asymp\xi^{4/3},\quad
				\left|\frac p{\sqrt q}-1\right|+\left|\frac q{p^2}-1\right|
				\lesssim b^{1/2}\xi^{-2/3},\quad
				|z|\lesssim_A Z_1\lesssim_A b^\eta.
			\end{equation}
			Here \eqref{definition-Z-general-ell} and $\eta\le1/10$ give the bound on $Z_1$;
			Increasing $s_{18}$ ensures $|z|\le1/2$.
			Both coefficients in \eqref{exact-identity-B-one-stable} are therefore
			$O(b^{1/2}\xi^{-2/3}+|z|)$, proving the middle bound in (i).
			
			\emph{For $\xi\ge b^{-\eta}$}, $q\asymp\xi^{4/3}$, while the shrinking-set bounds give
			$b^{-\eta/2}\lesssim_A\vartheta\lesssim_A b^{-1-\eta}$ and
			$|\mathcal F|\lesssim_A b^{-\eta/8}$. We rewrite \eqref{exact-identity-B-one-stable} as
			\begin{equation*}
				\widetilde{\mathcal B}_1(\varepsilon)
				=\left(1-\frac p{\sqrt\vartheta}\right)\mathcal F
				+2p\left(\frac1{\sqrt q}-\frac1{\sqrt\vartheta}\right)-\frac\varepsilon{p^2}.
			\end{equation*}
			These bounds and $b^{-\eta}\le\xi$ control the three terms by $C_A(\xi^{1/8}+\xi^{2/3})$, $C_A(1+\xi^{2/3})$, and $C_A(1+\xi^{1/\eta-1/3})$, respectively. Their sum is $O_{A,\eta}(\xi^{2/\eta})$, proving (i).
			
			\emph{Finite-dimensional forcing.} Write $\varphi_j=\varphi_{j,b,\beta}$, $\psi=\psi_{1,b,\beta}$, $\kappa=c_{1,0}/c_{0,0}$, and $\alpha=\beta-b'/b$. Proposition~\ref{diagonalisation-L-b}, the $T_0$ cancellation in \eqref{definition-adapted-eigendirection-strategy}, and \eqref{ode-beta-b} yield
			\begin{align}
				b^{1/4}|\varphi_j|+|\partial_s\psi|
				&\lesssim\xi\mathbf1_{\{\xi\le b^{3/4}\}}
				+b^{1/4}\mathfrak f_1\mathbf1_{\{\xi\ge b^{3/4}\}},\quad j=0,1,
				\label{pointwise-eigenfunction-estimates-R-final}\\*
				\left|\varepsilon_1'+\tfrac\beta2\varepsilon_1\right|
				+\left|\varepsilon_1'-\tfrac{b'}{2b}\varepsilon_1\right|&\lesssim b^{3/4}.
				\label{epsilon-one-modulation-errors}
			\end{align}
			The last bound follows by differentiating $\varepsilon_1=\mathbf n_1(\beta)b^{1/2}$. Since $\mathbf m_0(\beta)=\kappa\mathbf n_1(\beta)/2$,
			\begin{align}
				\partial_s\varepsilon_+&=-\tfrac\beta2\varepsilon_1(\varphi_1+\kappa\varphi_0)
				+\alpha\mathbf m_0(\beta)b^{1/2}\varphi_0+\mathcal E_s,
				\label{time-expansion-epsilon-plus-final}\\*
				\mathscr L_{b,\beta}\varepsilon_+&=-\tfrac\beta2\varepsilon_1(\varphi_1+\kappa\varphi_0)+\mathcal E_\mu,
				\qquad \Lambda_\xi Q_b=\mathbf m_0(\beta)b^{1/2}\varphi_0+\mathcal E_Q.
				\label{linear-expansion-epsilon-plus-final}
			\end{align}
			Here $\mathcal E_s=(\varepsilon_1'+\frac\beta2\varepsilon_1)\varphi_1-\kappa(\varepsilon_1'-\frac{b'}{2b}\varepsilon_1)\varphi_0+\varepsilon_1\partial_s\psi$. Bounds \eqref{pointwise-eigenfunction-estimates-R-final}--\eqref{epsilon-one-modulation-errors},
			the $O(b^{1/4})$ eigenvalue errors, and the remainder estimate for $\varphi_0$
			bound each error by the right-hand side of \eqref{global-estimate-R-one-b-stable}.
			Substitution into \eqref{definition-R-one-b-stable} gives $\mathcal R_{1,b}=-\mathcal E_s+\mathcal E_\mu+\alpha\mathcal E_Q$; since $|\alpha|\lesssim1$, (ii) follows.
		\end{proof}
		
		\subsection{Pointwise estimates under the bounds of \texorpdfstring{$\mathcal V_\ell$, $\ell\ge2$}{Vell, ell >= 2}}
		Assume $(\vartheta,b)(s)\in\mathcal V_\ell[A,\eta](s)$ on $[s_0,\bar s]$, with $\beta\equiv1$ and the notation \eqref{notation-beta-one}. As in \eqref{decom-varepvarphi-i-ell}, $\varepsilon=\varepsilon_++\varepsilon_-$, where
		\begin{equation}\label{definition-epsilon-plus-general-ell}
			\varepsilon_+=\varepsilon_\ell\psi_{\ell,b}+\sum_{j=1}^{\ell-1}\varepsilon_j\varphi_{j,b}.
		\end{equation}
		Use \eqref{definition-B-one-stable-corrected} for $\widetilde{\mathcal B}_\ell$ and \eqref{definition-R-one-b-stable} with $\beta=1$ for $\mathcal R_{\ell,b}$.
		
		\begin{lemma}[Unstable forcing estimates]\label{lem-unstable-forcing-estimates}
			For every integer $\ell\ge2$ and $A\ge1$, there exist
			$\eta_{19}(A,\ell)\in(0,1/(32\ell)]$ and
			$s_{19}=s_{19}(A,\eta,\ell)\ge1$, independent of $\bar s$,
			such that, for $0<\eta\le\eta_{19}(A,\ell)$ and $s_0\ge s_{19}$,
			the nonlinear bound in Lemma~\ref{lem-esti-tilde-R-1-ep-plus-stable}(i) holds with $\widetilde{\mathcal B}_\ell,Z_\ell$ in place of $\widetilde{\mathcal B}_1,Z_1$, with constants depending also on $\ell$, and
			\begin{equation}\label{global-estimate-R-general-ell}
				|\mathcal R_{\ell,b}(\varepsilon_+)|\lesssim_{A,\ell}
				(b^{1/2}+Ab^{1/4+\eta})\xi\mathbf1_{\{\xi\le b^{3/4}\}}
				+(b^{3/4}\mathfrak f_\ell+Ab^{3/4+\eta}\xi^{1/3})\mathbf1_{\{\xi\ge b^{3/4}\}}.
			\end{equation}
		\end{lemma}
		
		\begin{proof}
			The identity \eqref{exact-identity-B-one-stable} is unchanged.
			The mode bounds \eqref{compatibility-condi-epsilon-ell} and $\mathcal V_\ell$ give
			$|\varepsilon/Q_b|\lesssim_{A,\ell}b^\eta$ for $\xi\le b^{3/4}$ and
			$|\varepsilon/Q_b|\lesssim_{A,\ell}Z_\ell$ for $b^{3/4}\le\xi\le b^{-\eta}$.
			By \eqref{definition-Z-general-ell}, $Z_\ell\lesssim_{A,\ell}b^\eta$ on
			$[b^{3/4},1]$ and $Z_\ell\lesssim_{A,\ell}b^{1/2-(4\ell-1)\eta/3}=o(1)$ on
			$[1,b^{-\eta}]$, since $\eta\le\eta_{19}(A,\ell)\le1/(32\ell)$.
			Increasing $s_{19}$ makes these relative errors at most $1/2$.
			The inner and intermediate arguments above apply; Definition~\ref{defi-shinkingset-V-ell} gives the same far-field residual and lower bounds, and the stronger upper bound \eqref{far-field-vartheta-V-ell} for $\vartheta$.
			
			For the finite-dimensional term, set $\kappa_\ell=c_{\ell,0}/c_{0,0}$ and $\alpha=1-b'/b$. The projection argument of Section~\ref{subsection-proof-item-i-ell}
			applies for this choice of $\eta_{19}$ after increasing $s_{19}$.
			Thus \eqref{dynamics-b-item-i} gives $\alpha=2\ell+O_{A,\ell}(b^{2\eta})$. The expansion of $\Lambda_\xi Q_b$ in \eqref{linear-expansion-epsilon-plus-final} holds with $\beta=1$, and $\mathcal E_Q$ obeys the right-hand side of \eqref{global-estimate-R-general-ell}. The pointwise estimates used for $0\le j\le\ell$ are
			\begin{align*}
				b^{1/4}|\varphi_{j,b}|+|\partial_s\psi_{\ell,b}|
				&\lesssim_\ell\xi\mathbf1_{\{\xi\le b^{3/4}\}}+b^{1/4}\mathfrak f_\ell\mathbf1_{\{\xi\ge b^{3/4}\}},\\*
				|\partial_s\varphi_{j,b}|&\lesssim_\ell b^{-1/4}\xi\mathbf1_{\{\xi\le b^{3/4}\}}+b^{1/4}(\xi^{1/3}+\mathfrak f_\ell)\mathbf1_{\{\xi\ge b^{3/4}\}}.
			\end{align*}
			Differentiating \eqref{definition-epsilon-plus-general-ell} gives
			\begin{equation*}
				\begin{aligned}
					\mathcal R_{\ell,b}(\varepsilon_+)
					&={}\sum_{j=1}^{\ell}(-\varepsilon_j'+\mu_{j,b}\varepsilon_j)\varphi_{j,b}
					+[\kappa_\ell(\varepsilon_\ell'-\mu_{0,b}\varepsilon_\ell)+\alpha\mathbf m_0b^{1/2}]\varphi_{0,b}\\
					&\quad{}-\varepsilon_\ell\partial_s\psi_{\ell,b}-\sum_{j=1}^{\ell-1}\varepsilon_j\partial_s\varphi_{j,b}+\alpha\mathcal E_Q.
				\end{aligned}
			\end{equation*}
			By \eqref{equation-b-item-i} and \eqref{dynamics-modes-item-i},
			\begin{equation*}
				\max_{1\le j\le\ell}|-\varepsilon_j'+\mu_{j,b}\varepsilon_j|
				+|\kappa_\ell(\varepsilon_\ell'-\mu_{0,b}\varepsilon_\ell)+\alpha\mathbf m_0b^{1/2}|
				\lesssim_{A,\ell}b^{3/4}.
			\end{equation*}
			The $T_0$ term in $\partial_s\varphi_{j,b}$ gives the additional inner factor $b^{-1/4}$ and outer power $\xi^{1/3}$; it cancels only in the adapted direction. Combining $|\varepsilon_\ell|\lesssim_\ell b^{1/2}$ and $|\varepsilon_j|\le Ab^{1/2+\eta}$ for $j<\ell$ with the preceding bounds proves \eqref{global-estimate-R-general-ell}. In particular, $|\mathcal R_{\ell,b}|\lesssim_{A,\ell}b^{3/4}\mathfrak f_\ell$ for $\xi\ge1$, which is the bound used in the outer comparison.
		\end{proof}
		
		\section{Spectral analysis of the limiting operator and its semigroup}\label{section-spectral-L_0-beta}
		We study $\mathscr L_{0,\beta}$ from \eqref{defi-L_0_beta} in $L^2_{\rho_{0,\beta}}$, with the density and constant $a$ in \eqref{defi-rho-infty} and \eqref{notation-constant-a}. The change of variables $z=a\xi^{2/3}$ reduces the operator to a Laguerre-type Sturm--Liouville problem. Since $a^3=9/2$, we have $(a\xi^{2/3}\rho_{0,\beta})'=-\frac{3\beta}{4}\xi\rho_{0,\beta}$, hence
		\begin{equation*}
			\mathscr L_{0,\beta}q=\rho_{0,\beta}^{-1}(a\xi^{2/3}\rho_{0,\beta}q')'
			+\bigl(a^{-2}\xi^{-4/3}+\beta\bigr)q.
		\end{equation*}
		For $z=a\xi^{2/3}$ and $q(\xi)=\widetilde q(z)$,
		\begin{equation}\label{transf-xi-z}
			\partial_\xi q=\sqrt{2/z}\,\partial_z\widetilde q,\qquad
			\partial_\xi^2q=2z^{-1}\partial_z^2\widetilde q-z^{-2}\partial_z\widetilde q.
		\end{equation}
		After normalization, this change of variables gives the unitary equivalence $\mathscr L_{0,\beta}=\mathscr U(2\mathcal L+\beta)\mathscr U^{-1}$, where $\mathcal L$ acts in $L^2_{\widetilde\rho_\beta}$ and
		\begin{equation*}
			\mathcal L=\partial_z^2-\left(\frac1{2z}+\frac{\beta z}4\right)\partial_z+\frac1{2z^2},\qquad
			\widetilde\rho_\beta(z)=z^{-1/2}e^{-\beta z^2/8}.
		\end{equation*}
		
		\begin{lemma}\label{lemma-spectrum-L-z}
			For $\beta>0$, the Friedrichs realization $\mathcal A$ of $-\mathcal L$, initially defined on $C_c^\infty(0,\infty)$, has compact resolvent, simple eigenvalues $\lambda_n=\beta/4+\beta n/2$, $n\ge0$, and the orthonormal eigenbasis
			\begin{equation*}
				\varphi_n(z)=\widetilde{\mathcal N}_n z L_n^{(1/4)}(\beta z^2/8),\qquad
				\widetilde{\mathcal N}_n=\sqrt2(\beta/8)^{5/8}\sqrt{n!/\Gamma(n+5/4)}.
			\end{equation*}
			Here $L_n^{(1/4)}$ is the generalized Laguerre polynomial of degree $n$.
		\end{lemma}
		\begin{proof}
			For $f,g\in C_c^\infty(0,\infty)$, integration by parts gives
			\begin{equation*}
				\bigl\langle(-\mathcal L-\beta/4)f,g\bigr\rangle_{L^2_{\widetilde\rho_\beta}}
				=\int_0^\infty z^2(f/z)'(g/z)'\widetilde\rho_\beta\,dz.
			\end{equation*}
			Thus $-\mathcal L$ is symmetric and bounded below by $\beta/4$, so its Friedrichs realization is self-adjoint. Setting $\zeta=\beta z^2/8$ and $\varphi(z)=z g(\zeta)$, we compute
			\begin{equation*}
				-\mathcal L\varphi=\frac{\beta z}{2}\left[-\zeta g''-(5/4-\zeta)g'+\frac12g\right].
			\end{equation*}
			The Laguerre equation $\zeta g''+(5/4-\zeta)g'+ng=0$ therefore gives the stated eigenfunctions and positive eigenvalues of $\mathcal A$. These functions belong to the closed form domain: $\varphi_n/z$ is smooth at zero, where $z^2\widetilde\rho_\beta\sim z^{3/2}$ makes the cutoff cost vanish, and Gaussian decay controls their polynomial tails. Integration by parts against compactly supported tests, followed by form-domain density, then places $\varphi_n$ in $D(\mathcal A)$ with $\mathcal A\varphi_n=\lambda_n\varphi_n$.
			
			Laguerre orthogonality and the same change of variables give
			\begin{equation*}
				\|\varphi_n\|_{L^2_{\widetilde\rho_\beta}}^2
				=\frac{\widetilde{\mathcal N}_n^2}{2}(8/\beta)^{5/4}\frac{\Gamma(n+5/4)}{n!}=1.
			\end{equation*}
			Completeness of the Laguerre polynomials in $L^2(\zeta^{1/4}e^{-\zeta}\,d\zeta)$ transfers to $(\varphi_n)$ under this change of variables and multiplication by $z$. Hence the displayed eigenvalues exhaust the spectrum; they are distinct and tend to infinity, proving simplicity and compactness of the resolvent.
		\end{proof}
		
		We transfer this spectral description to the original variable.
		\begin{lemma}\label{lemma-spectrum-L0-beta-xi}
			For each $\beta>0$, let $\mathcal A$ be the Friedrichs realization of $-\mathcal L$ in Lemma~\ref{lemma-spectrum-L-z}. The self-adjoint realization
			$\mathscr L_{0,\beta}=\mathscr U(\beta-2\mathcal A)\mathscr U^{-1}$, with domain $\mathscr U D(\mathcal A)$, has simple eigenvalues $\lambda_{\beta,n}=\beta(\frac12-n)$, $n\ge0$. The functions $\phi_{\beta,n}$ in \eqref{defi-phi-beta-i}, normalized by \eqref{defi-N-beta-i}, form an orthonormal basis of $L^2_{\rho_{0,\beta}}$.
		\end{lemma}
		\begin{proof}
			With $z=a\xi^{2/3}$, the measures satisfy $\rho_{0,\beta}(\xi)\,d\xi=\frac{3}{2\sqrt a}\widetilde\rho_\beta(z)\,dz$. Thus $(\mathscr U h)(\xi)=(2\sqrt a/3)^{1/2}h(a\xi^{2/3})$ is unitary, and \eqref{transf-xi-z} gives the stated conjugation. Lemma~\ref{lemma-spectrum-L-z} then yields the domain, eigenvalues and complete orthonormal family; applying $\mathscr U$ to its eigenfunctions gives exactly \eqref{defi-phi-beta-i}--\eqref{defi-N-beta-i}.
		\end{proof}
		
		\begin{remark}
			Since $a^2/8=9/(16a)$, the Laguerre expansion gives
			\begin{equation*}
				\phi_{\beta,i}(\xi)=\sum_{j=0}^i\beta^{5/8+j}c_{i,j}\mathsf C_j\xi^{2/3+4j/3},\qquad
				c_{i,j+1}=-(i-j)c_{i,j}\quad(0\le j<i),
			\end{equation*}
			with $c_{i,j}$ and $\mathsf C_j$ from \eqref{defi-c-i-j} and \eqref{defi-C_i}. In particular,
			\begin{alignat*}{2}
				\phi_{\beta,i}(\xi)&=\beta^{5/8}c_{i,0}\mathsf C_0\xi^{2/3}\bigl(1+O_i(\beta\xi^{4/3})\bigr)&\qquad & (\xi\to0),\\*
				\phi_{\beta,i}(\xi)&=\beta^{5/8+i}c_{i,i}\mathsf C_i\xi^{2/3+4i/3}\bigl(1+O_i(\beta^{-1}\xi^{-4/3})\bigr)&\qquad & (\xi\to\infty).
			\end{alignat*}
			Consequently, $|\phi_{\beta,i}(\xi)|\lesssim_{i,\beta}\xi^{2/3}\langle\xi\rangle^{4i/3}$, as needed for fixed spectral modes.
		\end{remark}
		
		We now estimate $S_\beta(t)$ from \eqref{defi-S-beta-0-t} directly for constant $\beta>0$. Set $r=e^{-\beta t}$, $\kappa=\beta a^2/8$ and $\omega=2\kappa\sqrt r\,(\xi\zeta)^{2/3}/(1-r)$. The Hille--Hardy formula sums the spectral series $\sum_{n\ge0}e^{\beta(1/2-n)t}\phi_{\beta,n}(\xi)\phi_{\beta,n}(\zeta)$ to give the positive kernel
		\begin{equation}\label{heat-kernel-L0-2}
			G_\beta(t,\xi,\zeta)
			=\frac{\beta a^2e^{\beta t/2}}{6(1-r)r^{1/8}}(\xi\zeta)^{1/2}
			e^{-\kappa r(\xi^{4/3}+\zeta^{4/3})/(1-r)}I_{1/4}(\omega).
		\end{equation}
		Thus $S_\beta(t)f(\xi)=\int_0^\infty G_\beta(t,\xi,\zeta)f(\zeta)\rho_{0,\beta}(\zeta)\,d\zeta$, with the density in \eqref{defi-rho-infty}.
		
		\begin{lemma}\label{lem-simigroup-L-infty-beta=2}
			Fix $\beta>0$ and $k\ge0$, and let $f$ be measurable with $|f(\xi)|\le A_0\xi^{2/3}\langle\xi\rangle^k$. Then, for $t>0$ and $\xi>0$,
			\begin{subequations}\label{semi-group-beta-0}
				\begin{align}
					|S_\beta(t)f(\xi)|&\lesssim_{k,\beta}A_0e^{\beta t/2}\xi^{2/3}\langle\xi\rangle^k,\label{esti-Semi-f-}\\*
					|S_\beta(t)\Lambda_\xi f(\xi)|&\lesssim_{k,\beta}\frac{A_0e^{\beta t/2}}{\sqrt{1-e^{-\beta t}}}\xi^{2/3}\langle\xi\rangle^{k+2/3}.\label{Semi-hit-Derivative}
				\end{align}
			\end{subequations}
			Here $\Lambda_\xi$ is defined in \eqref{Lambda:op}; for nonsmooth $f$, its action is interpreted distributionally through integration by parts against the kernel. The constants are uniform when $\beta$ ranges over a compact subset of $(0,\infty)$.
		\end{lemma}
		\begin{proof}
			We use the kernel argument of \cite[Lemma~C1]{BDGJDE2025}, keeping $\beta$ arbitrary. For $\nu=1/4,5/4$, the standard Bessel estimates \cite{Watson1922} are
			\begin{equation}\label{beha-I-alpha-z-large}
				I_\nu(\omega)\asymp\omega^\nu\quad(0<\omega\le1),\qquad
				I_\nu(\omega)=\frac{e^\omega}{\sqrt{2\pi\omega}}\bigl(1+O(\omega^{-1})\bigr)\quad(\omega\ge1).
			\end{equation}
			Write $y=(\zeta^{2/3}-\sqrt r\,\xi^{2/3})/\sqrt{1-r}$. Combining \eqref{heat-kernel-L0-2} with \eqref{beha-I-alpha-z-large} gives
			\begin{equation*}
				G_\beta\rho_{0,\beta}(\zeta)\lesssim_\beta e^{\beta t/2}
				\begin{cases}
					(1-r)^{-5/4}\xi^{2/3}e^{-\kappa(r\xi^{4/3}+\zeta^{4/3})/(1-r)},&\omega\le1,\\
					(1-r)^{-1/2}r^{-3/8}\xi^{1/6}\zeta^{-1/2}e^{-\kappa y^2},&\omega\ge1.
				\end{cases}
			\end{equation*}
			In the first region, the substitution $\zeta=(1-r)^{3/4}q$ and $0<1-r<1$ yield
			\begin{equation}\label{small-argument-kernel-moment}
				\int_{\omega\le1}G_\beta\zeta^{2/3}\langle\zeta\rangle^k
				\left(1+\frac{\zeta^{4/3}}{1-r}\right)\rho_{0,\beta}(\zeta)\,d\zeta
				\lesssim_{k,\beta}e^{\beta t/2}\xi^{2/3}.
			\end{equation}
			In the second region, put $X=\sqrt{1-r}/(\sqrt r\,\xi^{2/3})$. Then $y\ge X/(2\kappa)-1/X$ and $\langle\zeta\rangle^k\lesssim_k\langle\xi\rangle^k\langle y\rangle^{3k/2}$. Changing variables to $y$ gives, for $m=0,1,2$,
			\begin{equation}\label{large-argument-kernel-moments}
				\frac{\int_{\omega\ge1}G_\beta\zeta^{2/3}\langle\zeta\rangle^k\langle y\rangle^m\rho_{0,\beta}(\zeta)\,d\zeta}
				{e^{\beta t/2}\xi^{2/3}\langle\xi\rangle^k}
				\lesssim_{k,\beta}1+X^{3/4}\int_{X/(2\kappa)-1/X}^\infty\langle y\rangle^{3k/2+3/4+m}e^{-\kappa y^2}\,dy
				\lesssim_{k,\beta}1.
			\end{equation}
			Indeed, the right-hand side is bounded for bounded $X$; for $X\ge\max(1,2\sqrt\kappa)$, its lower integration limit is at least $X/(4\kappa)$, so Gaussian decay absorbs $X^{3/4}$. The case $m=0$ and \eqref{small-argument-kernel-moment} prove \eqref{esti-Semi-f-}.
			
			For \eqref{Semi-hit-Derivative}, first take $f\in C_c^\infty(0,\infty)$. Integration by parts and $I_{1/4}'=I_{5/4}+I_{1/4}/(4\omega)$ give
			\begin{equation*}
				\begin{aligned}
					S_\beta(t)\Lambda_\xi f(\xi)&=\int_0^\infty H_\beta(t,\xi,\zeta)f(\zeta)\rho_{0,\beta}(\zeta)\,d\zeta,\\
					H_\beta&=G_\beta\left(\frac74+\frac\omega2\frac{I_{5/4}(\omega)}{I_{1/4}(\omega)}-\frac{\kappa\zeta^{4/3}}{1-r}\right).
				\end{aligned}
			\end{equation*}
			By \eqref{beha-I-alpha-z-large}, the Bessel ratio is $O(\omega)$ for $\omega\le1$ and $1+O(\omega^{-1})$ for $\omega\ge1$. In the latter region, the leading terms cancel since $\omega/2-\kappa\zeta^{4/3}/(1-r)=-\kappa\zeta^{2/3}y/\sqrt{1-r}$. Hence
			\begin{equation*}
				|H_\beta|\lesssim_\beta G_\beta
				\begin{cases}
					1+\zeta^{4/3}/(1-r),&\omega\le1,\\
					1+y^2+(1-r)^{-1/2}\sqrt r\,\xi^{2/3}|y|,&\omega\ge1.
				\end{cases}
			\end{equation*}
			Applying \eqref{small-argument-kernel-moment} and \eqref{large-argument-kernel-moments} bounds the integral by $C_{k,\beta}A_0e^{\beta t/2}\xi^{2/3}\langle\xi\rangle^k(1+(1-r)^{-1/2}\sqrt r\,\xi^{2/3})$, which implies \eqref{Semi-hit-Derivative}. These integrable kernel bounds extend both formulas to measurable $f$ with the stated growth by cutoff approximation and dominated convergence. No boundary term remains: $\zeta G_\beta f\rho_{0,\beta}=O(\zeta^{5/3})$ at zero and decays exponentially at infinity. All estimates are uniform for $\beta$ in compact subsets of $(0,\infty)$.
		\end{proof}
		The following localized bounds control the forcing in the Duhamel argument.
		\begin{lemma}[Localized forcing estimates]\label{lem-semi-tle-K-0-xi-le-K}
			Fix $L_0>0$, an integer $\ell\ge1$, $K,A\ge1$, and $0<\eta\le1/16$.
			Let $Z_\ell$ and $\mathfrak f_\ell$ be defined by
			\eqref{definition-Z-general-ell} and \eqref{defi-f-ell}, and assume
			$\|\mathcal F\|_{\mathfrak X_{b,\eta}^{(\ell)}}\le Ab^{1/2-\eta}$,
			as in the stable shrinking set for $\ell=1$ and
			\eqref{esti-V-ell-F} for $\ell\ge2$. There exists
			$b_{\rm ker}=b_{\rm ker}(A,K,\ell,\eta,L_0)\in(0,1)$ such that,
			for $0<b\le b_{\rm ker}$, the following bounds hold. Set
			\begin{align*}
				\widehat B_\ell&={}b^{1/4+\eta}\xi^{-1/3}\mathbf1_{\{\xi\le b^{3/4}\}}
				+\xi^{2/\eta}\mathbf1_{\{\xi\ge b^{-\eta}\}}\\*
				&\quad{}+(b^{1/2}\xi^{-2/3}+Z_\ell)(|\mathcal F|+Z_\ell)
				\mathbf1_{\{b^{3/4}\le\xi\le b^{-\eta}\}},\\*
				\widehat R_\ell&={}b^{1/2}\xi\mathbf1_{\{\xi\le b^{3/4}\}}
				+b^{3/4}\mathfrak f_\ell\mathbf1_{\{\xi\ge b^{3/4}\}},\\*
				\widehat E&={}Ab^{1/4+\eta}\xi\mathbf1_{\{\xi\le b^{3/4}\}}
				+Ab^{3/4+\eta}\xi^{1/3}\mathbf1_{\{\xi\ge b^{3/4}\}}.
			\end{align*}
			Thus $\widehat R_\ell+\widehat E$ is the full finite-dimensional
			envelope needed for $\ell\ge2$.
			Uniformly for $0<t\le L_0$ and $b^\eta\le\xi\le K$,
			\begin{equation*}
				|S_1(t)\widehat B_\ell(\xi)|\lesssim_{A,K,\ell,\eta,L_0}
				b^{3/4-5\eta/3}\xi^{2/3},\qquad
				|S_1(t)\widehat R_\ell(\xi)|\lesssim_{K,\ell,L_0}b^{3/4}\xi^{2/3}.
			\end{equation*}
			The additional envelope satisfies
			\begin{equation}\label{semigroup-extra-unstable-envelope}
				S_1(t)\widehat E(\xi)
				\lesssim_{K,\eta,L_0}A b^{3/4+2\eta/3}\xi^{2/3}.
			\end{equation}
			For a fixed integer $N\ge0$, define
			$P_Nf=f-\sum_{j=0}^N\langle f,\phi_{1,j}\rangle_{\rho_{0,1}}\phi_{1,j}$,
			with the moments understood as absolutely convergent integrals when
			$f$ is singular at zero. If $|f|\le\widehat B_\ell$, then
			\begin{equation}\label{projected-localized-forcing-bounds}
				|S_1(t)P_Nf(\xi)|
				\lesssim_{A,K,\ell,\eta,L_0,N}b^{3/4-5\eta/3}\xi^{2/3}.
			\end{equation}
			If $|f|\le\widehat R_\ell+\widehat E$, the corresponding bound is
			$C_{K,\ell,\eta,L_0,N}(1+Ab^{2\eta/3})b^{3/4}\xi^{2/3}$.
		\end{lemma}
		\begin{proof}
			Let $G_1$ be the kernel in \eqref{heat-kernel-L0-2} at $\beta=1$,
			and put $\Theta=1-e^{-t}$ and $\kappa=2^{-3/2}e^{-3L_0/4}$.
			Choose $b_{\rm ker}$ so that $b^{3/4}\le\kappa b^\eta$ and
			$b^{-\eta}\ge2K$ for $b\le b_{\rm ker}$; any further smallness
			requirements are absorbed into $b_{\rm ker}$.
			The Bessel bounds give, for $\zeta\le\kappa\xi$ and $0<t\le L_0$,
			\begin{equation*}
				G_1(t,\xi,\zeta)\rho_{0,1}(\zeta)
				\lesssim_{L_0}\xi^{2/3}\Theta^{-5/4}
				e^{-c_{L_0}\xi^{4/3}/\Theta}.
			\end{equation*}
			Optimizing in $\Theta$ gives
			$\sup_{0<\Theta\le1}\Theta^{-5/4}e^{-c\xi^{4/3}/\Theta}\lesssim\xi^{-5/3}$.
			Hence, for nonnegative $h$ supported in $[0,\kappa b^\eta]$,
			\begin{equation}\label{low-support-semigroup-envelope}
				S_1(t)h(\xi)\lesssim_{L_0}\xi^{-1}\|h\|_{L^1},\qquad\xi\ge b^\eta.
			\end{equation}
			The same Gaussian kernel gives, for every $N>0$,
			\begin{equation}\label{far-support-semigroup-envelope}
				S_1(t)[\zeta^{2/\eta}\mathbf1_{\{\zeta\ge b^{-\eta}\}}](\xi)
				\lesssim_{N,\eta,K,L_0}b^N\xi^{2/3},\qquad0<\xi\le K.
			\end{equation}
			
			Denote the intermediate term of $\widehat B_\ell$ by $B_{\rm mid}$.
			On $b^{3/4}\le\zeta\le\kappa b^\eta$, expansion of $Z_\ell$ and the
			weaker bound $|\mathcal F|\le Ab^{1/4-\eta}\mathfrak f_\ell$ give
			\begin{equation*}
				B_{\rm mid}\lesssim_{A,\ell}
				b^{3/4-\eta}(1+\zeta^{4/3})+b^{1+\eta}\zeta^{-4/3}.
			\end{equation*}
			Its integral is $O_{A,\ell,L_0}(b^{3/4})$, as is that of the inner term
			$b^{1/4+\eta}\zeta^{-1/3}$. Their kernel images are therefore bounded by
			$Cb^{3/4}\xi^{-1}\le Cb^{3/4-5\eta/3}\xi^{2/3}$, by
			\eqref{low-support-semigroup-envelope}.
			
			On $\kappa b^\eta\le\zeta\le b^{-\eta}$, expand the two factors of
			$B_{\rm mid}$. Each product has a power of $\zeta$ in
			$\{-4/3,0,4/3\}$ and a bracket power at most $8(\ell+1)/3$.
			Converting its explicit power to $\zeta^{2/3}$ costs at most
			$b^{-\eta|q-2/3|}$ for a power $q$. The smallest resulting $b$ power is
			$3/4-5\eta/3$, attained by the terms with $\mathcal F$ and one factor
			$b^{1/2}\zeta^{\pm2/3}$; all other products gain at least $b^\eta$
			when $\eta\le1/16$. Thus
			\begin{equation*}
				B_{\rm mid}\mathbf1_{\{\zeta\ge\kappa b^\eta\}}
				\lesssim_{A,\ell,L_0}b^{3/4-5\eta/3}
				\zeta^{2/3}\langle\zeta\rangle^{8(\ell+1)/3}.
			\end{equation*}
			Positivity and \eqref{semi-group-beta-0} control its kernel image by the
			required bound on $\xi\le K$. Adding \eqref{far-support-semigroup-envelope}
			proves the first estimate. Finally,
			$\widehat R_\ell\lesssim b^{3/4}\mathfrak f_\ell$ everywhere
			(the inner inequality uses $\zeta^{1/3}\le b^{1/4}$);
			\eqref{semi-group-beta-0} proves the second.

			\emph{The lower-mode terms.} Split $\widehat E$ at $\zeta=\kappa b^\eta$.
			The lower part satisfies
			\[
			\|\widehat E\mathbf1_{\{\zeta\le\kappa b^\eta\}}\|_{L^1}
			\lesssim_{L_0}A(b^{7/4+\eta}+b^{3/4+7\eta/3}).
			\]
			By \eqref{low-support-semigroup-envelope} and $\xi\ge b^\eta$, its image is at most
			$C_{L_0}A(b^{7/4-2\eta/3}+b^{3/4+2\eta/3})\xi^{2/3}$.
			Since $1-4\eta/3>0$, this has the size in
			\eqref{semigroup-extra-unstable-envelope}. On $\zeta\ge\kappa b^\eta$,
			$\widehat E\le C_{L_0}Ab^{3/4+2\eta/3}\zeta^{2/3}$;
			\eqref{semi-group-beta-0} proves the same bound for the remaining part.
			
			\emph{Finite spectral projections.} For fixed $j$, the product
			$\phi_{1,j}\rho_{0,1}$ is bounded at zero and has Gaussian decay up to a
			polynomial factor at infinity. The lower-region integrals and the polynomial
			bound for $B_{\rm mid}$ above therefore give
			\[
			\int_0^\infty\widehat B_\ell|\phi_{1,j}|\rho_{0,1}\,d\zeta
			\lesssim_{A,\ell,\eta,L_0,j}b^{3/4-5\eta/3}.
			\]
			The Gaussian tail is smaller than every power of $b$. Likewise,
			$\widehat R_\ell\lesssim b^{3/4}\mathfrak f_\ell$ and direct integration of
			$\widehat E$ give moments bounded by
			$C_{\ell,j}(b^{3/4}+Ab^{3/4+\eta})$.
			Since $|\phi_{1,j}(\xi)|\le C_{j,K}\xi^{2/3}$ for $0<\xi\le K$,
			subtracting the finitely many terms
			$e^{(1/2-j)t}\langle f,\phi_{1,j}\rangle_{\rho_{0,1}}\phi_{1,j}$
			proves \eqref{projected-localized-forcing-bounds} and the stated bound for
			$\widehat R_\ell+\widehat E$. The moments are absolutely convergent;
			no $L^2$ assumption on $f$ is needed.
		\end{proof}
		
		We finish by justifying the Duhamel formulas when the source has the wall singularity allowed in Section~12.
		\begin{lemma}[Duhamel formula for singular sources]\label{lemma-weak-duhamel-singular-source}
			Write $\rho=\rho_{0,1}$, $\kappa=a^2/8$, and $\omega(\xi)=e^{-\kappa\xi^{4/3}}$. Set
			\begin{equation*}
				H=L^2_\rho,\qquad
				V=\overline{C_c^\infty(0,\infty)}^{\,\|\cdot\|_V},\qquad
				\|v\|_V^2=\|v\|_H^2+\int_0^\infty|v'|^2\omega\,d\xi,
			\end{equation*}
			and use $H$ as pivot space in $V\subset H\subset V'$. This is the form domain of the Friedrichs realization of $-\mathscr L_{0,1}$. Let $J=[s_*,s^\dagger]$ be finite and let $P$ be the orthogonal projection onto the complement of finitely many eigenfunctions $\phi_{1,j}$. Suppose $c\in L^\infty(J)$ and $g$ is measurable, with
			\begin{equation*}
				|g(\tau,\xi)|\le C_J\left(\xi^{-1/3}\mathbf1_{\{0<\xi<1\}}
				+\langle\xi\rangle^M\mathbf1_{\{\xi\ge1\}}\right),\qquad M\ge0.
			\end{equation*}
			Then $g\in L^\infty(J;V')$, $\Lambda_\xi:V\to V'$ is bounded, and $P$ extends to $V'$ by duality. If $u\in C(J;H)\cap L^\infty(J;V)$ satisfies
			$\partial_\tau u=\mathscr L_{0,1}u+P[c(\tau)\Lambda_\xi u+g]$ in $V'$, then, for $s\in J$,
			\begin{equation}\label{weak-duhamel-singular-source}
				u(s)=S_1(s-s_*)u(s_*)+
				\int_{s_*}^s S_1(s-\tau)P[c(\tau)\Lambda_\xi u(\tau)+g(\tau)]\,d\tau
			\end{equation}
			in $H$, using the extension of $S_1(t)$ to $V'$. If also $|u(\tau,\xi)|\le C_J\xi^{2/3}\langle\xi\rangle^M$, then for $s>s_*$ the right-hand side defines a continuous spatial
			representative of $u(s)$ on $(0,\infty)$. The formula holds for this
			representative, with the drift term interpreted by weighted integration
			by parts in the kernel. In particular, it holds pointwise for the
			classical solutions used in the preceding sections.
		\end{lemma}
		\begin{proof}
			For $v\in C_c^\infty(0,\infty)$, integration of $(|v|^2\omega/\xi)'$ and $\omega'\le0$ give the weighted Hardy inequality
			\begin{equation*}
				\int_0^\infty\frac{|v|^2}{\xi^2}\omega\,d\xi
				\le4\int_0^\infty|v'|^2\omega\,d\xi.
			\end{equation*}
			The quadratic form $\mathfrak a$ of $-\mathscr L_{0,1}$ satisfies
			\begin{align*}
				\mathfrak a[v]+2\|v\|_H^2
				&=a\int|v'|^2\omega-a^{-2}\int\xi^{-2}|v|^2\omega+\|v\|_H^2\\
				&\ge\tfrac a9\int|v'|^2\omega+\|v\|_H^2,
			\end{align*}
			because $4/a^3=8/9$. The reverse bound is immediate, so the shifted form norm is equivalent to $\|\cdot\|_V$ and its closure has domain $V$.
			
			For $0<r\le1$, the zero trace and Cauchy--Schwarz give $|v(\xi)|\lesssim\xi^{1/2}\|v\|_V$ on $(0,1)$. Hence
			\begin{equation*}
				\left|\int_0^r g(\tau,\xi)v(\xi)\rho(\xi)\,d\xi\right|
				\lesssim C_J\int_0^r\xi^{-1}|v(\xi)|\,d\xi
				\lesssim C_J\sqrt r\,\|v\|_V.
			\end{equation*}
			The polynomial tail of $g$ belongs to $H$. Thus $g\in L^\infty(J;V')$, and cutting it off near zero and infinity gives $H$-valued sources converging to $g$ in $L^\infty(J;V')$. Since $\phi_{1,j}\in V$, finite spectral projections are bounded on $V$ and extend to $V'$. Their coefficients for $g$ are the ordinary absolutely convergent moments: near zero, $|g\phi_{1,j}\rho|=O(\xi^{-1/3})$.
			
			Integration of $(\xi^{1/3}|v|^2\omega)'$ gives
			\begin{equation*}
				\tfrac{4\kappa}{3}\|\xi^{2/3}v\|_H^2
				=\tfrac13\|v\|_H^2+2\int\xi^{1/3}vv'\omega\,d\xi.
			\end{equation*}
			Cauchy--Schwarz and absorption imply $\|\xi^{2/3}v\|_H\lesssim\|v\|_V$. Consequently,
			\begin{equation*}
				|\langle\Lambda_\xi u,v\rangle|
				\le\|u\|_H\|v\|_H
				+\tfrac34\|u'\|_{L^2_\omega}\|\xi^{2/3}v\|_H
				\lesssim\|u\|_V\|v\|_V.
			\end{equation*}
			All these estimates extend by density.
			
			Set $F=P(c\Lambda_\xi u+g)\in L^\infty(J;V')$. The spectral description in Lemma~\ref{lemma-spectrum-L0-beta-xi}, together with the form-norm equivalence, gives
			\begin{equation*}
				\|S_1(t)\|_{V'\to H}\lesssim_{L_0}t^{-1/2},\qquad0<t\le L_0.
			\end{equation*}
			Indeed, the eigenvalues of $2\operatorname{Id}-\mathscr L_{0,1}$ are $3/2+j$, and $\sup_{\lambda\ge3/2}\lambda^{1/2}e^{-t\lambda}\lesssim t^{-1/2}$. Thus the integral in \eqref{weak-duhamel-singular-source} converges in $H$. Testing the weak equation against each eigenfunction gives the scalar variation-of-constants formula; completeness then proves \eqref{weak-duhamel-singular-source}. Testing a difference equation by the difference, using the shifted form bound above and Young's inequality, also gives
			\begin{equation*}
				\|z\|_{C(J;H)}+\|z\|_{L^2(J;V)}
				\lesssim_J\|z(s_*)\|_H+\|F_1-F_2\|_{L^2(J;V')}
			\end{equation*}
			for solutions with prescribed forcings $F_1,F_2$. This proves uniqueness and justifies passage from truncated sources to the weak evolution.
			
			For the pointwise assertion, choose spatial cutoffs $\chi_{r,R}$, equal to one on $[2r,R]$ and supported in $[r,2R]$. Use $\chi_{r,R}g$ and $\Lambda_\xi(\chi_{r,R}u)$ in the kernel formulas, and then let $r\downarrow0$, $R\uparrow\infty$. The first source converges in $L^\infty(J;V')$. The second converges in $L^2(J;V')$: the derivative-cutoff cost at zero is $O(r^{-2}\int_r^{2r}\xi^{4/3}\,d\xi)=O(r^{1/3})$, the tails vanish by Gaussian decay, and the remaining derivative terms converge by dominated convergence in $L^2(J;L^2_\omega)$. On each compact spatial interval away from zero, the kernel applied to the source envelope is bounded for $0<t\le|J|$, by the separated-region estimates used in Lemma~\ref{lem-semi-tle-K-0-xi-le-K}. The differentiated kernel applied to $u$ is bounded by $C_Jt^{-1/2}$ there, by \eqref{Semi-hit-Derivative}. These bounds hold for the cutoffs as well, including the finite spectral subtractions, and are integrable in time. Dominated convergence therefore gives the pointwise kernel formula; the preceding energy estimate identifies it with the $H$-valued formula. In particular, removing the cutoffs produces no additional wall term.
		\end{proof}
		
		\section{Jacobians for the stable modulation and initial normalization maps}\label{appen-jacobian}
		Write $\varphi_j=\varphi_{j,b,\beta}$, $\rho=\rho_{b,\beta}$, and
		$\psi=\psi_{1,b,\beta}$ as in \eqref{definition-adapted-eigendirection-strategy}.
		Put $\kappa=c_{1,0}/c_{0,0}$, $n=\mathbf n_1(\beta)$, and
		$m=\mathbf m_0(\beta)$; \eqref{defi-m-0} and \eqref{defi-n-ell-1}
		give $n\kappa=2m$ and $n'/n=-5/(8\beta)$. All estimates below are
		uniform for $\beta\in[1/2,3/2]$ as $b\to0$.
		
		At fixed $\xi$, differentiating \eqref{spectral-beta-scaling} and
		\eqref{eigenfuc-L-b} gives, for $j=0,1$ and $\zeta=\xi/b^{3/4}$,
		\begin{align*}
			\beta\partial_\beta\varphi_j
			&=\left(\frac18+b\partial_b+\frac34\xi\partial_\xi\right)\varphi_j,\\*
			b\partial_b\varphi_j
			&=\sum_{k=0}^j c_{j,k}\beta^{k+5/8}b^{k+1/2}
			\left[(k+\tfrac12)T_k(\zeta)-\tfrac34\zeta T_k'(\zeta)\right]
			+b\partial_b\widetilde\varphi_j.
		\end{align*}
		The bracket cancels the leading outer power in \eqref{asym-T}, and the
		$T_0$ term cancels in $\psi$. Since $\partial_bQ_b=b^{-1}\Lambda_\xi Q_b$,
		\eqref{notation-scaling-normalization} and Proposition~\ref{diagonalisation-L-b}
		therefore imply
		\begin{equation*}
			\|b\partial_b\psi\|_\rho
			+\|(\beta\partial_\beta-13/8)\psi\|_\rho
			+\|b^{-1/2}\Lambda_\xi Q_b-m\varphi_0\|_\rho
			\lesssim b^{1/4}.
		\end{equation*}
		The weight derivatives follow directly from \eqref{defi-P-b}--\eqref{defi-rho-b-beta}:
		\begin{equation*}
			\beta\partial_\beta\log\rho=-P_b,\qquad
			b\partial_b\log\rho=-\frac{\Lambda_\xi Q_b}{2Q_b}
			+\frac{3\beta}{8}\int_0^\xi tQ_b(t)^{-3/2}\Lambda_tQ_b(t)\,dt.
		\end{equation*}
		Together with \eqref{bound-D-b} and the spectral estimates, these give
		\begin{equation*}
			\|b\partial_b\varphi_j+\varphi_j b\partial_b\log\rho\|_\rho\lesssim b^{1/4},
			\qquad \|\beta\partial_\beta\varphi_j-P_b\varphi_j\|_\rho\lesssim1.
		\end{equation*}
		The integral is $O(b^{1/2}\xi^{2/3})$, and the weight decays exponentially.
		
		For fixed $\Theta$, set $R=\Theta-Q_b-n\sqrt b\,\psi$ and
		$G_j=\langle R,\varphi_j\rangle_\rho/\|\varphi_j\|_\rho^2$,
		ordered as $G=(G_1,G_0)$ in \eqref{defi-G-modulation}.
		For $p=b,\beta$, differentiation of the weighted projection gives
		\begin{equation*}
			\partial_pG_j=
			\frac{\langle\partial_pR,\varphi_j\rangle_\rho
				+\langle R,\partial_p\varphi_j+\varphi_j\partial_p\log\rho\rangle_\rho}
			{\|\varphi_j\|_\rho^2}
			-G_j\partial_p\log\|\varphi_j\|_\rho^2.
		\end{equation*}
		The normalization term vanishes whenever $G=0$. At $R=0$, the moving-basis
		and weight terms also vanish; the preceding estimates and orthogonality give
		\begin{equation*}
			D_{(b,\beta)}G
			\begin{pmatrix}\sqrt b&0\\0&b^{-1/2}\end{pmatrix}
			=\begin{pmatrix}-n/2&-n/\beta\\0&2m/\beta\end{pmatrix}
			+O(b^{1/4}).
		\end{equation*}
		At a trapped normalized profile, $R=\varepsilon_-$ and
		$\|R\|_\rho\lesssim A^2b^{1/2+\eta}$; the two weighted derivative bounds
		above show that the additional error in the scaled matrix is $O(A^2b^\eta)$.
		Since the column scaling has determinant one, in both cases
		\begin{equation}\label{proof-Jacobian-G}
			\det D_{(b,\beta)}(G_1,G_0)
			=-\frac{nm}{\beta}+O(b^{1/4}+A^2b^\eta)\ne0
		\end{equation}
		for fixed $A,\eta>0$ and sufficiently small $b$.
		For the unnormalized scalar products in the initial-data map, at $G=0$
		the determinant is multiplied by
		$\|\varphi_1\|_\rho^2\|\varphi_0\|_\rho^2=1+o(1)$,
		so it has the same nonzero leading term.
		
		\medskip
		\noindent\textbf{Acknowledgements.}
		SI gratefully acknowledges support from NSF CAREER DMS2442781 and NSF DMS2306528.
		YM acknowledges the support of JSPS KAKENHI Grant Numbers 23K25779, 24H00185, 24H00184, 25K00915, and 26K00611.
		
		
	\end{document}